\documentclass[11pt]{amsart}

\usepackage[utf8]{inputenc}
\usepackage{amsmath, amssymb,amsthm,tikz}
\usepackage{tikz-cd}
\usepackage{tikz}
\usepackage{mathrsfs}
\usepackage{extarrows}
\usepackage{graphicx}
\usepackage{thmtools}
\usepackage{mathtools}
\usepackage{hyperref}

\usepackage{multirow}
\usetikzlibrary{bending}
\usepackage{nicematrix}
\usepackage{adjustbox}
\usepackage[margin=1in]{geometry}

\usepackage{dynkin-diagrams}

\usepackage[english]{babel}

\hypersetup{
    colorlinks,
    citecolor=blue,
    filecolor=blue,
    linkcolor=blue,
    urlcolor=blue
}

\usepackage[toc,page]{appendix}

\usepackage{pdfsync}
\usepackage{amscd}
\usepackage{eucal}
\usepackage{epsfig}
\usepackage{verbatim}
\usepackage{xcolor}
\usepackage{multirow}
\usepackage{rotating}
\usepackage{chngcntr}
\usepackage{latexsym,enumerate,amsxtra}
\newcommand{\Z}{\mathbb{Z}}
\newcommand{\R}{\mathbb{R}}
\newcommand{\C}{\mathbb{C}}

\newcommand{\F}{\mathbb{F}}

\newcommand{\Fq}{{\mathbb{F}_q}}
\newcommand{\FF}{{\mathbb{F}}}
\newcommand{\Qp}{{\mathbb{Q}_p}}

\newcommand{\mcB}{\mathcal{B}}
\newcommand{\mcD}{\mathcal{D}}
\newcommand{\mcG}{\mathcal{G}}
\newcommand{\mcJ}{\mathcal{J}}

\newcommand{\bbG}{\mathbb{G}}

\newcommand{\bbbT}{{\underline{\mathbb{T}}}}

\newcommand{\bbT}{\mathbb{T}}

\newcommand{\mcH}{\mathcal{H}}
\newcommand{\mcI}{\mathcal{I}}

\newcommand{\fg}{\mathfrak{g}}

\newcommand{\bA}{\mathbf{A}}
\newcommand\IC{{\mathrm{IC}}}
\newcommand\ii{{\mathrm{i}}}

\newcommand{\Fr}{{\mathrm{Fr}}}
\newcommand{\St}{{\mathrm{St}}}
\newcommand{\Sym}{{\mathrm{Sym}}}
\newcommand{\vol}{{\mathrm{Vol}}}

\newcommand\Perv{{\mathrm{Perv}}}
\newcommand\Unip{{\mathrm{U}}}
\newcommand\rG{{\mathrm{G}}}
\newcommand\rO{{\mathrm{O}}}
\newcommand{\cus}{\mathrm{c}}
\newcommand{\Oo}{\mathcal{O}}

\newcommand{\End}{\operatorname{End}}

\newcommand{\Image}{\operatorname{Im}}

\newcommand{\Id}{\operatorname{Id}}
\newcommand{\sign}{\operatorname{sgn}}

\newcommand{\Ind}{\operatorname{Ind}}

\newcommand{\Frob}{\mathrm{Fr}}
\newcommand{\rU}{\mathrm{U}}
\newcommand{\sr}{\mathrm{sr}}
\newcommand{\lr}{\mathrm{lr}}

\newcommand{\bG}{\mathbf{G}}
\newcommand{\bM}{\mathbf{M}}
\newcommand{\bP}{\mathbf{P}}
\newcommand{\bS}{\mathbf{S}}
\newcommand{\bT}{\mathbf{T}}

\newcommand{\ad}{\mathrm{ad}}
\newcommand{\ab}{\operatorname{ab}}

\newcommand{\un}{\operatorname{un}}
\newcommand{\unip}{\mathrm{unip}}

\newcommand{\uni}{\mathrm{uni}}

\newcommand{\rank}{\operatorname{rank}}

\newcommand{\Hom}{\operatorname{Hom}}

\newcommand{\cInd}{\operatorname{c-Ind}}
\newcommand{\Cent}{\operatorname{Z}}

\newcommand{\Nor}{\operatorname{N}}
\newcommand{\fR}{\mathfrak{R}}
\newcommand{\fB}{\mathfrak{B}}

\newcommand{\fini}{\mathrm{f}}
\newcommand{\res}{\mathrm{res}}
\newcommand{\Irr}{\operatorname{Irr}}

\newcommand{\fa}{\mathfrak{a}}

\newcommand{\fo}{\mathfrak{o}}
\newcommand{\fp}{\mathfrak{p}}

\newcommand{\ft}{\mathfrak{t}}

\newcommand{\fX}{{\mathfrak{X}}}

\newcommand{\der}{{\mathrm{der}}}
\newcommand{\nr}{{\mathrm{nr}}}
\newcommand{\red}{{\mathrm{red}}}

\newcommand{\enh}{{\mathrm{e}}}

\newcommand{\bbbG}{{\underline{\mathbb{G}}}}
\newcommand{\bbbH}{{\underline{\mathbb{H}}}}

\newcommand{\rZ}{{\mathrm{Z}}}

\newcommand{\bbH}{{\mathbb{H}}}

\newcommand{\Ad}{\operatorname{Ad}}

\newcommand{\Gal}{\operatorname{Gal}}
\newcommand{\sep}{\operatorname{sep}}
\newcommand{\GL}{\mathrm{GL}}

\newcommand{\PGL}{\operatorname{PGL}}

\newcommand{\Sp}{\operatorname{Sp}}
\newcommand{\SO}{\operatorname{SO}}

\newcommand{\GSp}{\operatorname{GSp}}
\newcommand{\SU}{\operatorname{SU}}

\newcommand{\rN}{\mathrm{N}}
\newcommand{\rS}{\mathrm{S}}

\newcommand{\ra}{\rightarrow}

\newcommand{\cusp}{\mathrm{cusp}}

\renewcommand{\tilde}{\widetilde}

\newtheorem{thm}{Theorem}[subsection]
\newtheorem*{thm*}{Theorem}
\newtheorem{theorem}[thm]{Theorem}
\newtheorem{corollary}[thm]{Corollary}
\newtheorem{lemma}[thm]{Lemma}
\newtheorem{proposition}[thm]{Proposition}

\newtheorem{notation}[thm]{Notation}

\theoremstyle{definition}
\newtheorem{definition}[thm]{Definition}

\newtheorem{remark}[thm]{Remark}
\newtheorem{example}[thm]{Example}

\newtheorem*{acknow}{Acknowledgments}

\usepackage[all]{xy}

\newcommand{\matr}[1]{\left(\begin{smallmatrix}#1\end{smallmatrix}\right)}
\usepackage{pdfsync}
\newcommand{\LLC}{{\mathrm{LLC}}}

\newcommand{\ZZ}{{\mathbb{Z}}}
\newcommand{\RR}{{\mathbb{R}}}

\newcommand\fdeg{{\mathrm{fdeg}}}
\newcommand\Qlbar{{\overline{\mathbb{Q}}_\ell}}
\newcommand\rA{{\mathrm{A}}}
\newcommand\til{\tilde}
\newcommand\SL{{\mathrm{SL}}}
\newcommand\fs{{\mathfrak{s}}}
\newcommand\cA{{\mathcal{A}}}
\newcommand\cC{{\mathcal{C}}}
\newcommand\cE{{\mathcal{E}}}
\newcommand\cF{{\mathcal{F}}}
\newcommand\cG{{\mathcal{G}}}
\newcommand\cH{{\mathcal{H}}}

\newcommand\cL{{\mathcal{L}}}
\newcommand\cM{{\mathcal{M}}}

\newcommand\cT{{\mathcal{T}}}
\newcommand\cU{{\mathcal{U}}}

\newcommand\ri{{\rm{i}}}
\newcommand\rrr{{\rm{r}}}
\newcommand\orb{{\mathcal C}}

\newcommand\rI{{\mathrm{I}}}

\newcommand\superc{{\mathrm{scusp}}}
\newcommand\triv{{\mathrm{triv}}}

\newcommand\Sc{{\mathrm{Sc}}}
\newcommand\bB{{\mathbf B}}
\newcommand\bH{{\mathbf H}}
\newcommand\bL{{\mathbf L}}
\newcommand\bU{{\mathbf U}}
\newcommand\bZ{{\mathbf Z}}
\newcommand\bt{{\mathbf t}}

\newcommand\cB{{\mathcal B}}
\newcommand\cD{{\mathcal D}}

\newcommand\cP{{\mathcal P}}

\renewcommand{\arraystretch}{1.5}

\makeatletter
  \def\title@font{\normalsize\bfseries}
  \let\ltx@maketitle\@maketitle
  \def\@maketitle{\bgroup%
    \let\ltx@title\@title%
    \def\@title{\resizebox{\textwidth}{!}{%
      \mbox{\title@font\ltx@title}%
    }}%
    \ltx@maketitle%
  \egroup}
\makeatother

\begin{document}

\title[Local Langlands correspondence for $\rG_2$]{THE EXPLICIT LOCAL LANGLANDS CORRESPONDENCE FOR $\bG_2$}
\author[Anne-Marie Aubert and Yujie Xu]{Anne-Marie Aubert and Yujie Xu
}

\address{Sorbonne Universit\'e and Universit\'e Paris Cit\'e, CNRS,
IMJ-PRG, F-75005 Paris, France }
\email{anne-marie.aubert@imj-prg.fr}

\address{M.I.T., 77 Massachusetts Avenue,
Cambridge, MA, USA}
\email{yujiexu@mit.edu}

\numberwithin{equation}{subsection}
\newtheorem{prop}[equation]{Proposition}

\newtheorem{lem}[equation]{Lemma}
\newtheorem{Coro}[equation]{Corollary}

\newtheorem{property}[equation]{Property}
\theoremstyle{definition}
\newtheorem{Defn}[equation]{Definition}
\newtheorem{numberedparagraph}[equation]{}
\newtheorem{conj}[equation]{Conjecture}

\begin{abstract}
We develop a general strategy for constructing the explicit Local Langlands Correspondences for $p$-adic reductive groups via reduction to LLC for supercuspidal representations of proper Levi subgroups, using Hecke algebra techniques. 

As an example of our general strategy, we construct the explicit Local Langlands Correspondence for the exceptional group $\rG_2$ over a nonarchimedean local field, with explicit $L$-packets and explicit matching between the group and Galois sides. We also give a list of characterizing properties for our LLC. In \cite{G2-stability}, we complete unique characterization using stability property of our $L$-packets. For intermediate series, we build on our previous results on Hecke algebras.~For principal series, we improve previous works of Muic etc.~and obtain more explicit descriptions on both group and Galois sides. 

Moreover, we show the existence of non-unipotent \textit{singular} 
supercuspidal representations of $\rG_2$, and exhibit them in \textit{mixed} $L$-packets mixing supercuspidal representations with non-supercuspidal ones. 
Furthermore, our LLC satisfies a list of expected properties, including the compatibility with cuspidal support. 

\end{abstract}

\maketitle

\bigskip

\tableofcontents

\section{ Introduction} \label{sec:Introduction}
Let $F$ be a non-archimedean local field and $\bG$ a connected reductive algebraic group over $F$. Let $G^{\vee}$ be the group of $\C$-points of the reductive group whose root datum is the coroot datum of $\bG$. The Local Langlands Conjecture predicts a surjective map\footnote{To avoid overunning the margins, we use abbreviations ``irred.'' for ``irreducible'', ``repres.'' for ``representations'', ``iso.'' for ``isomorphism'', ``cont.'' for ``continuous'' and ``conj.'' for ``conjugacy''.} 
\[\begin{Bmatrix}\text{irred. smooth}\\
\text{repres. } \pi\text{ of }\bG(F)\end{Bmatrix}/\text{iso.}\;\longrightarrow
\; \begin{Bmatrix}L\text{-parameters
}
\\ \text{i.e.~cont. homomorphisms}
\\
\varphi_{\pi}\colon W_F\times\SL_2(\C)\to G^{\vee}\rtimes W_F\end{Bmatrix}/\text{$G^\vee$-conj.},\]
where $W_F$ is the Weil group of $F$. The fibers of this map, called \textit{$L$-packets}, are expected to be finite. 
In order to obtain a bijection between the group side and the Galois side, the above Conjecture was later \textit{enhanced}.
On the Galois side, one considers \textit{enhanced $L$-parameters}. 

\noindent Many cases of the Local Langlands Conjecture have been established, most notably: 
\begin{itemize}
    \item for $\GL_n(F)$: \cite{Harris-Taylor, Henniart-LLC-GLn,Scholze-LLC};
    \item for $\SL_n(F)$: \cite{Hiraga-Saito} for $\mathrm{char}(F)=0$ and \cite{ABPS-SL} for $\mathrm{char}(F)>0$;
    \item quasi-split classical groups for $F$ of characteristic zero: \cite{Arthur-classification,Moeglin} etc.
\end{itemize}
    For classical groups, the main methods in literature are either (1) to classify representations of these groups in terms of representations of the general linear groups via twisted endoscopy, and to compare the stabilized twisted trace formula on the general linear group side and the stabilized (twisted) trace formula on the classical group side, or (2) to use the theta correspondence. 
    
    In this article, we take a completely different approach--from the above existing literature--to the construction of explicit Local Langlands Correspondences for $p$-adic reductive groups via reduction to LLC for supercuspidal representations of proper Levi subgroups. 
    We apply this general strategy and construct explicit Local Langlands Correspondence for the exceptional group $\rG_2$ over an arbitrary non-archimedean local field of residual characteristic $\neq 2,3$, with explicit $L$-packets and explicit matching between the group and Galois sides. Our methods are inspired by previous works such as \cite{Moussaoui-Bernstein-center,AMS18,AMS2,FOS-Jussieu,Solleveld-LLC-unipotent}.

    More precisely, we use a combination of the Langlands-Shahidi method, (extended affine) Hecke algebra techniques, Kazhdan-Lusztig theory and generalized Springer correspondence--in particular, the AMS Conjecture on cuspidal support \cite[Conjecture 7.8]{AMS18}. For \textit{intermediate series}, i.e.~Bernstein series with supercuspidal support ``in between'' a torus and $G$ itself (in our case, the supercuspidal support has to lie in a Levi subgroup isomorphic to $\GL_2$), we use our previous result on Hecke algebra isomorphisms and local Langlands correspondence for Bernstein series obtained in \cite{aubert-xu-Hecke-algebra}, which builds on the work of many others such as \cite{Bernstein-centre,Solleveld-endomorphism-algebra,Solleveld-Lparameters,Heiermann-intertwining-operators-Hecke-algebras}. For principal series (i.e.~Bernstein series with supercuspidal support in a torus), we improve on previous work of Muic's \cite{Muic-G2} and obtain more explicit description on the group side; and we use \cite{Roche-principal-series, Reeder-isogeny,ABPS-KTheory,ABP-G2,Ram} to match the group and Galois sides.

    For supercuspidal representations, we make explicit the theory of \cite{Kal-reg,Kaletha-nonsingular} for the non-singular supercuspidal representations and their $L$-packets. 
    For 
    \textit{singular}\footnote{which we define to be simply the ones that are \textit{not non-singular in the sense of \cite{Kaletha-nonsingular}}} supercuspidal representations, which are 
    not covered in \textit{loc.cit.}
    , we use \cite[Conjecture 7.8]{AMS18} (see Property \ref{property:AMS-conjecture-7.8}) to exhibit them in \textit{mixed} $L$-packets with non-supercuspidal representations. These mixed $L$-packets are drastically different from the supercuspidal $L$-packets of \cite{Kal-reg,Kaletha-nonsingular}.

Furthermore, our LLC satisfies several expected properties, including the expectation that $\Irr(S_{\varphi})$ parametrizes the internal structure of the $L$-packet $\Pi_{\varphi}(G)$, where $S_{\varphi}$ is the component group of the centralizer of the (image of the) $L$-parameter $\varphi$. In \cite{G2-stability}, we complete the unique characterization of our LLC using a strenghened version of \textit{atomic stability} for our $L$-packets. We achieved this by computing coefficients of local character expansions of Harish-Chandra characters building on methods in \cite{harish-chandra-local-character,Debacker-Sally-germs,Barbasch-Moy-LCE}.

To illustrate the general nature of our strategy for constructing LLC, we refer the reader to \cite{LLC-GSp4} where Suzuki and the second author recently constructed the explicit LLC for $\GSp_4$ and $\Sp_4$ and gave a unique characterization in terms of stability and other properties.

\subsection{Main results}
We now state our main results. Let $\Irr^{\fs}(G)$ be the Bernstein series attached to the inertial class $\fs=[L,\sigma]$ (for more details, see \eqref{eqn:Bernstein decomposition}). Let $\Phi_e(G)$ denote the set of $G^\vee$-conjugacy classes of enhanced $L$-parameters for $G$. Let $\Phi_e^{\fs^{\vee}}(G)\subset \Phi_e(G)$ be the Bernstein series on the Galois side, whose cuspidal support lies in $\fs^{\vee}=[L^{\vee},(\varphi_{\sigma},\rho_{\sigma})]$, i.e.~the image under LLC for $L$ of $\fs$ (for more details, see $\mathsection$\ref{subsec: Galois Bernstein}). 
For any $\fs=[L,\sigma]_G\in\fB(G)$, the LLC for $L$ given by $\sigma\mapsto(\varphi_\sigma,\rho_\sigma)$ is expected to induce a bijection (see \cite[Conjecture~2]{AMS18} and Conjecture~\ref{conj:matching}):
\begin{equation}\label{Bernstein-block-isom-intro}
    \mathrm{Irr}^{\fs}(G)\xrightarrow{1--1}\Phi_\enh^{\fs^{\vee}}(G). 
\end{equation}
For the group $\rG_2$, by \cite[Main Theorem]{aubert-xu-Hecke-algebra}, we have such a bijection \eqref{Bernstein-block-isom-intro} for each Bernstein series $\mathrm{Irr}^{\fs}(G)$ of \textit{intermediate series}. On the other hand, the analogous bijection to \eqref{Bernstein-block-isom-intro} holds for \textit{principal series} Bernstein blocks thanks to \cite{Roche-principal-series,Reeder-isogeny,ABPS-KTheory, AMS18}. 

Let $G=\rG_2(F)$ and $p\neq 2,3$. Combined with the detailed analysis in all of $\mathsection$\ref{sec:Galois-G2} through $\mathsection$\ref{sec:ppal-series}, we construct an explicit Local Langlands Correspondence
\begin{equation} \label{main-thm-bijection}
\begin{split}
   \LLC\colon \mathrm{Irr}(G)&\xrightarrow{1\text{-}1}\Phi_\enh(G)\\
    \pi &\mapsto (\varphi_{\pi},\rho_{\pi}),
\end{split}\end{equation}
and obtain the following result (see Theorem~\ref{main-thm}). 
\begin{theorem} \label{main-thm-summary} 
The explicit Local Langlands Correspondence  ~\eqref{main-thm-bijection} verifies\\ $\Pi_{\varphi_\pi}(G)\xrightarrow{1-1}\Irr(S_{\varphi_\pi})$ for any $\pi\in \Irr(G)$, and satisfies \eqref{Bernstein-block-isom-intro} for any $\fs\in\fB(G)$,
where $\fs^\vee=[L^\vee,(\varphi_\sigma,\rho_\sigma)]_{G^\vee}$, as well as
a list of properties, including stability, (see \S\ref{subsec:expected properties}) that uniquely determine it. 
\end{theorem}
In other words,  
\begin{enumerate}
    \item to each explicitly described $\pi\in\Irr(G)$, we attach an explicit $L$-parameter $\varphi_{\pi}$ and determine its enhancement $\rho_{\pi}$ explicitly; 
    \item to each $\varphi\in \Phi(G)$, we describe (the shape of) 
    its $L$-packet $\Pi_{\varphi}(G)$, and give an internal parametrization in terms of $\rho\in\Irr(S_{\varphi})$;
    \item Moreover, for non-supercuspidal representations, we specify the precise parabolic induced representation that it occurs in. 
\end{enumerate}

We now comment on several works related to our paper. 

A few weeks after our paper first appeared, Gan and Savin had also claimed a construction of a  
local Langlands correspondence using completely different methods.~For the convenience of the reader, we briefly comment on the differences between our work and theirs: 
\begin{enumerate}
    \item Our LLC is completely \textit{explicit}, with explicit tables, explicit Kazhdan-Lusztig triples, explicit information on the $L$-packets and their packet members, etc. 
    In particular, \textit{explicit} computations of unipotent classes (in Kazhdan-Lusztig triples) such as the ones in our paper have found many important applications in number theory, such as modularity lifting theorems (see for example \cite[Appendix A]{LLC-GSp4}). 
    \item Our construction and proofs of LLC are \textit{purely local}, whereas Gan-Savin's construction relies on the global construction of Kret-Shin \cite{Kret-Shin}. As such, our construction and proofs work for cases even in the lack of global constructions; yet on the other hand, due to the explicit nature of our construction, it is not hard to check local-global compatibility when, indeed, global constructions are made. 
    \item It is still unclear whether Gan-Savin's construction gives \textit{the} LLC nor the right unique characterization. Contrary to what is claimed, Gan-Savin's construction also involves many choices. 
    For example, Gan-Savin's construction relies on Bin Xu's construction of $L$-packets for $\mathrm{PGSp}_6$, which are not known to be canonical. Moreover, Gan-Savin's characterization does not establish stability.  
     \item On the other hand, our construction is compatible with that of 
    \cite{Kal-reg,Kaletha-LLC}, for which stability is known \cite{Fintzen-Kaletha-Spice}. Moreover, in \cite{G2-stability}, we use stability to uniquely pin down our $L$-packets.
    \item Our strategy for constructing the LLC works for arbitrary reductive groups and does not rely on the coincidental existences of theta correspondences. e.g. in \cite{LLC-GSp4}, we carry out our general strategy for $\GSp_4$ and $\Sp_4$.  
    \item Furthermore, our methods provide explicit local Langlands correspondence in a uniform way, independent of the characteristic of $F$.
\end{enumerate}
We also note the construction of an LLC for generic supercuspidal representations for $\rG_2$ in \cite{Harris-Khare-Thorne} using global methods and theta correspondences. 
It would be very interesting to compare our construction with that of \textit{loc. cit.}. 

Moreover, as mentioned earlier, our work complements that of \cite{Kal-reg,Kaletha-nonsingular} as our constructions cover even the non-supercuspidal representations, and specifically the $L$-packets that \textit{mix} supercuspidal representations with non-supercuspidal ones, i.e.~the $L$-packets that do not appear in \textit{loc. cit.}. Note that such ``mixed'' $L$-packets necessarily consist of \textit{singular} supercuspidal representations not addressed in the constructions of \textit{loc. cit.}. 

Lastly, it would also be interesting to compare our construction with those of \cite{Zhu-LLC}, \cite{Fargues-Scholze} and \cite{Genestier-Lafforgue}.

\bigbreak\noindent
{\sc Notation.} 
Let $F$ be a non-archimedean local  field. Let $\fo_F$ denote the ring of integers of $F$, $\fp_F$ the maximal ideal in $\fo_F$ and $k_F:=\fo_F/\fp_F$ the residue field of $F$. The group of units in $\fo_F$ will be denoted $\fo_F^\times$.
We assume that $k_F$ is finite and denote by  $q=q_F$ its cardinality. Let $\nu_F:=|\!|\;|\!|_F$ denote the normalized absolute value of $F$. 

We fix a separable closure $F_{\sep}$ of $F$ and denote by $W_F\subset\Gal(F_{\sep}/F)$ the absolute Weil group of $F$.  Let $I_F$ be the inertia group of $F$, and $P_F$ the wild inertia group (i.e., the maximal pro-$p$ open normal subgroup of $I_F$).
We denote by $F_{\nr}$ the maximal unramified extension of $F$ inside $F_{\sep}$ and by $\Frob_F$ the element of $\Gal(F_{\nr}/F)$ that induces the automorphism $\Frob\colon a\mapsto a^q$ on the residue field $\overline k_F$ of $F_{\nr}$. Then $W_F=I_F\rtimes\langle\Frob_F\rangle$. Let $W_F^{\der}$ denote the closure of the commutator subgroup of $W_F$, and write $W_F^{\ab}=W_F/W^{\der}_F$. We take  $W_F':=W_F\times\SL_2(\C)$ for the Weil-Deligne group of $F$. 

For $n$ a positive integer, let $\mu_n$ denote the group on $n$th roots of unity and let $\zeta_n$ be a primitive $n$th root of the unity.

For $b \in F^\times/F^{\times 2}$, let $\rU_b(1,1)$ denote the quasi-split unitary group, and $\rU_b(2)$ the compact unitary group in two variables in $F (\sqrt{b})$. 

For $\bH$ any reductive algebraic group, we denote by $\bH^\circ$ the identity component of $\bH$ and by $\rZ_\bH$ the center of $\bH$. Let $h$ be an element of $\bH$. We denote by $\Cent_{\bH}(h)$ the centralizer of $h$ in $\bH$ and by
$A_\bH(h):=\Cent_{\bH}(h)/\Cent_{\bH}(h)^\circ$ the component group of $\Cent_{\bH}(h)$. 

Let $\bG$ be a connected reductive algebraic group defined over $F$, let $\bG_{\ad}=\bG/\rZ_\bG$ be the adjoint group of $\bG$ and let $\bA_\bG$ denote the maximal split central subtorus of $\bG$. Let $W_\bG=W_\bG(\bT)$ denote the Weyl group of $\bG$ with respect to a maximal torus $\bT$. By a Levi subgroup of we mean an $F$-subgroup $\bL$ of $\bG$ which is a Levi factor of a parabolic $F$-subgroup of $\bG$. We denote by $G$, $A$, $T$, $L$ the groups of $F$-rational points of $\bG$, $\bA$, $\bT$, $\bL$, respectively. Let $X^*(L):= X^*(\bL)_F$ be the group of all rational characters $\chi\colon \bL\to\GL_1$.

Let $\mcB(\bG,F)$ denote the (enlarged) building of $G$. For $x$ a point in $\mcB(\bG,F)$, let $G_x$ denote the subgroup of $G$ fixing $x$, and  $G_{x,0}\subset G_x$ the associated parahoric subgroup. Let $G_{x,0^+}$ denote the pro-$p$ unipotent radical of $G_{x,0}$, and $\bbG_{x,0}$ the reductive quotient of $G_{x,0}$. If $\tau$ is a representation of $\bbG_{x,0}$ we denote by $\boldsymbol{\tau}$ its inflation to $G_{x,0}$. 

\smallskip

\begin{acknow}
The authors would like to thank Dan Ciubotaru, Stephen DeBacker, Tom Haines, Michael Harris, Tasho Kaletha, Dipendra Prasad and David Schwein for interesting conversations. We would also like to thank Dick Gross, Guy Henniart, Herv\'e Jacquet, David Kazhdan, Mark Kisin, George Lusztig, Barry Mazur, Dinakar Ramakrishnan, Sug Woo Shin, Jack Thorne, Akshay Venkatesh, David Vogan, Zhiwei Yun and Wei Zhang for their interest in the paper. We thank Chuan Qin and Kenta Suzuki for useful remarks on a previous version of our paper.

Our special thanks go to Cheng-Chiang Tsai, who had originally provided the content of section \ref{sec:Galois-G2} as an appendix and has now very generously offered to simply let us include it in the main text. We cannot thank him enough for his generosity, his brilliant mathematical input and numerous illuminating conversations. 

Parts of this project were done while Y.X.~was supported by a graduate student fellowship at Harvard University, and while she was supported by NSF grant DMS~2202677 at MIT. She heartily thanks both institutions for providing wonderful working environments. 
\end{acknow}

\section{ Review on enhanced \texorpdfstring{$L$}{L}-parameters} \label{sec:Lp}
Let $\bG$ be a connected reductive algebraic group defined over $F$, and $G:=\bG(F)$ the group of $F$-rational points of $\bG$. 
The main purpose of this section is to recall in \S\ref{subsec:expected properties} a list of properties that are expected to be satisfied by the local Langlands correspondence for $G$. In later sections, we will use these properties to build an explicit local Langlands correspondence for the exceptional group $\rG_2$.

\subsection{Definitions and Examples} \label{subsec:eLp defs}
Let $G^\vee$ be the complex reductive group with root datum dual to that of $G$. We suppose that the group $\bG$ is $F$-split. 

A Langlands parameter (or $L$-parameter) for $G$ is a continuous morphism
$\varphi\colon W_F\times\SL_2(\C)\to G^\vee$ such that 
$\varphi(w)$ is semisimple for each $w\in W_F$, and that the restriction $\varphi|_{\SL_2(\C)}$ of $\varphi$ to $\SL_2(\C)$ is a morphism of complex algebraic groups. Since $G$ is $F$-split, the Weil group $W_F$ acts trivially on $G^\vee$. Hence the $L$-group of $G$ is the direct product $G^\vee\times W_F$ and $L$-parameters may be taken with values in $G^\vee$. 
The group $G^\vee$ acts on the set of $L$-parameters, and we denote by $\Phi(G)$ the set of $G^\vee$-conjugacy classes of $L$-parameters for $G$.

\begin{definition} \label{defn:infinitesimal parameter} {\rm \cite{Vogan-LLC-1993}}
The \textit{infinitesimal parameter} of an $L$-parameter $\varphi$ for $G$ is a morphism $\lambda_\varphi\colon W_F\to G^\vee$ defined by
\begin{equation} \label{eqn:infinitesimal parameter}
\lambda_\varphi(w):=\varphi\left(w,\left(\begin{smallmatrix}|\!|w|\!|^{1/2}&0\\ 0&|\!|w|\!|^{-1/2}\end{smallmatrix}\right)\right)\quad\text{for any $w\in W_F$.}
\end{equation}
\end{definition}

Obviously, if $\varphi$ has trivial restriction to $\SL_2(\C)$, then it coincides with its infinitesimal parameter.

\begin{definition} \label{defn:Lparameters}
An $L$-parameter $\varphi$ (resp.~its infinitesimal parameter $\lambda_\varphi$) is 
\begin{enumerate}
    \item[{\rm (1)}] \textit{bounded} if $\varphi(W_F)$ is bounded;
    \item[{\rm (2)}] \textit{discrete} if $\varphi(W_F')$ (resp.~$\lambda_\varphi(W_F)$) is not contained in any proper Levi subgroup of $G^\vee$;
    \item[{\rm (3)}] \textit{supercuspidal} if it is discrete and has trivial restriction to $\SL_2(\C)$. 
\end{enumerate} 
\end{definition}

\begin{remark} \label{rem:discretness}
{\rm An $L$-parameter is supercuspidal if and only if its  infinitesimal parameter is supercuspidal (since they coincide). On the other hand, if $\lambda_\varphi$ is discrete, then $\varphi$ is also discrete. Indeed, if $\varphi$ is a non-discrete $L$-parameter, then $\varphi(W_F')\subset L^\vee$ for some proper Levi subgroup $L^\vee$ of $G^\vee$. Then,  \eqref{eqn:infinitesimal parameter} implies that $\lambda_\varphi(W_F)$ is contained in $L^\vee$, thus $\lambda_\varphi$ is also non-discrete. However, in general there exist discrete $L$-parameters with non-discrete infinitesimal parameters.} 
\end{remark}

Let $\Irr(G)$ be the set of isomorphism classes of irreducible smooth representations of $G$. 
A supercuspidal irreducible representation $\pi$ of $G$ is compact modulo center, and so a fortiori, it is square integrable modulo center. Let $\Irr_{\superc}(G)$ denote the subset of $\Irr(G)$ formed by the isomorphism classes of supercuspidal irreducible representations of $G$.
The Langlands correspondence for $G$ is expected to partition $\Irr(G)$ into finite sets, called $L$-packets, indexed by $\Phi(G)$. If $\pi\in\Irr(G)$, and $\varphi_\pi$ denotes the $L$-parameter of $\pi$, we denote by $\Pi_\varphi(G)$ the $L$-packet attached to $\varphi$:
\begin{equation} \label{eqn:L-packet}
\Pi_\varphi(G):=\left\{\pi\in\Irr(G)\,:\,\varphi_\pi=\varphi\right\}.
\end{equation}
In order to both characterize the images of supercuspidal representations under the LLC and
parametrize the elements of every $L$-packet, we need to \textit{enhance the $L$-parameter} by a representation of a certain finite group (as explained in Definition~\ref{defn:cuspidal-enhanced-L-parameter}). In order to to do this, we start by recalling in \S\ref{subsubsec:cuspidal-pairs} the notion of \textit{cuspidal unipotent pair} introduced by Lusztig in \cite[Definition~2.4]{Lusztig-IC}.

\subsubsection{Cuspidal unipotent pairs} \label{subsubsec:cuspidal-pairs}
Let $\cG$ be a complex Lie group, $u$ a unipotent element in $\cG$ and $\rho$ an irreducible representation of $A_{\cG}(u)$, which is the group of component of the centralizer of $u$ in $\cG$. The pair $(u,\rho)$ is called \textit{cuspidal} if it determines a $\cG$-equivariant cuspidal local system on the $\cG$-conjugacy class of $u$ in the sense of Lusztig \cite{Lusztig-IC}. We denote by $\fB(\Unip_\enh(\cG)$ the set of $\cG$-conjugacy classes of cuspidal unipotent pairs. The pairs $(u,\rho)$ such that $\rho$ is cuspidal are called \textit{cuspidal unipotent pairs}.  If $(u,\rho)$ is cuspidal, then $u$ is a distinguished unipotent element in $\cG$ (that is, $u$ does not meet the unipotent variety of  any proper Levi subgroup  of $\cG$), see \cite[Proposition~2.8]{Lusztig-IC}. However, in general, there exist distinguished unipotent elements whose conjugacy classes do not support cuspidal local systems.

\begin{example} \label{ex: Gvee torus} {\rm
The case $\cG=T^\vee$:
The pair $(T^\vee,\triv)$ is the unique cuspidal unipotent pair in $T^\vee$. }
\end{example}

For $\GL_n(\C)$ and $\SL_n(\C)$, the unipotent conjugacy classes are completely described by the sizes of the Jordan blocks of the elements in these classes. Therefore, the unipotent conjugacy classes correspond to partitions of $n$. We denote by $\orb_\nu$ (or just $\nu$ for abbreviation) the unipotent class corresponding to the partition $\nu$: it consists of unipotent matrices with Jordan blocks of sizes equal to the parts of the partition $\nu$. For example, the trivial class is parametrized by the partition $(1^n):=(1,1,\ldots,1)$, and the regular unipotent class corresponds to the partition $(n)$. Note that the centralizer of a unipotent element of $\GL_n(\C)$ is always connected. 

\begin{example} \label{ex:GLn} {\rm
The pair $((1),\triv)$ is a  cuspidal unipotent pair in $\GL_1(\C)$, and there is no cuspidal unipotent pair in $\GL_n(\C)$ for $n\ge 2$. Any Levi subgroup of $\GL_n(\C)$ is isomorphic to a product of the form $\GL_{n_1}(\C)\times\GL_{n_2}(\C)\times\cdots\times\GL_{n_r}(\C)$ with $n_1+n_2+\cdots+n_r=n$, thus admits a cuspidal pair only if $n_1=n_2=\cdots=n_r=1$, i.e. is a maximal torus $\cT$. The pair $((1^n),\triv)$ is a cuspidal unipotent pair in $\cT$. }
\end{example}

\begin{example} \label{ex:SLn} {\rm
The cuspidal unipotent pairs in $\SL_n(\C)$ are the pairs $(\orb_{(n)},\rho)$, where $\rho$ is an order-$n$ character of $\mu_n$ (see \cite[(10.3.2)]{Lusztig-IC}).}
\end{example}

\begin{example} {\rm
A Levi subgroup of $\SL_n(\C)$ is isomorphic to a product of the form $\rS(\GL_{n_1}(\C)\times\GL_{n_2}(\C)\times \cdots\times\GL_{n_r}(\C))$, i.e.~the subgroup of $\GL_{n_1}(\C)\times\GL_{n_2}(\C)\times\cdots\times\GL_{n_r}(\C)$ consisting of elements with determinant equal to $1$. We will denote it as $\cM_{n_1,\ldots,n_r}$. 

The case where $n=3$ will be of special interest to us, as we have $\SL_3$ as a pseudo-Levi subgroup of $G_{2,\C}$. Denote by $\rho[\zeta_3]$ and $\rho[\zeta_3^2]$ the two order-$3$ characters of $\mu_3$. Thus the cuspidal unipotent pairs in $\cM_{(3)}=\SL_3(\C)$ are $(\orb_{(3)},\rho[\zeta_3])$ and $(\orb_{(3)},\rho[\zeta_3^2])$.
The group $\cM_{2,1}$ has no cuspidal unipotent pair (see the proof of \cite[(10.3.2)]{Lusztig-IC}), and the unique cuspidal unipotent pair in the maximal torus $\cM_{(1^3)}$ of $\cG$ is $((1^3),\triv)$. }
\end{example}

\subsubsection{\texorpdfstring{$L$}{L}-packets} \label{subsec:Lpackets}
We shall use the following notations: 
\begin{equation} \label{eqn:Gphis}
\rZ_{G^\vee}(\varphi):=\Cent_{G^\vee}(\varphi(W'_F))\quad\text{and}\quad\cG_\varphi:=\Cent_{G^\vee}(\varphi(W_F)).
\end{equation} 
We also consider the following component groups
\begin{equation} \label{eqn:A-phi}
A_\varphi:=\rZ_{G^\vee}(\varphi)/\rZ_{G^\vee}(\varphi)^\circ\quad\text{and}\quad
S_\varphi:=\rZ_{G^\vee}(\varphi)/\rZ_{G^\vee}\cdot\rZ_{G^\vee}(\varphi)^\circ.
\end{equation}
We recall that $A_{\cG_\varphi}(u_\varphi)$ denotes the component group of $\rZ_{\cG_\varphi}(u_\varphi)$. By \cite[\S~3.1]{Moussaoui-Bernstein-center},  
\begin{equation} \label{eqn:Moussaoui-iso}
A_\varphi\simeq A_{\cG_\varphi}(u_\varphi), \text{where $u_\varphi:=\varphi\left(1,\left(\begin{smallmatrix}
1&1\cr
0&1
\end{smallmatrix}\right)\right)$.}
\end{equation}
\begin{remark} {\rm 
In the case where $\bG=\rG_2$, we have $\rZ_{G^\vee}=\{1\}$. Thus $S_\varphi=\Cent_{G^\vee}(\varphi)/\Cent_{G^\vee}(\varphi)^\circ=A_\varphi$.}
\end{remark}
\begin{definition} \label{defn:enhanced-L-parameter}
An \textit{enhancement} of $\varphi$ is an irreducible representation $\rho$ of $S_\varphi$.  The pairs $(\varphi,\rho)$ are called \textit{enhanced $L$-parameters} for $G$. 
\end{definition}
The group $G^\vee$ acts on the set of enhanced $L$-parameters in the following way
\begin{equation} \label{eqn:Gvee-action}
g\cdot (\varphi,\rho) = (g \varphi g^{-1},g \cdot \rho).
\end{equation}
We denote by $\Phi_\enh(G)$ the set of $G^\vee$-conjugacy classes of enhanced $L$-parameters.

\begin{definition} \label{defn:cuspidal-enhanced-L-parameter}
An enhanced $L$-parameter $(\varphi,\rho)\in\Phi_\enh(G)$ is called \textit{cuspidal} if $\varphi$ is discrete and $(u_\varphi,\rho)$ is a cuspidal unipotent pair in $\cG_\varphi$.
\end{definition}
We denote by $\Phi_{\enh,\cusp}(G)$ the subset of $\Phi_{\enh}(G)$ consisting of  $G^\vee$-conjugacy classes of cuspidal enhanced $L$-parameters.

\begin{example} \label{ex:SO4} {\rm
The proper Levi subgroups of $\SO_4(\C)$ are isomorphic to one of the following groups: $\GL_2(\C)$, $\GL_1(\C)\times\SO_2(\C)$, $\GL_1(\C)\times\GL_1(\C)$.

The group $\SO_4(\C)$ has a unique cuspidal unipotent pair: $((3,1),\rho)$ (see \cite[Corollary~13.4]{Lusztig-IC}). The group $\GL_1(\C)\times\SO_2(\C)$ has no cuspidal unipotent pair, and $((1^2),\triv)$ is the unique cuspidal unipotent pair in $\GL_1(\C)\times\GL_1(\C)$.}
\end{example}

The construction of enhanced $L$-parameters $(\varphi,\rho)$ for $G$ is based on \textit{the generalized Springer correspondence} for the group $\cG_{\varphi}$, which we now recall in \S\ref{subsubsection-generalized-Springer}.

\subsection{Complex groups} \label{subsec:complex groups}
In this subsection, we recall several results involving complex reductive groups. 
Let $\cG$ be a connected reductive group over $\C$, and let $\Unip(\cG)$ denote the unipotent variety of $\cG$. 

\subsubsection{Generalized Springer Correspondence} \label{subsubsection-generalized-Springer} 
Let $D^b_\cG (\Unip(\cG))$ be the constructible $\cG$-equivariant derived category on $\Unip(\cG)$, and $\Perv_\cG(\Unip(\cG))$ its subcategory of $\cG$-equivariant perverse sheaves. We denote by $\Unip_\enh(\cG)$ the set of $\cG$-conjugacy classes of pairs
$(u,\rho)$, with $u\in \cG$ unipotent and  $\rho\in\Irr(A_{\cG}(u))$, where $A_{\cG}(u):=\Cent_{\cG}(u)/\Cent_{\cG}(u)^\circ$. The elements of $\Unip_\enh(\cG)$ are called \textit{enhanced unipotent classes}.

Let $\IC(\cC,\cE)$ be the Deligne-Goresky-MacPherson intersection cohomology complex (see \cite{BBD}, \cite{Goresky-MacPherson} or \cite[(0.1)]{Lusztig-IC}) of the closure of $\cC$ with coefficients in $\cE$.
The simple objects in $\Perv_{\cG}(\Unip(\cG))$ are the $\IC(\cC,\cE)$, where $\cC$ is a unipotent class in $\cG$ and $\cE$ is an irreducible $\cG$-equivariant $\Qlbar$-local system on $\cC$. We recall that there is a canonical bijection $\rho\mapsto\cE_\rho$ between $\Irr(A_{\cG}(u))$, where $u\in\cC$, and the set of isomorphism classes of irreducible $\cG$-equivariant local systems on $\cC$.

Let $\cP=\cL\cU$ be a parabolic subgroup of $\cG$, with Levi factor $\cL$ and unipotent radical $\cU$. 
By a \textit{$\cP$-resolution} of an algebraic variety $X$, we mean a variety $Y$ endowed with a free $\cP$-action and a smooth $\cP$-equivariant morphism $Y\to X$.
From \cite[\S3.7]{Bernstein-Lunts}, we recall the integration functor 
\begin{equation} \label{eqn:integration-functor}
\gamma_\cP^\cG\colon D^b_\cP (\Unip(\cG))\to D^b_\cG (\Unip(\cG)),
\end{equation} 
given by: for any object $A$ of $D^b_\cP (\Unip(\cG))$ and $Y$ a $\cG$-resolution of $\Unip(\cG)$,
\begin{equation}
(\gamma_\cP^\cG A)(Y):=(q_Y)_!A(Y)[2 \dim \cG/\cP],
\end{equation}
where $q_Y\colon \cP\backslash Y\to \cG\backslash Y$ is the quotient functor and $A(Y)$ is defined by regarding $Y$ as a $\cP$-resolution of $\Unip(\cG)$. Let 
\begin{equation}
m\colon\Unip(\cP)\hookrightarrow \Unip(\cG)
\quad\text{and}\quad p\colon\Unip(\cP)\to\Unip(\cL)
\end{equation} 
denote the inclusion and projection, respectively. Then the parabolic induction functor is the functor 
\begin{equation}
    \rI_{\cL\subset \cP}^{\cG}:=\gamma_{\cP}^{\cG}\,\circ\,m_!\,\circ\,p^*\,\colon \Perv_{\cL}(\Unip(\cL))\to\Perv_{\cG}(\Unip(\cG)).
\end{equation}
If $\cF_L$ is a simple object in $\Perv_{\cL}(\Unip(\cL))$, then $\ri_{\cL\subset \cP}^{\cG}(\cF_{\cL})$ is semisimple. A simple object $\cF$ in $\Perv_{\cG}(\Unip(\cG))$ is called \emph{cuspidal} if for any simple object $\cF_{\cL}$ in $\Perv_{\cL}(\Unip(\cL))$, $\cF$ does not occur in $\ri_{\cL\subset \cP}^{\cG}(\cF_{\cL})$ (equivalently, if 
	$\rrr_{\cL\subset \cP}^{\cG}(\cF)=0$) for any proper parabolic subgroup $\cP$ of $\cG$ with Levi factor $\cL$.~The functor $\ri_{\cL\subset \cP}^{\cG}$ is left adjoint to $\rrr_{\cL\subset \cP}^{\cG}:=p_!\,\circ\,m^*$.

Let $\cF_\rho:=\IC(\cC,\cE_\rho)$, where $(\cC,\rho)\in\Unip_\enh(\cG)$.
Then $\cF_\rho$ occurs as a summand of $\ri_{\cL\subset\cP}^\cG(\IC(\cC_{\cusp},\cE_{\cusp}))$, for some quadruple $(\cP,\cL,\cC_{\cusp},\cE_{\cusp})$, where $\cP$ is a parabolic subgroup of $\cG$ with Levi subgroup $\cL$ and $(\cC_{\cusp},\cE_{\cusp})$ is a cuspidal enhanced unipotent class in $\cL$ (see \cite[\S~6.2]{Lusztig-IC}). Moreover, the triple $(\cP,\cL,\cC_{\cusp},\cE_{\cusp})$ is unique up to $\cG$-conjugation 
(see \cite[Proposition~6.3]{Lusztig-IC}). We denote by $\epsilon:=\rho_{\cE_{\cusp}}$ the (equivalence class of) irreducible representation of $A_{\cG}(u)$ which corresponds to  $\cE_{\cusp}$, and by
$\ft:=(\cL,(\cC_{\cusp},\epsilon))_{\cG}$ the $\cG$-conjugacy class of 
$(\cL,(\cC_{\cusp},\epsilon))$. We call $\ft$ the \textit{cuspidal support} of  $(\cC,\rho)$ and denote by $\fB(\Unip_\enh(\cG))$ the set of cuspidal supports for the group $\cG$. Then the \textit{cuspidal support map} for $\Unip_\enh(\cG)$ is defined to be the map
\begin{equation} \label{eqn:SCUe}
\Sc_{\cG}\colon\Unip_\enh(\cG)\to \fB(\Unip_\enh(\cG)),
\end{equation}
which sends the $\cG$-conjugacy class of $(\cC,\rho)$ 
to its cuspidal support $\ft=(\cL,(\cC_{\cusp},\epsilon))_{\cG}$. For simplicity, we often refer to the unipotent class $\cC_{\cusp}$ in $\cL$ by a unipotent element $v$ in it.

By \cite[Theorem~9.2]{Lusztig-IC}, the fiber of \eqref{eqn:SCUe} is in bijection with $\Irr(W_\ft)$, where $W_{\ft}:=\Nor_{\cG}(\cL)/\cL$ is a finite Weyl group, i.e.
\begin{equation} \label{eqn:generalized-Springer-correspondence}
    \Sc_{\cG}^{-1}(\ft)\simeq\Irr(W_\ft)\quad\text{for any cuspidal support $\ft=(\cL,(\orb_{\cusp},\epsilon))_{\cG}$.}
\end{equation}

\subsection{Cuspidal support map of enhanced \texorpdfstring{$L$}{L}-parameters} \label{subsec:cuspidal support}
Let $(\varphi,\rho)$ be an enhanced $L$-parameter for $G$. Recall that $u_\varphi:=\varphi\left(1,\left(\begin{smallmatrix} 1&1\cr 0&1
\end{smallmatrix}\right)\right)$. Then $u_\varphi$ is a unipotent element of the (possibly disconnected) complex reductive group $\cG_\varphi$ defined in \eqref{eqn:Gphis}, and $\rho\in\Irr(A_{\cG_\varphi}(u_\varphi))$ by \eqref{eqn:Moussaoui-iso}. Let $\ft_\varphi:=(\cL^\varphi,(v^\varphi,\epsilon^\varphi))$ denote the cuspidal support of $(u_\varphi,\rho)$, i.e.~
\begin{equation} \label{eqn:Sc u,rho}
(\cL^\varphi,(v^\varphi,\epsilon^\varphi)):=\Sc_{\cG_\varphi}(u_\varphi,\rho).
\end{equation}
In particular, $(v^\varphi,\epsilon^\varphi)$ is a cuspidal unipotent pair in $\cL^{\varphi}$. 

Upon conjugating $\varphi$ with a suitable element of $\rZ_{\cG_\varphi^\circ}(u_\varphi)$, we may assume that the identity component of $\cL^\varphi$ contains $\varphi\left(\left(1, \left(\begin{smallmatrix} z&0\cr 0&z^{-1}
\end{smallmatrix}\right)\right)\right)$ for all $z\in\C^\times$. 
Recall that by the Jacobson–Morozov theorem (see for example \cite[\S~5.3]{Carter-book}), any unipotent element $v$ of $\cL^\varphi$ can be extended to a homomorphism of algebraic groups
\begin{equation} \label{eqn:JM}
j_v\colon \SL_2(\C)\to \cL^\varphi \text{ satisfying }j_v\left(\begin{smallmatrix}1&1\\0&1\end{smallmatrix}\right)=v.
\end{equation}
Moreover, by \cite[Theorem~3.6]{Kostant}, this extension is unique up to conjugation in $\Cent_{\cL^\varphi}(v)^\circ$. We shall call a homomorphism $j_v$ satisfying these conditions to be \textit{adapted to $\varphi$}.

By \cite[Lemma~7.6]{AMS18}, up to $G^\vee$-conjugacy, there exists a unique homomorphism $j_{v}\colon \SL_2(\C)\to \cL^{\varphi}$ which is adapted to $\varphi$, and moreover, the cocharacter
\begin{equation} \label{eqn:cocharacter}
\chi_{\varphi,v}\colon z\mapsto \varphi\left(1, \left(\begin{smallmatrix} z&0\cr 0&z^{-1}
\end{smallmatrix}\right)\right)\cdot j_v\left(\begin{smallmatrix} z^{-1}&0\cr 0&z
\end{smallmatrix}\right)
\end{equation}
has image in $\rZ_{\cL^\varphi}^\circ$. 
We define an $L$-parameter $\varphi_{v} \colon W_F \times\SL_2(\C) \to\rZ_{G^\vee}(\rZ_{\cL^\varphi}^\circ)$ by
\begin{equation} \label{eqn:varphi_v}
\varphi_v(w,x):= \varphi(w,1)\cdot\chi_{\varphi,v}(|\!|w|\!|^{1/2})\cdot j_v(x)\quad\text{for any $w\in W_F$ and any $x\in \SL_2(\C)$.}
\end{equation}
\begin{remark} \label{remark:infinitesimal parameter}
{\rm Let $w\in W_F$ and $x_w:=\left(\begin{smallmatrix}|\!|w|\!|^{1/2}&0\\ 0&|\!|w|\!|^{-1/2}\end{smallmatrix}\right)$. By \eqref{eqn:infinitesimal parameter}, we have
\begin{equation}
\begin{matrix}
&\lambda_{\varphi_v}(w)=\varphi_v(w,x_w)=
\varphi(w,1)\cdot\chi_{\varphi,v}(|\!|w|\!|^{1/2})\cdot j_v(x_w)\cr
&=\varphi(w,1)\cdot\varphi(1,x_w)\cdot j_v(x_w^{-1})\cdot j_v(x_w)=\varphi(w,x_w)=\lambda_\varphi(w).
\end{matrix}
\end{equation}
}\end{remark}
\begin{definition} \label{defn:cuspidal support} {\rm \cite[Definition~7.7]{AMS18}}
The \textit{cuspidal support} of $(\varphi,\rho)$ is 
\begin{equation} \label{eqn:Sc Galois side}
\Sc(\varphi,\rho):=(\rZ_{G^\vee}(\rZ_{\cL^\varphi}^\circ),(\varphi_{v^\varphi},\epsilon^\varphi)).
\end{equation}
\end{definition}
\subsection{Bernstein series of \texorpdfstring{$L$}{L}-enhanced paramters} \label{subsec: Galois Bernstein}
Let $L^\vee$ be the Langlands dual group of $L$ and $\iota_{L^\vee}\colon L^\vee \hookrightarrow G^\vee$ the canonical embedding. Define
\begin{equation} \label{eqn:Xnrdual}
\fX_\nr(L^\vee):=\left\{\zeta\colon W_F/I_F\to \rZ_{L^\vee}^\circ\right\}.
\end{equation}
There is a canonical bijection between $\fX_\nr(L^\vee)$ and $\fX_\nr(L)$ (see \cite[\S3.3.1]{Haines-Bst}). The group $\fX_\nr(L^\vee)$ acts on the set of cuspidal enhanced $L$-parameters for $L$ in the following way: Given $(\varphi,\rho)\in\Phi_\enh(L)$ and $\xi\in \fX_\nr({}^LL)$, we define $(\xi\varphi,\varrho)\in\Phi_\enh(L)$ by 
$\xi\varphi:=\varphi$ on $I_F\times\SL_2(\C)$ and $(\xi\varphi)(\Frob_F):=\tilde \xi \varphi(\Frob_F)$. Here $\tilde \xi\in \rZ_{L^\vee\rtimes I_F}^\circ$ represents $z$. We denote by $\fX_\nr(L^\vee).(\varphi_\sigma,\rho_\sigma)$ the orbit of $(\varphi_\sigma,\rho_\sigma)\in\Phi_\enh(L)$.
\begin{definition} \label{defn:sdual}
Let $\fs^{\vee}:=[L^\vee,(\varphi_{\sigma},\rho_\sigma)]_{G^{\vee}}$ be the $G^\vee$-conjugacy class of
\[(L^\vee,\fX_\nr(L^\vee).(\varphi_\sigma,\rho_\sigma)).\] Let $\fB(G^\vee)$ be the set of such $\fs^\vee$.
\end{definition}
We define 
\begin{equation}
\Nor_{G^\vee}(\fs^\vee_{L^\vee}):=\left\{n\in \Nor_{G^\vee}(L^\vee)\,:\, {}^n(\varphi_\cus,\varrho_\cus)\simeq (\varphi_\cus,\varrho_\cus)\otimes\chi^\vee\;\text{ for some $\chi^\vee\in\fX_\nr({}^LL)$}\right\},
\end{equation}
where $\fs^\vee_{L^\vee}=[L^\vee,(\varphi_{\sigma},\rho_\sigma)]_{G^{\vee}}\in\fB(L^\vee)$ and denote by $W^{\fs^\vee}_{G^\vee}$ the extended finite Weyl group 
$W(L^\vee):=\Nor_{G^\vee}(\fs^\vee_{L^\vee})/L^\vee$.

By \cite[(115)]{AMS18}, the set $\Phi_\enh(G)$ is partitioned into \textit{series \`a la Bernstein} as
\begin{equation} \label{eqn:decPhi_e}
\Phi_\enh(G)=\prod_{\fs^\vee\in\fB(G^\vee)}\Phi(G)^{\fs^\vee},
\end{equation}
where for each $\fs^\vee$, the subset $\Phi_\enh(G)^{\fs^\vee}$
is defined to be the fiber of $\fs^\vee$ under the supercuspidal map $\Sc$ defined in \eqref{eqn:Sc Galois side}. 
\subsection{Explicit Construction of enhanced \texorpdfstring{$L$}{L}-parameters}\label{subsec:enhanced L-parameters}
We give a construction of enhanced $L$-parameters for the irreducible constituents of any parabolically induced representations from any supercuspidal pair $(L,\sigma)_G$ for which the ``crude local Langlands correspondence'' is established. We emphasize that it does not require knowledge of the internal structure of the corresponding $L$-packets for $L$, but only the definition of $\varphi_\sigma$. 
Our construction extends--and is inspired by--\cite{Kazhdan-Lusztig}, \cite[\S4.2]{Reeder-isogeny}, \cite[\S4.1]{ABPS-KTheory} and \cite[\S5]{ABPS-LMS}, which all treat principal series cases (i.e.~when $\bL$ is a torus).

Let $L$ be a Levi subgroup of $G$ and let $\sigma$ be an irreducible supercuspidal representation of $L$. Suppose that an $L$-parameter $\varphi_\sigma\colon W_F\times\SL_2(\C)\to L^\vee$ for $L$ has been constructed.
Then we define 
\begin{equation} \label{eqn:Hsigma}
t:=\varphi_\sigma(\Frob_F,1)\;\;\text{and}\;\;H^\vee_\sigma:=\Cent_{G^\vee}(\varphi_\sigma|_{I_F}).
\end{equation}
Let $u$ be a unipotent element of $H^\vee_\sigma$ such that $tut^{-1}=u^q$. Similar to \eqref{eqn:JM}, the unipotent element $u$ can be extended to a morphism
\begin{equation} \label{eqn:ju}
j_u\colon \SL_2(\C)\to H^\vee_\sigma \text{ satisfying }j_u\left(\begin{smallmatrix}1&1\\0&1\end{smallmatrix}\right)=u.
\end{equation}
\begin{definition} \label{defn:phisu}
We attach to the pair $(\sigma,u)$ an $L$-parameter $\varphi_{\sigma,u}\colon W_F\times\SL_2(\C)\to G^\vee$ defined by
\begin{equation}
\begin{split} 
\varphi_{\sigma,u} \left(w,\left(\begin{smallmatrix} 1&0\cr 0&1\end{smallmatrix}\right)\right)&:= \varphi_\sigma(w)\cdot j_u\left(\begin{smallmatrix}|\!|w|\!|^{1/2}&0\\0&|\!|w|\!|^{-1/2}\end{smallmatrix}\right)^{-1}, \quad\text{for any $w\in W_F$.}\cr
\varphi_{\sigma,u}\left(1,\left(\begin{smallmatrix} 1&1\cr 0&1\end{smallmatrix}\right)\right)&:=u.
\end{split}
\end{equation}
\end{definition}
We define the group
\begin{equation} \label{eqn:GLphisigma}
 \cG_{\sigma,u}:=\cG_{\varphi_{\sigma,u}}=\Cent_{G^\vee}(\varphi_{\sigma,u}(W_F)),
\end{equation}
and we observe that $H^\vee_\sigma=\cG_{\varphi_{\sigma,u}|_{I_F}}$.
\subsubsection{Loop groups} \label{subsubsection:loop groups}
\begin{proposition}\label{loop} {\rm \cite[Proposition~3.8]{Reeder-torsion}} Suppose $\cG$ is a connected reductive group over $\C$ and $\boldsymbol{\sigma}$ is a semisimple automorphism of $\cG$. Then $(\cG^{\boldsymbol{\sigma}})^\circ$ is a reductive quotient of the twisted loop group $L^{\theta}\cG$ where $\theta$ is a pinned automorphism of $\cG$ in the same outer class as $\boldsymbol{\sigma}$. Moreover $\cG^{\boldsymbol{\boldsymbol{\sigma}}}=(\cG^{\boldsymbol{\sigma}})^\circ$ is connected if $\cG$ is simply connected.
\end{proposition}
\begin{proof} By a classical result of Steinberg's \cite[p.51]{Steinberg-Memoir}, $\boldsymbol{\sigma}$ stabilizes a Borel subgroup and a maximal torus $\cT\subset\cB\subset\cG$. Thus we may write by \cite[Lemma~3.2]{Reeder-torsion}  $\boldsymbol{\sigma}=\boldsymbol{\theta}\circ c_t$, where $\boldsymbol{\theta}$ preserves some pinning for $(\cB,\cT)$, $t\in\cT^{\boldsymbol{\theta}}$ and $c_t$ is the conjugation by $t$. Write $t=s\cdot t_{\uni}$, where $s\in X_*(\cT^{\boldsymbol{\theta}})\otimes\RR_+$ and $t_{\uni}\in X_*(\cT^{\boldsymbol{\theta}})\otimes S^1$. By taking $\boldsymbol{\sigma}_{\uni}:=\boldsymbol{\theta}\circ c_{t_{\uni}}$, we have $(\cG^{\boldsymbol{\sigma}})^\circ=((\cG^{\boldsymbol{\sigma}_{\uni}})^\circ)^{c_s}$. 

Consider $\RR/\ZZ\cong S^1$ with which we take a lift $\til{t}_{\uni}\in X_*(\cT^{\theta})\otimes\R$ of $t_{\uni}$. Then $\til{t}_{\uni}$ corresponds to a point on some facet $\cF$ on the apartment in the building of $L^{\boldsymbol{\theta}}\cG
$ associated to the maximal untwisted loop (i.e. $\C((\varpi))$-split) torus $L^{\boldsymbol{\theta}}((\cT^{\boldsymbol{\theta}})^\circ)=L((\cT^{\boldsymbol{\theta}})^\circ)$. We have that $(\cG^{\boldsymbol{\sigma_{\uni}}})^\circ$ is the reductive quotient at this facet $\cF$. Consider $\RR\xrightarrow[\sim]{\exp}\RR_+$, and let $\til{s}\in X_*(\cT^{\boldsymbol{\theta}})\otimes\RR$ be the pullback of $s$. Let $\cF'$ be the facet containing $\cF+\epsilon\til{s}$ for all $0<\epsilon\ll1$. Since $(\cG^{\boldsymbol{\sigma}})^\circ=((\cG^{\boldsymbol{\sigma}_{\uni}})^\circ)^{c_s}$ is a Levi subgroup of $(\cG^{\boldsymbol{\sigma}_{\uni}})^\circ$, we have that $(\cG^{\boldsymbol{\sigma}})^\circ$ is the reductive quotient at $\cF'$. This proves the proposition except for the last statement, which is also a result of Steinberg's.
\end{proof}

\section{ Review on representations of \texorpdfstring{$p$}{p}-adic groups} \label{sec:padic groups}
Let $\fR(G)$ denote the category of all smooth complex representations of $G$.
This is an abelian category admitting arbitrary coproducts.

\subsection{Supercuspidal support} \label{subsec:supercuspidal support}
Let $\pi\in\Irr(G)$. There exists a parabolic subgroup $P=LU$ of $G$ and a supercuspidal irreducible representation $\sigma$ of $L$ such that $\pi$ embeds in $\ii_P^G\sigma$. If $P'=L'U'$ is a parabolic subgroup of $G$ and $\sigma'$ a supercuspidal irreducible representation of $L'$, then $\pi$ is isomorphic to a subquotient of $\ii_{P'}^G\sigma'$ if and only if there exists an element of $G$ conjugating $(L,\sigma)$ and $(L',\sigma')$. The $G$-conjugacy class $(L,\sigma)_G$ of $(L,\sigma)$ is called the \textit{supercuspidal support} of $\pi$. 
We denote by $\Sc$ the map defined by $\Sc(\pi):=(L,\sigma)_G$.

Two supercuspidal pairs $(L,\sigma_1)$ and $(L,\sigma_2)$ are $G$-conjugate if and only if $\sigma_1$ and $\sigma_2$ are in the same orbit under $W_G(L):=\Nor_G(L)/L$.

\subsection{Langlands classification} \label{subsec:Langlands classification}
Let $\bB$ be a Borel subgroup of $\bG$ defined over $F$, and let $\bT\subset\bB$ be a maximal $F$-torus in $\bG$. A parabolic subgroup $P$ of $G$, with Levi subgroup $L$ is said to be \textit{standard} if $P\supset \bB(F)$ and $L\supset\bT(F)$. Let $\fa^*_L:=\RR\otimes X^*(L)$. We denote by $\nu\mapsto \chi_\nu$ the isomorphism from $\fa^*_L$ to the group of positive real valued unramified characters of $L$ as defined in \cite[(2)]{Silberger-Zink}.

\begin{definition}
A \textit{standard triple} $(P,\pi,\nu)$ for $G$ consists of:
\begin{itemize}
\item $P$ a standard parabolic subgroup of $G$;
\item $\pi$ an irreducible tempered representation of the standard Levi subgroup $L$ of $P$;
\item $\nu\in\fa^*_L$ a real parameter which is \textit{regular} and \textit{positive}, i.e. properly contained in the positive chamber determined by $P$ (for more details, see for instance \cite[\S1.3]{Silberger-Zink}).
\end{itemize}
\end{definition}
Let $(P,\pi,\nu)$ be a standard triple for $G$. Let $\ii_P^G\colon \fR(L)\to \fR(G)$ be the normalized parabolic induction functor.~Then $\ii_P^G(\pi\otimes\chi_\nu)$ has a unique irreducible quotient
\begin{equation} \label{eqn:Langlands quotient}
J(P,\pi,\nu) := J(\ii_P^G(\pi\otimes\chi_\nu)),
\end{equation}
called the \textit{Langlands quotient} of $\ii_P^G(\pi\otimes\chi_\nu)$, 
and the map $J$ defines a bijection between the set of all standard triples for $G$ and $\Irr(G)$ (see \cite[Theorem~3.5]{Konno}).

Let $\bU$ denote the unipotent radical of the Borel subgroup $\bB$ and $U:=\bU(F)$. A character $\chi\colon U\to\C^\times$ is called \textit{generic} if the stabilizer of $\chi$ in $\bT(F)$ is exactly $\bZ_{\bG}(F)$. Such a pair ${\mathfrak w}:=(U,\chi)$ is called a \textit{Whittaker datum} for $G$. There are only finitely many $G$-conjugacy classes of Whittaker data for $G$, since the group $\bG_{\ad}(F)$ acts transitively on the set of these pairs.  

\begin{definition} \label{defn:generic}
A smooth irreducible representation $\pi$ of $G$ is called \textit{${\mathfrak w}$-generic} if 
$\Hom_G(\pi,\Ind_U^G\chi)\ne 0$. 
\end{definition}
\subsection{Bernstein series} \label{subsec:Bernstein}
Let $L$ be a Levi subgroup of a parabolic subgroup $P$ of $G$, and let $\fX_\nr(L)$ denote the group of unramified characters of $L$. Let $\sigma$ be an irreducible supercuspidal smooth representation of $L$. 
\begin{notation} \label{notation:inertial class}
We write:
\begin{itemize}
\item $(L,\sigma)_G$ for the $G$-conjugacy class of the pair $(L,\sigma)$;
\item $\fs:=[L,\sigma]_G$ for the $G$-conjugacy class of the pair $(L,\fX_\nr(L)\cdot\sigma)$. 
\end{itemize}
We denote by $\fB(G)$ the set of such classes $\fs$. We set $\fs_L:=[L,\sigma]_L$.
\end{notation}
Let $\fR^\fs(G)$ be the full subcategory of $\fR(G)$ whose objects are the representations $(\pi,V)$ such that every $G$-subquotient of $\pi$ is equivalent to a subquotient of a parabolically induced representation $\ii_P^G(\sigma')$, where $\ii_P^G$ is the functor of normalized parabolic induction  and $\sigma'\in\fX_\nr(M)\cdot\sigma$. 
The categories $\fR^\fs(G)$ are indecomposable and split the full smooth category $\fR(G)$ into a direct product (see \cite{Bernstein-centre}):
\begin{equation} \label{eqn:Bernstein decompositionI}
\fR(G)=\prod_{\fs\in\fB(G)}\fR^\fs(G).
\end{equation}
We denote by $\Irr^\fs(G)$ the classes of irreducible objects in $\fR^\fs(G)$ --i.e.~irreducible representations whose supercuspidal support lies in $\fs$-- and call $\Irr^\fs(G)$ the \textit{Bernstein series} attached to $\fs$. By \eqref{eqn:Bernstein decompositionI}, Bernstein series give a partition of the set $\Irr(G)$ of isomorphism classes of irreducible smooth irreducible representations of $G$:
\begin{equation} \label{eqn:Bernstein decomposition}
\Irr(G)=\bigsqcup_{\fs\in\fB(G)}\Irr^\fs(G).
\end{equation}
For $\fs=[L,\sigma]_G\in \fB(G)$, we define
\begin{equation} \label{eqn:NsI}
    \Nor_G(\fs_L):=\left\{g\in G\,:\, \text{${}^gL=L$ and ${}^g\sigma\simeq\chi\otimes\sigma$, for some $\chi\in\fX_\nr(L)$}\right\},
\end{equation}
where $\fs_L=[L,\sigma]_L\in \fB(L)$ and denote by $W_G^\fs$ the extended finite Weyl group $\Nor_G(\fs_L)/L$.
\subsection{Formal degrees.}
Let $\bA$ denote the maximal $F$-split torus of the center of $\bG$, and $A:=\bA(F)$. Let $\pi$ be an irreducible square-integrable representation of $G$ on a Hilbert space $V$. Let $\fdeg(\pi)\in \R_{>0}$ denote the formal degree of $\pi$. 
By definition, we have
\begin{equation} \label{eqn:defn formal degree}
    \int_{G/A}(\pi(g)v_1,v_1')\overline{(\pi(g)(v_2,v_2')} dg =\fdeg(\pi)^{-1}\cdot (v_1,v_1')\overline{(v_2,v_2')}
\end{equation}
holds for all $v_1,v_1',v_2,v_2'\in V$, where $(\cdot,\cdot)$ is a $\bG(F)$-invariant inner product on $V$. We remark that $\fdeg(\pi)$ depends on the choice of the Haar measure $dg$ on $G$. We normalize the measure $dg$ as follows. When $\bG$ splits over an unramified extension of $F$, by \cite[(6)]{FOS-Jussieu}, the volume of the parahoric subgroup $G_{x,0}$ is then
\begin{equation} \label{eqn:DR}
    \mathrm{Vol}(G_{x,0})
    =q^{-\dim(\bbG_{x,0})/2}|\bbG_{x,0}|_{p'}.
\end{equation}
\subsection{Unipotent representations} \label{subsec:unipotents}
Let $x\in\mcB(G,F)$ and let $\tau$ be an irreducible cuspidal unipotent representation of $\bbG_{x,0}$. Let $\boldsymbol{\tau}^\vee$ denote the contragredient of the inflation $\boldsymbol{\tau}$ of $\tau$ to $G_{x,0}$. Consider the algebra $\mcH(G,\boldsymbol{\tau})$ of locally constant locally supported $\End(\boldsymbol{\tau}^\vee)$-valued functions $f$ satisfying $f(p_1gp_2)=\boldsymbol{\tau}^\vee(p_1)f(g)\boldsymbol{\tau}^\vee(p_2)$ for any $p_1,p_2\in G_{x,0}$ and $g\in G$.

We have an anti-involution on $\mcH(G,\boldsymbol{\tau})$ given on the Iwahori-Matsumoto basis $(T_w)$ by $\overline{T_w}:=T_{w^{-1}}$. It induces a Hermitian form $h$ on $\mcH(G,\boldsymbol{\tau})$ defined by
$h(f_1,f_2):=(f_1\overline{f_2})(1)$. Let $\mcH^2(G,\boldsymbol{\tau})$ denote the corresponding completion. A finite dimensional simple left $\mcH(G,\boldsymbol{\tau})$-module $\pi^{\mcH}$ is called square-integrable if it may be realized as an $\mcH(G,\boldsymbol{\tau})$-submodule $M_{\pi^\mcH}$ of $\mcH^2(G,\boldsymbol{\tau})$. There is a positive 
real number $d(\pi^{\mcH})$, depending only on the isomorphism class of $\pi^{\mcH}$, such that
\begin{equation} \label{eqn:dpiH}
 h(f_1,f_2)\cdot\overline{h(g_1,g_2)}=d(\pi^\mcH)\cdot h(f_1\overline{g_1},f_2\overline{g_2})\quad \text{for any $f_1,f_2,g_1,g_2\in M_{\pi^{\mcH}}$.}   
\end{equation}
\begin{proposition} \label{prop: formal degrees for non-supercuspidals}
Let $\pi$ be a unipotent square-integrable irreducible representation of $G$ such that $\pi^\cH:=\Hom_{G_{x,0}}(\boldsymbol{\tau},\pi)\ne 0$. We have
\begin{equation} \label{eqn:formal degree square integrable}
\fdeg(\pi)=\frac{\dim\tau}{\mathrm{Vol}(G_{x,0})}\cdot d(\pi^\cH). 
\end{equation}
\end{proposition}
\begin{proof}
See \cite[Proposition 9.1]{Reeder-HA}.
\end{proof}

Generalizing his work with Kazhdan \cite{Kazhdan-Lusztig}, Lusztig proved in \cite{Lu-padicI} that, when $\bG$ is simple of adjoint type, the unipotent representations of $G$ are in bijective correspondence with $G^\vee$-conjugacy classes of triples $(t,u,\rho)$, where $s\in G^\vee$ is semisimple, $u$ is unipotent  such that $tut^{-1}=u^q$, and $\rho$ is the isomorphism class of an irreducible representation of the component group of the mutual centralizer in $G^\vee$ of $t$ and $u$, such that $\rho$ is trivial on the center of $G^\vee$.  This result was extended to an arbitrary group $G$ in \cite{FOS-Jussieu}.

One can replace $u$ by $n = {\mathrm{ln}}(u)$, a nilpotent element in the Lie algebra $\fg ^\vee$ of $G^\vee$ so that an indexing triple is $(t,n,\rho)$, where
$t$ is a semisimple element in $G^\vee$, $n$ a nilpotent element in $\fg ^\vee$ such that $\Ad(t)n = qn$,
and an irreducible representation $\rho$ of the component group $A(t,n):= \rZ_{G^\vee}(t, n)/\rZ_{G^\vee}(t,n)^\circ$, where $\rZ_{G^\vee}(s,n):= \rZ_{G^\vee}(t)\cap \rZ_{G^\vee}(n)$ and $\rZ_{G^\vee}(n)$ is taken with respect to the adjoint action of $G^\vee$ on $\fg^\vee$. We will use this form of the indexing triples in the examples in later sections.
\subsection{Supercuspidal representations} \label{subsec:supercuspidals}
In this section, we briefly recall the general construction of supercuspidal representations of $G=\bG(F)$, where $\bG$ is an arbitrary tamely ramified reductive group and the residual characteristic $p$ of $F$ does not divide the order of the Weyl group of $\bG$. We also describe different kinds of supercuspidal representations that will occur: regular (resp. non-singular) supercuspidal representations, depth-zero supercuspidal representations and  unipotent supercuspidal representations. 

\subsubsection{Deligne-Lusztig theory} \label{subsubsec:DL} 
Let $\bbbG$ be a connected reductive algebraic group defined over a finite field $\Fq$. We denote by $\bbbG^\vee$ a connected reductive algebraic group defined over $\Fq$ with root datum dual to that of $\bbbG$. 
Recall from classical Deligne-Lusztig theory \cite[\S10]{Deligne-Lusztig} and \cite[(8.4.4)]{Lusztig-characters-Princeton-book} that the set of equivalence classes of irreducible representations of $\bbbG(\Fq)$ decomposes into a disjoint union
\begin{equation} \label{eqn:Lusztig-series}
\Irr(\bbbG(\Fq))\xrightarrow{\sim}\coprod\limits_{(s)}\mathcal{E}(\bbbG(\Fq),s),
\end{equation}
where $(s)$ is the $\bbbG^{\vee}(\Fq)$-conjugacy class of a semisimple element $s$ of $\bbbG^{\vee}(\Fq)$, and $\mathcal{E}(\bbbG,s)$ is the Lusztig series defined as below in \eqref{defn-Deligne-Lusztig-series}. More precisely, the decomposition is obtained as follows: as proved in \cite[Corollary~7.7]{Deligne-Lusztig}, for any irreducible representation $\tau$ of $\bbbG(\Fq)$, there exists an $\Fq$-rational maximal torus $\bbbT$ of $\bbbG$ and a character $\theta$ of $\bbbT(\Fq)$ such that $\tau$ occurs in the Deligne-Lusztig (virtual) character $R_\bbbT^\bbbG(\theta)$, i.e.~such that
$\langle\tau, R_\bbbT^\bbbG(\theta)\rangle_{\bbbG(\Fq)}\ne 0$, where the character of $\tau$ is also denoted by $\tau$, and $\langle\;,\;\rangle_{\bbbG(\Fq)}$ is the usual scalar 
product on the space of class functions on $\bbbG(\Fq)$: 
\begin{equation} \label{eqn:scalar product}
\langle f_1,f_2\rangle_{\bbbG(\Fq)} = |\bbbG(\Fq)|^{-1}\sum_{g\in \bbbG(\Fq)}f_1(g)\,\overline{f_2(g)} .
\end{equation}
If $\theta=1$ (i.e.~the trivial character of $\bbbT(\Fq)$), then the representation $\tau$ is called \textit{unipotent}. 

If the pairs $(\bbbT,\theta)$ and $(\bbbT',\theta')$ are not $\bbbG(\Fq)$-conjugate, then $R_\bbbT^\bbbG(\theta)$ and $R_{\bbbT'}^\bbbG(\theta')$ are orthogonal to each other with respect to $\langle\,,\,\rangle_{\bbbG(\Fq)}$, but they may have a common constituent as they are virtual characters. This has motivated the introduction of the following weaker notion of conjugacy:
the pairs $(\bbbT,\theta)$ and $(\bbbT',\theta')$ are called \textit{geometrically conjugate} if there exists a $g\in\bbbG$, such that $\bbbT'={}^g\bbbT$ and such that for any non-negative integer $m$, we have 
$\theta'\circ\Nor_{\mathrm{Fr}^m|\mathrm{Fr}}=\theta\circ\Nor_{\mathrm{Fr}^m|\mathrm{Fr}}\circ\ad(g)$, 
where $\Fr\colon\bbbG\to\bbbG$ is the geometric Frobenius endomorphism associated to the $\Fq$-structure of $\bbbG$ (hence we have $\bbbG^{\Fr}=\bbbG(\Fq)$), and $\Nor_{\Fr^m|\Fr}\colon \bbbT^{\Fr^m}\to\bbbT^{\Fr}$ is the norm map given by $\Nor_{\Fr^m|\Fr}(t):=t\cdot\Fr(t)\cdot\Fr^2(t)\cdots\Fr^{m-1}(t)$.

The $\bbbG(\Fq)$-conjugacy classes of pairs $(\bbbT,\theta)$ as above are in one-to-one correspondence with the $\bbbG^\vee(\Fq)$-conjugacy classes of pairs $(\bbbT^\vee,s)$ where $s$ is a semisimple element of $\bbbG^\vee(\Fq)$, and $\bbbT^\vee$ is an $\Fq$-rational maximal torus of $\bbbG^\vee$ containing $s$.
Then the \textit{Lusztig series} $\mathcal{E}(\bbbG,s)$ is defined as
\begin{equation}\label{defn-Deligne-Lusztig-series}
  \{\tau\in\Irr(\bbG)\colon
  \text{$\tau$ occurs in $R_\bbbT^\bbbG(\theta)$ where $(\bbbT,\theta)_{\bbG}$ corresponds to $(\bbbT^\vee,s)_{\bbG^\vee}$}\}.
\end{equation}
By definition, $\mathcal{E}(\bbbG,1)$ consists only of unipotent representations. 

By the various works of Lusztig, there is a bijection
\begin{equation}\label{Lusztig-unipotent-decomposition}
    \begin{matrix}
    \mathcal{E}(\bbbG(\Fq),s)\xrightarrow{1--1} \mathcal{E}(\bbbG^\vee_s(\Fq),1)\cr
    \tau\mapsto\tau_\unip,
    \end{matrix}
\end{equation}
where $\bbbG^\vee_s:=\Cent_{\bbbG^\vee}(s)$ denotes the centralizer of $s$ in $\bbbG^\vee$, and hence plays the role of an ``endoscopic'' group of $\bbbG$:
\begin{enumerate}
\item[$\bullet$]
In the case where $\bbbG$ has connected center, the existence of the bijection~\eqref{Lusztig-unipotent-decomposition} was established in \cite[Theorem~4.23]{Lusztig-characters-Princeton-book}. 
\item[$\bullet$]
For an arbitrary $\bbbG$, the group $\bbbG^\vee_s$ may be disconnected, and one needs to first extend the notion of Deligne-Lusztig character to this case in the following way. Let $\bbbG_s^{\vee,\circ}$ be the identity component of $\bbbG_s^\vee$. For $\bbbT^\vee$ an $\Fq$-rational maximal torus of $\bbbG_s^{\vee,\circ}$ and $\theta^\vee$ a character of $\bbbT^\vee(\Fq)$, we define
\begin{equation}
R_{\bbbT^\vee}^{\bbbG^\vee_s}(\theta^\vee):=\Ind_{\bbbG^{\vee,\circ}_s(\Fq)}^{\bbG_s^\vee(\Fq)}(R_{\bbbT^{\vee}}^{\bbbG^{\vee,\circ}_s}(\theta^\vee)).
\end{equation}
Then $\mathcal{E}(\bbbG^\vee_s(\Fq),1)$ is defined as the set of irreducible coonstituents of $R_{\bbbT^\vee}^{\bbbG^\vee_s}(1)$ and the equivalence \eqref{Lusztig-unipotent-decomposition} was established in \cite[\S12]{Lusztig-disconnected-center}. 
\end{enumerate}
\begin{proposition} \label{property-Lusztig-DL-map}
The bijection \eqref{Lusztig-unipotent-decomposition} satisfies the following properties:
\begin{enumerate}
    \item[{\rm (1)}] It sends a Deligne-Lusztig character $R_\bbbT^\bbbG(\theta)$ in $\bbG$ (up to a sign) to a Deligne-Lusztig character $R_{\bbbT^{\vee}}^{\bbbG^\vee_s}(1)$, where $1$ denotes the trivial character of $\bbbT^{\vee}$ (see \cite[\S12]{Lusztig-disconnected-center}). 
    \item[{\rm (2)}] It preserves cuspidality in the following sense: \\
if $\tau\in\Irr(\bbbG(\Fq),s)$ is cuspidal, then
\begin{itemize}
    \item[{\rm (a)}]
if $s\in\bbbG^\vee(\Fq)$, then the largest $\Fq$-split torus in the center of $\bbbG^\vee_{s}$ coincides with the largest $\Fq$-split torus in the center of $\bbbG^\vee$ (see \cite[(8.4.5)]{Lusztig-characters-Princeton-book}),
    \item[{\rm (b)}] the unipotent representation $\tau_\unip$ is cuspidal.
 \end{itemize}
    \item[{\rm (3)}]
The dimension of $\tau$ is given by
\begin{equation} \label{eqn:dim tau}
\dim(\tau)=\frac{|\bbG_{x_0}|_{p'}}{|(\Cent_{\bbbG_{x_0}^\vee}(s))(\Fq)|_{p'}}\,\dim(\tau_\unip), 
\end{equation}
where $|\cdot|_{p'}$ is the largest prime-to-$p$ factor of $|\cdot|$ (see \cite[Remark~13.24]{Digne-Michel}).
\end{enumerate}
\end{proposition}
Let $W_\bbbG(\bbbT):=\Nor_\bbbG(\bbbT)/\bbbT$ and 
$\theta\in\Irr(\bbbT(\Fq))$. We define
\begin{equation}
W_\bbbG(\bbbT,\theta)^{\Fr}:=\{w\in W_\bbbG(\bbbT)^{\Fr}\,:\,{}^w\theta=\theta\}.
\end{equation}
By \cite[Theorem~6.8]{Deligne-Lusztig}, we have
\begin{equation} \label{eqn:DLproduct}
\langle R_\bbbT^\bbbG(\theta),R_\bbbT^\bbbG(\theta)\rangle_{\bbbG(\Fq)}=|W_\bbbG(\bbbT,\theta)^{\Fr}|,
\end{equation}
where $\langle \;,\;\rangle_{\bbbG(\Fq)}$ is as in \eqref{eqn:scalar product}. 
We recall that $\theta$ is said to be \textit{in general position} if $W_\bbbG(\bbbT,\theta)^{\Fr}=\{1\}$ (see \cite[Definition~5.15]{Deligne-Lusztig}).

\begin{lem}\label{torus-centralizer-lemma}
Let $\tau\in\cE(\bbbG(\Fq),s)$ be such that $\bbbG^\vee_s:=\Cent_{\bbbG^\vee}(s)$ is a torus. Then $\tau=\pm R_\bbbT^\bbbG(\theta)$, where $\theta$ is in general position. 
\end{lem}
\begin{proof} When $\bbbG^\vee_s:=\Cent_{\bbbG^\vee}(s)$ is a torus, we have $\tau_{\unip}=1=R_{\bbbT^\vee}^{\Cent_{\bbbG^\vee}(s)}(1)$. 
By Proposition~\ref{property-Lusztig-DL-map}(1), we obtain $\tau=\pm R_\bbbT^\bbbG(\theta)$. Therefore
\begin{equation}
\langle R_\bbbT^\bbbG(\theta),R_\bbbT^\bbbG(\theta)\rangle_{\bbG}=\langle \tau,\tau\rangle_\bbG=1.
\end{equation}
Then it follows from \eqref{eqn:DLproduct} that
$|W_\bbbG(\bbbT,\theta)^{\Fr}|=1$, i.e.~$\theta$ is in general position.
\end{proof}

\subsubsection{Depth-zero supercuspidal representations} \label{subsec:depth-zero supercuspidal}
Let $\bG$ be a connected reductive algebraic group defined over $F$ and $\bS$ an $F$-torus of $\bG$. Then $\bS$ possesses an ft-N\'eron model (a smooth group scheme over $\fo_F$ of finite type), the neutral connected component of it is called the \textit{connected N\'eron model} of $\bS$ and denoted by $\mathfrak{S}^\circ$ (see \cite[\S10]{Bosch-Lutkebohmer-Raynaud}). Let $S:=\bS(F)$ and $S_0:=\mathfrak{S}^\circ(\fo_F)$.

We suppose that the torus $\bS$ is a \textit{maximally unramified elliptic maximal torus} (as defined in \cite[Definition~3.4.2]{Kal-reg}), i.e.~such that $\bS=\Cent_\bG(\bT)$, where $\bT$ is the maximal unramified subtorus of $\bS$. Since $\bT$ is a maximal split torus over $F_\nr$, we have the apartment 
$\cA_\red(\bS,F_\nr)\subset \cB_\red(\bG,F_\nr)$. Since $\bS$ is defined over $F$ and elliptic, this apartment is Frobenius stable, and contains a unique Frobenius-fixed point, which we denote by $x$. Then $x$ is a vertex of $\cB_\red(\bG,F_\nr)$ by \cite[Lemma 3.4.3]{Kal-reg}. Therefore $G_{x,0}$ is a maximal parahoric subgroup of $G$. We denote by  $\bbG_{x,0}$ its reductive quotient, which is the group of $\Fq$-points of a connected reductive algebraic group $\bbbG_{x,0}$ over $\overline\FF_q$.

The special fiber of the (automatically connected) ft-Neron model of $\bT$ embeds canonically as an elliptic (i.e.~anisotropic modulo center) maximal torus $\bbbT$ of the reductive group $\bbbG_x$.~More explicitly, $\bbbT(k_{F'})\subset \bbbG_x(k_{F'})$ is the image in $\bG(F')_{x,0+}$ of $\bS(F')\cap\bG(F')_x$, or equivalently of
 $\bT(F')\cap\bG(F')_x$, for every unramified extension $F'$ of $F$ \cite[Lemma~2.2.1(3)]{DeBacker-parameterizing-conjugacy-classes}. Every elliptic maximal torus of $\bbbG_{x,0}$ arises in this way by \cite[Lemma~2.3.1]{DeBacker-parameterizing-conjugacy-classes}.

In \cite[Definition~5.15]{Deligne-Lusztig}, Deligne and Lusztig defined two regularity conditions for a character $\theta$ of $\bbbT(\Fq)$, which we now recall:
\begin{itemize}
    \item 
$\theta$ is said to be \textit{in general position} if its stabilizer in $(\Nor_\bbbG(\bbbT)/\bbbT)(\Fq)$ is trivial. 
\item $\theta$ is said to be \textit{non-singular} if it not orthogonal to any coroot.
\end{itemize} 
If the centre of $\bbbG$ is connected, $\theta$ is non-singular if and only if it is in general position by \cite[Proposition~5.16]{Deligne-Lusztig}.~The character $\theta$ is said to be \textit{regular} if its stabilizer in $(\Nor_\bG(\bS)/\bS)(F)$ is trivial (see \cite[Definition 3.4.16]{Kal-reg}).
If $\theta$ is regular, then it is in general position (see \cite[Fact 3.4.18]{Kal-reg}). 

\begin{definition} \cite[Definition 3.4.16]{Kal-reg} Let $\theta_F$ be a depth-zero character of $S$. Then $\theta_F$ is said to be \textit{regular} if its restriction to $S_0$ equals the inflation of a regular character of $\bbbT(\Fq)$.
\end{definition}

Let $\pi$ be an irreducible depth-zero supercuspidal representation of $G$. Then there exists a vertex $x\in\cB_\red(G,F)$ and an irreducible cuspidal representation $\tau$ of $\bbbG_x(\Fq)$, such that the restriction of $\pi$ to $G_{x,0}$ contains the inflation of $\tau$ (see \cite[\S1-2]{Morris-ENS} or  \cite[Proposition~6.6]{Moy-PrasadII}). The normalizer $\rN_{G}(G_{x,0})$ of $G_{x,0}$ in $G$ is a  totally disconnected group that is compact mod center, which, by \cite[proof of (5.2.8)]{Bruhat-Tits-II}, coincides with the fixator $G_{[x]}$ of $[x]$ under the action of $G$ on the reduced building of $\bG$. The $\pi$ is compactly induced from a representation of $\rN_G(G_{x,0})$: 
\begin{equation} \label{eqn:depth zero supercuspidal}
\pi=\cInd_{G_{[x]}}^G(\boldsymbol{\tau}).
\end{equation}

\begin{definition}\label{defn:regular depth-zero}\cite[Definition 3.4.19]{Kal-reg}~The representation $\pi$ is said to be \textit{regular} if $\tau=\pm R_\bbbT^{\bbbG_{x,0}}(\theta)$ for some pair $(\bbbT,\theta)$, where $\bbT$ is an elliptic maximal torus of $\bbbG_{x,0}$ and $\theta$ is a regular character of $\bbbT(\Fq)$. 
\end{definition}

Let $\pi$ be an irreducible depth-zero supercuspidal representation of $G$. Let $x\in\cB_\red(G,F)$ and $\tau\in\Irr_\cusp(\bbbG_x(\Fq))$ such that the restriction of $\pi$ to $G_{x,0}$ contains $\boldsymbol{\tau}$.
\begin{proposition} \label{prop:torus-centralizer}
If $\bG$ is simply connected, then $\pi$ is regular if and only if $\tau=\pm R_\bbbT^{\bbbG_{x,0}}(\theta)$, where $\bbT$ is an elliptic maximal torus of $\bbbG_{x,0}$ and $\theta$ is a character of $\bbbT(\Fq)$ that is in general position. \\
In particular, if $\tau\in\cE(\bbbG_{x,0}(\Fq),s)$ and $\Cent_{\bbbG_{x,0}^\vee}(s)$ is a torus, then $\pi$ is regular. 
\end{proposition}
\begin{proof}
By \cite[Lemma~3.4.10]{Kal-reg}, we have an exact sequence
\begin{equation} \label{eqn:Kal-exact}
1\to \Nor_{G_{x,0}}(\bS)/S_0\to \Nor_\bG(\bS)(F)/S\to G_x/[G_{x,0}: S]\to 1,
\end{equation}
and the natural map 
\begin{equation} \label{eqn:Kal-bij}
     \Nor_{G_{x,0}}(\bS)/S_0\longrightarrow \Nor_{\bbbG_x}(\bbbT)(\Fq)/\bbbT(\Fq)
\end{equation}
is bijective. 
Since $\bG$ is simply connected, the parahoric subgroup $G_{x,0}$ coincides with the fixator $G_x$ in $G$ of $x$ (see \cite[\S 5.2.9]{Bruhat-Tits-II}).
Combining \eqref{eqn:Kal-exact} and \eqref{eqn:Kal-bij}, we have 
\begin{equation} \label{eqn:normalizers}
\Nor_{G_x}(\bS)/S_0\simeq \Nor_\bG(\bS)(F)/S\simeq \Nor_{\bbbG_x}(\bbbT)(\Fq)/\bbbT(\Fq).
\end{equation}
Thus by \eqref{eqn:normalizers}, $\theta$ is regular if and only if it is in general position. 
The last assertion follows from Lemma~\ref{torus-centralizer-lemma}.
\end{proof}
Suppose $\bG$ is simply connected, then $G_{[x]}=G_{x,0}$ for any $x\in\mcB(\bG,F)$. Let $\pi\in\Irr(G)$ be a depth-zero supercuspidal representation of $G$. By \eqref{eqn:depth zero supercuspidal}, there is $\tau\in\Irr(\bbG_{x,0})$ such that 
\begin{equation} \label{eqn:iGxG}
\pi=\cInd_{G_{x,0}}^G\boldsymbol{\tau}=:\ii_{\bbG_{x,0}}^G\tau.
\end{equation}

\begin{proposition} \label{formal-degree-formula}
The formal degree of $\pi$ is given by
\begin{equation} \label{formal-degree-depth0-sc-formula}
    \fdeg(\pi)=\frac{q^{\dim(\mathbb{G}_{x,0})/2}\dim(\tau_{\unip})}{|(\Cent_{\bbbG_{x,0}}(s))^{\vee}(\Fq)|_{p'}}.
\end{equation}
\end{proposition}
\begin{proof}
We have (see for instance \cite[Lemma~18]{Schwein} or \cite[\S~5.3]{DeBacker-Reeder})
\begin{equation} \label{eqn:fdeg induced}
    \fdeg(\pi)=\frac{\dim\tau}{\mathrm{Vol}(G_{x,0})
}.
\end{equation}
Combining \eqref{eqn:dim tau} and \eqref{eqn:DR}, we have 
\begin{equation}
\fdeg(\pi)=\frac{|G_{x,0}|_{p'}\dim(\tau_{\unip})}{|(\Cent_{\bbbG_{x,0}^\vee}(s))(\Fq)|_{p'}\cdot q^{-\dim(\mathbb{G}_{x,0})/2}|G_{x,0}|_{p'}}=\frac{q^{\dim(\mathbb{G}_{x,0})/2}\dim(\tau_{\unip})}{|(\Cent_{\bbbG_{x,0}^\vee}(s))^{\vee}(\Fq)|_{p'}}.
\end{equation}
\end{proof}

\subsubsection{Positive-depth supercuspidal representations} \label{subsec:positive depth supercuspidal}
In this section, we assume that $\bG$ splits over a tamely ramified extension of $F$, and that $p$ does not divide the order of the Weyl group $W_{\bG}$ of $\bG$.
Let $E/F$ be a finite extension. By a twisted $E$-Levi subgroup of $\bG$, we mean an $F$-subgroup $\bG'$ of $\bG$ such that $\bG'\otimes_F E$ is a Levi subgroup of $\bG\otimes_F E$. If $E/F$ is tamely ramified, then $\bG'$ is called a tamely ramified twisted Levi subgroup of $\bG$.
A tamely ramified twisted Levi sequence in $\bG$ is a finite sequence $\vec \bG=(\bG^0,\bG^1,\cdots,\bG^d)$ of twisted $E$-Levi subgroups of $\bG$, with $E/F$ tamely ramified (see~\cite[p~586]{Yu}).~We recall that an anisotropic algebraic group over $F$ is a linear algebraic group that does not contain non-trivial $F$-split tori. 
\begin{definition} \label{defn:cuspidal datum}
A \textit{cuspidal $\bG$-datum} is a tuple $\cD:=(\vec \bG,y,\vec r, \pi^0,\vec\phi)$ consisting of 
\begin{enumerate}
\item a tamely ramified Levi sequence $\vec\bG=(\bG^0\subset\bG^1\subset\cdots\subset\bG^d=\bG)$ of twisted $E$-Levi subgroups of $\bG$, such that $\rZ_{\bG^0}/\rZ_{\bG}$ is anisotropic;
\item a point $y$ in $\cB(\bG^0,F)\cap\cA(\bT,E)$, whose projection to the reduced building of $G^0$ is a vertex, where $\bT$ is a maximal torus of $\bG^0$ (hence of $\bG^i$) that splits over $E$;
\item a sequence $\vec r=(r_0,r_1,\ldots,r_d)$ of real numbers such that $0< r_0<r_1<\cdots< r_{d-1}\le r_d$ if $d>0$, and $0\le r_0$ if $d=0$;
\item an irreducible depth-zero supercuspidal representation $\pi^0$ of $G^0$;
\item
 a sequence $\vec\phi=(\phi_0,\phi_1,\ldots,\phi_d)$ of characters, where $\phi_i$ is a character of $G^i$ which is trivial on $G^i_{y,r_i+}$ and nontrivial on $G^i_{y,r_i}$ for $0\le  i\le d-1$, such that
 \begin{itemize}
     \item $\phi_d$ is trivial on $G^d_{y,r_d+}$ and nontrivial on $G^d_{y,r_d}$ if $r_{d-1}< r_d$, and $\phi_d=1$ if $r_{d-1}=r_d$ (with $r_{-1}$ defined to be $0$). 
     \item Moreover, $\phi_i$ is $\bG^{i+1}$-generic of depth $r_i$ relative to $y$ in the sense of \cite[\S9]{Yu} for $0\le i\le d-1$.
 \end{itemize}
\end{enumerate}
\end{definition}

All supercuspidal representations arise from cuspidal $\bG$-data if $p$ does not divide the order of the Weyl group of $G$ (see \cite{Fintzen}). When $\bG=\rG_2$, this condition says that $p\neq 2,3$.

Let $\bS$ be a maximal torus in $\bG$ and  be a character of depth zero. 
As before, let $\bT$ denote the maximal unramified subtorus of $\bS$. Let $F'/F$ be an unramified extension splitting $\bT$ and $R_\res(\bG,\bT)$ denote the set of restrictions to $\bT$ of the absolute roots in $R(\bG,\bS)$. 

\begin{definition} \label{defn:F-ns} \cite[Definition~3.1.1]{Kaletha-nonsingular}
A depth-zero character $\theta_F\colon\bS(F)\to\C^\times$ is called \textit{$F$-non-singular} if for every $\alpha\in R_\res(\bG,\bT)$ the character
$\theta_F\circ\Nor_{F'|F}\circ\alpha^\vee\colon F'\to \C^\times$ has non-trivial restriction to $\fo_{F'}$.
\end{definition}

\begin{definition} \label{defn: Fnse} \cite[Definition~3.4.1]{Kaletha-nonsingular}
Let $\theta_F\colon S\to \C^\times$ be a character. The pair $(\bS,\theta_F)$ is said to be \textit{tame $F$-non-singular elliptic} in $\bG$ if
\begin{itemize}
\item $\bS$ is elliptic and its splitting extension $E/F$ is tame;
\item  Inside the connected reductive subgroup $\bG^0$ of $\bG$ with maximal torus $\bS$
and root system
\[R:= \{\alpha\in R(\bG,\bS)|\theta_F(\Nor(\alpha^\vee(E_{0+}^\times)))=1\},\]
the torus $\bS$ is maximally unramified.
\item The character $\theta_F$ is $F$-non-singular with respect to
$\bG^0$ in the sense of Definition~\ref{defn:F-ns}.
\end{itemize}
\end{definition}
In \cite{Kal-reg, Kaletha-nonsingular}, Kaletha describes how to construct  supercuspidal representations $\pi_{\bS,\theta_F}$ of $G$ from tame $F$-non-singular elliptic pairs $(\bS,\theta_F)$ in $\bG$. The representation $\pi_{\bS,\theta_F}$ is obtained in two steps. One starts by unfolding the pair $(\bS,\theta_F)$ into a cuspidal $\bG$-datum $\cD_{\bS,\theta_F}:=(\vec \bG,y,\vec r, \pi^0,\vec\phi)$ as in Definition~\ref{defn:cuspidal datum}. The properties of $\bS$ and $\theta_F$ provided by Definition~\ref{defn: Fnse} allow us to go to the reductive quotient and use the theory of Deligne-Lusztig cuspidal representations in order to construct $\pi^0$, the so-called depth-zero part of the datum $\cD_{\bS,\theta_F}$. The second step involves applying Yu's construction \cite{Yu} to the obtained $\bG$-datum. 

Since $p$ does not divide the order of the Weyl group of $\bG$, it does not divide the order of the fundamental group of $\bG_\der$.  

\begin{definition}\label{defn:Kaletha regular} \cite[Definition 3.7.9]{Kal-reg}
A supercuspidal representation of $G$ is said to be \textit{regular} if it arises via Yu's construction from a cuspidal $G$-datum $\cD$ such that the representation $\pi^0$ of $G^0$ is regular in the sense of Definition~\ref{defn:regular depth-zero}.
\end{definition}
Let $\bS_0\subset\bS$ be the connected component of the intersection of kernels of all elements of $R_{0+}$, and $R(\bG,\bS_0)$ be the image of $R(\bG,\bS)\backslash R_{0+}$ under the restriction map $X^*(\bS)\to X^*(\bS_0)$.

\begin{definition} \cite[Definition 4.1.4]{Kaletha-nonsingular} A \textit{torally wild supercuspidal $L$-packet datum} of $\bG$ is a tuple $(\bS,j^\vee,\chi_0,\theta_F)$, where
\begin{itemize}
    \item $\bS$ is a torus of dimension equal to the absolute rank of $\bG$, defined over $F$ and split over a tamely ramified extension of $F$;
\item $j^\vee\colon\bS^\vee\hookrightarrow\bG^\vee$ is an embedding of complex reductive groups whose $\bG^\vee$-conjugacy class is $\Gamma_F$-stable;
\item
$\chi_0=(\chi_{\alpha_0})_{\alpha_0}$ are tamely ramified $\chi$-data for $R(\bG,\bS_0)$;
\item $\theta_F\colon\bS(F)\to\C^\times$ is a character;
subject to the condition that $(\bS,\theta_F)$ is a tame $F$-non-singular elliptic pair. 
\end{itemize}
\end{definition}
Formal degrees of arbitrary-depth tame supercuspidal representations in the sense of \cite{Yu} can be computed as in \cite[Theorem A]{Schwein}. 
Let $\bG$ be a semisimple $F$-group, and let $\mathcal{D}$ be a cuspidal $\bG$-datum with associated supercuspidal representation $\pi$. Let $R_i$ denote the absolute root system of $\bG^i$, for the twisted Levi sequence $(\bG^i)_{0\leq i\leq d}$. Let $\exp_q(t):=q^t$.
\begin{proposition}
The formal degree of $\pi$ is given by
\begin{equation}
    \fdeg(\pi)=\frac{\dim\rho}{[G^0_{[y]}:G^0_{y,0+}]}\exp_q\left(\frac{1}{2}\dim G+\frac{1}{2}\dim \bbG^0_{y,0}+\frac{1}{2}\sum\limits_{i=0}^{d-1}r_i(|R_{i+1}|-|R_i|)\right).
\end{equation}
\end{proposition}

\begin{remark} \label{rem:formal degrees}
The Formal Degree Conjecture of \cite{Hiraga-Ichino-Ikeda}, which describes the formal degree $\fdeg(\pi)$ of any irreducible smooth representation $\pi$ of $G$ in terms of adjoint gamma factor, has been proved for regular supercuspidal representations in \cite[Theorem~B]{Schwein}, for non-singular supercuspidal representations in \cite[Theorem~9.2]{Ohara-fdegr}, and for unipotent supercuspidal representations in \cite[Theorem~3]{FOS}.  
\end{remark}

\section{ Preliminaries on \texorpdfstring{$\mathrm{G}_2$}{\mathrm{G}}} \label{sec:preliminariesG2}

Let $\bG=\rG_2$ be the exceptional group of type $\rG_2$ over $F$, and let $G=\rG_2(F)$ denote its group of $F$-rational points. Let $\bT$ be a maximal $F$-split torus in $\bG$ and 
$(\varepsilon_1,\varepsilon_2,\varepsilon_3)$ be the canonical
basis of $\RR^3$, equipped with the scalar product $(\;|\;)$ for
which this basis is orthonormal. Then
$\alpha:=\varepsilon_1-\varepsilon_2$ and
$\beta:=-2\varepsilon_1+\varepsilon_2+\varepsilon_3$ define a
basis of the set $R(\bG,\bT)$ of
roots of $\bG$ with respect to $\bT$, and
\begin{equation} \label{eqn:G2roots}
R(\bG,\bT)^+:=\left\{\alpha,\beta,\alpha+\beta,2\alpha+\beta,3\alpha+\beta,
3\alpha+2\beta\right\}
\end{equation}
is a subset of positive roots in $R(\bG,\bT)$
(see \cite[Planche IX]{Bou}). We have
\begin{equation} \label{products}
(\alpha|\alpha)=2,\;\;\; (\beta|\beta)=6,\;\;\; \text{ and
}\;\;(\alpha|\beta)=-3.
\end{equation}
The root $\alpha$ is a short root, and $\beta$ is a long root. 
\begin{figure}
\begin{center} 
\begin{tikzpicture}
   [ -{Straight Barb[bend,
       width=\the\dimexpr10\pgflinewidth\relax,
       length=\the\dimexpr12\pgflinewidth\relax]},]
    \foreach \i in {0, 1, ..., 5} {
      \draw[thick, black] (0, 0) -- (\i*60:2);
      \draw[thick, black] (0, 0) -- (30 + \i*60:3.5);
    }
    \draw[thin, red] (1.5, 0) arc[radius=1.5, start angle=0, end angle=5*30];
    \node[right] at (2, 0) {$\alpha$};
    \node[above left] at (45:2.4) {$2\alpha+\beta$};
    \node[above right] at (-48.5:-2.25) {$\alpha+\beta$};
    \node[above right] at (51:-3) {$-2\alpha-\beta$};
    \node[above right] at (-80:2.4) {$-\alpha-\beta$};
    \node[right] at (-2.8, 0) {$-\alpha$};
    \node[right] at (-0.9,3.9) {$3\alpha+2\beta$};
    \node[] at (-0.05, -3.6) {$-3\alpha-2\beta$};
    \node[above left, inner sep=.2em] at (5*30:3.5) {$\beta$};
    \node[above left, inner sep=.2em] at (5*30:-4.2) {$-\beta$};
    \node[above right, inner sep=.2em] at (-5*30:4.4) {$-3\alpha-\beta$};
    \node[above right, inner sep=.2em] at (-5*30:-3.5) {$3\alpha+\beta$};
    \node[right,red] at (15:1.5) {$5\pi/6$};
     \end{tikzpicture}
\caption{Diagram for $\rG_2(F)$}
\label{fig:G2-root-diagram}
\end{center}
\end{figure}
Let $\bB=\bT\bU$ be the corresponding Borel subgroup in $\bG$ and $\overline \bB=\bT \overline \bU$ for the opposite Borel subgroup. 
For any root $\gamma\in R(\bG,\bT)$, we denote by $\gamma^\vee$ the corresponding coroot and by $s_\gamma$ the fundamental reflection  in $W$ defined by $\gamma$:
\begin{equation}
    s_\gamma(\gamma'):=\gamma'-\langle\gamma',\gamma^\vee\rangle\gamma=\gamma'-\frac{2(\gamma'|\gamma)}{(\gamma|\gamma)}\gamma, \quad\text{for any $\gamma'\in R(\bG,\bT)$.}\end{equation}
We set $a:=s_\alpha$, $b:=s_\beta$ and $r:=ba$. 

The group $X^*(T)$ of rational characters of $T$ has the following description
\begin{equation} \label{XT}
X^*(T)=\ZZ(2\alpha+\beta)+\ZZ(\alpha+\beta).
\end{equation}
We identify $T$ with $F^{\times} \times F^{\times}$ by
\begin{equation} \label{etaa}
\begin{matrix}
  \eta_\alpha\colon& T \longrightarrow F^{\times} \times F^{\times}\hphantom{(2\alpha+\beta)(t),\alpha}\cr
&\bt \mapsto ( (2\alpha+\beta)(\bt), (\alpha+\beta)(\bt) ).
\end{matrix}
\end{equation}
Let $\eta'_\alpha\colon F^{\times} \times F^{\times}\to T$ be the inverse morphism of $\eta_\alpha$.  
We have
\begin{equation} \label{eqn:roots computation}
\begin{cases} \alpha (\eta'_\alpha(t_{1}, t_{2}))= t_{1}t_{2}^{-1},\ \ \beta (\eta'_\alpha
(t_{1}, t_{2}))=t_{1}^{-1}
t_{2}^{2} \\
a(\eta'_\alpha(t_{1}, t_{2}))=(t_{2}, t_{1}),\ \ b(\eta'_\alpha(t_{1}, t_{2})) =(t_{1},
t_{1}t_{2}^{-1})
\end{cases}.
\end{equation}
Any pair
$(\xi_1,\xi_2)$ of characters of $F^\times$ define a character of $T$ that we denote by $\xi_1\otimes\xi_2$.
For each root $\gamma\in R(\bG,\bT)$, we fix root group homomorphisms $x_\gamma\colon F\to G$ and
$\ZZ$-homomorphisms $\zeta_\gamma\colon\SL_2(F)\to G$ 
as in \cite[(6.1.3)~(b)]{Bruhat-Tits-I}.
We have
\begin{equation}x_\gamma(u)=\zeta_\gamma\left(\begin{matrix}
1&u\cr
0&1\end{matrix}\right),\;\;\;\;
x_{-\gamma}(u)=\zeta_{-\gamma}\left(\begin{matrix}
1&0\cr
u&1\end{matrix}\right)
\;\;\text{
and }\;\;
\gamma^\vee(t)=\zeta_\gamma\left(\begin{matrix}
t&0\cr
0&t^{-1}\end{matrix}\right).
\end{equation}
For $\gamma\in\{\alpha,\beta\}$, let $\bP_\gamma$ be the maximal standard parabolic subgroup of $\bG$ generated by $\gamma$, and
$\bM_\gamma$ be the centralizer of the image of $(\gamma')^\vee$ in $\bG$, where $\gamma'$ is
the unique positive root orthogonal to $\gamma$, i.e.
\begin{equation}
\gamma':=\begin{cases}
3\alpha+2\beta&\text{ if $\gamma=\alpha$,}\cr
3\alpha+\beta&\text{ if $\gamma=\beta$.}
\end{cases}
\end{equation}
Then $\bM_\gamma$ is a Levi factor for $\bP_\gamma$, and $\bM_\alpha$, $\bM_\beta$ are representatives for the two conjugacy classes of maximal Levi subgroups of $\bG$.

We extend $\zeta_\gamma\colon\SL_2(F)\to M_\gamma:=\bM_\gamma(F)$ to an isomorphism $\zeta_\gamma\colon\GL_2(F)\to M_\gamma$ via
\begin{equation}
    \zeta_\gamma\left(\left(\begin{matrix}t&0\cr 0&1\end{matrix}\right)
\right):=\zeta_{\gamma'}\left(\left(\begin{matrix}t&0\cr
0&t^{-1}\end{matrix}\right)\right),
\;\;\;\;\text{for $t\in\ F^\times$.}
\end{equation}
Then the restriction to $T$ of the inverse map of $\zeta_\gamma$ coincides with the isomorphism $
\eta_\gamma\colon T{\overset\sim\to} F^\times\times F^\times$,
where $\eta_\alpha$ is as in~(\ref{etaa}), while
\begin{equation} \label{eqn:eta_beta}
\begin{matrix}
  \eta_\beta\colon& T \longrightarrow F^{\times} \times F^{\times}\hphantom{(2\alpha+\beta)}\cr
&\bt \longmapsto ((\alpha+\beta)(\bt),\alpha(\bt)).
\end{matrix}
\end{equation}
It follows from \eqref{eqn:roots computation} that
\begin{equation} \label{eqn:change of realization of T}
    (\eta_\beta\circ \eta'_\alpha)(t_1,t_2)=(t_2,t_1t_2^{-1}),\quad \text{for $t_1,t_2\in F^\times$.}
\end{equation}
The long roots $\{\pm\beta,\pm(3\alpha+\beta),\pm(3\alpha+2\beta)\}$ form the root system of $\SL_3(F)$. Hence $\SL_3(F)$ embeds in $\rG_2(F)$ as the subgroup generated by the groups $U_\gamma$ with $\gamma$ long. 

The Weyl group $W$ of $\bG$ is generated by the reflections $s_\alpha$ and $s_\beta$ satisfying relations 
\begin{equation}
s_\alpha^2=s_\beta^2=1\quad\text{and}\quad (s_\alpha s_\beta)^6=1.
\end{equation}
Thus $W$ is the dihedral group of order $12$. Let $r(\vartheta)\in W$ denote the rotation through $\vartheta$ with
center at the origin. The elements of $W$ are described in Table~\ref{table:1}, along
with their actions on the character  $\xi_1\otimes\xi_2$ of $T
\cong F^{\times} \times F^{\times}$.

\begin{table}[ht]
\begin{tabular}{ |c|c|c| }
 \hline
$w$  & $w$ &  $w(\xi_1\otimes\xi_2)$ \\ [0.5ex] 
 \hline\hline
$1$&$1$&$\xi_1\otimes\xi_2$\\
$s_\alpha$&$a$&$\xi_2\otimes\xi_1$\\
$s_\beta$&$b$&$\xi_1\xi_2\otimes\xi^{-1}_2$ \\
$r(\pi/3)$&$ab$&$\xi^{-1}_2\otimes \xi_1\xi_2$\\
$r(5\pi/3)$&$ba$&$\xi_1\xi_2\otimes\xi^{-1}_1$\\
$s_{3\alpha+\beta}$&$aba$&$\xi^{-1}_1 \otimes \xi_1\xi_2$\\
$s_{\alpha+\beta}$&$bab$&$\xi_1\otimes \xi^{-1}_1\xi^{-1}_2$ \\
$r(2\pi/3)$&$abab$&$\xi^{-1}_1\xi^{-1}_2\otimes \xi_1$\\
$r(4\pi/3)$&$baba$&$\xi_2 \otimes \xi^{-1}_1\xi^{-1}_2$\\
$s_{2\alpha+\beta}$&$ababa$&$\xi^{-1}_1\xi^{-1}_2 \otimes\xi_2$ \\
$s_{3\alpha+2\beta}$&$babab$&$\xi^{-1}_2\otimes \xi^{-1}_1$ \\
$r(\pi)$&$bababa$&$\xi^{-1}_1\otimes \xi^{-1}_2$\\
 \hline
 \end{tabular}
   \vskip0.2cm
     \caption{ \label{table:1} $\qquad$ $\qquad$ $\quad$}
\end{table}

The group $\bG$ is simply connected. The nodes of its extended Dynkin diagram are $\delta:=-3\alpha-2\beta$ (the opposite of the highest root), $\alpha$ and $\beta$ (in particular $\delta$ is the extended node, and $\alpha$, $\beta$ are respectively short and long roots):
\begin{center}
    \includegraphics[scale=0.62]{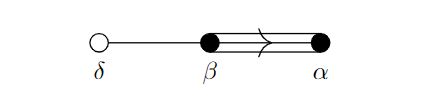}
\end{center}
The affine Weyl group of $G$ is generated by $s_\delta$, $s_\alpha$ and $s_\beta$.
There are three $G$-conjugacy classes of maximal parahoric subgroups of $G$, obtained by deleting one node from the extended Dynkin diagram. We denote by $G_{x_0}$ (deleting node $\delta$), $G_{x_1}$ (deleting node $\alpha$) and $G_{x_2}$ (deleting node $\beta$) a set of representatives of these classes. Their reductive quotients are  $\bbG_{x_0}\simeq \rG_2(\Fq)$, $\bbG_{x_1}\simeq\SL_3(\Fq)$  and $\bbG_{x_2}\simeq\SO_4(\Fq)\simeq \SL_2(\Fq)\times\SL_2(\Fq)/\{\pm 1\}\simeq\SO_4(\Fq)$. 
\begin{remark} \label{ex:Iwahori G2}
By \eqref{eqn:DR} the volume of the Iwahori subgroup $\mcI_{\rG_2}$ of $\rG_2$ is
\begin{equation} \label{eqn:volume Iwahori}
    \mathrm{Vol}(\mcI_{\rG_2})
    =\frac{(q-1)^2}{q},
\end{equation}
since the reductive quotient of $\mcI$ is a torus isomorphic to $\Fq^\times\times\Fq^\times$.
\end{remark} 

Let $G^\vee$ denote the Langlands dual group of $G$, i.e.~the complex Lie group with root datum dual to that of $G$. We identify $T^\vee$ with $\C^{\times} \times \C^{\times}$ by 
\begin{equation} \label{eta-betavee}
\begin{split}
  \eta_{\beta^\vee}\colon  T^\vee &\longrightarrow \C^{\times} \times \C^{\times}\hphantom{(2\alpha+\beta)(t),\alpha}\\
\bt &\longmapsto ( (2\beta^{\vee}+\alpha^\vee)(\bt), (\beta^\vee+\alpha^\vee)(\bt) ).
\end{split}
\end{equation}
The group $G^\vee$ is an exceptional group of type $\rG_2$, with positive roots the set
\begin{equation} \label{eqn:coroots}
R(G^\vee,T^\vee)^+=\left\{\alpha^\vee,\beta^\vee,\alpha^\vee+\beta^\vee,\alpha^\vee+2\beta^\vee,\alpha^\vee+3\beta^\vee,
2\alpha^\vee+3\beta^\vee\right\},
\end{equation}
in which $\alpha^\vee$ is a long root and $\beta^\vee$ a short root. Let $\fg^\vee$ and $\ft^\vee$ denote the Lie algebras of $G^\vee$ and $T^\vee$ respectively.
The adjoint action of $G^\vee$ on $\fg^\vee$ defines a Cartan decomposition 
\begin{equation} \label{eqn: Cartan decomposition}
\fg^\vee=\ft^\vee\oplus  \bigoplus_{\gamma^\vee\in R(G^\vee,T^\vee)}\fg^\vee_\gamma,\quad\text{with $\fg^\vee_\gamma=\C e_{\gamma^\vee}$}.
\end{equation}
For each $\gamma^\vee\in R(G^\vee,T^\vee)$,  let $U_{\gamma^\vee}^\vee$ be the associated root subgroup in $G^\vee$.~We fix root group homomorphisms $x_{\gamma^\vee}\colon \C \to U_{\gamma^\vee}^\vee$.

\begin{numberedparagraph} \label{unipotent orbits}
Let $\Unip(G^\vee)$ denote the unipotent variety of $G^\vee$. For two unipotent classes $\orb$, $\orb'$ in
$G^\vee$, we write $\orb'\le\orb$ if $\orb'$ is contained in the Zariski
closure of $\orb$. The relation $\le$ defines a partial ordering on $\Unip(G^\vee)$.
In the Bala-Carter classification, the unipotent classes in $G^\vee$ are
\begin{equation}
1\le \rA_1\le\widetilde \rA_1\le \rG_2(a_1)\le \rG_2.
\end{equation}
They are described in the following Table \ref{tab:Bala-Carter-G2} (see \cite[\S8.4]{Collingwood-McGovern}).
Amongst these five unipotent classes, three are special:~$1$, $\rG_2(a_1)$ and $\rG_2$. In \cite{Ram}, $1$, $\rA_1$,
$\widetilde \rA_1$, $\rG_2(a_1)$ and $\rG_2$ are referred to as the \emph{trivial}, \emph{minimal},
\emph{subminimal}, \emph{subregular} and \emph{regular} (also called \emph{principal}) unipotent orbit, respectively.

\begin{table}[ht]
\begin{tabular}{ |c|c|c| }
 \hline
 Label& Diagram& Dimension\cr
 \hline
$1$&$0\;\;0$&$0$\cr
\hline
$\rA_1$&$1\;\;0$&$6$\cr
\hline
$\widetilde\rA_1$&$0\;\;1$&$8$\cr
\hline
$\rG_2(a_1)$&$2\;\;0$&$10$\cr
\hline
$\rG_2$&$2\;\;2$&$12$\cr  
 \hline
\end{tabular}
\vskip0.2cm
    \caption{\label{tab:Bala-Carter-G2}:  {Bala-Carter classification for $\rG_2(\C)$}.$\qquad$ $\qquad$ $\qquad$}.
\end{table}
By \eqref{eqn: Cartan decomposition}, for any $\gamma\in R(G^\vee,T^\vee)^+$, we can choose $e_\gamma\in\fg_\gamma^\vee$, $f_\gamma\in\fg_{-\gamma}^\vee$ and $h_\gamma\in \ft^\vee$ such that 
\begin{equation}
[h_{\gamma'},e_\gamma]=\gamma(h_{\gamma'})e_\gamma,\quad [h_{\gamma'},f_\gamma]=-\gamma(h_{\gamma'})f_\gamma\quad \text{and}\quad [e_{\gamma},f_\gamma]=h_{\gamma}.
\end{equation}
The weighted Dynkin diagram of the minimal orbit $\rA_1$ shows that $\beta^\vee(h)=1$ and $\alpha^\vee(h)=0$, if $(e,f,h)$ is an $\SL_2$-triple where $e$ is a nilpotent element which corresponds to $\rA_1$. From the table above, we know that $h=h_{2\beta^\vee+\alpha^\vee}$ works. It follows that $e_{2\beta^\vee+\alpha^\vee}$ is a representative of the nilpotent orbit corresponding to $\rA_1$. The root $\beta^\vee+\alpha^\vee$ is short. Since all the short roots are conjugate under the Weyl group $W^\vee=\Nor_{G^\vee}(T^\vee)/T^\vee$ of $G^\vee$, the corresponding spaces $\fg^\vee_{\gamma^\vee}$ are also conjugate under $W^\vee$.~Since the orbits are closed under multiplication by a non-zero scalar and $\dim\fg_{\gamma^\vee}^\vee=1$, any non-zero vector in a space $\fg_{\gamma^\vee}^\vee$ with $\gamma^\vee$ a short root spans the minimal nilpotent orbit of $G^\vee$.

We recall the Springer correspondence for $\rG_2(\C)$ from \cite[p.427]{Carter-book} in Table~\ref{tab:table-unipotent-Gsigma=G2}: to each irreducible character, it attaches a pair $(u,\rho)$ consisting of (the conjugacy class of) a unipotent element $u$ in $G^\vee=\rG_2(\C)$ and an irreducible representation $\rho$ of the group $A_u:=\pi_0(\rZ_{G^\vee}(u))$. There is one $\rho$ missing, which is the case where the component group $A_u$ is the symmetric group $S_3$ and $\rho$ is the sign character of $S_3$. 
\begin{center} 
\begin{table}[ht]
    \begin{tabular}{|c|c|c|c|}
    \hline
         characters of $W$   & unipotent &$A_u$& enhancement \\
         \hline
        $\phi_{1,6}$ & triv&triv&triv\\
        \hline
        $\phi_{1,3}''$  &$A_1$  &triv&triv\\
        \hline
        $\phi_{2,2}$ &$\widetilde{A}_1$  &triv&triv\\
        \hline
        $\phi_{2,1}$ &$\rG_2(a_1)$ &$S_3$ &triv\\
        \hline 
        $\phi_{1,3}'$ &$\rG_2(a_1)$ &$S_3$ &$\rho_{2,1}$\\
        \hline
        $\phi_{1,0}$ &$\rG_2$ &triv&triv\\
        \hline
    \end{tabular}
    \vskip0.2cm
    \caption{\label{tab:table-unipotent-Gsigma=G2}:  {Springer correspondence for $\rG_2(\C)$}. $\qquad$ $\qquad$ $\quad$}
\end{table}
\end{center}

\end{numberedparagraph}

\smallskip

\begin{numberedparagraph}
\textbf{The group $\SO_4(\C)$:} 
We realize the group $\SO_4(\C)\simeq\SL_2^\lr\times\SL_2^\sr/{\{\pm 1\}}$ as the subgroup of $\rG_2(\C)$ generated by $T^\vee$ and the images of $x_{\alpha^\vee}$ and $x_{2\beta^\vee+\alpha^\vee}$. Then the Weyl group of $\SO_4(\C)$ with respect to $T^\vee$ is $\{1, s_{\alpha^\vee},s_{2\beta^\vee+\alpha^\vee}, s_{\alpha^\vee}\cdots s_{2\beta^\vee+\alpha^\vee}\}$. 
The group $\SO_4(\C)$ admits $4$ unipotent classes that are labelled as $(3,1)$, $(2,2)'$, $(2,2)''$,
$(1)$ (see for instance \cite[\S~11.3]{Lusztig-IC}). A representative of the nilpotent class corresponding to $(2,2)'$ is $e_{\alpha^\vee}$, and a representative of the nilpotent class corresponding to $(2,2)''$ is $e_{2\beta^\vee+\alpha^\vee}$.
The closure order on the unipotent classes of $\SO_4(\C)$ is given by the
following:
\begin{equation}
\vcenter{\xymatrix@R=5pt@C=2.5pt{
&(3,1) \ar@{-}[dl] \ar@{-}[dr] \\
(2,2)' \ar@{-}[dr] && (2,2)'' \ar@{-}[dl] \\
&(1)}}.
\end{equation}
The classes $(3,1)$ and $(1)$ are said to be non-degenerate, and the classes $(2,2)'$, $(2,2)''$ degenerate. 

\begin{table} [ht]
\begin{tabular}{|c|c|c|}
  \hline  
        unipotent class in $\SO_4(\C)$& unipotent class in $\rG_2(\C)$&$A_{\SO_4}(u)$   \\
        \hline
        $(3,1)$ & $\rG_2(a_1)$&$\mu_2$  \\
        $(2,2)'$&$A_1$ &$\{1\}$\\
        $(2,2)''$&$\widetilde{A}_1$&$\{1\}$\\
        $(1)$ & 1&$\{1\}$\cr\hline
    \end{tabular}
    \vskip0.2cm
    \centering
    \caption{Unipotent classes of $\SO_4(\C)$.$\qquad$ $\quad$ $\quad$} \label{table:SO4}
\end{table}
\end{numberedparagraph}

\begin{numberedparagraph}
\textbf{The group $\SL_3(\C)$:}
The long roots $\{\pm\alpha^\vee,\pm(3\beta^\vee+\alpha^\vee),\pm(3\beta^\vee+2\alpha^\vee)\}$ form the root system of $\SL_3(\C)$. Hence $\SL_3(\C)$ embeds in $\rG_2(\C)$ as the subgroup generated by $T^\vee$ and the images of $x_{\gamma^\vee}$ with $\gamma^\vee$ long. We can define an explicit such embedding $j_3\colon\SL_3(\C)\hookrightarrow \rG_2(\C)$ as
\begin{equation}
j_3\left(\begin{matrix} 1&u&\cr
&1&\cr
&&1\end{matrix}\right)=x_{\alpha^\vee}(u),\quad
j_3\left(\begin{matrix} 1&&\cr
&1&\cr
&u&1\end{matrix}\right)=x_{3\beta^\vee+\alpha^\vee}(u)
\end{equation}
\begin{equation}
j_3\left(\begin{matrix} 1&&\cr
u&1&\cr
&&1\end{matrix}\right)=x_{-\alpha^\vee}(u)\quad\text{and}\quad
j_3\left(\begin{matrix} t_1&&\cr
&t_2&\cr
&&t_1^{-1}t_2^{-1}
\end{matrix}\right)=\eta'_{\beta^\vee}(t_1,t_2),\end{equation}
where 
\begin{equation} \label{eqn:mu-beta-vee}
\eta'_{\beta^\vee}\colon \C^{\times} \times C^{\times}\to T^\vee
\end{equation} is the inverse morphism of $\eta_{\beta^\vee}$.

In general, the center of $\cG:=\SL_n(\C)$ is isomorphic to the group $\mu_n$ of $n$-th roots of unity of $n$. If $u\in\orb_{(n)}$ (the regular unipotent class), then the component group $A_u$ of the centralizer of $u$ in $\cG$ is isomorphic to $\mu_n$. In particular, the unipotent classes of $\SL_3(\C)$ can be summarized as in Table~\ref{tab:table-unipotent-Gsigma=SL3}.
\begin{center}
\begin{table} [ht]
\begin{tabular}{|c|c|c|}
  \hline  
         unipotent class in $\SL_3(\C)$& unipotent class in $\rG_2(\C)$& $A_{\SL_3}(u)$  \\
         \hline
         $3$& $\rG_2(a_1)$&$\mu_3$\\
         \hline
         $(2,1)$ &$\rA_1$&$\{1\}$\\
         \hline
         $(1,1,1)$ & $1$ &$\{1\}$\cr
         \hline
    \end{tabular}
    \vskip0.2cm
     \caption{\label{tab:table-unipotent-Gsigma=SL3}: Unipotent classes of $\SL_3(\C)$. $\qquad$ $\qquad$ $\quad$}
\end{table}
\end{center}
\end{numberedparagraph}

\begin{numberedparagraph} \textbf{The group $\GL_2(\C)$:} 
There are two different cases depending on whether $\GL_2(\C)$ corresponds to a Levi subgroup of $\rG_2(\C)$ attached to a short or a long root $\gamma^\vee$. We have
$W_G^{\fs}=\{1,s_{\gamma^\vee}\}$.
The unipotent classes in this case are summarized in Table \ref{tab:table-unipotent-Gsigma=GL2}. 
\begin{center}
\begin{table} [ht]
\resizebox{.97\hsize}{!}{\begin{tabular}{|c|c|c|c|}
  \hline  
         unipotent class in $\GL_2(\C)$& representative of nilpotent orbit&  unipotent class in $\rG_2(\C)$&$A_{\GL_2}(u)$  \\
         \hline
         regular& $e_{\gamma^\vee}$& $\begin{cases}\widetilde{\rA}_1&\text{if $\gamma^\vee$ is long}\cr
         \rA_1&\text{if $\gamma^\vee$ is short}\end{cases}$&$1$\\
         \hline
         trivial &0&$1$&$1$\\
         \hline
    \end{tabular}}
    \vskip0.2cm
     \caption{\label{tab:table-unipotent-Gsigma=GL2}: Unipotent classes of $\GL_2(\C)$. $\qquad$ $\qquad$ $\quad$}
\end{table}
\end{center}
\end{numberedparagraph}

\section{ The Galois side for \texorpdfstring{$\mathrm{G}_2$}{\mathrm{G2}}}\label{sec:Galois-G2}

In this section, we study $L$-parameters into $G^{\vee}=\rG_2(\C)$, and determine all non-supercuspidal members in their $L$-packets assuming the properties in \S\ref{subsec:expected properties}. We assume $p\neq 2,3$.

\begin{proposition}\label{singprin} The following are equivalent for $\varphi:W_F\times \SL_2(\C)\rightarrow G^{\vee}$:
\begin{enumerate}
    \item $\varphi(W_F)$ is abelian. 
    \item $\Cent_{G^{\vee}}(S)$ contains a maximal torus of $G^{\vee}$. 
\end{enumerate}
\end{proposition}

\begin{proof} We claim that any abelian subgroup of $G^{\vee}=\rG_2(\C)$ is contained in a torus except for a single exception that is a unique conjugacy class of abelian subgroups isomorphic to $\mu_2^3$. The claim is equivalent to the statement that $\Cent_{G^{\vee}}(S)$ contains a maximal torus of $G^{\vee}$ for any $S\subset G^{\vee}$ abelian. Since closed subgroup of $G^{\vee}$ satisfies the descending chain condition, it suffices to prove this for finitely generated abelian group. Using Proposition \ref{loop}, for $s\in S$ non-trivial we have $\Cent_{G^{\vee}}(s)$ is isomorphic to one of $\SL_3(\C)$, $\SO_4(\C)$, $\GL_2(\C)$ or $\GL_1(\C)^2$. The whole centralizer $\Cent_{G^{\vee}}(S)=\Cent_{\Cent_{G^{\vee}}(s)}(S)$ contains a maximal torus of $\SL_3(\C)$, $\SO_4(\C)$, $\GL_2(\C)$ or $\GL_1(\C)^2$, except for the case of $\SO_4(\C)$ for which an exception happens when $S$ maps to the unique subgroup of $\SO_4(\C)/\mu_2\cong\left(\PGL_2(\C)\right)^2$ that is isomorphic to $\mu_2^2$. Together with $s$, this implies that $S$ is an order $8$ subgroup of $G^{\vee}$ isomorphic to $\mu_2^3$.
\end{proof}
Recall that we have chosen in \eqref{eqn:coroots} a set of positive roots for $G^{\vee}$ and its maximal torus $T^{\vee}$. We denote by $\lambda\colon(\C^{\times})^2\cong T^{\vee}$ the isomorphism satisfying $\alpha(\lambda(z_1,z_2))=z_1$ and $\beta(\lambda(z_1,z_2))=z_2$. 
The Weyl group $W_{G^{\vee}}(T^{\vee})$ is generated by two reflections, one sending $\lambda(z_1,z_2)\mapsto\lambda(z_1^{-1},z_1z_2)$ and the other sending $\lambda(z_1,z_2)\mapsto\lambda(z_1z_2^3,z_2^{-1})$.

\begin{example} $\lambda(1,\zeta_3)$ has centralizer isomorphic to $\SL_3(\C)$, and is conjugate to $\lambda(1,\zeta_3^{-1})$, making its orbit the unique one with this centralizer.
\end{example}
\begin{numberedparagraph}\label{two-strategies-Galois-side}
There are two obvious strategies to start with: 
\begin{enumerate}
    \item Classify $\varphi$ in terms of $\varphi|_{\SL_2(\C)}$;
    \item Classify $\varphi$ in terms of $\varphi|_{W_F}$.
\end{enumerate}
The two strategies are somewhat 
orthogonal. As in line with standard literature, we also use the second strategy in the cases where $\varphi(\SL_2(\C))=1$. 
One of the upshots in our approach is that we use a combination of both strategies. 
\end{numberedparagraph}

\subsection{Classification of $\varphi$ in terms of $\varphi|_{\SL_2(\C)}$}\label{subsec:restr-to-SL2}
    This is the first strategy.~For a general $L$-parameter $\varphi\colon W_F\times \SL_2(\C)\rightarrow G^{\vee}$, the possible shapes for $\varphi|_{\SL_2(\C)}$ are in bijection with unipotent orbits of $G^{\vee}$ (see $\mathsection$\ref{unipotent orbits} for the notations for unipotent orbits). We have the following cases.
    \begin{enumerate}
    \item $\varphi\left(\begin{smallmatrix}1&1\\0&1\end{smallmatrix}\right)$ is regular, i.e.~it lies in the $\rG_2$ unipotent class. In this case, we have $\Cent_{G^{\vee}}(\varphi(\SL_2(\C))=\{1\}$ and the $L$-parameter $\varphi$ is discrete, by Property~\ref{property:size of L-packets}, we obtain a singleton $L$-packet consisting of the Steinberg representation $\St_{\rG_2}$ of $G$.
    \item $\varphi\left(\begin{smallmatrix}1&1\\0&1\end{smallmatrix}\right)$ is subregular, i.e.~it lies in the subregular unipotent class $\rG_2(a_1)$. In this case, $\Cent_{G^{\vee}}(\varphi(\SL_2(\C))\cong S_3$. To see this: $\varphi|_{\SL_2(\C))}$ factors through $\SL_2(\C)\hookrightarrow \SL_3(\C)\hookrightarrow G^{\vee}$. Up to $G^{\vee}$-conjugation, we may assume that $\varphi\left(\matr{t&0\\0&t^{-1}}\right)=\lambda(t,1)$, such that $\Cent_{G^{\vee}}(\{\lambda(t,1)\;|\;t\in\C^{\times}\})\cong \GL_2(\C)$ is given by the short root $\beta$ of $G^{\vee}$. This $\GL_2(\C)$ acts on the sum of root spaces corresponding to $\{\alpha^{\vee},\alpha^{\vee}+\beta^{\vee},\alpha^{\vee}+2\beta^{\vee},\alpha^{\vee}+3\beta^{\vee}\}$, which is a $4$-dimensional representation of $\GL_2(\C)$ that is isomorphic to $\Sym^3(\mathrm{std})$. One can check that $\varphi\left(\matr{1&1\\0&1}\right)$ is an element in general position in $\Sym^3(\mathrm{std})$ and thus has centralizer $S_3$, i.e.~the permutation group of the $3$ roots of a separable polynomial of degree $3$. 
    
  In this case, $\varphi(W_F)$ is any subgroup of $S_3$, and $\varphi$ is always discrete. We remark that the order $2$ element in $S_3\cong \Cent_{G^{\vee}}(\varphi(\SL_2(\C))$ has centralizer isomorphic to $\SO_4(\C)$, while the order $3$ element has centralizer isomorphic to $\SL_3(\C)$. The subregular unipotent $\varphi\left(\matr{1&1\\0&1}\right)$ lies in the regular unipotent orbit in both $\SO_4(\C)$ and $\SL_3(\C)$ by Tables~\ref{table:SO4} and \ref{tab:table-unipotent-Gsigma=SL3}. 
    Since the group $\varphi(W_F)$ can have order $1$, $2$, $3$ or $6$, we have an $L$-packet of size $2+1$, $1+1$, $1+2$ or $1+0$ respectively (here we write the number of non-supercuspidal representations plus the number of supercuspidal representations). 
    \begin{itemize}
    \item The case where $|\varphi(W_F)|=1$ is the unipotent case, and will be addressed in Theorem \ref{constituents-part-one}\eqref{nd6} and Table \ref{table:unram-unip-ppal-series-case}. 
    More precisely, this $L$-packet has size $3$: it contains two unipotent representations in the principal series, the representations $\pi(1)$ and $\pi(1)'$ from Table \ref{table:unram-unip-ppal-series-case}, and one unipotent supercuspidal representation from $\mathsection$\ref{subsec:udL-packets}(1). 
        \item  The case where $|\varphi(W_F)|=2$ will be discussed in \S\ref{d0p}(\ref{toy}).
        \item The case where $|\varphi(W_F)|=3$ will be discussed in \S\ref{d0p}(\ref{chal}). 
        \item Lastly the case where $|\varphi(W_F)|=6$ will be discussed in \S\ref{d0p}(\ref{single}).
    \end{itemize}

    \item When $\varphi\left(\matr{1&1\\0&1}\right)$ is in the $\widetilde{A}_1$ (resp.~$A_1$) unipotent class, i.e.~$\varphi(\SL_2(\C))$ is a short (resp.~long) root $\SL_2(\C)$ in $G^{\vee}$,
    its centralizer $\Cent_{G^{\vee}}(\varphi(\SL_2(\C)))\cong \SL_2(\C)$ is a long (resp.~short) root $\SL_2(\C)$. We have two sub-cases:
    \begin{enumerate}
        \item $\varphi(W_F)\subset \Cent_{G^{\vee}}(\varphi(\SL_2(\C)))$ is not contained in a torus of $\Cent_{G^{\vee}}(\varphi(\SL_2(\C)))$, in which case $\varphi$ is discrete. This case is treated in \S\ref{SL2}.
        \item\label{reducible} $\varphi(W_F)$ is contained in a torus $T^{\vee}\subset \Cent_{G^{\vee}}(\varphi(\SL_2(\C)))$. In this case, either  $\Cent_{G^{\vee}}(\varphi)=T^{\vee}$, or $\Cent_{G^{\vee}}(\varphi)=\Cent_{G^{\vee}}(\varphi(\SL_2(\C)))$. We always have $S_{\varphi}=1$, and the $L$-packet consists of a single principal series representation by Proposition \ref{singprin}.
    \end{enumerate}
    \item $\varphi(\SL_2(\C))$ is trivial. We leave it to later case-by-case discussions in $\mathsection$\ref{d0p} and $\mathsection$\ref{d+p}.
    \end{enumerate}
As a consequence of the above discussions, we immediately obtain the following. 
    \begin{corollary}\label{lemma-nondisc} When $\varphi$ is non-discrete, either $\varphi(\SL_2(\C))$ is trivial, or $\varphi(\SL_2(\C))$ is a short or long $\SL_2(\C)$ in $G^{\vee}$. In the latter case(s), $S_{\varphi}=\{1\}$ and the $L$-packet consists of a single principal series representation.
    \end{corollary}
Let $F_{2,2}$ denote the bi-quadratic extension over $F$ with $\Gal(F_{2,2}/F)\cong \mu_2^2$. We now study our LLC using Properties \ref{property:L-packets}, \ref{property:size of L-packets}, \ref{property:generic tempered L-packets} and \ref{property:AMS-conjecture-7.8}. 
\begin{theorem} \label{constituents-part-one} 
A parabolically induced representation of $G=\rG_2(F)$ has two or more non-isomorphic constituents exactly in the following cases:
\begin{enumerate}
            \item\label{length-two-intermediate-series-thm-constituents} The parabolically induced representation $\ii_P^G\sigma$ from a Levi subgroup $M\cong \GL_2(F)$ of a supercuspidal representation $\sigma$ of $M$ s.t.
            \begin{enumerate}
                \item\label{nd11} The Levi $M=M_\beta\cong\GL_2(F)$ is associated to the long root and $\sigma\cong\sigma_{S_3}\otimes\nu_F^{\pm1}$, where $\sigma_{S_3}$ is the unique supercuspidal representation of $\GL_2$ whose corresponding $L$-parameter under the LLC for $\GL_2$ has image $S_3$, and $\otimes\nu_F^{\pm1}$ denotes tensoring with $\nu_F^{\pm1}\circ\det:\GL_2(F)\ra\C^{\times}$. In this case, the parabolic induction $\ii_P^G(\sigma)$ has a discrete constituent $\pi(\sigma)$ and another non-tempered constituent $J(\sigma)$, where $\pi(\sigma)$ corresponds to the $L$-parameter in \ref{d0p}(\ref{single}) such that $\varphi|_{\SL_2(\C)}$ meets the $\rG_2(a_1)$ unipotent class, while $J(\sigma)$ has trivial $\varphi|_{\SL_2(\C)}$.
                \item\label{nd12} The central character $\omega_{\sigma}$ of $\sigma$ has order $2$ (and is non-trivial); moreover, if $\sigma\cong\sigma_{S_3}$, then $\GL_2(F)$ cannot be the one attached to the long root.\\
                In this case, $\ii_P^G\sigma$ has two non-isomorphic constituents $\pi$ and $\pi'$ that are both non-discrete and tempered, one of which is generic. Both $\pi$ and $\pi'$ correspond to the $L$-parameters $\varphi$ in \ref{d0p}(\ref{packetoftwo}) for the depth-zero case, or \ref{d+p}(\ref{packetoftwo+}) for the positive-depth case. In particular, we always have $\varphi|_{\SL_2(\C)}\equiv1$ in this case.
                \item\label{nd13} The central character $\omega_{\sigma}\equiv\nu_F^{\pm1}$. In this case, the parabolic induction has a discrete constituent $\pi(\sigma)$ which corresponds to the $L$-parameters in \S\ref{SL2}, and another non-tempered constituent $J(\sigma)$ which has trivial $\varphi|_{\SL_2(\C)}$. 
            \end{enumerate}
            \item\label{nd22} The parabolically induced representation from the order $2$ character $\chi$ of $(F^{\times})^2$ such that $\varphi_{\chi}\colon W_F\rightarrow(\C^{\times})^2$ satisfies $\ker(\phi)=W_{F_{2,2}}$ 
            for the bi-quadratic extension $F_{2,2}$ over $F$. \\
            In this case, the parabolically induced representation has two non-isomorphic constituents $\pi$ and $\pi'$ that are both non-discrete and tempered, one of which is generic. Both $\pi$ and $\pi'$ correspond to the $L$-parameter $\varphi$ given in \ref{d0p}(\ref{twoprin}). In particular, we have $\varphi|_{\SL_2(\C)}\equiv 1$.
        \end{enumerate}
        Representations in the above two cases \eqref{length-two-intermediate-series-thm-constituents} and \eqref{nd22} are the only non-discrete representations of $G$ whose $L$-packets are not singleton packets. In both cases, the $L$-packet has two members that are exactly the two irreducible constituents in the same parabolic induction.
 
        \begin{enumerate}
        \setcounter{enumi}{2}
            \item\label{constituents-part-two} The parabolically induced representation of 
            $\theta\colon T\rightarrow\C^{\times}$ with $\theta\circ\lambda^{\vee}=\nu_F$ for some coroot $\lambda\colon F^{\times}\rightarrow T$. We have the following cases:
            
\begin{enumerate}
    \item\label{nd1}
    There is a unique such $\lambda\in R^{\vee}(\bG,\bT)$. In this case, the 
    parabolically induced representation has two non-isomorphic constituents. They are both non-discrete. 
                \item\label{nd2} There are exactly two such $\lambda\in R^{\vee}(\bG,\bT)$ that are both short with $120^o$ angle, or both long with $60^o$ angle. In this case, the parabolic 
                induction 
                again has two non-isomorphic constituents, both non-discrete.
                \item\label{nd3} There are two such $\lambda\in R^{\vee}(\bG,\bT)$ and they are perpendicular to each other. In this case, the parabolic 
                induction has four constituents. One of them is discrete and corresponds to an $L$-parameter with $\Cent_{G^{\vee}}(\varphi(W_F))=\SO_4(\C)$ and $\varphi|_{\SL_2(\C)}$ subregular. 
                \item\label{nd4} There are exactly two such $\lambda\in R^{\vee}(\bG,\bT)$ that are both long with $120^o$ angle. In this case, the 
                parabolic induction has four constituents. One of them is discrete and corresponds to an $L$-parameter with $\Cent_{G^{\vee}}(\varphi(W_F))=\SL_3(\C)$ and $\varphi|_{\SL_2(\C)}$ subregular. 
                \item\label{nd5} There are two such $\lambda\in R^{\vee}(\bG,\bT)$ with $150^o$ angle in between. In this case, the 
                parabolic induction has the Steinberg representation as a constituent, and has three other non-discrete constituents. 
                \item\label{nd6} There are four such $\lambda\in R^{\vee}(\bG,\bT)$. In this case, the 
                parabolic induction has 5 constituents. Two of them are discrete and correspond to the two non-supercuspidal members in the $L$-packet (which is a unipotent $L$-packet of size $3$ containing also the supercuspidal representation $\pi[-1]$) for the $L$-parameter with $\varphi(W_F)=1$ and $\varphi|_{\SL_2(\C)}$ subregular. 
            \end{enumerate}
        \end{enumerate}
        In cases (\ref{nd2}), (\ref{nd5}) and (\ref{nd6}), the character $\theta\colon T\rightarrow\C^{\times}$ is unique 
        up to the action of the Weyl group $W_G$ 
        (one character for each of the two configurations in (\ref{nd2})). In case (\ref{nd3}), we have that $\theta$ 
        (up to $W_G$-action) is determined by an order-$2$ character $F^{\times}\rightarrow\C^{\times}$, i.e. by a quadratic extension. In case (\ref{nd4}), we have that $\theta$ 
        (up to $W_G$-action) is determined by an order-$3$ character up to inversion, i.e. 
        it is determined up to a Galois cubic extension.
\end{theorem}

    \begin{proof} For non-discrete $L$-parameters, a case-by-case study conducted in \S\ref{d0p} and \S\ref{d+p} shows that the only cases where $S_{\varphi}\neq 1$ are: \S\ref{d0p}(\ref{twoprin}),  \S\ref{d0p}(\ref{packetoftwo}) and \S\ref{d+p}(\ref{packetoftwo+}).~These cases belong to case (\ref{nd22}) in Theorem \ref{constituents-part-one}. For discrete $L$-parameters, Lemma \ref{cusp} below implies that $\varphi|_{\SL_2(\C)}\not\equiv1$ is necessary to obtain non-supercuspidal representations. 
    Hence one can refer to the classification of $\varphi$ in $\mathsection$\ref{subsec:restr-to-SL2}, and match them with the case-by-case study in sections \S\ref{d0p}, \S\ref{d+p} and \S\ref{SL2} to obtain a complete list of discrete non-supercuspidal representations of $G$. These results are listed in cases (\ref{nd11}), (\ref{nd13}), (\ref{nd3}), (\ref{nd4}), (\ref{nd5}) and (\ref{nd6}).

    Outside the above discussed cases, we have $S_{\varphi}=\{1\}$. Therefore, in order to have two different constituents, we need two non-supercuspidal $L$-parameters with the same supercuspidal support. 
We start with a $L$-parameter $\varphi\colon W_F\times \SL_2(\C)\rightarrow G^{\vee}$, and consider the {\it trivially} enhanced $L$-parameter $(\varphi,1)$. Then Property~\ref{property:AMS-conjecture-7.8} 
specifies the supercuspidal support as follows: We take a maximal torus $T^{\vee}\subset \cG_{\varphi}:=\Cent_{G^{\vee}}(\varphi(W_F))$. Let $L^{\vee}:=\Cent_{G^{\vee}}(T^{\vee})$. The cuspidal support $\Sc(\varphi,1)$ of $(\varphi,1)$ is the trivially enhanced $L$-parameter $(\bar{\varphi},1)$, where $\bar{\varphi}\colon W_F\rightarrow L^{\vee}$ is 
\begin{equation}\bar{\varphi}(w):=\varphi(w,\matr{|\!|w|\!|^{1/2}&0\\0&|\!|w|\!|^{-1/2}}).
\end{equation}
    
Suppose there exists more than one $\varphi$ with the same $\bar{\varphi}$. Denote by $\phi$ their common $\bar{\varphi}$. We have $\varphi|_{\SL_2(\C)}\not\equiv1$ for at least one $\varphi$. The case where $\varphi$ is discrete is considered in the case-by-case study in sections \S\ref{d0p}, \S\ref{d+p} and \S\ref{SL2}. We devote the rest of this proof to the case where $\varphi$ is non-discrete. 

By Corollary \ref{lemma-nondisc}, when $\varphi$ is non-discrete, we have that $H^{\vee}=T^{\vee}$ is a maximal torus. The group $L^{\vee}:=T^{\vee}\cdot\varphi(\SL_2(\C))$ is a (short root or long root) Levi subgroup isomorphic to $\GL_2(\C)$, for which we denote by $\lambda^{\vee}\in R^{\vee}(L^{\vee},T^{\vee})$ given by $\lambda^{\vee}(t)=\varphi\left(\matr{t&0\\0&t^{-1}}\right)$ with $\lambda$ as the corresponding root. The fact that $\varphi|_{W_F}\colon W_F\rightarrow T^{\vee}$ commutes with $\varphi(\SL_2(\C))$ is equivalent to the fact that $\lambda\circ\varphi|_{W_F}\equiv 1$, which is moreover equivalent to saying that
    \begin{equation}\label{reduciblepoint}
    \lambda\circ\phi\equiv |\!|\cdot|\!|\text{ as characters on }W_F.
    \end{equation}
By Local Langlands for tori, $\phi=\bar{\varphi}\colon W_F\rightarrow T^{\vee}$ corresponds to some $\theta\colon T\rightarrow\C^{\times}$, and \eqref{reduciblepoint} is equivalent to saying that
    $\theta\circ\lambda\equiv|\!|\cdot|\!|$. 
    Every step of the argument above is reversible, therefore, if we start with a $\phi\colon W_F\rightarrow T^{\vee}$ satisfying \eqref{reduciblepoint} for some root $\lambda\in R(G^{\vee},T^{\vee})$, we obtain a corresponding non-discrete $\varphi$ with $\varphi|_{\SL_2(\C)}\not\equiv1$ and $\bar{\varphi}=\phi$. 
It then remains to resolve the question of whether there can be more than one choice for such $\lambda$. 

If there is only one choice for such $\lambda$, then we obtain two $\varphi$'s satisfying $\bar{\varphi}=\phi$; one of them satisfies $\varphi|_{\SL_2(\C)}\equiv1$ and $\varphi|_{W_F}=\phi$. This is precisely case (\ref{nd1}). 

If, however, there are more than one such $\lambda$, then they generate a lattice \[\Lambda:=\langle\lambda\in R(G^{\vee},T^{\vee})\;|\;\lambda\circ\phi\equiv |\!|\cdot|\!|\rangle.\]
We have $\Lambda\subset X^*(T^{\vee})$ with index at most $3$. When $\Lambda=X^*(T^{\vee})$, we have that $\phi$ is uniquely determined by $\lambda$'s. There can be either two or four such $\lambda$'s of the configuration in case (\ref{nd2}), (\ref{nd5}) or (\ref{nd6}). In case (\ref{nd2}), the two choices of $\lambda$ are conjugate under the Weyl group action, and thus there is only one possible $\varphi$ satisfying $\bar{\varphi}=\phi$ and $\varphi|_{\SL_2(\C)}\not\equiv1$. 
    
    In case (\ref{nd5}), we have two different choices for $\lambda$, giving rise to two non-discrete $\varphi$'s satisfying $\bar{\varphi}=\phi$. Nevertheless, there is also a discrete $\varphi$ with $\bar{\varphi}=\phi$, which is the one satisfying $\varphi|_{W_F}\equiv 1$ and $\varphi\left(\matr{1&1\\0&1}\right)$ is the regular unipotent class.
    
    In case (\ref{nd6}), we have two different choices for $\lambda$, giving rise to two non-discrete $\varphi$'s satisfying $\bar{\varphi}=\phi$. Nevertheless, there is also a discrete $\varphi$ with $\bar{\varphi}=\phi$, which is the one satisfying $\varphi|_{W_F}\equiv 1$ and $\varphi\left(\matr{1&1\\0&1}\right)$ is the subregular unipotent class. Moreover, there is a non-trivial local system on the subregular unipotent orbit (which is the unique one of rank $2$) that lives in the principal series. This gives one another non-trivial enhancement of the same $\varphi$ giving our $\theta$ as the supercuspidal support.
    
    In case (\ref{nd3}), we have two different choices for $\lambda$, giving rise to two non-discrete $\varphi$'s satisfying $\bar{\varphi}=\phi$. We also have a non-discrete $L$-parameter satisfying $\varphi|_{W_F}=\phi$ and $\varphi|_{\SL_2(\C)}\equiv 1$. Suppose that the two choices for $\lambda$ are denoted by $\lambda_1$ and $\lambda_2$, where $\lambda_1$ is short. Write $\lambda'_1:=\frac{3}{2}\lambda_1-\frac{1}{2}\lambda_2$ and $\lambda'_2:=\frac{1}{2}\lambda_1+\frac{1}{2}\lambda_2$. Thus we also have $\lambda'_i\in R(G^{\vee},T^{\vee})$, but $\lambda'_i\circ\phi$ is different from $|\!|\cdot|\!|$ by an order-two character. This order-two character corresponds to $\eta:W_F\rightarrow T^{\vee}$ where $\Cent_{G^{\vee}}(\eta)\cong \SO_4(\C)$ such that $\lambda_1,\lambda_2\in R(\Cent_{G^{\vee}}(\eta),T^{\vee})$.
    We then consider the discrete $L$-parameter $\varphi\colon W_F\times \SL_2(\C)\rightarrow G^{\vee}$ such that $\varphi|_{W_F}=\eta$ and $\varphi(\SL_2(\C))\subset \Cent_{G^{\vee}}(\eta)\cong \SO_4(\C)$ meets the regular unipotent orbit in $\SO_4(\C)$. 
    The trivial enhancement of this $L$-parameter also has supercuspidal support $\phi$ (while the other enhancement is cuspidal).
    
    Case (\ref{nd4}) is completely analogous to (\ref{nd3}), except that $\eta$ now has order three and $\Cent_{G^{\vee}}(\eta)\cong \SL_3(\C)$. We note that there exists a Weyl group element that stabilizes everything else but sends $\eta$ to $\eta^{-1}$. This finishes the proof.
    \end{proof}
\begin{numberedparagraph}
    Before we proceed, let us summarize the results of Theorem~\ref{constituents-part-one} in the following Table \ref{tableprin} and Table \ref{tableint}. 
    
    For a character $\theta\colon T\rightarrow\C^{\times}$, we write $\theta=\xi_1\otimes\xi_2$, where $\xi_1=\theta\circ\beta^{\vee}$ and $\xi_2=\theta\circ(\alpha^{\vee}+\beta^{\vee})$ with $\alpha^{\vee}$, $\beta^{\vee}$ as in \eqref{eqn:coroots}.~We remark that this identification is compatible with (\ref{etaa}). Thanks to Theorem~\ref{constituents-part-one}, we have a table for parabolic inductions 
    $I(\xi_1\otimes\xi_2)$ that contain at least two non-isomorphic irreducible constituents that are in the principal series. 
    In the second to last column of Table \ref{tableprin}, each item begins with $\varphi\left(\matr{1&1\\0&1}\right)$ as its first tag, followed by tags among $\{\mathrm{d.},\mathrm{n.d.},\mathrm{t.},\mathrm{n.t.},\mathrm{g.},\mathrm{n.g.}\}$ for the adjectives ``discrete'', ``non-discrete'', ``tempered'', ``non-tempered'', ``generic'', and ``non-generic'', respectively. 
    On the other hand, all the principal series not appearing in Table \ref{tableprin} should have only one irreducible constituent up to isomorphism.

 \begin{equation*}
     \begin{array}{|c|c|c|c|c|}
    \hline \text{Label}&\text{Input}&\xi_1\otimes\xi_2&\text{Constituents of }\ii_P^G(\xi_1\otimes\xi_2) &\text{Table}\\
    \hline\begin{array}{c}(\ref{nd22})\\\S\ref{d0p}(\ref{twoprin})\end{array}&\begin{array}{c}\eta_2^2\equiv(\eta_2')^2\equiv 1,\\
    1\not\equiv\eta_2\not\equiv\eta_2'\not\equiv1
    \end{array}&\eta_2\otimes\eta_2'&\begin{array}{c}(1,\mathrm{n.d.},\mathrm{t.},\mathrm{g.})\\ (1,\mathrm{n.d.},\mathrm{t.},\mathrm{n.g.})\end{array} &\text{Table}~\ref{table:unitary-length-two}\\ 
    \hline\begin{array}{c}(\ref{nd1})\\(\ref{nd2})\end{array}&\begin{array}{c}\eta:F^{\times}\rightarrow\C^{\times}\text{ s.t.}\\\eta^2\not\equiv1\not\equiv(\eta\nu_F)^2,\\\eta\not\equiv\nu_F^2,\;\eta\not\equiv\nu_F^{-3}\end{array}&\eta\otimes\nu_F&\begin{array}{c}(\widetilde{A}_1,\mathrm{n.d.}),(1,\mathrm{n.t.}),\\\text{the first tempered} \\ \text{if and only if }|\!|\eta\nu_F^{1/2}|\!|\equiv1\end{array}&\begin{array}{c} {\text{Table}~\ref{table:nu-tensor-xi-shorttable-ppal-series-case2-ramified-neither}}\end{array}\\
    \hline\begin{array}{c}(\ref{nd1})\\(\ref{nd2})\end{array}&\begin{array}{c}\eta:F^{\times}\rightarrow\C^{\times}\text{ s.t.}\\\eta^3\not\equiv1,\nu_F^{-3}\end{array}&\eta\otimes\eta\nu_F&
    \arraystretch{1.5}\begin{array}{c}(A_1,\mathrm{n.d.}),(1,\mathrm{n.t.}),\\\text{the first tempered} \\ \text{if and only if }|\!|\eta\nu_F^{1/2}|\!|\equiv1\end{array}&\begin{array}{c}{\text{Table}~\ref{table:xi-tensor-xi-shorttable-ppal-series-case5}}\\
    \text{Table}~\ref{table:xi-tensor-xi-shorttable-ppal-series-case5-ramified}\\
    {\text{Table}~\ref{table:xi-tensor-one-shorttable-ppal-series-case2}}
    \end{array}\\
    \hline\begin{array}{c}(\ref{nd3})\\\S\ref{d0p}(\ref{toy})\end{array}&E/F\text{ quadratic}&\nu_F\otimes\eta_{E/F}&\begin{array}{c}(\rG_2(a_1),\mathrm{d.},\mathrm{g.}),\\(\widetilde{A}_1,\mathrm{n.t.}),\;(A_1,\mathrm{n.t.}),\;(1,\mathrm{n.t.})\end{array}&\begin{array}{c}
    \text{Table}~\ref{table:unip-ppal-series-case-3aquad}\\ 
    \text{Table}~\ref{table:unip-ppal-series-case-3aquad-ramified}
    \end{array}\\
    \hline\begin{array}{c}(\ref{nd4})\\\S\ref{d0p}(\ref{chal})\end{array}&E/F\text{ Galois cubic}&\eta_{E/F}\otimes\eta_{E/F}\nu_F&\begin{array}{c}(\rG_2(a_1),\mathrm{d.},\mathrm{g.}),\\(A_1,\mathrm{n.t.}),\;(A_1,\mathrm{n.t.}),\;(1,\mathrm{n.t.})\end{array}&
    \begin{array}{c}\text{Table}~\ref{table:unip-ppal-ser3c-unramified}\\
    \text{Table}~\ref{table:unip-ppal-series3c-ramified}
    \end{array}\\
    \hline(\ref{nd5})&\text{None}&\nu_F\otimes\nu_F^2&\begin{array}{c}(\rG_2,\mathrm{d.},\mathrm{g.}),\\(\widetilde{A}_1,\mathrm{n.t.}),\;(A_1,\mathrm{n.t.}),\;(1,\mathrm{n.t.})\end{array}&\text{Table}~\ref{table:unip-ppal-series-case-3bbis}\\
    \hline(\ref{nd6})&\text{None}&1\otimes\nu_F&\begin{array}{c}(\rG_2(a_1),\mathrm{d.},\mathrm{g.}),\;(\rG_2(a_1),\mathrm{d.},\mathrm{n.g.}),\\(\widetilde{A}_1,\mathrm{n.t.}),\;(A_1,\mathrm{n.t.}),\;(1,\mathrm{n.t.})\end{array}&\text{Table}~\ref{table:unram-unip-ppal-series-case}\\
    \hline
    \end{array}
    \end{equation*}
    \vskip-.4cm
  \begin{table}[h]
        \centering
        \caption{\label{tableprin}Principal series with more than one non-isomorphic constituents.
        $\;\;\;\;\;\;\;\;\;\;\;\;\;\;\;\;\;\;\;\;$}
    \end{table}
    \end{numberedparagraph}
    
  \begin{numberedparagraph}
On the other hand, for a supercuspidal representation $\sigma$ of $\GL_2(F)$, we also consider the parabolic induction $\ii_P^G(\sigma)$, whose irreducible constituents are referred to as \textit{intermediate series}. We denote by $\omega_{\sigma}\colon\GL_1(F)\rightarrow\C^{\times}$ the central character of the supercuspidal representation $\sigma$. 
We also denote by $\sigma\otimes\nu_F^s$ the representation $\sigma$ tensored with $\GL_2(F)\xrightarrow{\det}\GL_1(F)\xrightarrow{\nu_F^s}\C$, such that $\omega_{\sigma\otimes\nu_F^s}=\omega_{\sigma}\cdot\nu_F^{2s}$. 

We shall also consider a particular supercuspidal representation of $\GL_2(F)$ that is denoted as $\sigma_{S_3}$ and is characterized by the property that: its corresponding $L$-parameter $W_F\rightarrow \GL_2(\C)$ has image isomorphic to $S_3$; such a supercuspidal representation exists if and only if $q\equiv -1(3)$, in which case it is unique, at depth-zero and has central character $\omega_{\sigma}$ unramified of order two. We have the following Table \ref{tableint} for parabolic inductions that contain at least two non-isomorphic irreducible constituents that are intermediate series.  
    \begin{equation*}
    \begin{array}{|c|c|c|c|}
    \hline\text{Cases}&\text{Choice of }\GL_2(F)&\text{Condition on $\sigma$}&\text{Constituents of }\ii_P^G(\sigma)\\
    \hline\begin{array}{c}\ref{constituents-part-two}\eqref{nd13}\\\ref{SL2}\end{array}&\text{short}&\omega_{\sigma}\equiv\nu_F^{\pm1}&(\widetilde{A}_1,\mathrm{d.},\mathrm{g.}),(1,\mathrm{n.t.})\\
    \hline\begin{array}{c}\ref{constituents-part-two}\eqref{nd13}\\\ref{SL2}\end{array}&\text{long}&\omega_{\sigma}\equiv\nu_F^{\pm1}&
    \begin{array}{c}(A_1,\mathrm{d.},\mathrm{g.}),(1,\mathrm{n.t.})\end{array}\\
    \hline\begin{array}{c}\ref{constituents-part-two}\eqref{nd12}\\\ref{d0p}(\ref{packetoftwo})\\\ref{d+p}(\ref{packetoftwo+})\end{array}&\text{short}&\begin{array}{c}\omega_{\sigma}^2\equiv 1,\\\omega_{\sigma}\not\equiv1\end{array}&(1,\mathrm{n.d.},\mathrm{t.},\mathrm{g.}),(1,\mathrm{n.d.},\mathrm{t.},\mathrm{n.g.})\\
    \hline\begin{array}{c}\ref{constituents-part-two}\eqref{nd12}\\\ref{d0p}(\ref{packetoftwo})\\\ref{d+p}(\ref{packetoftwo+})\end{array}&\text{long}&\begin{array}{c}\omega_{\sigma}^2\equiv 1,\\\omega_{\sigma}\not\equiv1,\;\sigma\not\cong\sigma_{S_3}\end{array}&(1,\mathrm{n.d.},\mathrm{t.},\mathrm{g.}),(1,\mathrm{n.d.},\mathrm{t.},\mathrm{n.g.})\\
    \hline\begin{array}{c}\ref{constituents-part-two}\eqref{nd11}\\\ref{d0p}(\ref{single})\end{array}&\text{long}&\sigma\cong\sigma_{S_3}\otimes\nu_F^{\pm1}&(\rG_2(a_1),\mathrm{d.},\mathrm{g.}),(1,\mathrm{n.t.})\\
    \hline
    \end{array}
    \end{equation*}
    \vskip-.3cm
    \begin{table}[h]
        \centering
        \caption{\label{tableint}Intermediate series with more than $1$ non-isomorphic constituents.
        $\;\;\;\;\;\;\;\;\;\;\;\;\;\;\;\;\;\;\;\;$}
    \end{table}

  \end{numberedparagraph}

\subsection{$L$-parameters for $\rG_2$}     
\subsubsection{Depth-zero $L$-parameters}\label{d0p}

    In the rest of this chapter, we follow the second strategy outlined in $\mathsection$\ref{two-strategies-Galois-side} to finish the classification of $\varphi$, i.e.~we classify $\varphi$ in terms of $\varphi|_{W_F}$.~Let us denote by $G_0^{\vee}:=\Cent_{G^{\vee}}(\varphi(P_F))$ and $S^{\vee}:=\Cent_{G^{\vee}}(\varphi(I_F))=\Cent_{G_0^{\vee}}(\varphi(I_F))$.  Recall that $\cG_{\varphi}:=\Cent_{G^{\vee}}(\varphi(W_F))=\Cent_{S^{\vee}}(\varphi(\Fr))$, and we have $S_{\varphi}=\pi_0(\Cent_{\cG_{\varphi}}(\varphi(\SL_2(\C))))$.

For each configuration of $\varphi|_{W_F}$, we have that $\varphi$ is discrete if and only if $(\cG_{\varphi})^\circ$ is semisimple and $\varphi(1,\matr{1&1\\0&1})$ is a distinguished unipotent element $u=u_\varphi$ in $(\cG_{\varphi})^\circ$. In this case, $\rho\in\Irr(S_{\varphi})$ is cuspidal in the sense of $\mathsection$\ref{subsubsec:cuspidal-pairs}. 
In fact, the following holds (thanks to the small rank of $\rG_2$). 
\begin{lemma}\label{cusp} For discrete $L$-parameters, except in the case where $\varphi|_{W_F}$ is trivial and $S_{\varphi}\cong S_3$, we have that $u$ is regular unipotent in $\cG_{\varphi}$, and $\rho\in\Irr(S_{\varphi})$ is cuspidal if and only if either $\rho$ is non-trivial or $\cG_{\varphi}$ is finite.
\end{lemma}

\begin{proof} The assertion is obvious when $\cG_{\varphi}$ is finite. We now assume that $\dim \cG_{\varphi}>0$. Each connected Dynkin component for $\cG_{\varphi}$ is of type $\rA_1$,  $\rA_2$ or $\rG_2$. The only non-regular but distinguished unipotent appears in the $\rG_2$ case for which $\varphi|_{W_F}$ is trivial and $S_{\varphi}\cong S_3$. For the remaining cases, $u$ is regular. Pulling back to the simply-connected cover of $\cG_{\varphi}$, any local system on the regular unipotent orbit of $\SL_2(\C)$, $\SL_3(\C)$ or $\rG_2(\C)$ is cuspidal if and only if it is non-trivial. 
\end{proof}
We remark that when $u$ is regular unipotent, we have $S_{\varphi}=\rZ_{\cG_{\varphi}}$. 

We are now ready to proceed with a detailed case-by-case discussion.~The centralizer $S^{\vee}:=\Cent_{G^{\vee}}(\varphi(I_F))$ is the fixed subgroup of a finite-order automorphism of $G^{\vee}$. Since $G^{\vee}$ is simply-connected, $S^{\vee}$ is connected.~It can be one of the following cases:
\begin{enumerate}
	\item\label{d0-unipotent} $S^{\vee}\cong G^{\vee}$, i.e.~$\varphi(I_F)$ is trivial. In this case, the $L$-parameter is unipotent and an explicit LLC has been constructed by \cite{Lu-padicI} (see also \cite[p.480]{Reeder-Iwahori-spherical-discrete-series} and \cite{Morris-ENS}). 
	\item $S^{\vee}\cong \SL_3(\C)$. In this case, $\varphi(I_F)$ is generated by a conjugate of $\lambda(1,\zeta_3)$. Since $\lambda(1,\zeta_3)$ is also conjugate to $\lambda(1,\zeta_3^{-1})$, this makes $\varphi|_{I_F}$ unique. By Proposition \ref{loop}, $\cG_{\varphi}:=\Cent_{S^{\vee}}(\varphi(\Fr))$  can be $\SL_3$, $\GL_2$, $\GL_1^2$, $\SL_2$, $\SO_3$ or $\GL_1$. In the case $\cG_{\varphi}=\GL_2$, $\GL_1^2$ or $\GL_1$, we have that $\cG_{\varphi}$ is not semisimple and thus $\varphi$ is not discrete. In these cases $\Cent_{G^{\vee}}(\varphi)$ is always a torus making $S_{\varphi}=\{1\}$. In the cases $\cG_{\varphi}=\GL_2$ or $\GL_1^2$, $\Fr$ acts on $\SL_3(\C)$ by inner automorphism, necessarily fixing $\rZ_{\SL_3(\C)}=\varphi(I_F)$. Hence $\varphi(W_F)$ is abelian. By Proposition \ref{singprin} the $L$-packet consists of a single principal series representation. When $\cG_{\varphi}\cong \GL_1$, $\varphi(W_F)$ is non-abelian and by Proposition \ref{singprin} we get an $L$-packet consists of a single non-discrete constituent of a parabolic induction from $\GL_2(F)$. We discuss the rest of the cases:
	\begin{enumerate}
		\item\label{chal} $\cG_{\varphi}=\SL_3$. This is to say that $\varphi(W_F)\subset \Cent_{G^{\vee}}(\SL_3(\C))=\varphi(I_F)$, i.e. $\varphi|_{W_F}$ factors through $\Gal(E/F)$ for some degree $3$  ramified Galois extension, which necessarily forces $q\equiv1(3)$. In such case there are $3$ choices of $E$. For each choice we have a unique discrete $L$-parameter with $S_{\varphi}=\mu_3$ by Lemma \ref{cusp}. Any such $L$-packet consists of one non-supercuspidal representation and two supercuspidal ones. When $\varphi$ is not discrete, the centralizer $\Cent_{\cG_{\varphi}}(\varphi(\SL_2(\C))$ is either $\SL_3(\C)$ or $\GL_1(\C)$ so that $S_{\varphi}=\{1\}$. By Proposition~\ref{singprin} we obtain an $L$-packet consisting of a single principal series representation.
		
		\item\label{2b} $\cG_{\varphi}=\SL_2(\C)$. In this case $\varphi(\Fr)$ acts on $S^{\vee}=\SL_3(\C)$ by an outer automorphism. The unique outer class for $\SL_3$ is represented by $g\mapsto (g^t)^{-1}$ which acts as inversion on $\rZ_{\SL_3(\C)}=\varphi(I_F)$. This implies $q\equiv -1(3)$. Moreover, such an outer automorphism with fixed subgroup $\SL_2(\C)$ is unique up to $\SL_3(\C)$-conjugacy because in Proposition~\ref{loop} there is a unique $\SL_2$ vertex of the loop group $\SU_3/_{\C((t))}$. Hence there exists a unique such $\varphi$ that is discrete, with $S_{\varphi}=\rZ_{\SL_2(\C)}=\mu_2$. The $L$-packet consists of one non-supercuspidal representation and one supercuspidal representation.  
		By Property \ref{property:AMS-conjecture-7.8}, the non-supercuspidal is the generic constituent of the parabolic induction from $M_\beta$ --the long root $\GL_2(F)$-- of the depth-zero supercuspidal representation of $\GL_2(F)$ corresponding to $W_F\xrightarrow{\varphi^{\tau}} (\Z/3)\rtimes(\Z/4)\xrightarrow{\tau} \GL_2(\C)$ with some unramified twist specified by Property \ref{property:AMS-conjecture-7.8}. 
		Here $\tau$ is the unique irreducible yet faithful representation of $(\Z/3)\rtimes(\Z/4)$, and $\varphi^{\tau}$ maps $I_F$ onto $\Z/3$ and some lift $\Fr$ to $(0,1)$. More precisely, this depth-zero supercuspidal representation of $\GL_2(F)$ corresponds to the tame $L$-parameter $W_F\rightarrow \GL_2(\C)$ that sends a generator of tame inertia subgroup to $\matr{\zeta_3&0\\0&\zeta_3^{-1}}$ and a lift of $\Fr$ to $\matr{0&iq^{1/2}\\iq^{1/2}&0}$. In fact, this $L$-parameter is also of long $\SL_2$-type as in \S\ref{SL2}, and the above description is generalized there.
		
		If otherwise $\varphi$ is non-discrete, then $\varphi(\SL_2(\C))=1$, $S_{\varphi}=\pi_0(\cG_{\varphi})=\{1\}$. By Property \ref{property:AMS-conjecture-7.8}, we have a singleton $L$-packet consisting of a tempered representation contained in the parabolic induction from the long root $\GL_2(F)$ given by the finite-order homomorphism $\tau\colon W_F\rightarrow \GL_2(\C)$ described in the previous paragraph.
		\item\label{single} $\cG_{\varphi}=\SO_3$. Again in this case $\varphi(\Fr)$ acts on $S^{\vee}=\SL_3(\C)$ by an outer automorphism and hence $q\equiv -1(3)$. There is again a unique discrete such $\varphi$. In this case $S_{\varphi}=\{1\}$ and the $L$-packet consists of a single non-supercuspidal generic discrete series. 
		Under~\ref{property:AMS-conjecture-7.8} it is given by the generic constituent of the parabolically induced representation from the long root $\GL_2(F)$ of the depth-zero supercuspidal representation of $\GL_2(F)$ corresponding to $W_F\twoheadrightarrow  S_3\hookrightarrow \GL_2(\C)$ with some unramified twist specified by Property \ref{property:AMS-conjecture-7.8}. 
		Specifically, it is the depth-zero supercuspidal representation corresponding to the tame $L$-parameter $W_F\rightarrow \GL_2(\C)$ that sends a generator of the tame inertia subgroup to $\matr{\zeta_3&0\\0&\zeta_3^{-1}}$ and a lift of $\Fr$ to $\matr{0&q^{1/2}\\q^{1/2}&0}$.		We remark that compared to the non-supercuspidal representation in (\ref{2b}), the two depth-zero supercuspidal representations of $\GL_2(F)$ differ by an order four unramified twist. 

		If otherwise $\varphi$ is non-discrete, then $\varphi(\SL_2(\C))=1$, $S_{\varphi}=\pi_0(\cG_{\varphi})=\{1\}$. By the \ref{property:AMS-conjecture-7.8}, we have a singleton $L$-packet consisting of a tempered generic representation contained in the parabolic induction from the long root $\GL_2(F)$ given by the finite-order homomorphism from $W_F$ to $\GL_2(\C)$ that sends a generator of tame inertia to $\matr{\zeta_3&0\\0&\zeta_3^{-1}}$ and some lift of $\Fr$ to $\matr{0&i\\i&0}$ so that $\varphi(W_F)\cong S_3$, somewhat similar to the previous paragraph.
	\end{enumerate}
	\item $S^{\vee}\cong \SO_4(\C)$. In this case, $\varphi(I_F)=\rZ_{\SO_4(\C)}$ is up to conjugate the unique order $2$ subgroup of $G^{\vee}$. This is because $\varphi(I_F)$ has to be generated by either $\lambda(1,-1)$, $\lambda(-1,1)$ or $\lambda(-1,-1)$, yet the above three elements are themselves conjugate.
	\begin{enumerate}
	    \item\label{toy} If $\varphi$ is discrete,  we need the centralizer $\cG_{\varphi}:=\Cent_{S^{\vee}}(\varphi(\Fr))$ to be semisimple, which is only possible when $\cG_{\varphi}=S^{\vee}$, i.e. $\varphi(\Fr)\in \Cent_{G^{\vee}}(S^{\vee})=\rZ_{S^{\vee}}=\mu_2$. In this case $\varphi|_{W_F}$ factors through $\Gal(E/F)$ for some quadratic ramified Galois extension. There are two such choices. For each choice we have a unique $L$-parameter with $S_{\varphi}=\rZ_{\SO_4(\C)}=\mu_2$. Each $L$-packet consists of one non-supercuspidal representation and one supercuspidal representation. 
	    \item Suppose $\varphi$ is non-discrete. Then $\varphi(\Fr)$ commutes with $\varphi(I_F)=\rZ_{S^{\vee}}\cong\mu_2$. The centralizer $\Cent_{G^{\vee}}(\varphi)\subset S^{\vee}\cong \SO_4(\C)$ is connected except for one case when $\cG_{\varphi}=\Cent_{G^{\vee}}(\varphi)=\rS(\rO_2(\C)\times \rO_2(\C))$, in which case $\varphi(\Fr)$ is mapped to a non-central order $2$ element in $\SO_4(\C)$ (and all of them live in a single orbit). Outside the exception case, by Proposition~\ref{singprin} we always get a singleton $L$-packet consisting of a representation in the principal series.
	    \item\label{twoprin} For the exception case, we have a unique such $L$-packet. By Property~\ref{property:AMS-conjecture-7.8}, it has two tempered members. Each of them is contained in a principal series coming from a biquadratic order $2$ character.
	\end{enumerate}
	
	\item\label{d0pSL2} $S^{\vee}\cong \GL_2(\C)$. In this case, $\cG_{\varphi}=\Cent_{S^{\vee}}(\varphi(\Fr))$ can be $\GL_2$, $\GL_1$, $\SL_2$ or $\rO_2$. 
	\begin{enumerate}
	    \item The only discrete $\varphi$ arises in the unique semisimple case $\cG_{\varphi}=\SL_2$. In this case $\varphi$ is of $\SL_2$-type in \S\ref{SL2}.
	    \item Otherwise $\varphi$ is non-discrete. We note that $\varphi(I_F)$ is contained in $\rZ_{S^{\vee}}\cong \rZ_{\GL_2(\C)}=\GL_1(\C)$. In the case when $\cG_{\varphi}\cong \GL_2(\C)$, $S_{\varphi}=\{1\}$ and $\varphi(\Fr)$ commutes with $\rZ_{S^{\vee}}$, and hence we fall into the situation of Proposition \ref{singprin}. When $\cG_{\varphi}\cong \SL_2$ or $\GL_1$, $\varphi(W_F)$ is non-abelian yet $S_{\varphi}=\{1\}$. By the \ref{property:AMS-conjecture-7.8}, we get an $L$-packet consisting of a single representation contained in the parabolic induction from $\GL_2(F)$ (short root $\GL_2(F)$ in $G$ if $S^{\vee}\subset G^{\vee}$ is short root $\GL_2(\C)$, and both long otherwise) given by $\varphi|_{W_F}$ which has non-abelian image and thus necessarily supercuspidal.
	    \item\label{packetoftwo} Lastly, when $\cG_{\varphi}=\rO_2(\C)$, we have $\varphi(\SL_2(\C))=1$, but $S_{\varphi}\cong\mu_2$. In this case, $S^{\vee}\cong \GL_2(\C)$ is contained in its normalizer which is isomorphic to $\SO_4(\C)\subset G^{\vee}$, while $\varphi(\Fr)$ normalizes $S^{\vee}$ in an outer manner. Up to inner conjugation, there is a unique automorphism of $\GL_2(\C)$ with fixed subgroup isomorphic to $\rO_2(\C)$. Hence within $\SO_4(\C)$, up to conjugation, 
	    \[
	    \hskip0.7cm\varphi(I_F)\subset\left\{\matr{s&&&\\&s&&\\&&s^{-1}&\\&&&s^{-1}}\;|\;s\in S^1\subset\C^{\times}\right\},\;\varphi(\Fr)=\matr{&&1&\\&&&1\\1&&&\\&1&&}.
	    \]
	    Also, recall $\varphi(\SL_2(\C))=1$. We need $\varphi(I_F)$ to be large enough such that
	    \[
	    \hskip1.1cm \Cent_{G^{\vee}}(\varphi)=\left\{\matr{z&&&\\&z^{-1}&&\\&&z&\\&&&z^{-1}}\;|\;z\in\C^{\times}\}\sqcup\{\matr{&z&&\\z^{-1}&&&\\&&&z\\&&z^{-1}&}\;|\;z\in\C^{\times}\right\}.
	    \]
	    By Property~\ref{property:AMS-conjecture-7.8}, the $L$-packet consists of two non-discrete tempered representations, both contained in a parabolic induction from $\GL_2(F)$ (short root $\GL_2(F)$ in $G$ if $S^{\vee}\subset G^{\vee}$ is short root $\GL_2(\C)$, and both long otherwise) given by $\varphi|_{W_F}$ that has non-abelian image and thus necessarily supercuspidal. 
	    
	    We now make a few remarks.~We have $\Cent_{\SO_4(\C)}(\rO_2(\C))\cong \rO_2(\C)$, while $\Cent_{\SO_4(\C)}(\SO_2(\C))\cong \GL_2(\C)$. In particular, the supercuspidal representation of $\GL_2(F)$ is characterized by the property that its $L$-parameter has image in $\rO_2(\C)$ where $\Fr$ is sent to the non-trivial component and $I_F$ is sent to $\SO_2(\C)$. In other words, the central character of the supercuspidal representation of $\GL_2(F)$ is the unramified order-two character. Conversely, any supercuspidal representation of $\GL_2(\C)$ whose central character has order-two has its corresponding $L$-parameter mapping into some $\rO_2(\C)$ (but not into $\SO_2(\C)$) by Lemma \ref{solv}. By Property \ref{property:AMS-conjecture-7.8}, its parabolic induction to $G$ has a non-discrete constituent living in an $L$-packet of $2$ members. By Property \ref{property:L-packets}, these $2$ members are both tempered. See also \S\ref{d+p}(\ref{packetoftwo+}) for the case where the supercuspidal representation of $\GL_2(F)$ has positive-depth, and its central character is again of order two but possibly ramified.
	\end{enumerate}

	\item $S^{\vee}\cong \GL_1(\C)^2$ is a maximal torus of $G^{\vee}$. We have the following cases:
	\begin{enumerate}
	    \item\label{regular-sc-depthzero} For discrete $\varphi$, $\varphi(\Fr)$ acts on $S^{\vee}$ through an elliptic element in $W_{\bG}$, such that $S_{\varphi}\cong(\mu_2)^2$, $\mu_3$ and $\mu_1$, respectively, when the Weyl group element has order $2$, $3$ or $6$. In this case, the $L$-parameter is regular supercuspidal, and the Local Langlands has been determined in \cite{Kal-reg}. 
	    The size of the $L$-packet is $4$, $3$ or $1$ respectively.
        \item For non-discrete $\varphi$, $\varphi(\Fr)$ acts on $S^{\vee}$ through some non-elliptic element $w$ in the Weyl group $W_{\bG}$, giving us a non-discrete $\varphi$. Thus we have that $w$ is either trivial or a reflection. When $w$ is trivial,  $S_{\varphi}=\pi_0(S^{\vee})$ is connected and Proposition \ref{singprin} applies. When $w$ is a reflection, we have $S_{\varphi}=\pi_0(\Cent_{S^{\vee}}(\Fr))\cong (X_*(S^{\vee})_{\Fr})_{\mathrm{tor}}=\{1\}$. Hence we have a singleton $L$-packet consisting of an irreducible constituent of the parabolic induction of a supercuspidal representation of $\GL_2(F)$.
	\end{enumerate}

\end{enumerate}

\subsubsection{Positive-depth $L$-parameters}\label{d+p}

In this case, $G_0^{\vee}$ is a proper Levi subgroup of $G^{\vee}$, which is either $\GL_2(\C)$ or $\GL_1(\C)^2$. We have
\begin{enumerate}
	\item\label{posdep-GL1} When $G_0^{\vee}\cong \GL_1(\C)^2$ is a maximal torus of $G^{\vee}$, the whole $W_F$ acts on $G_0^{\vee}$ through a subgroup of the Weyl group, giving $S_{\varphi}\cong (\mu_2)^2$, $\mu_3$, $\mu_2$ or trivial, and thus an $L$-packet with $1$, $2$, $3$ or $4$ members. The members are regular supercuspidal representations if $X_*(G_0^{\vee})^{W_F}$ is trivial. Otherwise, $W_F$ acts on $X_*(G_0^{\vee})$ via at most a single reflection. In this case we have $S_{\varphi}=\{1\}$ (as in \S\ref{d0p}(\ref{regular-sc-depthzero})). When $W_F$ acts trivially on $X_*(G_0^{\vee})$, Proposition \ref{singprin} applies. Otherwise, when $W_F$ acts by a reflection we have a singleton $L$-packet consisting of a constituent of the parabolic induction from a supercuspidal representation of $\GL_2(F)$.
	\item When $G_0^{\vee}\cong \GL_2(\C)$, there are several cases for which $S^{\vee}=\Cent_{G_0^{\vee}}(\varphi(I_F))$. We first assume $\varphi$ is discrete: 
	\begin{enumerate}
		\item\label{posdep-2a} $S^{\vee}\supset \SL_2(\C)$. In this case we necessarily have $\cG_{\varphi}=\SL_2(\C)$ for $\varphi$ to be discrete. This implies that $\varphi$ is of $\SL_2$-type, and we refer to \S\ref{SL2}.

		\item\label{posdep-2b} $S^{\vee}$ is a maximal torus, in which case $\cG_{\varphi}=S_{\varphi}=S^{\vee}[2]\cong(\mu_2)^2$ is a regular supercuspidal $L$-parameter with $4$ members in its $L$-packet.
		\item\label{posdep-2c} $S^{\vee}$ is a torus of rank $1$. In this case $\cG_{\varphi}=S_{\varphi}=S^{\vee}[2]\cong\mu_2$ is a regular supercuspidal $L$-parameter with $2$ members in its $L$-packet.
		\item\label{ns} $S^{\vee}\cong \rO_2(\C)$ is disconnected. In this case $\cG_{\varphi}=S_{\varphi}\cong(\mu_2)^2$ is a non-singular supercuspidal $L$-parameter with $4$ members in its $L$-packet.
	\end{enumerate}
	Lastly, we turn to the non-discrete case in which $\dim \cG_{\varphi}>0$. We claim that we always have $\cG_{\varphi}$ connected and thus $S_{\varphi}=\{1\}$ except when $\Cent_{G^{\vee}}(\varphi)=\cG_{\varphi}=S^{\vee}\cong \rO_2(\C)$. This is done by investigating the possibilities for $\cG_{\varphi}$ in cases (\ref{posdep-2a})(\ref{posdep-2b})(\ref{posdep-2c}) above. In the last case (\ref{ns}) when $S^{\vee}\cong \rO_2(\C)$, to have non-discrete $L$-parameter we have either $\cG_{\varphi}\cong \SO_2(\C)$ or $\cG_{\varphi}\cong \rO_2(\C)$. In the former case we have again $S_{\varphi}=\{1\}$. Yet in the latter case:
	\begin{enumerate}
	\setcounter{enumii}{4}
	    \item\label{packetoftwo+} When $\cG_{\varphi}\cong \rO_2(\C)$, the $L$-parameter has the same centralizer and double centralizer (a different $\rO_2(\C)$) as in \S\ref{d0p}(\ref{packetoftwo}). Hence the $L$-packet consists of two non-discrete tempered constituents of parabolic inductions of supercuspidal representations of $\GL_2(F)$ with the same properties as described in \S\ref{d0p}(\ref{packetoftwo}), except for one difference: the central character of the supercuspidal representation of $\GL_2(F)$ may be unramified or ramified. Since $\cG_{\varphi}\cong \rO_2(\C)$, we have $\Cent_{G^{\vee}}(\cG_{\varphi})\cong \rO_2(\C)$, therefore the central character is the order-two character specified by $\varphi\colon W_F\rightarrow\pi_0(\Cent_{G^{\vee}}(\cG_{\varphi}))$. 
	\end{enumerate}
\end{enumerate}

\subsubsection{$L$-parameters and representations of $\SL_2$-type}\label{SL2}

\begin{definition} \label{defn:SL2-type} We say that an $L$-parameter $\varphi:W_F\times \SL_2(\C)\to \rG^{\vee}$ is \textit{of short $\SL_2$-type}, if it is discrete and $\varphi(\SL_2(\C))$ is a short root $\SL_2$ in $\rG^{\vee}=\rG_2(\C)$. We say that an $L$-parameter $\varphi:W_F\times \SL_2(\C)\to \rG^{\vee}$ is \textit{of long $\SL_2$-type}, if it is discrete and $\varphi(\SL_2(\C))$ is a long root $\SL_2$ in $\rG^{\vee}=\rG_2(\C)$. We say that it is \textit{of $\SL_2$-type} if it is of either short or long $\SL_2$-type.
\end{definition}

When $\varphi(\SL_2(\C))$ is a short (resp.~long) root $\SL_2(\C)$, we have that $\Cent_{G^{\vee}}(\varphi(\SL_2(\C)))$ is a long (resp.~short) root $\SL_2(\C)$, such that the two $\SL_2(\C)$'s live in an $\SO_4(\C)\subset G^{\vee}=\rG_2(\C)$. In particular, we have the following.

\begin{lemma} A discrete $L$-parameter is of short (resp.~long) $\SL_2$-type if and only if $\varphi(W_F)$ lies in a long (reps.~short) $\SL_2(\C)\subset G^{\vee}$. An $L$-parameter for which $\varphi(W_F)$ is contained in such an $\SL_2(\C)$ is discrete if and only if $\varphi|_{W_F}$ is a discrete $L$-parameter into $\SL_2(\C)$. Moreover, two $L$-parameters of short (resp.~long) $\SL_2$-type are conjugate if and only if the corresponding discrete $L$-parameter into $\SL_2(\C)$ are conjugate.
\end{lemma}

\begin{proof} It only remains to show the last statement. 
Note that the normalizer of a short or long $\SL_2(\C)$ in $G^{\vee}$ is the aforementioned $\SO_4(\C)$. The conjugacy actions of this $\SO_4(\C)$ on $\SL_2(\C)$ are all inner (factoring through the corresponding $\PGL_2(\C)$), hence the claim.
\end{proof}

Recall from \ref{defn:Lparameters}(3) the definition for a discrete $L$-parameter $\varphi\colon W_F\ra H^{\vee}$ (for some $H^{\vee}$) to be \textit{supercuspidal}--a terminology that makes sense since it is expected, and indeed implied by Property \ref{property:AMS-conjecture-7.8}, that the $L$-packet $\Pi_\varphi(G)$ consists of only supercuspidal representations. 
We will study the supercuspidal $L$-parameters $\varphi$ such that $\varphi(W_F)$ is contained in $\SL_2(\C)$. We have the following group-theoretic result on the LLC for $\GL_2$ (or $\PGL_2$):
\begin{lemma}\label{solv} Suppose $H\subset \GL_2(\C)$ is a solvable subgroup with $\rZ_{\SL_2(\C)}(H)$ finite. Then $H$ is contained in the normalizer of a unique maximal torus $T^{\vee}\subset \GL_2(\C)$, and $\rZ_{\SL_2(\C)}(H)=\rZ_{\SL_2(\C)}=\mu_2$.
\end{lemma}
\begin{numberedparagraph}\label{paragraph-after-SL2-type}
Hence any supercuspidal $L$-parameter into $\SL_2(\C)$ always corresponds to a \textit{compound $L$-packet} (i.e.~a Vogan $L$-packet) of two members, consisting of an irreducible supercuspidal representation of $\PGL_2(F)$ and another finite-dimensional irreducible representation of the compact inner form of $\PGL_2(F)$. We observe that for our original $L$-parameter $\varphi$, we have $\rZ_{G^{\vee}}(\varphi)=\rZ_{\SL_2(\C)}(\varphi(W_F))=\mu_2$ by Lemma \ref{solv}. Hence the $L$-packet $\Pi_\varphi(G)$ consists of a non-supercuspidal representation and a supercuspidal representation by Lemma~\ref{cusp}. Let $P$ be a parabolic subgroup of $G$ with Levi subgroup isomorphic to $\GL_2(F)$. Here the $\GL_2(F)$ corresponds to the short (resp.~long) root if $\varphi$ is of short (resp.~long) $\SL_2$-type. By Property~\ref{property:AMS-conjecture-7.8}, the non-supercuspidal member has the following supercuspidal support: $\varphi|_{W_F}\colon W_F\rightarrow \SL_2(\C)$, which gives us a supercuspidal $\tau$ of $\PGL_2(F)$. We consider $\tau_s:=\tau\otimes(\nu_F\circ\det)^s$ as a representation of $\GL_2(F)$. Then the non-supercuspidal of $G$ is the generic constituent in $\ii_{P}^{G}\tau_{\pm\frac{1}{2}}$ (both $s=\pm\frac{1}{2}$ work).
\end{numberedparagraph}

\subsubsection{Table for discrete $L$-parameters}
As a summary, we give a table for discrete $L$-parameters that are neither unipotent nor supercuspidal. In each row, any Galois-theoretic input as indicated gives a (unique) discrete $L$-parameter.

\[
\begin{array}{|c|c|c|c|c|c|}
	\hline
	\text{Label}&\text{Galois-theoretic input}&\varphi(W_F)&\cG_{\varphi}&S_{\varphi}&\varphi\left(\matr{1&1\\0&1}\right)\\\hline
	\begin{array}{c}\S\ref{d0p}(\ref{chal})\\\S\ref{d0s}(\ref{chals})\end{array}&\begin{array}{c}E/F\text{ ramified Galois cubic}\\\text{(only exist when }q\equiv 1(3))\end{array}&\Z/3&\SL_3&\mu_3&\text{subreg.}\\\hline
	\begin{array}{c}\S\ref{d0p}(\ref{2b})\\\S\ref{d0s}(\ref{u3})\end{array}&q\equiv -1(3)&(\Z/3)\rtimes(\Z/4)&\SL_2\text{ (long)}&\mu_2&\text{long}\\\hline
	\S\ref{d0p}(\ref{single})&q\equiv -1(3)&S_3&\SO_3\text{ (subreg.)}&1&\text{subreg.}\\\hline
	\begin{array}{c}\S\ref{d0p}(\ref{toy})\\\S\ref{d0s}(\ref{toys})\end{array}&E/F\text{ ramified quadratic}&\Z/2&\SO_4&\mu_2&\text{subreg.}\\\hline
	\S\ref{SL2}&\begin{array}{c}\text{supercuspidal $L$-parameter}\\\text{into }\SL_2(\C)\end{array}&\begin{array}{c}\text{non-split extension}\\\text{of }\{\pm1\}\text{ by }\Z/n,\\n\ge 4\text{ even }\end{array}&\SL_2\text{ (short)}&\mu_2&\text{short}\\\hline
	\S\ref{SL2}&\begin{array}{c}\text{supercuspidal $L$-parameter}\\\text{into }\SL_2(\C)\end{array}&\begin{array}{c}\text{non-split extension}\\\text{of }\{\pm1\}\text{ by }\Z/n,\\n\ge 4\text{ even }\end{array}&\SL_2\text{ (long)}&\mu_2&\text{long}\\\hline
	\end{array}
\]
Note that the second row is indeed a special case of the last row when the $L$-parameter is tame, $n=6$ and $|\varphi(I_F)|=3$.

\section{ Supercuspidal representations of \texorpdfstring{$\mathrm{G}_2$}{\mathrm{G2}}} \label{sec:supercuspidals-G2}

Our goal here is to attach, to every irreducible supercuspidal representation of $\rG_2(F)$, a cuspidal enhanced $L$-parameter into $G^{\vee}$. When the supercuspidal representation is unipotent, this has been done in \cite{Lu-padicI}, and when the supercuspidal representation is non-singular, this has been done in \cite{Kal-reg,Kaletha-nonsingular}. In this section, we shall focus on the cases \textit{not} covered in pre-existing literature. 
By \cite{Kim-exhaustion,Fintzen}, all supercuspidal representations of $\bG(F)$ are \textit{tame} supercuspidal representations in the sense of \cite{Yu} when $p\neq 2,3$. 

\subsubsection{Depth-zero supercuspidal representations}\label{d0s} As recalled in \S\ref{sec:preliminariesG2}, there are three maximal parahoric subgroups with reductive quotient $\rG_2$, $\SL_3$ and $\SO_4$ over $k_F=\Fq$.~By \eqref{eqn:depth zero supercuspidal}, depth-zero irreducible supercuspidal representations are in bijection with irreducible cuspidal representations on these reductive quotients $\bbG_x$.~By \eqref{Lusztig-unipotent-decomposition} and Proposition~\ref{property-Lusztig-DL-map}, any such cuspidal representation is labeled with 
\begin{itemize}
    \item[(i)] a semisimple conjugacy class $(s)$ with
    $s$ an \textit{isolated} element of $\bbG_x^\vee$, 
i.e. $\Cent_{\bbbG_x^\vee}(s)$ is not contained in a $\Fq$-rational Levi subgroup of any proper $\Fq$-rational parabolic subgroup of $\Cent_{\bbbG_x^{\vee}}(s)$;
    \item[(ii)] a unipotent cuspidal representation of the group of the $k_F$-rational points of the reductive algebraic group  $\bbbH:=\Cent_{\bbbG_x^{\vee}}(s)^\vee$, where $k_F=\Fq$. We denote $\bbH:=\bbbH(k_F)$.
\end{itemize}

\begin{enumerate}
	\item\label{Gx=G2}
	We begin with the case $\bbbG_x=\rG_2$, i.e. ~$x=x_0$ and $G_{x_0}$ is the hyperspecial maximal parahoric subgroup of $G$.~By~\cite[Prop.~4.9, Theorem~5.1 and Table I]{Bonnafe}, the group $\bbbH^{\vee}$ can be $\rG_2$, $\SL_3$, $\SU_3$, (split) $\SO_4$, $\rU_2$ or any elliptic maximal torus. 
By \cite{Chang-Ree}, when $\bbbH^{\vee}$ is of type $\rA_2$,
\begin{equation} \label{eqn:type centralizer}
\bbbH^\vee(\Fq)=(\Cent_{\bbbG_{x_0}^{\vee}}(s))(\Fq)=\begin{cases} 
 \SL_3(q)&\text{if $q\equiv 1\mod 3$,}\cr
 \SU_3(q)&\text{if $q\equiv -1\mod 3$.}
 \end{cases}
\end{equation}
By the classification of unipotent cuspidal representations in \cite{Lusztig-Madison},
\begin{itemize}
    \item 
neither $\SL_3(\Fq)$ nor $\SO_4(\Fq)$ has cuspidal unipotent representations. Indeed, for $m\ge 2$, the group $\SO_{2m}(\Fq)$ has a (unique) cuspidal unipotent representation only if $m$ is the square of an integer;
\item $\SU_n(\Fq)$, $n=k(k+1)/2$ are the only projective unitary groups that possess unipotent cuspidal representations, and each one of them has a unique unipotent cuspidal representation, which corresponds to the partition $(k, k-1,\ldots,1)$ of $n$.
\end{itemize}
Thus the only possibilities for $\bbbH(\Fq)$ equipped with a cuspidal unipotent representation are $\rG_2(\Fq)$, $\SU_3(\Fq)$, or torus. 

\begin{enumerate}
	\item\label{d0s-Gx=G2=unipotent} In the first case, i.e.~when $\bbbH=\rG_2$, the corresponding supercuspidal and cuspidal representations are unipotent. The group $\rG_2(\Fq)$ has four cuspidal unipotent  irreducible representations (see \cite[Theorem~3.28]{Lusztig-Madison}), denoted as $\rG_2[1]$, $\rG_2[-1]$, $\rG_2[\zeta_3]$ and $\rG_2[\zeta_3^2]$. Their dimensions are computed in \cite[p.~460]{Carter-book} and recorded in Table~\ref{table:G2FQ}. 
The corresponding irreducible supercuspidal unipotent representations are
\begin{equation} \label{eqn:usc}
\pi[\zeta_3]:=\ii_{\bbG_{x_0}}^G \rG_2[\zeta]\quad \text{for $\zeta\in \{1,-1,\zeta_3,\zeta_3^2\}$.}
\end{equation}
With Haar measure on $G$ normalized as in 
\eqref{eqn:DR}, by \eqref{formal-degree-depth0-sc-formula} the representation
$\pi[\zeta]$ has formal degree
\begin{equation} \label{eqn:formal degree unipotent supercuspidal}
\fdeg(\pi[\zeta])=\frac{q}{|\rG_2(\Fq)|_{p'}}\cdot \dim \rG_2[\zeta]=\frac{q}{(q^6-1)(q^2-1)}\cdot \dim \rG_2[\zeta].
\end{equation}
\begin{center} 
\begin{table} [ht]
\begin{tabular}{|c|c|c|c|}
\hline
Representation $\rG_2[\zeta]$ &Dimension of $\rG_2[\zeta]$ & $\fdeg(\pi[\zeta])$ \cr
\hline\hline
$\rG_2[1]$&  $\frac{q(q-1)^2(q^3+1)}{6(q+1)}$ & $\frac{q^2-q+1}{6(q+1)^2(q^4+q^2+1)}$\cr
\hline
 $\rG_2[-1]$&$\frac{q(q-1)(q^6-1)}{2(q^3+1)}$ &$\frac{q^2}{2(q^3+1)(q+1)}$ \cr
 \hline
 $\rG_2[\zeta_3]$,  $\rG_2[\zeta_3^2]$ & $\frac{q(q^2-1)^2}{3}$ &$\frac{1}{3(q^4+q^2+1)}$ \cr
  \hline
\end{tabular}
\vskip0.2cm
     \caption{\label{table:G2FQ} {Unipotent supercuspidal  irreducible representations of $\rG_2(F)$}. $\qquad$ $\qquad$ $\quad$}
\end{table}
\end{center}
\begin{remark}
Let $P(X):=\sum_{w\in W_G}X^{\ell(w)}$ be the Poincar\'e polynomial of $W_G$, 
with $\ell$ the length function on $W_\fini$.
We have
\begin{equation}
P(X)=\frac{(X+1)(X^6-1)}{X-1}.
\end{equation}
By \eqref{eqn:formal degree unipotent supercuspidal} and \eqref{eqn:volume Iwahori}, we have
\begin{equation} \label{eqn:fgusc}
\fdeg(\pi[\zeta])=\frac{q}{(q-1)^2\cdot P(q)}\cdot \dim \rG_2[\zeta]=\frac{\dim \rG_2[\zeta]}{\vol(\mcI_{\rG_2})\cdot P(q)}.
\end{equation}
\end{remark}
	\item\label{regular-d0s} In the case where $\bbbH^{\vee}$ is any elliptic maximal torus, the representation $\tau$ is $k_F$-non-singular. Since $G_{x_0}$ is a hyperspecial maximal parahoric subgroup of $G$, the notion of $k_F$-non-singularity and $F$-non-singularity agree in this case. Thus the supercuspidal representation $\pi$ is non-singular (equivalently it is also regular).
\item\label{u3} In the last remaining case where $\bbbH^\vee=\SU_3$. Note that $\rZ_{\SU_3}=\rZ_{\bbbH^{\vee}}$ needs to contain $s$ as a non-zero rational point, which implies that $q\equiv-1(3)$. Conversely, if $q\equiv -1(3)$, the Coxeter torus of $\bbG_x$ has two order-$3$ elements whose centralizers are $\SU_3$. In fact, all order-$3$ elements $s\in \rG_2(\bar{\F}_q)$ with $\rZ_{\rG_2}(s)\cong \SL_3(\bar{\F}_q)$ lie in a single geometric orbit. Since $\SU_3$ is connected ($\rG_2$ is simply connected), all such rational order-$3$ elements lie in a rational orbit. Hence there is one choice for $\bbbH^{\vee}\cong \SU_3$ when $q \equiv -1(3)$, and no choice when $q \equiv 1(3)$. When $q\equiv -1\mod 3$, the group $(\Cent_{\bbbG_{x_0}^\vee}(s))(\Fq)$ is the special unitary group $\SU3(\Fq)$, and $\tau_{\unip}$ is the unique irreducible cuspidal unipotent representation of $\SU3(\Fq)$, which is parametrized by the partition $(2,1)$ of $3$. 
It is clear that in this case, the representation $\tau$ is $k_F$-singular, in particular, the representation $\pi$ is singular by \cite[3.1.4]{Kaletha-nonsingular}. Consequently, there exists one singular non-unipotent cuspidal representation $\tau$ of $\bbbG_x(\Fq)$ when $q\equiv -1(3)$, and none otherwise. 
Therefore, when $q\equiv -1(3)$, there is a singular depth-zero supercuspidal representation of $G$, and we need to find the non-supercuspidal representation in its $L$-packet. We do so by computing the formal degree of $\pi$. 
By \cite[\S13.7]{Carter-book}, $\dim(\tau_\unip)=q(q-1)$. By \eqref{eqn:dim tau} 
\begin{equation} \label{eqn:dim tau x0}
\dim(\tau)=\frac{|\bbG_{x_0}|_{p'}}{|(\Cent_{\bbbG_{x_0}^\vee}(s))^{\vee}(\Fq)|_{p'}}\,\dim(\tau_\unip).
\end{equation}
We have 
\begin{equation}\label{eqn:order-PU3}
  |\SU_3(\Fq)|=q^6(q^{3/2}+1)(q-1)(q^{1/2}+1). 
\end{equation}
By \eqref{eqn:DR}, we have 
\begin{equation}\label{eqn-DR-applied-to-sc}
    \mathrm{Vol}(G_{x_0})=q^{-\mathrm{rk}(G)/2}|\mathbb{G}_{x_0}|_{p'}=q^{-1}|\mathbb{G}_{x_0}|_{p'}.
\end{equation}
Since $\pi=\ii_{\bbG_{x_0}}^G\tau$, by combining \eqref{eqn:dim tau}, \eqref{eqn:order-PU3} and \eqref{eqn-DR-applied-to-sc}, we have 
\begin{equation}\label{fdeg-singular-nonunipotent-sc}
\begin{split}
   \fdeg(\pi)=\frac{\dim\tau}{\mathrm{Vol}(G_{x_0})}=\frac{\dim(\tau_\unip)}{|(\Cent_{\bbbG_{x_0}^\vee}(s))^{\vee}(\Fq)|_{p'}\cdot q^{-1}}.
\end{split}
\end{equation}
	\end{enumerate}
\item\label{Gx=SL3} Next we look at the case where $\bbbG_x=\SL_3$, i.e.~$x=x_1$.
	The center of $\SL_3$, which consists of the scalar matrices $a\Id_3$ such that $a^3=1$, is finite, hence disconnected. By \cite[Proposition 5.2]{Bonnafe}, the possibilities for $\bbbH^{\vee}$ are $\PGL_3$ (i.e.~when $d=1$ and $s=\Id$ \textit{loc.cit.}) and
$\bbbT^\vee\rtimes\mu_3$ (i.e.~when $d=3$ and $s=\mathrm{diag}\left(1,\zeta_3,\zeta_3^2\right)\mod \F_q^{\times}$ for the primitive third root of unity $\zeta_3$, by \textit{loc.cit.}~in this case the Weyl group of $\bbbH^\circ=\Cent_{\bbG_{x_1}^{\vee}}(s)^\circ$ is trivial, thus $\bbbH^\circ=\bbbT^\vee$) or $\bbbT^\vee$ is an elliptic maximal torus of $\PGL_3$ (they are all rationally conjugate).  Let us discuss these cases:
	\begin{enumerate}
		\item\label{Gx=SL3-unipotent} $\bbbH=\SL_3$, 
The group $\SL_3$ does not admit cuspidal unipotent representations, so the case $\bbbH=\SL_3$ does not occur.
\item\label{Gx=SL3-regular} $\bbbH^{\vee}=\bbbT$. This happens if $s\in \bbbT(\Fq)$ is not of order $1$ or $3$. In this case, one checks that $\rZ_{G^{\vee}}(s)$ is a torus, i.e. we get only regular supercuspidal representations.
\item\label{chals} 
		$\bbbH^{\vee}\cong \bbbT\rtimes\mu_3$. This happens if $s$ is an order-$3$ element in $\bbbT(\Fq)\cong\F_{q^3}^{\times}/\Fq^{\times}$. Such an element exists if and only if $q\equiv 1(3)$. Note that while all the $T$'s are conjugate, the two order-$3$ elements are \textit{not} rationally conjugate in $\PGL_3(\Fq)$, because they map to the two different non-trivial classes in $\operatorname{coker}(\SL_3(\Fq)\rightarrow \PGL_3(\Fq))$. Hence when $q\equiv 1(3)$, we have two choices of such $s$'s. The Deligne-Lusztig induction for each $s$ is a direct sum of $3$ cuspidal representations of $\SL_3(\Fq)$ of the same dimension. Hence we get $2$ families, of $3$ singular supercuspidal representations all of the same formal degree. 
		
Using \eqref{formal-degree-depth0-sc-formula}, we compute the formal degree as follows (note that $p\neq 2,3$): we have $|\mu_3|_{p'}=3$ and $|\mathbb{T}|=\frac{q^3-1}{q-1}=q^2+q+1$, thus 
\begin{equation}
    \fdeg(\ii_{\bbG_{x_1}}^G\tau)=\frac{q^{2/2}\cdot 1}{3(q^2+q+1)}=\frac{q}{3(q^2+q+1)}.
\end{equation}
Hence we obtain
\begin{equation}\label{eqn:fdegx1}
    \fdeg(\ii_{\bbG_{x_1}}^G\tau)=\frac{q}{3(q^2+q+1)}.
\end{equation}
	\end{enumerate}
	\item\label{Gx=SO4} Lastly, we look at the case $\bbbG_x=\SO_4$. A centralizer subgroup in $\bbbG_x^{\vee}\cong \SO_4$ has unipotent cuspidal representations only if its identity component is a torus. Hence we begin with $\bbbT^{\vee}\subset\bbbG_x^{\vee}$ any representative of the unique rational conjugacy class of elliptic maximal torus. Abstractly, $\bbbT^{\vee}\cong \rU_1\times\rU_1$ over $\Fq$ so that $\bbbT^{\vee}(\Fq)\cong(\Z/(q+1))^2$. We discuss three scenarios:
	\begin{enumerate}
		\item\label{Gx=SO4-regular} $s\in \mathbb{T}^{\vee}(\Fq)$ is not perpendicular to any coroot of $G$. In this case we have non-singular supercuspidal representations, and in fact regular ones because $G^{\vee}$ is simply connected.
	\item\label{tameSL2} $s\in \bbbT^{\vee}(\Fq)$ is perpendicular to a unique pair of coroots $\pm\alpha^{\vee}$ of $G$. The cocharacter lattice is generated by $\alpha^{\vee}$ and another $\lambda$. If $\alpha^{\vee}$ is a short coroot, then $s$ is not perpendicular to any other coroots if and only if $(2\lambda)(s)=\lambda(s)^2\not=1$. If $\alpha^{\vee}$ is a long coroot, then $s$ is not perpendicular to any other coroots if and only if $\lambda(s)^2\not=1\not=\lambda(s)^3$. The rational conjugacy class of $s$ is determined by $\{\lambda(s)^{\pm1}\}$. Hence there are in general $\frac{q-1}{2}$ such $s$ unless $q\equiv -1(3)$ and $\alpha^{\vee}$ is a long coroot, in which case there are $\frac{q-3}{2}$ such $s$. Each $s$ gives a singular supercuspidal matched with those in $\mathsection$\ref{d0p}(\ref{d0pSL2}), or rather the tamely ramified $L$-parameters in $\mathsection$\ref{SL2}.
		\item\label{toys} 
		$s\in \mathbb{T}^{\vee}$ is perpendicular to two pairs of coroots of $G$, which correspond to $\SO_4(\C)\subset \rG_2(\C)$. This is the case when $s$ is the unique (up to rational conjugacy) order $2$ element in $\bbbT^{\vee}(\Fq)$ with $\Cent_{\bbbG_x^{\vee}}(s)^{\circ}=\bbbT^{\vee}$. In this case $\Cent_{\bbbG_x^{\vee}}(s)\cong\bbbT^{\vee}\rtimes\mu_2$. Indeed, \begin{equation} \label{eqn:centralizer}
(\Cent_{\bbbG_{x_2}^{\vee}}(s))(\Fq)=\left\{(g_1,g_2)\in\rO_2^-(\Fq)\times\rO_2^-(\Fq)\;:\;\det(g_1)=\det(g_2)\right\}
,\end{equation}
where $\rO_2^-$ denotes the non-split form of $\rO_2$.
The group $\SO_2^-(\Fq)$ has a unique cuspidal unipotent representation, the trivial representation.  Thus $\rO_2^-(\Fq)$ admits two cuspidal unipotent representations: the trivial one and the sign representation. So $(\Cent_{\bbbG_{x_2}^{\vee}}(s))^\vee(\Fq)$ has two cuspidal unipotent irreducible representations: $1\otimes 1$ and $\sign\otimes\sign$. Both have dimension one.
In this case, one can check that $\Cent_{G^{\vee}}(\widetilde{s})^{\circ}=\Cent_{G^{\vee}}(\widetilde{s})=\SO_4(\C)$, and thus $\Cent_{G^{\vee}}(\varphi|_{I_F})=\SO_4(\C)$. In particular, $\Cent_{G^{\vee}}(\varphi|_{W_F})$ is not finite, and the corresponding supercuspidal representations of $\rG_2(F)$ are singular.  The Deligne-Lusztig representation is the direct sum of two cuspidal representations of the same dimension. Hence we get two supercuspidal representations to be matched in the $L$-packets in \ref{d0p}(\ref{toy}).

We have
$|\rO_2^-(\Fq)|=2(q+1)$, which gives $|\rO_2^-(\Fq)\times\rO_2^-(\Fq)|=4(q+1)^2$ by \cite{Carter-book}, and hence
\begin{equation} \label{eqn:order}
|(\Cent_{\bbbG_{x_2}^{\vee}}(s))(\Fq)|=\frac{|\rO_2^-(\Fq)\times\rO_2^-(\Fq)|}{2}=2(q+1)^2.
\end{equation}
By \eqref{formal-degree-depth0-sc-formula}, the formal degree of $\ii_{\bbG_{x_2}}^G\tau$ is given as follows:
\begin{equation} \label{eqn:fdegx2}
\fdeg(\ii_{\bbG_{x_2}}^G\tau)=\frac{q^{4/2}}{2(q+1)^2}=\frac{q^2}{2(q+1)^2},
\end{equation}
which agrees with the formal degree for $\pi(\eta_2)$ from Table \ref{table:unip-ppal-series-case-3aquad-ramified}. 
\end{enumerate}
\end{enumerate}

\subsubsection{Positive-depth supercuspidal representations}\label{d+s} Since $p\ne 2,3$ any supercuspidal representation $\pi$ of $G$ is constructed via a Yu datum $\cD$. The latter includes a tamely ramified twisted Levi sequence in $\bG$ is a finite sequence $\vec\bG=(\bG^0\subset\bG^1\subset\cdots\subset\bG^d=\bG)$ of twisted Levi subgroups of $\bG$ that splits over a tamely ramified extension $E$ of $F$ (i.e.,  $\bG_E^i:=\bG^i\times_F E$ is a split Levi subgroup of $\bG_E^{i+1}$ for each $i$), such that $\rZ_{\bG^0}$ is anisotropic (since $\rZ_{\bG}=\{1\}$), a positive-depth character on $G^0:=\bG^0(F)$, and a depth-zero supercuspidal of $G^0$. If the depth of $\pi$ is positive then we have $d\ge 1$. The group $\bG$ being $F$-split, we have $\bG_E=\bG$. The only possibilities for $\bG_E^0$ are $\bG$, $\GL_2(E)$ or a maximal $E$-split torus $\bS$. Thus the only possibilities for $\bG^0$ are $\bG$, a torus or  a unitary group (a priori it could be either $\rU(1,1)$ or $\rU_2$, but since $\rZ_{\bG^0}$ is anisotropic, only the compact unitary group $\rU_b(2)$, for the quadratic extension $F(\sqrt b)$ is possible)).

The representation $\pi$ is regular (resp.~non-singular) if and only if the depth-zero supercuspidal of $G^0$ is. Hence the question is largely about classifying such $G^0$ and their depth-zero supercuspidal representations. We note that $\bG^0$ is determined by $\bZ_0:=\rZ_{\bG^0}$, an anisotropic subtorus of $\bG$ satisfying $\rZ_{\Cent_\bG(\bZ_0)}=\bZ_0$. Also since we need $G^0$ to have some positive-depth character, $1\le \dim \bZ_0\le\rank \bG=2$. We now discuss the possibilities for $\bZ_0$.

\begin{enumerate}
	\item\label{posdep-most-regular} $\dim \bZ_0=2$, then $\bZ_0$ is a maximal torus of $\bG$. In this case, we obtain only regular supercuspidal representations of positive depths.
	\item $\dim \bZ_0=1$, such that $\bG^0/\bZ_0\cong \PGL_2$ is split. In this case, all depth-zero supercuspidal representations of $G^0$ coming from cuspidal representations of reductive quotients of $G^0$ are non-singular. 
	\begin{enumerate}
	    \item\label{posdep-PGL2-reg} They are regular unless, 
	    \item\label{posdep-PGL2-nonsing} $\bZ_0$ is ramified such that $\bG^0\cong \rU_2$ is ramified for which the reductive quotient $\SL_2$ has a unique pair of non-singular (but not regular) supercuspidal representations. Since there are two ramified quadratic extensions, we obtain two possible groups $G^0$ and four such non-singular supercuspidal representations in the $L$-packet in $\mathsection$\ref{d+p}(\ref{ns}). 
	\end{enumerate}
	\item\label{wilddih} $\dim \bZ_0=1$ such that $\bG^0/\bZ_0$ is anisotropic, i.e.~a compact inner form of $\PGL_2$. In this case, 
	\begin{enumerate}
	    \item\label{inner-form-PGL2-reg} $G^0$ has regular depth-zero supercuspidal representations;
	\item\label{inner-form-PGL2-singular} $G^0$ also has singular depth-zero supercuspidal representations. Any such singular depth-zero supercuspidal  representation has the property that its restriction to the derived subgroup is the trivial representation. Since $\bG_{\der}^0:=(\bG^0)_{\der}$ is simply connected, such a singular supercuspidal representation is a character of $G^0/\bG_{\der}^0(F)$. All these representations arise from $L$-parameters \textit{of $\SL_2$-type} as in \S\ref{SL2} and live in a mixed $L$-packet together with a non-supercuspidal representation, as described in \S\ref{paragraph-after-SL2-type}.  
	\end{enumerate}
\end{enumerate}

\section{ Non-supercuspidal representations of \texorpdfstring{$\mathrm{G}_2$}{\mathrm{G2}}} \label{subsec:class-s}\
Let $\fs=[L,\sigma]_G$, where $L$ is a proper Levi subgroup $L$ of $G=\rG_2(F)$. Let $\sigma$ be an irreducible supercuspidal representation of $L$. In the case of $G=\rG_2(F)$, the Levi $L$ being isomorphic to either $\GL_2(F)$ or $F^\times\times F^\times$, all of its irreducible supercuspidal representations are regular. Thus $\sigma$ is a regular supercuspidal, and we can consider an $L$-parameter $\varphi_\sigma\colon W_F\to L^\vee$ for $\sigma$, as defined in \cite{Kal-reg} or \cite{Bushnell-Henniart}. Both constructions coincide as shown in \cite{Oi-Tokimoto}.

\begin{lemma} \label{lemma:cases}\ 

\begin{enumerate}
\item[{\rm(1)}]
If $M=T$, we have $W^\fs_G\ne\{1\}$ if and only if $\fs=[T,\xi\otimes\xi]_G$ or
$\fs=[T,\xi\otimes 1]_G$ with $\xi$ an irreducible character of
$F^\times$.
\item[{\rm(2)}]
If $M=M_\alpha$ or $M=M_\beta$, we have $W^\fs_G\ne\{1\}$ if and only if $\sigma$ is self-dual.
\end{enumerate}
\end{lemma}
\begin{proof}
(1) Recall the definition from \ref{eqn:NsI}, we have
\begin{equation}
W^\fs_G=\left\{w\in W\;:\;w\cdot(\xi_1\otimes\xi_2)=\chi(\xi_1\otimes\xi_2)
\;\text{ for some $\chi\in \fX_{\un}(T)$}\right\}.
\end{equation}
Let $\sigma^\circ:=
\xi_1|_{\fo_F^\times}\otimes\xi_2|_{\fo_F^\times}$. Then we have 
\begin{equation}
W^\fs_G=\{w\in W\;:\;w\cdot\sigma^\circ=\sigma^\circ\}.
\end{equation}
Let $\xi_i^\circ:=\xi_i|_{\fo_F^\times}$.
By Table~\ref{table:1}, it follows that we have
$W^\fs_G=\{1\}$ if and only if
\[\xi_1^\circ\ne 1,\;\; \xi_2^\circ\ne 1,\;\;
\xi_1^\circ\xi_2^\circ\ne 1,\;\;\xi_1^\circ\ne\xi_2^\circ,\;\;
(\xi_1^\circ)^2\xi_2^\circ\ne 1,\;\;\xi_1^\circ(\xi_2^\circ)^2\ne 1.
\]
Hence we have $W^\fs_G\ne\{1\}$ if and only if we are in one of the
following cases:
\begin{enumerate}
\item[(i)] We have $\xi_1^\circ=\xi_2^\circ$.
We may and do assume that $\xi_1=\xi_2=\xi$.
\item[(ii)] We have $\xi_2^\circ=1$. We may and do assume that
$\xi_1=\xi$ and $\xi_2=1$.
\end{enumerate}
(2) This is clear. See for example \cite[4.1.9]{aubert-xu-Hecke-algebra} (or \cite{ShahidiThird}). 
\end{proof}
\begin{proposition} \label{principal-series-classification-s-WGs}
Suppose that $W^\fs_G\ne\{1\}$.  Then the possibilities for $\fs$ and $W^\fs_G$ are as follows:
\begin{itemize}
\item[(1)]
$\fs=[T,1]_G$. Here $W^\fs_G=W$. 
\item[(2)]
$\fs=[T,\xi\otimes 1]_G$ with $\xi$ ramified non-quadratic. Here $W^\fs_G\simeq\ZZ/2\ZZ$.
\item[(3)]
$\fs=[T,\xi\otimes\xi]_G$ with $\xi$ ramified, neither quadratic nor
cubic. Here $W^\fs_G\simeq\ZZ/2\ZZ$. 
\item[(4)]
$\fs=[T,\xi\otimes\xi]_G$ with $\xi$ ramified quadratic. Here $W^\fs_G\simeq\ZZ/2\ZZ \times \ZZ/2\ZZ$.
\item[(5)]
$\fs=[T,\xi\otimes\xi]_G$ with $\xi$ ramified cubic. Here $W^\fs_G\simeq S_3$.
\item[(6)]
$\fs=[M_\beta,\sigma]_G$ with $\sigma$ self-dual. Here $W^\fs_G\simeq\ZZ/2\ZZ$.
\item[(7)]
$\fs=[M_\alpha,\sigma]_G$ with $\sigma$ self-dual. Here $W^\fs_G\simeq\ZZ/2\ZZ$.
\end{itemize}
\end{proposition}
\begin{proof}
We first consider the case $M=T$ and describe the $W$-orbits. We have
\begin{equation} \label{firstorbit} 
W\cdot(\xi\otimes\xi)=\{\xi\otimes\xi, \ \,\xi^{-1}\otimes\xi^{-1},
\ \, \xi^2\otimes \xi^{-1}, \ \, \xi^{-1}\otimes\xi^2, \ \,
\xi\otimes\xi^{-2}, \ \, \xi^{-2}\otimes\xi\}.\end{equation}
It follows that
\begin{equation} \label{eqn:orbits}
W\cdot(\xi\otimes\xi)=
\begin{cases}\xi\otimes\xi,\ \,\xi\otimes 1,
\ \,1\otimes\xi&\text{if $\xi$ is quadratic,}\cr
\xi\otimes\xi,\ \,\xi\otimes\xi^{-1},\ \,\xi^{-1}\otimes \xi,\ \,
\xi^{-1}\otimes\xi^{-1}&\text{if $\xi$ is
cubic.}\end{cases}
\end{equation}
We have
\begin{equation}
|W\cdot(\xi\otimes\xi)|=\begin{cases}
1&\text{if $\xi$ is trivial,}\cr
3&\text{if $\xi$ is quadratic,}\cr
4&\text{if $\xi$ is cubic,}\cr
6&\text{otherwise.}
\end{cases}
\end{equation}
On the other hand, we have
\begin{equation} \label{secondorbit}
W\cdot(\xi\otimes 1)=\{\xi\otimes 1, \ \,1\otimes\xi,
\ \, \xi\otimes \xi^{-1}, \ \, \xi^{-1}\otimes 1, \ \,
1\otimes\xi^{-1}, \ \, \xi^{-1}\otimes\xi\}.\end{equation}
If $\xi$ is quadratic, then we have
\begin{equation} \label{eqn:orbit}
W\cdot(\xi\otimes 1)=\{\xi\otimes\xi,\ \,\xi\otimes 1, \ \,1\otimes\xi\}.
\end{equation}
Thus we have
\begin{equation}
|W\cdot(\xi\otimes 1)|=\begin{cases}
1&\text{if $\xi$ is trivial,}\cr
3&\text{if $\xi$ is quadratic,}\cr
6&\text{otherwise.}
\end{cases}
\end{equation}
This gives the following possibilities for $\fs$:

\noindent
(2) 
If $\xi_2=1$ and $\xi_1=\xi$ with $\xi$ a ramified non-quadratic character, by~(\ref{secondorbit}) we have
\begin{equation}
\begin{array}{ccccccc}\fs&=&[T,\xi\otimes 1]_G&=&[T,1\otimes\xi]_G&=&
[T,\xi\otimes \xi^{-1}]_G\cr
&=&[T,\xi^{-1}\otimes 1]_G&=&[T,1\otimes\xi^{-1}]_G&=&[T,\xi^{-1}\otimes\xi]_G.
\end{array}
\end{equation}
It follows from Table~\ref{table:1} that
\begin{equation}
W^\fs_G=\{e,b\}\cong \ZZ/2\ZZ.
\end{equation}

\noindent
(3) If $\xi_1=\xi_2=\xi$ with $\xi$ a
ramified character that is neither quadratic nor cubic, by~(\ref{firstorbit}) we have
\begin{equation}
\begin{array}{ccccccc}
\fs&=&[T,\xi\otimes\xi]_G&=&[T,\xi^{-1}\otimes\xi^{-1}]_G&=&
[T,\xi^2\otimes \xi^{-1}]_G\cr
&=&[T,\xi^{-1}\otimes\xi^2]_G&=&[T,\xi\otimes\xi^{-2}]_G
&=&[T,\xi^{-2}\otimes\xi]_G.
\end{array}
\end{equation}
It follows from Table~\ref{table:1} that
\begin{equation}
W^\fs_G=\{e,a\}\cong \ZZ/2\ZZ.\end{equation}

\noindent
(4) If $\xi_1=\xi_2=\xi$ with $\xi$ a ramified quadratic character, 
since $\fs=[T,\xi_1\otimes\xi_2]_G =[T,w\cdot(\xi_1\otimes\xi_2)]_G$, we have 
$\fs=[T,\xi\otimes\xi]_G=[T,\xi\otimes 1]_G=[T,1\otimes \xi]_G$. Then Table~\ref{table:1} gives
\begin{eqnarray}
W^{\fs}&=&\left\{e,a,babab,bababa\right\}
\,=\,\left\{e,a,r^3,ar^3\right\}\nonumber\\
&=&\langle s_\alpha, s_{3\alpha + 2\beta}\rangle
\,\cong\, \ZZ/2\ZZ \times \ZZ/2\ZZ.\nonumber\end{eqnarray}
\noindent
(5) If $\xi_1=\xi_2=\xi$, with $\xi$ a
ramified character of order $3$, we have
\begin{equation}
\fs=[T,\xi\otimes\xi]_G = [T,\xi^{-1}\otimes\xi^{-1}]_G=
[T,\xi\otimes\xi^{-1}]_G=[T,\xi^{-1}\otimes\xi]_G.
\end{equation}
It follows from Table~\ref{table:1} that
\begin{equation}
W^\fs_G=\{e,a,bab,abab,baba,ababa\}\cong S_3.
\end{equation}
(6)\&(7) We now consider the cases $M=M_\alpha$ and  $M=M_\beta$. We have $\Nor_G(M)/M\simeq\ZZ/2\ZZ$. Since $\sigma$ is self-dual, by \cite[4.1.9]{aubert-xu-Hecke-algebra} we have $W^\fs_G=\Nor_G(M)/M$, thus $W^{\fs}_G\simeq \Z/2\Z$. 
\end{proof}

\section{ The intermediate series} \label{sec:interm_series}\
Let $\fs=[M,\sigma]_G$, where $\sigma$ is a supercuspidal irreducible representation of $M\in\{M_\alpha,M_\beta\}$. Let $\omega_{\sigma}$ denote the central character of $\sigma$. Let $P$ be a parabolic subgroup of $G$ with Levi subgroup $M$, and let $\ii_P^G(\sigma)$ denote the (normalized) parabolic induction of $\sigma$. Since $M\simeq\GL_2(F)$, the representation $\sigma$ is a \textit{regular} supercuspidal in the sense of Definition~\ref{defn:Kaletha regular}. Therefore, the $L$-parameter $\varphi_{\sigma}$ has trivial restriction to $\SL_2(\C)$. Thus the corresponding unipotent class in $\GL_2(\C)$ is trivial, and any cuspidal local system on this class is trivial (see $\mathsection$\ref{subsubsec:cuspidal-pairs} for details).  

\subsubsection{The short root case}
When $M\simeq M_{\alpha}$, by \cite[Proposition 6.2]{ShahidiLanglandsconjecture}:
\begin{enumerate}
    \item When $\omega_\sigma\neq 1$, $\ii_P^G(\sigma)$ is reducible and there are no complementary series.
    \item When $\omega_\sigma=1$, $\ii_P^G(\sigma)$ is irreducible, and $\ii_P^G(\sigma\otimes\nu_F^s)$ is irreducible unless $s=\pm 1/2$. 
    \begin{itemize}
        \item $\ii_P^G(\sigma\otimes\nu_F^{1/2})$ has has length $2$: it has a unique generic discrete series subrepresentation $\pi(\sigma)$ and a unique irreducible pre-unitary non-tempered Langlands quotient, $J(\sigma)$.  
        \item All the representations $\ii_P^G(\sigma\otimes\nu_F^s)$ for $0<s<1/2$ are in the complementary series and $s=1/2$ is the edge of complementary series. 
    \end{itemize}
\end{enumerate}
Since $\pi(\sigma)$ is a discrete series, by \cite[Proposition~9.3]{Solleveld-endomorphism-algebra}, it corresponds to a discrete series of the Hecke algebra $\mathcal{H}^{\fs}(G)$. 
On the other hand, by \cite[Tables~4.5.6 and 4.5.8]{aubert-xu-Hecke-algebra}, the Hecke algebra $\mathcal{H}^{\fs}(G)$ is isomorphic to an extended affine Hecke algebra which is either of type $\widetilde{\rA}_1(q,q)$ or has trivial parameters. 
In the first case, we can apply \cite[Table 2.1]{Ram}. Since the representation $\pi(\sigma)$ is tempered, by \cite[Table 2.1]{Ram}, it has unipotent class $e_{\alpha_1}$, which corresponds to $\widetilde{\rA}_1$ by Table \ref{tab:table-unipotent-Gsigma=GL2}. When $\mathcal{H}^{\fs}(G)$ has trivial parameters, we obtain the enhanced $L$-parameters as in Table \ref{tableint} (using the explicit formula in Definition \ref{defn:phisu}). 
\begin{center} 
\begin{table} [ht]
\begin{tabular}{|c|c|c|c|c|}
\hline
Representation 
&indexing triple & enhanced $L$-parameter \cr
\hline\hline
$\pi(\sigma)$&$(t_a,e_{\alpha_1},1)$ &$(\varphi_{\sigma,\widetilde{\rA}_1},1)$  \cr
\hline
 $J(\sigma)$&$(t_a,0,1)$ &$(\varphi_{\sigma,1},1)$  \cr
  \hline
\end{tabular}
\vskip0.2cm
    \caption{\label{tab:LLC}: LLC for the irred.~constituents of $\ii_P^G(\sigma\otimes\nu_F^{1/2})$ attached to $\fs=[M_\alpha,\sigma]_G$ where $\omega_\sigma$ unramified.}
\end{table}
\end{center}

\subsubsection{The long root case}
When $M\simeq M_\alpha$, by \cite[Theorem 8.1, Proposition 8.3]{ShahidiAnnalsAproofof}:
\begin{enumerate}
\item If $\varphi_\sigma=\Ind_{W_K}^{W_F}(\chi)$, $[K:F]=2$, $\chi\in\widehat{K}^*$, then $\ii_P^G(\sigma)$ is irreducible if and only if either $\chi|_{F^\times}=\eta$, or $\chi|_{F^\times}=1$ and $\chi^3=1$. 
\item The representation $I(\widetilde{\alpha}/i,\sigma)$ is reducible with a unique $\chi$-generic subrepresentation $\pi(\sigma)$ which is in the discrete series. Its Langlands quotient $J(\sigma)$ is never generic. It is a pre-unitary non-tempered representation. 
\end{enumerate}

\begin{remark}
Since $p$ is odd, the case where $\sigma$ is extraordinary (so-called \textit{exceptional} in \cite{Bushnell-Henniart-GL2}) mentioned in \cite{ShahidiAnnalsAproofof} does not occur (see \cite[Theorem~34.1 and Theorem~44.1]{Bushnell-Henniart-GL2}).
\end{remark}

\begin{center} 
\begin{table} [ht]
\begin{tabular}{|c|c|c|c|}
\hline
Representation 
&indexing triple & enhanced $L$-parameter\cr
\hline\hline
$\pi(\sigma)$&$(t_a,e_{\alpha_1},1)$ &$(\varphi_{\sigma,A_1},1)$  \cr
\hline
 $J(\sigma)$&$(t_a,0,1)$ &$(\varphi_{\sigma,1},1)$  \cr
  \hline
\end{tabular}
\vskip0.2cm
    \caption{\label{tab:}: {LLC for the irred.~constituents of $\ii_P^G(\sigma)$ attached to $\mathfrak{s}=[M_\beta,\sigma]_G$.}}
\end{table}
\end{center}
Since the representation $\pi(\sigma)$ is a discrete series, by \cite[Proposition~9.3]{Solleveld-endomorphism-algebra}, it corresponds to a discrete series of the Hecke algebra $\mathcal{H}^{\fs}(G)$. 
On the other hand, by \cite[Tables~4.5.2 and 4.5.4]{aubert-xu-Hecke-algebra}, the Hecke algebra $\mathcal{H}^{\fs}(G)$ is isomorphic to an extended affine Hecke algebra which is of type $\widetilde{\rA}_1(q^3,q)$, $\widetilde{\rA}_1(q,q)$ or has trivial parameters. 
In the $\widetilde{\rA}_1(q,q)$ case, we can apply \cite[Table 2.1]{Ram}. Since $\pi(\sigma)$ is tempered, by \cite[Table 2.1]{Ram}, it has unipotent class $e_{\alpha_1}$, which corresponds to $\rA_1$ by Table \ref{tab:table-unipotent-Gsigma=GL2}.

When $\mathcal{H}^{\fs}(G)$ has trivial parameters, similarly, we obtain the enhanced $L$-parameters as in Table \ref{tableint} (using the explicit formula in Definition \ref{defn:phisu}). When $\mathcal{H}^{\fs}(G)$ is of type $\widetilde{\rA}_1(q^3,q)$, we use Property \ref{property:AMS-conjecture-7.8} to find the \textit{singular} supercuspidal representation that shares the same $L$-packet as this intermediate series representation, and then use \eqref{defn:phisu} again to explicitly construct the enhanced $L$-parameter. This is done in Theorem \ref{constituents-part-one}. 

\subsubsection{Formal degrees for depth-zero intermediate series} 
Consider a depth-zero representation $\pi\in\Irr_\fs(G)$ where $\fs=[M,\sigma]_G$ and $M\in\{M_\alpha,M_\beta\}$. Since the depth is preserved via parabolic induction, the supercuspidal representation $\sigma$ of $M\simeq\GL_2(F)$ has also depth-zero and we have $\sigma=\cInd_{F^\times M_y,0}^M(\boldsymbol{\tau})$, where $M_{y,0}\simeq \GL_2(\fo_F)$ and $\boldsymbol{\tau}$ denotes the extension to $F^\times M_{y,0}$ of the inflation to $M_{y,0}$ of a cuspidal irreducible representation of $\mathbb{M}_{y,0}\simeq\GL_2(\Fq)$.  
By \eqref{Lusztig-unipotent-decomposition} the representation $\tau$ corresponds to the unipotent representation $\tau_{\unip}\in \mathcal{E}(\Cent_{\mathbb{M}_{y,0}^{\vee}}(s),1)$. In this case, $\Cent_{\mathbb{M}_{y,0}^{\vee}}(s)$ is either $\GL_2(\F_q)$ or a torus. Since $\GL_2(\F_q)$ carries no cuspidal unipotent representations, $\Cent_{\mathbb{M}_{y,0}^{\vee}}(s)$ can only be a torus, and thus $\tau_{\unip}=1$. In particular, $\dim(\tau_{\unip})=1$. On the other hand, we have
\begin{equation}
    |\mathbb{M}_{y,0}|=|\GL_2(\F_q)|=q(q^2-1)(q-1). 
\end{equation}
Thus, since $|\bbT]=(q-1)^2$, we have
\begin{equation}
    \dim(\tau)=\frac{|\mathbb{M}_{y,0}|_{p'}}{|\mathbb{T}|_{p'}}\cdot\dim(\tau_{\unip})=\frac{(q^2-1)(q-1)}{(q-1)^2}=q+1.
\end{equation}
By \cite{Kim-Yu-types}, the pair $(G_{y,0},\boldsymbol{\tau})$ is an $\fs$-type in $G$, and a $G$-cover of $(M_{y,0},\boldsymbol{\tau})$
 (see  also \cite{Morris-Level-zero-types} or \cite{Moy-Prasad-types}).
Since $\mathbb{G}_{y,0}=\mathbb{M}_{y,0}$, by \eqref{eqn:DR} we have
\begin{equation}
    \mathrm{Vol}(G_y)=q^{-\mathrm{rk}(G)/2}\cdot |\mathbb{G}_{y,0}|_{p'}=q^{-1}(q^2-1)=\frac{q^2-1}{q}.
\end{equation}
Therefore by 
a generalized version of 
{Proposition~\ref{prop: formal degrees for non-supercuspidals}}, we have 
\begin{equation}
    \fdeg(\pi)=\frac{\dim\tau}{\mathrm{Vol}(G_y)}\cdot d(\pi^{\mathcal{H}})=\frac{(q+1)q}{q^2-1}d(\pi^{\mathcal{H}}).
\end{equation}

\begin{numberedparagraph}
In the following table, we compute the quantity $d(\pi_s^{\mcH})$ and $d(\pi_0^{\mcH})$ as in \cite[$\mathsection$9]{Reeder-HA}. Recall that if $q_{\alpha}\neq q_0$, we must have $|R(\Oo)|=1$, and the formal degrees of square-integral modules for $\mathcal{H}(q_{\alpha},q_0)$ are given by 
\begin{equation}
    d(\pi_s^{\mcH})=\frac{q_{\alpha}q_0-1}{(q_{\alpha}+1)(q_0+1)}\qquad d(\pi_0^{\mcH})=\frac{|q_{\alpha}-q_0|}{(q_{\alpha}+1)(q_0+1)}.
\end{equation}
On the other hand, if the Hecke algebra $\mathcal{H}(q_{\alpha},q_0)$ is of rank one, and $q_{\alpha}=q_0$, there are $|R(\Oo)|$ square-integrable $\mathcal{H}(q_{\alpha},q_0)$-modules $\eta\otimes\pi_s$, one for each character $\eta$ of $R(\Oo)$. They all have the same formal degree 
\begin{equation}
    d(\eta\otimes\pi_s^{\mcH})=|R(\Oo)|^{-1}\frac{q_{\alpha}-1}{q_{\alpha}+1}.
\end{equation}
We compute these quantities $d(\pi)$ in the following table \ref{num: M long 0}, following \cite[Table 4.5.2]{aubert-xu-Hecke-algebra}. To distinguish between simple modules of Hecke algebras and irreducible representations, we use the notation $d(\pi^{\mathcal{H}})$ to denote formal degrees for simple modules for Hecke algebras $\mathcal{H}(q_{\alpha},q_0)$. 
\end{numberedparagraph}

\begin{numberedparagraph}\label{num: M long 0}
\textit{Table for long root essentially depth zero cases.}\\
We refer the reader to \cite{aubert-xu-Hecke-algebra} for the notations. \\

\adjustbox{scale=0.7,center}{
\begin{tabular}{|c|c|c|c|c|c|c|c|c|}
\hline
$r$ & $\mcD$ &  $\omega_{\sigma}$ &$\chi^2\chi'^{-1}$  & $R(\mathcal{O})$  
 & $W_{\Oo}$  & $\begin{array}{c}\mathcal{H}(G,\rho)\\ q_F^{\lambda(\alpha)},q_F^{\lambda^*(\alpha)}\end{array}$ &$d(\pi^{\mathcal{H}})$&$\fdeg(\pi)$ \\ \hline\hline
\multirow{4}{*}{$r=0$}& \multirow{4}{*}{$((G,M),(y,\iota),(M_{y,0},\rho_M))$} & $=1$ &$\begin{array}{c} \text{unramified} \\ \chi \text{ cubic}\end{array}$
   & $=1$ 
& $\neq 1$  & $\begin{array}{c}\text{non-comm,} \\q_F^3, q_F\end{array}$  & $\begin{array}{c}\frac{q^4-1}{(q^3+1)(q+1)};\\ \frac{(q^3-q)}{(q^3+1)(q+1)}\end{array}$ & $\begin{array}{c}
       \frac{(q^2+1)q}{q^3+1}\\
       \frac{q^2}{q^3+1}
\end{array}$
\\ \cline{3-9}
&  & $\ne 1$ &$\begin{array}{c} \text{unramified}\\ \chi \text{ cubic}\end{array}$  & $=1$     
& $\neq 1$  & $\begin{array}{c}\text{non-comm,}\\ q_F^2, q_F^2\end{array}$ & $\frac{q^2-1}{q^2+1}$ &$\frac{q(q+1)}{q^2+1}$
\\ \cline{3-9}
&  &$=1$ & $\begin{array}{c}\text{ramified}\\ \chi \text{ cubic}\end{array}$ & $=1$  
  &  $\neq 1$
  & $\begin{array}{c}\text{non-comm,}\\ q_F,q_F\end{array}$ &$\frac{q-1}{q+1}$ & $\frac{q}{q+1}$
\\ \cline{3-9}
& &$\neq 1$ &$\begin{array}{c}\text{ ramified}\\ \chi \text{ cubic}\end{array}$ &$*$   
&$=1$  &  $\C[R(\mathcal{O})]\ltimes\C[\Oo]$ & $0$ & $0$\\
\cline{3-9}
& &$=1$ &\multirow{2}{*}{$\begin{array}{c}\chi\text{ not cubic}\\ \text{N/A}\end{array}$} &$=1$  
&$\neq 1$  & $\begin{array}{c}\text{non-comm,}\\ q_F,q_F\end{array}$ &$\frac{q-1}{q+1}$ & $\frac{q}{q+1}$\\
\cline{3-3}\cline{5-9}
&&$\neq 1$ &&$*$  
&$=1$ &$\C[R(\mathcal{O})]\ltimes\C[\Oo]$ &$0$ &$0$ \\
\hline
\multirow{2}{*}{$r\neq 0$}
& \multirow{2}{*}{$(((M,G),M),(y,\iota),(r,0),(\phi,1),(M_{y,0},\rho_M))$}
& \multirow{2}{*}{$\ne 1$}
  &\multirow{2}{*}{N/A}   & \multirow{2}{*}{$=1$}
 & \multirow{2}{*}{$=1$}  &\multirow{2}{*}{$\C[\Oo]$}&\multirow{2}{*}{$0$} &\multirow{2}{*}{$0$}
\\ 
& 

&
&
& 
&  
&
& &
\\ \hline
\end{tabular}}
\smallbreak
\begin{center} {\sc Table} \ref{num: M long 0}.
\end{center}
\end{numberedparagraph}
\begin{numberedparagraph}\label{long-root-intermediate-series-mixed-packet}
As can be seen from the table, the first row has equal formal degree as in the singular non-unipotent supercuspidal representation from \eqref{fdeg-singular-nonunipotent-sc}, i.e.~they both have formal degree $\frac{q^2}{q^3+1}$. In particular, we have obtained an $L$-packet mixing a singular non-unipotent supercuspidal with an intermediate series representation with supercuspidal support being a supercuspidal representation associated to the long root $M\simeq M_\beta$. 
\end{numberedparagraph}
\begin{remark}
Recall that the singular supercuspidal representation from \eqref{fdeg-singular-nonunipotent-sc} can only occur when $q\equiv -1\mod 3$, which is also precisely the condition for this intermediate series representation in the first row of Table \ref{num: M long 0} to occur. Recall from \cite[(4.2.6)]{aubert-xu-Hecke-algebra} that
\begin{equation}
    \omega_\sigma(\varpi_F)+\chi^2\chi'^{-1}(\varpi_L)=0,
\end{equation}
which can only occur when $q\equiv -1\mod 3$. 
\end{remark}

\section{ The principal series} \label{sec:ppal-series} \

\subsection{Notation and background} \label{subsec:ppal-series-backgroud}
Let $\bT$ be an $F$-split maximal torus in $\bG$, and let $T^\vee$ denote its Langlands dual torus, which is a maximal torus in $G^\vee$. The principal series consists of all $G$-representations that \textit{occur in} parabolic inductions of characters of $T:=\bT(F)$ to $\rG_2(F)$. 
Let $T_0$ denote the maximal compact subgroup of $T$. Since $\bT$ is $F$-split, 
\begin{equation}\label{eq:cT0}
T \simeq F^\times \otimes_{\Z} X_* (T)\simeq(\fo_F^\times \times \Z)\otimes_{\Z} X_*(T)=T_0\times X_* (T).
\end{equation}
Since $W=W_\bG$ (the Weyl group of $\bG$ with respect to $\bT$) acts trivially on $F^\times$, these isomorphisms are $W$-equivariant if
we endow the right hand side with the diagonal $W$-action.
Thus \eqref{eq:cT0} determines a $W$-equivariant isomorphism of character groups 
\begin{equation}\label{split}
\Irr(T)\cong\Irr(T_0)\times\Irr(X_*(T))=\Irr(T_0)\times X_{\nr}(T).
\end{equation}
How $\Irr(T_0)$ is embedded depends on the choice of $\varpi_F$. However, the
isomorphisms 
\begin{equation}
\Irr(T_0)\cong\Hom(\fo_F^\times,T^\vee) \quad\text{and}\quad
X_{\nr}(T)\cong\Hom(\Z,T^\vee)=T^\vee. 
\end{equation}
are canonical.
Let $\varphi_\sigma$ be the image of $\sigma\in\Irr(T)$ under the canonical map 
\begin{equation} \label{paragraph:starting-LLC}
\Irr(T)=\Hom(F^\times\otimes_{\Z}X_*(T),\C^\times)
\simeq \Hom(F^\times,\C^\times\otimes_\Z X^*(T))\simeq \Hom(F^\times,T^\vee).
\end{equation} 
We define
\begin{equation} \label{eqn:cs map}
    c^\fs:=\varphi_\sigma|_{\fo_F^\times}\colon \fo_F^\times\to T^\vee.
\end{equation}
Note that unramified twists do not modify $\varphi_\sigma|_{\fo_F^\times}$, and conversely, \eqref{paragraph:starting-LLC}
determines 
$c^\fs$ up to unramified twists. On the other hand, two elements of $\Irr(T)$ are $G$-conjugate if and only if they are $W$-conjugate, so the $W$-orbit of  $c^\fs$ contains the same amount of information as $\fs$.
We define
\begin{equation} \label{eqn:H group}
 \mcJ^\fs:=\rZ_{G^\vee}(\Image(c_\fs)).  
\end{equation}
By \cite[pp.~394–395]{Roche-principal-series}, when $\bG$ has connected center, the group $\mcJ^\fs$ is connected and its Weyl group is isomorphic to the group $W^\fs$ (see \eqref{eqn:NsI}). Let $J^\fs$ denote the group of $F$-rational points of the $F$-split reductive algebraic $F$-group whose Langlands dual is the group $\mcJ^\fs$.

By \cite[Theorem~8.2]{Roche-principal-series},~the algebra $\cH(G,\tau^\fs)$ and $\cH(J^\fs,1_{\mathcal{I}})$ are isomorphic via a family of $*$-preserving,~support-preserving (and hence inner-product preserving) isomorphisms.~By \cite[proof of Theorem~10.7]{Roche-principal-series}, these isomorphisms preserve square-integrability and, for corresponding $\pi\in\Irr^2(G)$ and $\pi^1\in\Irr^2(J^\fs)$, with the appropriate normalization of volume factors, we have 
\begin{equation}\fdeg(\pi)=d^{\mathcal{H}}(\pi^1).
\end{equation}

\subsection{Principal series for \texorpdfstring{$\rG_2$}{\rG_2}}
For $\gamma\in\{\alpha,\beta\}$, we set
\begin{equation} \label{eqn:IgammaGL}
I^\gamma(\xi_1\otimes\xi_2):=i_{T(U\cap M_\gamma)}^{M_\gamma}(\xi_1\otimes\xi_2).
\end{equation}
By using \eqref{eqn:change of realization of T} we obtain
\begin{equation} \label{eqn:twist}
I^\alpha(\xi_1\otimes\xi_2)\otimes (\nu_F^s\circ\det)=
I^\alpha(\xi_1\nu_F^s\otimes\xi_2\nu_F^s)\quad\text{and}\quad
I^\beta(\xi_1\otimes\xi_2)\otimes (\nu_F^s\circ\det)=
I^\beta(\xi_2\nu_F^s\otimes\xi_1\xi_2^{-1}).
\end{equation}
As recalled in \cite[Proposition~1.1]{Muic-G2}, the principal series $I^\gamma(\xi_1\otimes\xi_2)$ reduces if and only if $(\xi_1\otimes\xi_2)\circ\gamma^\vee=\nu_F^{\pm 1}$. Let $\delta(\xi)$ denote the unique irreducible subrepresentation of $\ii_B^{\GL_2(F)}(\nu_F^{1/2}\xi\otimes\nu_F^{-1/2}\xi)$, it is also the unique irreducible quotient of $\ii_B^{\GL_2(F)}(\nu_F^{-1/2}\xi\otimes\nu_F^{1/2}\xi)$.  Similarly, $\xi\circ\det$ is the unique irreducible quotient (resp. subrepresentation) of $\ii_B^{\GL_2(F)}(\nu_F^{1/2}\xi\otimes\nu_F^{-1/2}\xi)$ (resp. $\ii_B^{\GL_2(F)}(\nu_F^{-1/2}\xi\otimes\nu_F^{1/2}\xi)$). Thus, in the Grothendieck groups $R(M_\alpha)$, $R(M_\beta)$ we have 
\begin{equation} \label{eqn:inM-case-1}
I^\alpha(\nu_F^{1/2}\xi\otimes\nu_F^{-1/2}\xi)=\delta(\xi)+\xi\circ\det\quad\text{and}\quad
I^\beta(\nu_F^{-1/2}\xi\otimes\nu_F)=\delta(\xi)+\xi\circ\det,
\end{equation}
\begin{equation} \label{eqn:inM-case-2}
I^\alpha(\nu_F^{-1/2}\xi\otimes\nu_F^{1/2}\xi)=\delta(\xi)+\xi\circ\det\quad\text{and}\quad
I^\beta(\nu_F^{1/2}\xi\otimes\nu_F^{-1})=\delta(\xi)+\xi\circ\det.
\end{equation}
Let $s\in\C$. For an admissible representation $\pi$ of $\GL_2(F)$, we write 
\begin{eqnarray} \label{eqn:IgammaG2}
I_\gamma(s,\pi):=\ii_{P_\gamma}^G(\pi\otimes(\nu_F^s\circ\det)),\quad\text{and}\quad I_\gamma(\pi):=I_\gamma(0,\pi).
\end{eqnarray}
When $\pi$ a tempered irreducible representation and $s>0$, 
the representation $I_\gamma(s,\pi)$ has a unique irreducible quotient denoted by $J_\gamma(s,\pi)$.

We denote by $I(\xi_1\otimes\xi_2)$ the parabolically induced representation of $\sigma=\xi_1\otimes\xi_2$.
We will describe the irreducible components of the parabolically induced representation $I(\xi_1\otimes\xi_2)$ following the methods developed in \cite[\S3 and \S4]{Muic-G2}.

\begin{remark} \label{rem:reduction}
{\rm Recall (see \cite[Theorem 2.9]{Bernstein-Zelevinsky}) that, for any $w\in W$, the Jordan-H\"older series of $I(\xi_1\otimes\xi_2)$ and $I({}^w(\xi_1\otimes\xi_2))$ have the same irreducible quotients.} 
\end{remark}

If $\omega_1$, $\omega_2$ are unitary characters of $F^\times$ and $s_1>s_2>0$, then the representation $I(\nu_F^{s_1}\omega_1\otimes\nu_F^{s_1}\omega_2)$ has a unique irreducible quotient denoted by $J(s_1,s_2,\omega_1,\omega_2)$.

\begin{lem} \label{lem:generic} {\rm \cite[Theorem~7]{Rodier-Whittaker-models}}
We suppose that $G$ is $F$-split. Let $P$ be a parabolic subgroup of $G$ with Levi subgroup $L$, and $\sigma$ an irreducible smooth representation of $L$.  
\begin{itemize}
    \item[(a)]
If $\sigma$ is generic, then $\ii_P^G(\sigma)$ contains a unique generic irreducible component.
    \item[(b)]
If $\sigma$ is not generic, then no irreducible component of $\ii_P^G(\sigma)$ is generic.
\end{itemize}
\end{lem}
By transitivity of the parabolic induction we have 
\begin{equation} 
I(\xi_1\otimes\xi_2)=\ii_{P_\gamma}^G\left(I^\gamma(\xi_1\otimes\xi_2)\right)=I_\gamma\left(s,I^\gamma(\xi_1\otimes \xi_2)\otimes \nu_F^{-s}\circ\det\right)).
\end{equation}
By using \eqref{eqn:twist}, we get
\begin{equation} \label{eqn:transitivity}
I(\xi_1\otimes\xi_2)=
I_\alpha(s,I^\alpha(\xi_1\nu_F^{-s}\otimes \xi_2\nu_F^{-s}))=
I_\beta(s,I^\beta(\xi_2\nu_F^{-s}\otimes \xi_1\xi_2^{-1})).
\end{equation}

\begin{lem} \label{lem:main-induced}\ 

\begin{enumerate}
\item
If $\xi_1=\nu_F
\xi_2$, then we have
\begin{equation} \label{decomposing-Inus1-tensor-Inus2-case-1}
\begin{split}
I(\xi_1\otimes\xi_2)&=
I_\alpha(s,\delta(\nu_F^{1/2-s}\xi_2))+I_\alpha(s,\nu_F^{1/2-s}\xi_2\circ\det)\cr
&=I_\beta(s,\delta(\nu_F^{1/2-s}\xi_2))+I_\beta(s,\nu_F^{1/2-s}\xi_2\circ\det).
\end{split}
\end{equation}
\item If $\xi_1=\nu_F^{-1}\xi_2$, then we have 
\begin{equation} \label{decomposing-Inus1-tensor-Inus2-case-2}
\begin{split}
I(\xi_1\otimes\xi_2)&=
I_\alpha(s,\delta(\nu_F^{-1/2-s}\xi_2))+I_\alpha(s,\nu_F^{-1/2-s}\xi_2\circ\det)\cr
&=I_\beta(s,\delta(\nu_F^{-1/2-s}\xi_2))+I_\beta(s,\nu_F^{-1/2-s}\xi_2\circ\det).
\end{split}
\end{equation}
\end{enumerate}
\end{lem}
\begin{proof}
The result follows from the combination of  \eqref{eqn:transitivity} with \eqref{eqn:inM-case-1} and \eqref{eqn:inM-case-2}. 
\end{proof}
For $\gamma\in\{\alpha,\beta\}$, denote by $\rrr_\gamma\colon R(G)\to R(M_\gamma)$ the normalized Jacquet restriction functor $\rrr^G_{M_\gamma}$. For $\gamma,\gamma'\in\{\alpha,\beta\}$, let
$W^{M_{\gamma'},M_\gamma}$ be the subset of $W$ defined in \cite[\S2.11]{Bernstein-Zelevinsky}. By \cite[Lemmas~5.4-5.6]{Digne-Michel}, it consists in the elements $w\in W$ satisfying
\begin{equation} \label{eqn:reduced elements}
\ell(s_\gamma w)=\ell(s_\gamma)+\ell(w)\quad\text{and}\quad\ell(w s_{\gamma'})=\ell(w) +\ell(s_{\gamma'}).
\end{equation}
Then, from Table~\ref{table:1}, we obtain 
\begin{equation} \label{eqn:WMaMa}
    W^{M_\alpha,M_{\alpha}}=\left\{1,s_\beta,s_{\alpha+\beta},s_{3\alpha+\beta}\right\}, \quad 
    W^{M_{\beta},M_\alpha}=\left\{1,ba,baba\right\}
    \end{equation}
    \begin{equation}\label{eqn:WMbMb}
    W^{M_{\alpha},M_\beta}=\left\{1,ab,abab\right\},\quad 
    W^{M_\beta,M_{\beta}}=\left\{1,s_\alpha,s_{2\alpha+\beta},s_{3\alpha+2\beta}\right\}
    \end{equation}
    \begin{equation}\label{eqn:WMaT}
    W^{M_{\alpha},T}=\left\{1,s_\beta,ab,s_{\alpha+\beta},abab, babab\right\},\quad 
    W^{T,M_{\alpha}}=\left\{1,s_\beta,ba,s_{\alpha+\beta},baba, babab\right\}.
\end{equation}
\begin{lem}
Let $\xi\in\Irr(F^\times)$. We have
\begin{equation} \label{eqn:Jacquet-aa}
\rrr_\alpha\left(I_\alpha(s,\delta(\xi))\right)=\nu_F^{s}\delta(\xi)+\nu_F^{-s}\delta(\xi^{-1})
+I^\alpha(\nu_F^{2s}\xi^{2}\otimes\nu_F^{1/2-s}\xi^{-1})
+I^\alpha(\nu_F^{s+1/2}\xi\otimes\nu_F^{-2s}\xi^{-2}).
\end{equation}
\begin{equation} \label{eqn:Jacquet-ab}
\rrr_\alpha\left(I_\beta(s,\delta(\xi))\right)=\delta(\nu_F^s\xi)+\nu_F^s\xi\circ\det+
I^\alpha(\nu_F^{2s}\xi^2\otimes\nu_F^{-s-1/2}\xi^{-1})+
I^\alpha(\nu_F^{s-1/2}\xi\otimes \nu_F^{-2s}\xi^{-2}).
\end{equation}
\end{lem}
\begin{proof}
(1) Since $I_\alpha(s,\delta(\xi))=I_\alpha(\nu_F^{s}\delta(\xi))$ and $\Nor_G(M_\alpha)/M_{\alpha}=\{1,s_{3\alpha+\beta}\}$, by \cite[Geometrical Lemma in \S2.12]{Bernstein-Zelevinsky}, we get 
\[\rrr_\alpha\left(I_\alpha(s,\delta(\xi))\right)=\nu_F^{s}\delta(\xi)+
s_{3\alpha+\beta}(\nu_F^{s}\delta(\xi))+I^\alpha\circ s_{\alpha+\beta}\circ\rrr_T^{M_\alpha}(\nu_F^{s}\delta(\xi))+I^\alpha\circ s_{\beta}\circ\rrr_T^{M_\alpha}(\nu_F^{s}\delta(\xi)).\]
By \cite[Proposition~1.11.(b)]{Zelevinsky}, we know that $\nu_F^{s}\delta(\xi)$ is the unique submodule of the length two representation $I^\alpha(\nu_F^{s+1/2}\xi\otimes\nu_F^{s-1/2}\xi)$, and $\rrr_T^{M_\alpha}(\nu_F^{s}\delta(\xi))=\nu_F^{s+1/2}\xi\otimes\nu_F^{s-1/2}\xi$. 
On the other hand, we have 
$s_{3\alpha+\beta}(\nu_F^{s}\delta(\xi))=\nu_F^{-s}\delta(\xi^{-1})$. Hence, we get
\[\rrr_\alpha\left(I_\alpha(s,\delta(\xi))\right)=\nu_F^{s}\delta(\xi)+
\nu_F^{-s}\delta(\xi^{-1})+I^\alpha\circ s_{\beta}(\nu_F^{s+1/2}\xi\otimes\nu_F^{s-1/2}\xi)+I^\alpha\circ s_{\alpha+\beta}(\nu_F^{s+1/2}\xi\otimes\nu_F^{s-1/2}\xi).\]
By Table~\ref{table:1}, we have
\[s_{\beta}(\nu_F^{s+1/2}\xi\otimes\nu_F^{s-1/2}\xi)=\nu_F^{2s}\xi^{2}\otimes\nu_F^{1/2-s}\xi^{-1} 
\quad\text{and}\quad
s_{\alpha+\beta}(\nu_F^{s+1/2}\xi\otimes\nu_F^{s-1/2}\xi)=\nu_F^{s+1/2}\xi\otimes\nu_F^{-2s}\xi^{-2}.
\]
Finally, we obtain \eqref{eqn:Jacquet-aa}.

(2) By \eqref{eqn:WMbMb} and \cite[Geometric Lemma]{Bernstein-Zelevinsky}, we have in $R(M_\alpha)$: 
\begin{equation}\label{r-alpha-I-beta}
\rrr_\alpha\circ I_\beta=I^\alpha\circ\rrr_{T}^{M_\beta}+I^\alpha\circ ba\circ\rrr_{T}^{M_\beta}+I^\alpha\circ baba\circ\rrr_{T}^{M_\beta}.
\end{equation}
Again by \cite[1.11]{Zelevinsky}, we have 
\[\rrr_T^{M_\beta}(\nu_F^{s}\delta(\xi))=\nu_F^{s+1/2}\xi\otimes\nu_F^{s-1/2}\xi.\]
By Table~\ref{table:1}, we have 
$ba(\nu_F^{s+1/2}\xi\otimes\nu_F^{s-1/2}\xi)=\nu_F^{2s}\xi^2\otimes \nu_F^{-s-1/2}\xi^{-1}$ and 
$baba(\nu_F^{s+1/2}\xi\otimes\nu_F^{s-1/2}\xi)=\nu_F^{s-1/2}\xi\otimes \nu_F^{-2s}\xi^{-2}$. 
Hence, we get 
\begin{align*}(ba\circ\rrr_{T}^{M_\beta})(\nu_F^{s}\delta(\xi))=\nu_F^{2s}\xi^2\otimes\nu_F^{2s+1/2}\xi^{-1}\\
(baba\circ\rrr_{T}^{M_\beta})(\nu_F^{s}\delta(\xi))=\nu_F^{s-1/2}\xi\otimes \nu_F^{-2s}\xi^{-2}
\end{align*}
Then \eqref{r-alpha-I-beta} gives
\[\rrr_\alpha\left(I_\beta(s,\delta(\xi))\right)=I^\alpha(\nu_F^{s+1/2}\xi\otimes\nu_F^{s-1/2}\xi)+I^\alpha(\nu_F^{2s}\xi^2\otimes\nu_F^{-s-1/2}\xi^{-1})+
I^\alpha(\nu_F^{s-1/2}\xi\otimes \nu_F^{-2s}\xi^{-2})
.\]
Since 
$I^\alpha(\nu_F^{s+1/2}\xi\otimes\nu_F^{s-1/2}\xi)=\delta(\nu_F^s\xi)+\nu_F^s\xi\circ\det$ 
by \eqref{eqn:inM-case-1}, we have 
\[\rrr_\alpha\left(I_\beta(s,\delta(\xi))\right)=\delta(\nu_F^s\xi)+\nu_F^s\xi\circ\det+
I^\alpha(\nu_F^{2s}\xi^2\otimes\nu_F^{-s-1/2}\xi^{-1})+
I^\alpha(\nu_F^{s-1/2}\xi\otimes \nu_F^{-2s}\xi^{-2})
.\]
\end{proof}

\begin{lem}\label{Jacquet-restriction-alpha-nu-tensor-xi2} Notations as above. We have 
\[\rrr_\alpha (I(\nu_F^{\mp 1}\otimes\xi_2))=\begin{split}
I^{\alpha}(\nu_F^{\mp 1}\otimes\xi_2)+I^{\alpha}(\nu_F^{\pm 1}\otimes\nu_F^{\mp 1}\xi_2)+I^{\alpha}(\nu_F^{\mp 1}\xi_2\otimes\xi_2^{-1}) 
+I^{\alpha}(\nu_F^{\mp 1}\xi_2\otimes\xi_2^{-1})\\+I^{\alpha}(\nu^{\mp 1}_F\xi_2\otimes\nu^{\pm 1}_F)+I^{\alpha}(\nu^{\mp 1}_F\otimes\xi_2)+I^{\alpha}(\xi_2\otimes\nu^{\pm 1}_F\xi_2^{-1}). 
\end{split}
\]
\end{lem}
\begin{proof} By \eqref{eqn:transitivity}, we have
\begin{equation}
I(\nu_F^{\mp 1}\otimes\xi_2)=
I_\alpha(s,I^\alpha(\nu_F^{\mp 1-s}\otimes \xi_2\nu_F^{-s})).
\end{equation}
First note that $I_{\alpha}(s,I^{\alpha}(\nu_F^{\mp 1-s}\otimes\xi_2\nu_F^{-s}))=I_{\alpha}(\nu_F^s\otimes I^{\alpha}(\nu_F^{\mp 1-s}\otimes\xi_2\nu_F^{-s}))=I_{\alpha}(I^{\alpha}(\nu_F^{\mp 1}\otimes\xi_2))$. 
By \cite[Geometrical Lemma]{Bernstein-Zelevinsky}, we have in $R(M_{\alpha})$:
\begin{align}\label{raIa-Ia-nu-xi2}
    \rrr_\alpha\left(I_{\alpha}(s,I^{\alpha}(\nu_F^{\mp 1-s}\otimes\xi_2\nu_F^{-s})\right)&=I^{\alpha}(\nu_F^{\mp 1}\otimes\xi_2)+s_{3\alpha+\beta}\circ I^{\alpha}(\nu_F^{\mp 1}\otimes\xi_2)+\ii_T^{M_{\alpha}}\circ s_{\beta}\circ\rrr_T^{M_{\alpha}}(I^{\alpha}(\nu_F^{\mp 1}\otimes\xi_2))\\
    &+\ii_T^{M_{\alpha}}\circ s_{\alpha+\beta}\circ\rrr_T^{M_{\alpha}}(I^{\alpha}(\nu_F^{\mp 1}\otimes\xi_2)).\label{last-term-raIa-Ia-nu-xi2}
\end{align}
The second term in \eqref{raIa-Ia-nu-xi2} gives $s_{3\alpha+\beta}\circ I^{\alpha}(\nu_F^{\mp 1}\otimes\xi_2)=I^{\alpha}(\nu_F^{\pm 1}\otimes\nu_F^{\mp 1}\xi_2)$. By \cite[Geometrical Lemma]{Bernstein-Zelevinsky}, we have in $R(T)$
\begin{equation} \label{eqn:BZ}
\rrr_T^{M_{\alpha}}\left(I^{\alpha}(\nu_F^{\mp 1}\otimes\xi_2)\right)=\nu_F^{\mp 1}\otimes\xi_2+ s_\alpha(\nu_F^{\mp 1}\otimes\xi_2)=\nu_F^{\mp 1}\otimes\xi_2+\xi_2\otimes \nu_F^{\mp 1}.\end{equation}
Since $s_{\beta}(\nu_F^{\mp 1}\otimes\xi_2)=\nu_F^{\mp 1}\xi_2\otimes\xi_2^{-1}$ and $s_{\beta}(\xi_2\otimes \nu_F^{\mp 1})=\nu_F^{\mp 1}\xi_2\otimes \nu_F^{\pm 1}$, we get
\begin{equation}
    \ii_T^{M_{\alpha}}\circ s_{\beta}\circ\rrr_T^{M_{\alpha}}(I^{\alpha}(\nu_F^{\mp 1}\otimes\xi_2))=I^{\alpha}(\nu_F^{\mp 1}\xi_2\otimes\xi_2^{-1})+I^{\alpha}(\nu^{\mp 1}_F\xi_2\otimes\nu^{\pm 1}_F). 
\end{equation}
For the last term in \eqref{last-term-raIa-Ia-nu-xi2}, since $s_{\alpha+\beta}(\nu^{\mp 1}_F\otimes\xi_2)=\nu_F^{\mp 1}\otimes \nu_F^{\pm 1}\xi_2^{-1}$ and $s_{\alpha+\beta}(\xi_2\otimes \nu^{\mp 1}_F)=\xi_2\otimes\nu^{\pm 1}_F\xi_2^{-1}$, we obtain 
\begin{equation}
    \ii_T^{M_{\alpha}}\circ s_{\alpha+\beta}\circ\rrr_T^{M_{\alpha}}I^{\alpha}(\nu_F^{\mp 1}\otimes\xi_2)=I^{\alpha}(\nu^{\mp 1}_F\otimes\xi_2)+I^{\alpha}(\xi_2\otimes\nu^{\pm 1}_F\xi_2^{-1}).
\end{equation}
Thus the result follows.
\end{proof}

\begin{remark}{\rm 
We observe that \eqref{eqn:Jacquet-aa} coincides with the one computed in \cite[Proof of Proposition~4.1]{Muic-G2} in the particular case where $s=1/2$ and $\xi$ is quadratic. We have also noticed the following typos in \cite[\S2]{Muic-G2}: in the computation of $\rrr_\alpha\circ I_\alpha(\pi)$ the element $w_{2\alpha+\beta}$ must be replaced by $w_{\alpha+\beta}$, and moreover $w_{3\alpha+2\beta}$ must be replaced by $w_{3\alpha+\beta}$ in the computation of $\rrr_\alpha\circ I_\beta(\pi)$, it should be $I^\alpha$ instead of $I^\beta$. }
\end{remark}

\begin{lem}\label{Jacquet-restriction-r-empty}
Notations as above. We have 
\begin{align*}
    \rrr_{\varnothing}I(\nu_F^{\mp 1}\otimes\xi_2)=&\nu_F^{\mp 1}\otimes\xi_2+\xi_2\otimes\nu_F^{\mp1}+\nu_F^{\mp1}\xi_2\otimes\xi_2^{-1}+\xi_2^{-1}\otimes\nu_F^{\mp1}\xi_2+\nu_F^{\mp1}\xi_2\otimes\nu_F^{\pm1}+\nu_F^{\pm1}\otimes\nu_F^{\mp1}\xi_2\\
    &+\nu_F^{\mp1}\otimes\nu_F^{\pm1}\xi_2^{-1}+\nu_F^{\pm1}\xi_2^{-1}\otimes\nu_F^{\mp1}+\xi_2\otimes\nu_F^{\pm1}\xi_2^{-1}+\nu_F^{\pm1}\xi_2^{-1}\otimes\xi_2+\xi_2^{-1}\otimes\nu_F^{\pm1}+\nu_F^{\pm1}\otimes\xi_2^{-1}.
\end{align*}
\end{lem}
\begin{proof}
By \cite[Geometric Lemma]{Bernstein-Zelevinsky} the Jacquet restriction is simply the sum of Weyl group orbits of $\nu_F^{\pm1}\otimes\xi_2$.
\end{proof}

In the following, we will classify the principal series cases according to the possibilities for $\fs$.
We will use the dual figure of Figure \ref{fig:G2-root-diagram} in which the role of $\alpha$ and $\beta$ is played by $\beta^\vee$ and $\alpha^\vee$, respectively to help us analyze the roots. 
Note that $\alpha^\vee$ (resp.~$\beta^{\vee}$), which is long (resp.~short), corresponds to the long (resp.~short) root $\alpha_1$ (resp.~$\alpha_2$) in the notations of \cite[$\mathsection$~6]{Ram}, as shown in \cite[Fig.6.1]{Ram}. 
Since the root $\alpha^\vee$ (resp.~$\beta^{\vee}$) is long (resp.~short), $e_{\alpha^\vee}=e_{\alpha_1}$ (resp.~$e_{\beta^{\vee}}=e_{\alpha_2}$) is a representative of $\rA_1$ (resp.~$\widetilde{\rA}_1$). 
Since $\alpha^\vee+\beta^\vee$ is short, $e_{\alpha^\vee+\beta^\vee}$ is a representative of $\widetilde\rA_1$ (see \cite[Table~6.3]{Ram}.

\subsection{Case \texorpdfstring{$\fs=[T,\xi\otimes\xi]_G$}{$\fs=[T,\xi\otimes\xi]_G$}}
\label{principal-series-first-case}
We may write
$\sigma=\chi_1\xi\otimes\chi_2\xi$, with $\xi$ a \textbf{ramified} character of $F^\times$, and $\chi_1$, $\chi_2$ unramified characters of $F^\times$.
We write $\chi_1=\nu_F^{s_1}$ and $\chi_2=\nu_F^{s_2}$ for $s_1,s_2\in\C$.

--If both $\nu_F^{s_1}\xi$ and $\nu_F^{s_2}\xi$ are unitary, by \cite[Theorem~$\rG_2$]{Keys}, $I(\nu_F^{s_1}\xi\otimes\nu_F^{s_2}\xi)$ is reducible if and only if $\nu_F^{s_1}\xi$ and $\nu_F^{s_2}\xi$ are two distinct quadratic characters, in which case it is of multiplicity $1$ and length $2$. In this case, the two irreducible constituents are in the same $L$-packet and they are both tempered. One is generic and the other one is not. See Theorem~\ref{constituents-part-two}. 

--Otherwise (i.e., at least one of $\nu_F^{s_1}\xi$, $\nu_F^{s_2}\xi$ is non-unitary), by \cite[Proposition 3.1]{Muic-G2},  $I(\nu_F^{s_1}\xi\otimes\nu_F^{s_2}\xi)$ is irreducible unless we are in one of the following situations:
\begin{enumerate}
    \item $s_1=\pm 1$, $\xi=1$, $s_2$ arbitrary (resp.~$s_1$ arbitrary, $s_2=\pm 1$, $\xi=1$);
    \item $\nu_F^{s_1+s_2}\xi^2=\nu_F^{\pm 1}$;
    \item $s_1-s_2=\pm 1$;
    \item $\nu_F^{2s_1+s_2}\xi^3=\nu_F^{\pm 1}$;
    \item $\nu_F^{s_1+2s_2}\xi^3=\nu_F^{\pm 1}$.
\end{enumerate}
We now explain how cases (1) and (2) are equivalent, using Remark~\ref{rem:reduction} and Table~\ref{table:1}: in case (1), $\xi_1\otimes\xi_2=\nu_F^{\pm 1}\otimes\nu_F^{s_2}={}^{s_{\beta}}(\nu_F^{s_2\pm 1}\otimes \nu_F^{-s_2})$ by Table~\ref{table:1}, which then falls under case (2) by Remark~\ref{rem:reduction}. Similarly, case (5) is equivalent to case (4) by using the action of $s_\alpha$.

We will treat case (2), case (3) and case (4) separately.

We recall first that the character $\xi$ admits a polar decomposition 
\begin{equation}
    \xi=\nu_F^{t}\omega, \quad\text{where $t\in \R$ and $\omega$ is unitary,}
\end{equation}
(see for instance \cite[(2.2.1)]{Tate-Corvallis}). Writing
$s:=s_2+t\pm 1/2$, we obtain 
\begin{equation} \label{eqn:omega}
    \omega:=\nu_F^{s_2-s\pm 1/2}\xi.
\end{equation}

\subsubsection{{\bf Case (2) within Case $\fs=[T,\xi\otimes\xi]_G$}} \label{subsubsec:case 2-1} 
In case (2), we have 
\begin{equation} \label{eqn:xi-1}
\nu_F^{s_1}\xi=\nu_F^{-s_2\pm 1}\xi^{-1}.
\end{equation}
Thus we have
\begin{equation} \label{eqn:change-case}
    I(\xi_1\otimes\xi_2)=I(\nu_F^{-s_2\pm 1}\xi^{-1}\otimes\nu_F^{s_2}\xi)\overset{ababa}{=}I(\nu_F^{\mp 1}\otimes\xi_2)
\end{equation}

(a) Case $\nu_F^{\mp 1}\otimes\xi_2\notin\left\{\nu_F\otimes 1,\nu_F\otimes\nu_F^2,\nu_F^{-1}\otimes 1,\nu_F^{-1}\otimes\nu_F^{-2}\right\}$: In this case, $I^\alpha(\nu_F^{\mp 1-s}\otimes \xi_2\nu_F^{-s}))$ is irreducible, but $I(\nu_F^{\mp 1}\otimes\xi_2)$ is reducible. We compute the irreducible constituents in the following.

\begin{enumerate}
    \item[(i)] When $\xi_2$ is quadratic, this is covered in Tables~\ref{table:unip-ppal-series-case-3aquad} and \ref{table:unip-ppal-series-case-3aquad-ramified} (depending on whether $\xi_2$ is unramified or not). In this case, $I(\nu_F^{\mp 1}\otimes\xi_2)$ has length $4$.
    \item[(ii)] When $\xi_2$ is 
    non-quadratic, since $I(\nu_F^{\mp 1}\otimes\xi_2)\overset{bababa}{=}I(\nu_F^{\pm 1}\otimes\xi_2^{-1})$. By Lemma~\ref{Jacquet-restriction-alpha-nu-tensor-xi2}, we have 
    \begin{equation}\label{ralpha-I-nu-tensor-xi-inverse}
        \rrr_\alpha I(\nu_F^{\pm 1}\otimes\xi_2^{-1})=I^{\alpha}(\nu_F^{\pm 1}\otimes\xi_2^{-1})+I^{\alpha}(\nu_F^{\mp 1}\otimes\nu_F^{\pm 1}\xi_2^{-1})+I^{\alpha}(\nu_F^{\pm 1}\xi_2^{-1}\otimes\xi_2)+I^{\alpha}(\nu_F^{\pm 1}\otimes\nu_F^{\mp 1}\xi_2).
    \end{equation}
    Let $\pi(\nu_F^{\mp 1}\otimes\xi_2)\cong \pi(\nu_F^{\pm 1}\otimes\xi_2^{-1})$ be an irreducible subrepresentation of $I(\nu_F^{\mp 1}\otimes\xi_2)\cong I(\nu_F^{\pm 1}\otimes\xi_2^{-1})$. Comparing \eqref{ralpha-I-nu-tensor-xi-inverse} with Lemma~\ref{Jacquet-restriction-alpha-nu-tensor-xi2}, we obtain that $\rrr_\alpha I(\nu_F^{\mp 1}\otimes\xi_2)$ and $\rrr_\alpha I(\nu_F^{\pm 1}\otimes\xi_2^{-1})$ share the terms $I^{\alpha}(\nu_F^{\mp 1}\otimes\nu_F^{\pm 1}\xi_2^{-1})$ and $I^{\alpha}(\nu_F^{\pm 1}\otimes\nu_F^{\mp 1}\xi_2)$. Thus we have 
    \begin{equation}
        \rrr_\alpha\pi(\nu_F^{\mp 1}\otimes\xi_2)=I^{\alpha}(\nu_F^{\mp 1}\otimes\nu_F^{\pm 1}\xi_2^{-1})+I^{\alpha}(\nu_F^{\pm 1}\otimes\nu_F^{\mp 1}\xi_2)=\rrr_\alpha\pi(\nu_F^{\pm 1}\otimes\xi_2^{-1}).
    \end{equation}
By \cite[Corollary,p.419]{Rodier-principal-series}, $S=\{s_{3\alpha+\beta}\}$ and $I(\nu_F^{\mp 1}\otimes\xi_2)$ has length $2^{|S|}=2$. Let $J(\nu_F^{\mp 1}\otimes\xi_2)$ denote the irreducible quotient of $I(\nu_F^{\mp 1}\otimes\xi_2)$, and we have 
    \begin{equation}
        I(\nu_F^{\mp 1}\otimes\xi_2)=\pi(\nu_F^{\mp 1}\otimes\xi_2)+J(\nu_F^{\mp 1}\otimes\xi_2).
    \end{equation}

In this case, we have the following two Tables depending on what $\xi_2$ is like. 
    \begin{enumerate}
        \item When $\xi_2$ is ramified cubic, $\mathcal{J}^{\fs}=\SL_3(\C)$. By \cite[p.~145]{ABP-G2} this case corresponds to case $t_b$ in \cite[Table 4.1]{Ram}. By \cite[Theorem~9.4]{Roche-principal-series} the representation $\pi(\nu_F^{\mp 1}\otimes\xi_2)$ corresponds to an Iwahori-spherical, non-square-integrable, tempered representation of $\PGL_3(F)$ and hence is tempered. It is indexed in \cite[Table~4.2]{Ram} by the triple $(t_b,e_{\alpha_2},1)$ with $t_b=s_2s_1t$. The list of irreducible constituents is summarized in 
        Table~\ref{table:nu-tensor-xi-shorttable-ppal-series-case2-cubic-ramified}.

\begin{center}
\begin{table} [ht]
\begin{tabular}{|c|c|c|}
\hline
Representation & enhanced $L$-parameter\cr
\hline\hline
$\pi(\nu_F^{\mp 1}\otimes\xi_2)$&$(\varphi_{\sigma,\rA_1},1)$  \cr
 $J(\nu_F^{\mp 1}\otimes\xi_2)$&$(\varphi_{\sigma,1},1)$ \cr
  \hline
\end{tabular}
\vskip0.2cm
\caption{\label{table:nu-tensor-xi-shorttable-ppal-series-case2-cubic-ramified}: LLC for irred. constituents of $I(\nu_F^{\mp 1}\otimes\xi_2)$ where $\xi_2$ is ramified cubic.}
\end{table}
\end{center}
\item When $\xi_2$ is ramified neither quadratic nor cubic, $\mathcal{J}^{\fs}=\GL_2(\C)$. The list of irreducible constituents is summarized in Table~\ref{table:nu-tensor-xi-shorttable-ppal-series-case2-ramified-neither}. \\
More precisely, this corresponds to case $t_a$ in \cite[Table 2.1]{Ram}, where the tempered representation corresponds to the regular unipotent class $e_{\alpha_1}$ in $\GL_2(\C)$. By Proposition~\ref{principal-series-classification-s-WGs}, we have $W_G^{\fs}=\{1,s_{\beta}\}$.  Thus by Table~\ref{tab:table-unipotent-Gsigma=GL2}, this corresponds to the unipotent class $\widetilde{A}_1$ in $\rG_2(\C)$. 
\begin{center}
\begin{table} [ht]
\begin{tabular}{|c|c|c|c|}
\hline
Representation & enhanced $L$-parameter\cr
\hline\hline
$\pi(\nu_F^{\mp 1}\otimes\xi_2)$&$(\varphi_{\sigma,\widetilde{A}_1},1)$ \cr
 $J(\nu_F^{\mp 1}\otimes\xi_2)$&$(\varphi_{\sigma,1},1)$  \cr
  \hline
\end{tabular}
\vskip0.2cm
\caption{\label{table:nu-tensor-xi-shorttable-ppal-series-case2-ramified-neither}: LLC for irred. constituents of $I(\nu_F^{\mp 1}\otimes\xi_2)$ where $\xi_2$ is ramified neither cubic nor quadratic.}
\end{table}
\end{center}    

\item When $\xi_2$ is unramified non-quadratic, $\mathcal{J}^{\fs}=\rG_2(\C)$. The list of irreducible constituents is summarized in Table~\ref{table:nu-tensor-xi-shorttable-ppal-series-case2-unramified}. We observe that
\begin{equation}
I(\nu^{1-s_2}\xi^{-1}\otimes\nu^{s_2}\xi)=
I_\beta(1/2-s_2,\xi^{-1}\otimes\St_{\GL_2})\oplus I_\beta(1/2-s_2,\xi^{-1}\circ\det).
\end{equation}
It shows that $\pi(\nu_F^{\mp 1}\otimes\xi_2):=I_\beta(1/2-s_2,\xi^{-1}\otimes\St_{\GL_2})$ is tempered. Thus this case corresponds to the case $t_g$ in \cite[Table~6.1]{Ram}.
    
\begin{center}
\begin{table} [ht]
\begin{tabular}{|c|c|c|}
\hline
Representation & enhanced $L$-parameter\cr
\hline\hline
$\pi(\nu_F^{\mp 1}\otimes\xi_2)$&$(\varphi_{\sigma,\rA_1},1)$   \cr
 $J(\nu_F^{\mp 1}\otimes\xi_2)$&$(\varphi_{\sigma,1},1)$ \cr
  \hline
\end{tabular}
\vskip0.2cm
\caption{\label{table:nu-tensor-xi-shorttable-ppal-series-case2-unramified}: LLC for irred. constituents of $I(\nu_F^{\mp 1}\otimes\xi_2)$ where $\xi_2$ is unramified non-quadratic.}
\end{table}
\end{center}
\end{enumerate}
\end{enumerate}

\smallskip
\noindent
(b) Case $\nu_F^{\mp 1}\otimes\xi_2\in\left\{\nu_F\otimes 1,\nu_F\otimes\nu_F^2,\nu_F^{-1}\otimes 1,\nu_F^{-1}\otimes\nu_F^{-2}\right\}$: 
Since, by Table~\ref{table:1}, we have $I(\nu_F\otimes 1)\overset{bababa}{=}I(\nu_F^{-1}\otimes 1)$ and $I(\nu_F\otimes \nu_F^2)\overset{bababa}{=}I(\nu_F^{-1}\otimes \nu_F^{-2})$, we are reduced to consider 
$\nu_F^{\mp 1}\otimes\xi_2\in\left\{\nu_F\otimes 1,\nu_F\otimes\nu_F^2\right\}$.
The LLC for the irreducible components of $I(\nu_F\otimes\nu_F^2)$ will be computed in Table~\ref{table:unip-ppal-series-case-3bbis}.

By \cite[Proposition~4.3]{Muic-G2}, the representation $I(\nu_F\otimes 1)$ contains exactly two irreducible subrepresentations $\pi(1)$ and $\pi(1)'$, which are square-integrable, and we have in $R(G)$
\begin{align} \label{eqn:list4.3}
    I_\alpha(1/2,\delta(1))=\pi(1)'+J_\alpha(1/2,\delta(1))+J_\beta(1/2,\delta(1))\\
    I_\beta(1/2,\delta(1))=\pi(1)+\pi(1)'+J_\beta(1/2,\delta(1))\\
    I_\alpha(1/2,1_{\GL_2})=\pi(1)+J_\beta(1/2,\delta(1))+J_\beta(1/2,\pi(1,1))\\
    I_\beta(1/2,1_{\GL_2})=J_\beta(1,\pi(1,1))+J_\beta(1/2,\delta(1))+J_\alpha(1/2,\delta(1)).
\end{align}
Hence $I(\nu_F\otimes 1)$  has length $6$: its other irreducible constituents  are  $J_\alpha(1/2,\delta(1))$, $J_\beta(1/2,\delta(1))$, $J_\beta(1,\pi(1,1))$, the latter occurring with multiplicity $2$.

\begin{lem} \label{pi-1-prime-generic}
The representation $\pi(1)'$ is generic.
\end{lem}
\begin{proof}
By \cite[Theorem~8]{Rodier-Whittaker-models}, the representation $\delta(1)=\St_{\GL_2}$ is generic. Hence by \eqref{eqn:list4.3}, both $I_\alpha(1/2,\delta(1))$ and $I_\beta(1/2,\delta(1))$ contain a generic irreducible component, and that none of the representations $\pi(1)$, $J_\beta(1/2,\delta(1))$, $J_\beta(1,\pi(1,1))$ and $J_\alpha(1/2,\delta(1))$ is generic. Thus, $\pi(1)'$ is generic. 
\end{proof}
 
Using \cite[Table 6.1]{Ram} Case $t_e$, we obtain the enhanced unipotent conjugacy classes, which give the enhanced $L$-parameters, and the dimensions of the involved irreducible representations, as in Table~\ref{table:unram-unip-ppal-series-case}. More precisely, $J_{\alpha}(1/2,\delta(1))$ has dimension $1$ and corresponds to $e_{\alpha^\vee}$, which corresponds to the minimal unipotent class $\rA_1$. On the other hand, $J_{\beta}(1/2,\delta(1))$ corresponds to trivial unipotent class. The two square-integrable representations $\pi(1)$ and $\pi(1)'$ correspond to $e_{\alpha^\vee}+e_{\alpha^\vee+2\beta^\vee}$, which, as mentioned in \cite[Table~6.3]{Ram}, is a representative of the nilpotent class corresponding to  $\rG_2(a_1)$. For any representative $u$ of  $\rG_2(a_1)$, we have $A_{\rG_2(\C)}(u)\simeq
S_3$, the symmetric group on three elements, which has irreducible representations indexed by the partitions $(3)$, $(21)$, $(13)$ of $3$, where $(3)$ is the trivial one. 
The generic representation $\pi(1)'$ corresponds to  $(t_e,\rG_2(a_1),(3))$, while $\pi(1)$ corresponds to  $(t_e,\rG_2(a_1),(2,1))$. 
There are five indexing triples, as in Table~\ref{table:unram-unip-ppal-series-case}. Note that this table is consistent with the list for $\pi_{\sqrt{q}}^{-1}(t_e)$ in \cite[p.143]{ABP-G2}. 
\begin{center} 
\begin{table} [ht]
\begin{tabular}{|c|c|c|}
\hline
Indexing triple &unipotent orbit & representation\\
\hline
$(t_e,0,1)$& 1 &$J_{\beta}(1,\pi(1,1))$ \\
$(t_e, e_{\alpha^{\vee}},1)$&$\rA_1$ &$J_{\alpha}(1/2,\delta(1))$\\
$(t_e,e_{\beta^{\vee}+\alpha^{\vee}},1)$ &$\widetilde\rA_1$ & $J_{\beta}(1/2,\delta(1))$\\
$(t_e,e_{\alpha^\vee}+e_{2\beta^\vee+\alpha^\vee},(21))$& $\rG_2(a_1)$ &$\pi(1)$\\
$(t_e,e_{\alpha^\vee}+e_{2\beta^\vee+\alpha^\vee},(3))$&$\rG_2(a_1)$&$\pi(1)'$\\
\hline
\end{tabular}
\end{table}
\end{center}

\begin{center}
\begin{table} [ht]
\begin{tabular}{|c|c|c|}
\hline
Representation
& enhanced $L$-parameter&dimension
\cr
\hline\hline
$\pi(1)'$&$(\varphi_{\sigma,\rG_2(a_1)},(3))$&3
\cr
\cline{1-3}
$\pi(1)$&$(\varphi_{\sigma,\rG_2(a_1)},(21))$&1
\cr
\cline{1-3}
$J_\beta(1,\pi(1,1))$&$(\varphi_{\sigma,1},1)$&$3$
\cr
\cline{1-3}
$J_\beta(1/2,\delta(1))$&$(\varphi_{\sigma,\widetilde\rA_1},1)$&$2$
\cr
 \cline{1-3}
$J_\alpha(1/2,\delta(1))$&$(\varphi_{\sigma,\rA_1},1)$&$1$
\cr
 \hline
\end{tabular}
\vskip0.2cm
\caption{\label{table:unram-unip-ppal-series-case}: LLC for the irreducible components of $I(\nu_F\otimes 1)$. $\qquad$ $\qquad$}
\end{table}
\end{center}

\subsubsection{{\bf Case (3) within Case $\fs=[T,\xi\otimes\xi]_G$}} \label{subsubsec:Case 3-1} \ 

In case (3), $\xi_1=\nu_F^{s_1}\xi=\nu_F^{s_1-s_2}\xi_2=\nu_F^{\pm 1}\xi_2$. 
By \eqref{decomposing-Inus1-tensor-Inus2-case-1} we have 
\begin{equation}
\begin{split}
I(\xi_1\otimes\xi_2)&=I_{\alpha}(s,\delta(\nu_F^{\pm 1/2-s}\xi_2))+I_{\alpha}(s,\nu_F^{\pm 1/2-s}\xi_2\circ\det)\\
&=I_{\alpha}(s,\delta(\nu_F^{s_2-s \pm 1/2}\xi))+I_{\alpha}(s,\nu_F^{s_2-s\pm 1/2}\xi\circ\det).
\end{split}
\end{equation}
By \cite[Theorem 3.1 (i)]{Muic-G2}, if the character $\omega$ defined in \eqref{eqn:omega} is unitary, then  $I_{\alpha}(s,\delta(\nu_F^{s_2-s\pm 1/2}\xi))$ and $I_{\alpha}(s,\nu_F^{s_2-s\pm 1/2}\xi\circ\det)$ are irreducible unless 
\[
\text{$s=\pm 1/2$, $\nu_F^{s_2-s\pm 1/2}\xi$ is quadratic,}\quad \text{$s=\pm 3/2$, $\nu_F^{s_2-s\pm 1/2}\xi=1$}\quad \text{or $s=\pm 1/2$, $\nu_F^{s_2-s\pm 1/2}\xi$ is cubic.}
\]
When they are irreducible, we compute the $L$-parameters for irreducible constituents and record them in Table~\ref{table:xi-tensor-xi-shorttable-ppal-series-case5} (when $\xi$ is unramified) and  Table~\ref{table:xi-tensor-xi-shorttable-ppal-series-case5-ramified} (when $\xi$ is ramified). \\

\hrule

\smallskip

When they are reducible, we have in $R(G)$:

(a) Case ($s=\pm 1/2$, $\nu_F^{s_2-s\pm 1/2}\xi=\eta_2$ quadratic):
\[I(\xi_1\otimes\xi_2)=I(\nu_F^{s\pm 1/2}\eta_2\otimes\nu_F^{s\mp 1/2}\eta_2)=I(\nu_F^{\pm 1}\nu_F^{s\mp 1/2}\eta_2\otimes \nu_F^{s\mp 1/2}\eta_2)
,\]
and thus Lemma~\ref{lem:main-induced} applies to this case, and we have
\begin{equation}\label{decompose-I-xi1-tensor-xi2-eta2}
    I(\xi_1\otimes\xi_2)=I_{\alpha}(s, \delta(\eta_2))+I_{\alpha}(s,\eta_2\circ\det).
\end{equation}
We now plug in $s=\pm 1/2$ and obtain
\[I(\xi_1\otimes\xi_2)=I(\nu_F^{\pm 1}\eta_2\otimes\eta_2)
=I_\alpha(\pm 1/2,\delta(\eta_2))+I_\alpha(\pm 1/2,\eta_2\circ\det).\]
Note that by \cite[Lemma 5.4(iii)]{BDK}, we have
\begin{equation}
I_{\alpha}(-1/2,\delta(\eta_2))\simeq I_{\alpha}(1/2,\delta(\eta_2))\quad\text{and}\quad
I_{\alpha}(-1/2,\eta_2\circ\det)\simeq I_{\alpha}(1/2,\eta_2\circ\det).
\end{equation}
Therefore in the above we treat $\pm 1/2$ cases together. 

Recall from \eqref{eqn:omega}, we have that $\eta_2=\omega$ is unitary. Thus \cite[Proposition 4.1 (ii)]{Muic-G2} applies, and 
\begin{align}\label{pi-eta2-occuring-in-Ialpha}
    I_{\alpha}(s,\delta(\eta_2))&=\pi(\eta_2)+J_{\alpha}(s,\delta(\eta_2))\\
    I_{\alpha}(s,\eta_2\circ\det)&=J_{\beta}(1,\pi(1,\eta_2))+J_{\beta}(s,\delta(\eta_2)).
\end{align}
Combined with \eqref{decompose-I-xi1-tensor-xi2-eta2}, we have 
\begin{equation}
I(\xi_1\otimes\xi_2)
=\pi(\eta_2)+J_\beta(1,\pi(1,\eta_2))+J_\beta(1/2,\delta(\eta_2))+J_{\alpha}(1/2,\delta(\eta_2)).
\end{equation}
We observe that $I(\nu_F^{-2}\otimes\nu_F^{1})\overset{babbaba}{=}I(\nu_F^2\otimes \nu_F)$ in $R(G)$. 

$\bullet$ When $\eta_2$ is \textit{unramified} quadratic, $\mathcal{J}^{\mathfrak{s}}=\rG_2(\C)$ as in \cite[$\mathsection$ 4]{ABP-G2}, and the $L$-parameters for irreducible constituents are recorded in Table~\ref{table:unip-ppal-series-case-3aquad}.   More precisely, we are in case $t_d$ in \cite[Table 6.1]{Ram} (see also \cite[p.17]{ABP-G2}), there are four indexing triples
\begin{center} 
\begin{table} [ht]
\begin{tabular}{|c|c|c|}
\hline
Indexing triple &unipotent orbit & representation\\
\hline
$(t_d,0,1)$& 1 &$J_{\beta}(1,\pi(1,\eta_2))$\\
$(t_d, e_{\alpha^\vee}, 1)$&$\rA_1$ &$J_{\alpha}(1/2,\delta(\eta_2))$\\
$(t_d,e_{2\beta^\vee+\alpha^\vee},1)$ & $\widetilde{\rA}_1$ & $J_{\beta}(1/2,\delta(\eta_2))$\\
$(t_d,e_{\alpha^\vee}+e_{\alpha^\vee+2\beta^\vee},1)$& $\rG_2(a_1)$ &$\pi(\eta_2)$\\
\hline
\end{tabular}
\end{table}
\end{center}
Note that $(t_d,0,1)$ corresponds to the trivial unipotent in $\mathcal{J}^{\fs}$. In the Langlands classification (see for instance \cite{Silberger-Zink}) $(s_1t_d,\{1\})$ corresponds to the unique $J_{\alpha}$ term; since $(s_1t_d,\{1\})$ corresponds to $e_{\alpha^\vee}$, which is a representative of the minimal nilpotent orbit in $\mathcal{J}^{\fs}$. 
Likewise, the parameters $(t_d,\{2\})$ and $(s_2s_1s_2t_d,\{2\})$ correspond to the two $J_{\beta}$ terms. The element $e_{\alpha^\vee}+e_{\alpha^\vee+2\beta^\vee}$ corresponds to the subregular unipotent orbit $\rG_2(a_1)$ in $\mathcal{J}^{\fs}$ (see \cite[Table~6.3]{Ram}). 
\begin{center} 
\begin{table} [ht]
\begin{tabular}{|c|c|}
\hline
Representation
& enhanced $L$-parameter
\cr
\hline\hline
$\pi(\eta_2)$&$(\varphi_{\sigma,\rG_2(a_1)},1)$\cr
$J_\beta(1,\pi(1,\eta_2))$&$(\varphi_{\sigma,1},1)$ 
\cr
$J_\beta(1/2,\delta(\eta_2))$&$(\varphi_{\sigma,\widetilde{\rA}_1},1)$ \cr
$J_{\alpha}(1/2,\delta(\eta_2))$&$(\varphi_{\sigma,\rA_1},1)$ \cr
 \hline
\end{tabular}
\vskip0.2cm
\caption{\label{table:unip-ppal-series-case-3aquad} LLC for the irreducible components of $I(\nu_F\otimes \eta_2)$ for $\eta_2$ unramified quadratic. $\qquad$ $\qquad$}
\end{table}
\end{center}

$\bullet$ When $\eta_2$ is \textit{ramified} quadratic, by Proposition~\ref{principal-series-classification-s-WGs}, we have $\mathcal{J}^{\mathfrak{s}}=\SO_4(\C)$ as in \cite[\S~8]{ABP-G2}. By \cite[Theorem~4.7]{ABPS-KTheory}, we are reduced to compute the Kazhdan-Lusztig triples for the $F$-split group $\SO_4(F)$. 
Since $\SO_4(\C)\simeq\SL_2(\C)\times\SL_2(\C)/\{\pm 1\}$, we have $\cH(\SO_4(\C),\mcI)\simeq (\SL_2(\C)\times\SL_2(\C))^{\Z/2\Z}$ by \cite[\S1.5]{Reeder-isogeny}.  By \cite[proof of Lemma~8.3]{ABP-G2}, both $I(\nu_F\otimes \eta_2)$ and $I(\nu_F\eta_2\otimes \eta_2)$ have length $4$.  The $L$-parameters for irreducible constituents are recorded in Table~\ref{table:unip-ppal-series-case-3aquad-ramified}. More precisely: 

\begin{lem}
$\pi(\eta_2)$ is generic.
\end{lem}
\begin{proof}
The proof is very similar to that of Lemma~\ref{pi-1-prime-generic}. Recall from \eqref{pi-eta2-occuring-in-Ialpha} that $I_{\alpha}(s,\delta(\eta_2))=\pi(\eta_2)+J_{\alpha}(s,\delta(\eta_2))$. Since $\delta(\eta_2)$ is generic, $I_{\alpha}(1/2,\delta(\eta_2))$ must contain a generic representation as an irreducible constituent. On the other hand, since $\eta_2\circ\det$ is not generic, by Lemma~\ref{lem:generic}, all of the irreducible constituents occuring in $I_{\beta}(1/2,\eta_2\circ\det)$ are non-generic. In particular, by \cite[4.1(ii)]{Muic-G2}, $J_{\alpha}(1/2,\delta(\eta_2))$ is not generic. Therefore $\pi(\eta_2)$ must be generic.  
\end{proof}
Since $\pi(\eta_2)$ is generic, it has trivial enhancement as in the first row of Table~\ref{table:unip-ppal-series-case-3aquad-ramified}. By \cite{Roche-principal-series}, we have 
\begin{equation}\label{Roche-isomorphism}
\Irr(\mathcal{H}(G,\tau^\fs))\xrightarrow{1--1} \Irr(\mathcal{H}(J^{\fs},1)).
\end{equation} 
Since $\pi(\eta_2)$ is a discrete series, its image under \eqref{Roche-isomorphism} is also a discrete series, which by \cite[\S8.0.3]{ABP-G2} corresponds to a Kazhdan-Lusztig triple with unipotent class in $\SO_4(\C)$ the regular unipotent class, i.e. by Table~\ref{table:SO4} the subregular unipotent $\rG_2(a_1)$ in $\rG_2$. 
On the other hand, consistent with the tables in \cite{Ram} and our previous tables, $J_{\beta}(1,\pi(1,\eta_2))$ corresponds to the trivial unipotent class. 
It now remains to distribute $J_{\beta}(1/2,\delta(\eta_2))$ and $J_{\alpha}(1/2,\delta(\eta_2))$ among the two unipotent classes $(2,2)'$ and $(2,2)''$ in $\SO_4(\C)$. 
By the indexing triples in \cite[\S2]{Ram}, $J_{\beta}(1/2,\delta(\eta_2))$ corresponds to $e_{\alpha_2}$, which corresponds to $e_{\beta^{\vee}}$, which then corresponds to $\widetilde{\rA}_1$. 

We now compute the formal degree of $\pi(\eta_2)$: by \cite[Theorem~10.7]{Roche-principal-series}, up to normalization factors of volumes, we have
\begin{equation}
    \fdeg(\pi(\eta_2))=d(\St_{\SO_4}^{\mathcal{H}}),
\end{equation}
where $\St_{\SO_4}$ is the Steinberg representation of $\SO_4(F)$. 
Since $|Z(G^{\vee})|=2$ and $|A(s,u)|=1$, by~\cite[Theorem 4.1]{Ciubotaru-Kato-Kato}, we have
\begin{equation}    
d(\St_{\SO_4}^{\mathcal{H}})=\frac{1}{2}\cdot\frac{q-1}{q^2-1}\cdot \frac{q-1}{q^2-1}\cdot q^2
    =\frac{q^2}{2(q+1)^2}.
\end{equation}
Thus we have
\begin{equation}\label{fdeg-pi-eta2}
    \fdeg(\pi(\eta_2))=\frac{q^2}{2(q+1)^2},
\end{equation}
which agrees with the formal degree for the singular supercuspidal computed in \eqref{eqn:fdegx2}.

\begin{center} 
\begin{table} [ht]
\begin{tabular}{|c|c|}
\hline
Representation
& enhanced $L$-parameter
\cr
\hline\hline
$\pi(\eta_2)$&$(\varphi_{\sigma,\rG_2(a_1)},1)$ \cr
$J_\beta(1,\pi(1,\eta_2))$&$(\varphi_{\sigma,1},1)$ \cr
$J_\beta(1/2,\delta(\eta_2))$&$(\varphi_{\sigma,\widetilde{A}_1},1)$ \cr
$J_{\alpha}(1/2,\delta(\eta_2))$&$(\varphi_{\sigma,A_1},1)$ \cr
 \hline
\end{tabular}
\vskip0.2cm
\caption{\label{table:unip-ppal-series-case-3aquad-ramified} LLC for the irreducible components of $I(\nu_F\otimes \eta_2)$ for $\eta_2$ ramified quadratic. $\qquad$ $\qquad$}
\end{table}
\end{center}
(b) Case ($s=\pm 3/2$, $\nu_F^{s_2-s\pm 1/2}\xi=1$):  $\xi_2=\nu_F^{s_2}\xi=\nu_F^{s\mp 1/2}$, and $\xi_1=\nu_F^{\pm}\xi_2=\nu_F^{\pm 1/2+s}$. 
\[I(\xi_1\otimes \xi_2)=I(\nu_F^{\pm 1/2+s}\otimes \nu_F^{s\mp 1/2})=I(\nu_F^{\pm 1}\nu_F^{s\mp 1/2}\otimes\nu_F^{s\mp 1/2}),\]
and thus Lemma~\ref{lem:main-induced} applies to this case, and we have 
\begin{equation}
    I(\xi_1\otimes \xi_2)=I_{\alpha}(s,\delta(\nu_F^{\pm 1/2-s}\xi_2))+I_{\alpha}(s,\nu_F^{\pm 1/2-s}\xi_2\circ\det).
\end{equation}
We now plug in $s=\pm 3/2$ and the explicit values for $\xi_2$, and obtain
\[I(\xi_1\otimes\xi_2)=I(\nu_F^{\pm 2}\otimes\nu_F^{\pm 1})
=I_{\alpha}(\pm 3/2,\delta(1))+I_{\alpha}(\pm 3/2,1_{\GL_2}).\]
By \cite[Lemma 5.4(iii)]{BDK}, we have
\begin{equation}
I_{\alpha}(-3/2,\delta(1))\simeq I_{\alpha}(3/2,\delta(1))\quad\text{and}\quad
I_{\alpha}(-3/2,1_{\GL_2})\simeq I_{\alpha}(3/2,1_{\GL_2}).
\end{equation}
Therefore it suffices to treat the $\pm 3/2$ cases together. When $s=3/2$, by \cite[Proposition~4.4]{Muic-G2}, we have the following decomposition into irreducible constituents:
\begin{equation}
I_{\alpha}(3/2,\delta(1))=\St_{\rG_2}+J_\alpha(3/2,\delta(1))\quad\text{and}\quad
I_{\alpha}(3/2,1_{\GL_2})=1_{\rG_2}+J_\beta(5/2,\delta(1)).
\end{equation}
We thus obtain
\begin{equation}
I(\xi_1\otimes\xi_2)=I(\nu_F^{\pm 2}\otimes\nu_F^{\pm 1})=\St_{\rG_2}+J_\alpha(3/2,\delta(1))+1_{\rG_2}+J_\beta(5/2,\delta(1)).
\end{equation}

In the notation of \cite[Table~6.1]{Ram} the central character of the irreducible components of $I(\nu_F^2\otimes \nu_F)$ is denoted $t_a$. Thus two representations have dimension $1$ (the trivial representation $1_{\rG_2}$ and the Steinberg representation $\St_{\rG_2}$) and the two have dimension $5$. These four irreducible representations are also listed in \cite[p.143]{ABP-G2} as the elements of $\pi^{-1}_{\sqrt{q}}(t_a)$. For each root $\gamma^\vee$ in $R(G^\vee,T^\vee)$ let denote by $e_{\gamma^\vee}$ an element of the root space $\fg^\vee_{\gamma^\vee}$. The representation $\St_{\rG_2}$ is square-integrable, hence in particular it is tempered. From \cite[Table~6.1]{Ram}, it has indexing triple $(t_a,e_{\alpha^\vee}+e_{\beta^\vee},1)$. Then the first row of \cite[Table~6.3]{Ram} shows that it is attached to the regular unipotent class $\rG_2$. By \cite[Table~6.1]{Ram}, the indexing triple of $1_{\rG_2}$ is $(t_a,0,1)$, and those of $J_\alpha(3/2,\delta(1))$ and $J_\beta(5/2,\delta(1))$ are $(t_a,e_{\alpha^\vee},1)$ and $(t_a,e_{\beta^\vee},1)$ or $(t_a,e_{\beta^\vee},1)$ and $(t_a,e_{\alpha^\vee},1)$. Also, by \cite[Table~6.3]{Ram}, since $\beta^\vee$ is short, we have that $e_{\beta^\vee}$  corresponds to the minimal unipotent class $A_1$, while $e_{\alpha^\vee}$ corresponds to the subminimal one $\widetilde \rA_1$. 
In this case, $\mathcal{J}^{\fs}=\rG_2(\C)$ as in \cite[\S4]{ABP-G2}, and the $L$-parameters for irreducible constituents are recorded in Table~\ref{table:unip-ppal-series-case-3bbis}.
More precisely, we are in case $t_a$ in \cite[Table 6.1]{Ram} and there are four indexing triples as in Table~\ref{table:unip-ppal-series-case-3bbis} (note that this table is consistent with the case $\pi_{\sqrt{q}}^{-1}(t_a)$ in \cite[p.143]{ABP-G2}). 
Table~\ref{table:unip-ppal-series-case-3bbis} shows that there are four $L$-packets. Each $L$-packet is a singleton, hence the attached group $\mcG_{\sigma,u}$ is connected.
\begin{center} 
\begin{table} [ht]
\begin{tabular}{|c|c|c|}
\hline
Indexing triple &unipotent orbit & representation\\
\hline
$(t_a,0,1)$& 1 &$1_{G_2}$ \\
$(t_a, e_{\beta^\vee}, 1)$&$\widetilde{\rA}_1$ &$J_{\beta}(5/2,\delta(1))$\\
$(t_a,e_{\alpha^\vee},1)$ &$A_1$ & $J_{\alpha}(3/2,\delta(1))$\\
$(t_a,e_{\alpha^\vee}+e_{\beta^\vee},1)$& $\rG_2$ &$\St_{\rG_2}$\\
\hline
\end{tabular}
\end{table}
\end{center}
\begin{center}
\begin{table} [ht]
\begin{tabular}{|c|c|}
\hline
Representation
& enhanced $L$-parameter 
\cr
\hline\hline
$\St_{\rG_2}$&$(\varphi_{\sigma,\rG_2},1)$
\cr
$J_\alpha(3/2,\delta(1))$&$(\varphi_{\sigma,A_1},1)$
\cr
$J_\beta(5/2,\delta(1))$&$(\varphi_{\sigma, \widetilde{A}_1},1)$
\cr
$1_{\rG_2}$&$(\varphi_{\sigma,1},1)$
\cr
 \hline
\end{tabular}
\vskip0.2cm
\caption{\label{table:unip-ppal-series-case-3bbis}: LLC for the irreducible components of $I(\nu_F^{2}\otimes\nu_F$).}
\end{table}
\end{center}
(c) Case ($s=\pm 1/2$, $\nu_F^{s_2-s\pm 1/2}\xi=\eta_3$ cubic): 
\[
I(\xi_1\otimes\xi_2)=I(\nu_F^{s\pm 1/2}\eta_3\otimes\nu_F^{s\mp 1/2}\eta_3)=I(\nu_F^{\pm 1}\nu_F^{s\mp 1/2}\eta_3\otimes\nu_F^{s\mp 1/2}\eta_3),
\]
and thus Lemma~\ref{lem:main-induced} applies to this case, and we have the decomposition into irreducible constituents:
\begin{equation}
    I(\xi_1\otimes\xi_2)=I_{\alpha}(s,\delta(\eta_3))+I_{\alpha}(s,\eta_3\circ\det).
\end{equation}
We now plug in $s=\pm 1/2$ and obtain
\[I(\xi_1\otimes\xi_2)=I(\nu_F^{\pm 1}\eta_3\otimes\eta_3)=I_{\alpha}(\pm 1/2,\delta(\eta_3))+I_{\alpha}(\pm 1/2,\eta_3\circ\det).\]
By \cite[Lemma~5.5(iii)]{BDK}, we have
\[I_{\alpha}(-1/2,\delta(\eta_3))=I_{\alpha}( 1/2,\delta(\eta_3^{-1}))\quad\text{and}\quad I_{\alpha}(- 1/2,\eta_3\circ\det)=I_{\alpha}(1/2,\eta_3^{-1}\circ\det).\]
The following essentially follows from \cite[Proposition~4.2(ii)]{Muic-G2}, for the reader's convenience, we include a more detailed proof than \textit{loc.cit.}. When $\xi_2$ is cubic, consider the following two cases. 
    \begin{enumerate}
        \item If $3s+1/2=\pm 1$, i.e. $s=1/6$ or $-1/2$, then $I^\alpha(\nu_F^{2s}\xi^2\otimes\nu_F^{-s-1/2}\xi^{-1})$ in  \eqref{eqn:Jacquet-ab} becomes
        \[I^\alpha(\nu_F^{2s}\xi^2\otimes\nu_F^{-s-1/2}\xi^{-1})=\delta(\nu_F^{2s\mp 1/2}\xi^2)+\nu_F^{2s\mp 1/2}\xi^2\circ\det,\]
        and the term $I^\alpha(\nu_F^{s+1/2}\xi\otimes\nu_F^{-2s}\xi^{-2})$ in \eqref{eqn:Jacquet-aa} becomes
        \[I^\alpha(\nu_F^{s+1/2}\xi\otimes\nu_F^{-2s}\xi^{-2})=\delta(\nu_F^{-2s\pm 1/2}\xi)+\nu_F^{-2s\pm 1/2}\xi\circ\det.\]
        By \cite[Theorem 3.1]{Muic-G2}, we only need to consider the case where $s=-1/2$. In this case, we have
        \begin{equation}    \rrr_\alpha(I_{\alpha}(s,\delta(\xi)))=\nu_F^s\delta(\xi)+\nu_F^{-s}\delta(\xi^{-1})+\nu_F^{s+1}\delta(\xi)+\nu_F^{s+1}\xi\circ\det+I^{\alpha}(\nu_F^{-1}\xi^{-1}\otimes\nu_F\xi^{-1}).
        \end{equation}
        Thus plugging $s=-1/2$ in the above and comparing with the analogous formula for $\rrr_\alpha(I_{\alpha}(s),\delta(\xi^{-1}))$, we obtain that $\rrr_\alpha(I_{\alpha}(-1/2,\delta(\xi)))$ and $\rrr_\alpha(I_{\alpha}(-1/2,\delta(\xi^{-1})))$ share the common terms $\nu_F^{1/2}\delta(\xi)+\nu_F^{1/2}\delta(\xi^{-1})$ (note that this is only true for very specially chosen values of $s$), thus 
        \begin{equation}
            \rrr_\alpha(\pi(\xi))=\nu_F^{1/2}\delta(\xi)+\nu_F^{1/2}\delta(\xi^{-1}),
        \end{equation}
        and collecting the differing terms we obtain
        \begin{equation}
            \rrr_\alpha(J_{\alpha}(-1/2,\delta(\xi)))=\delta(\nu_F^{-1/2}\xi)+\nu_F^{1/2}\xi\circ\det+I^{\alpha}(\nu_F^{-1}\xi^{-1}\otimes\nu_F\xi^{-1}).
        \end{equation}
        In summary, we have 
        \begin{equation}
            I_{\alpha}(-1/2,\delta(\xi))=\pi(\xi)+J_{\alpha}(-1/2,\delta(\xi)).
        \end{equation}
        
        \item If $3s-1/2=\pm 1$, i.e. $s=1/2$ or $-1/6$; similarly, by \cite[Theorem 3.1]{Muic-G2}, we only need to consider the case where $s=1/2$. In this case, we have 
        \begin{equation}
            \rrr_\alpha(I_{\alpha}(s,\delta(\xi)))=\nu_F^s\delta(\xi)+\nu_F^{-s}\delta(\xi^{-1})+\nu_F^{1-s}\delta(\xi^{-1})+\nu_F^{1-s}\xi^{-1}\circ\det+I^{\alpha}(\nu^{s+1/2}\xi\otimes\nu_F^{-2s}\xi).
        \end{equation}
        Thus plugging $s=1/2$ in the above and comparing with the analogous formula for $\rrr_\alpha(I_{\alpha}(s,\delta(\xi^{-1})))$, we obtain that $\rrr_\alpha(I_{\alpha}(1/2,\delta(\xi)))$ and $\rrr_\alpha(I_{\alpha}(1/2,\delta(\xi^{-1})))$ share the common terms $\nu_F^{1/2}\delta(\xi)+\nu_F^{1/2}\delta(\xi^{-1})$, thus 
        \begin{equation}
            \rrr_\alpha(\pi(\xi))=\nu_F^{1/2}\delta(\xi)+\nu_F^{1/2}\delta(\xi^{-1}),
        \end{equation}
        and collecting the differing terms we obtain 
        \begin{equation}
            \rrr_\alpha(J_{\alpha}(1/2,\delta(\xi)))=\nu_F^{-1/2}\delta(\xi^{-1})+\nu_F^{1/2}\xi^{-1}\circ\det+I^{\alpha}(\nu_F\xi\otimes\nu_F^{-1}\xi).
        \end{equation}
        In summary, we still have 
        \begin{equation}
            I_{\alpha}(1/2,\delta(\xi))=\pi(\xi)+J_{\alpha}(1/2,\delta(\xi)).
        \end{equation}
    \end{enumerate}
The Jacquet restrictions for $J_{\alpha}(s,\xi\circ\det)$ can be computed similarly. 

\begin{numberedparagraph}
Returning to the specific setting involving $\eta_3$, we obtain:
\[I_{\alpha}(1/2,\delta(\eta_3^{\pm 1}))=\pi(\eta_3^{\pm 1})+J_\alpha(1/2,\delta(\eta_3^{\pm 1}))\quad\text{and}\]
\[
I_{\alpha}(1/2,\eta_3^{\pm 1}\circ\det)=J_\beta(1,\pi(\delta_3^{\pm 1},\delta_3^{\mp 1})) + J_\alpha(1/2,\delta(\eta_3^{\mp 1})).\]
In summary, we obtain the following decomposition into irreducible constituents:
\begin{equation}
I(\xi_1\otimes\xi_2)= I(\nu_F^{\pm 1}\eta_3\otimes\eta_3)=\pi(\eta_3^{\pm 1})+J_\alpha(1/2,\delta(\eta_3^{\pm 1}))+J_\beta(1,\pi(\delta_3^{\pm 1},\delta_3^{\mp})) + J_\alpha(1/2,\delta(\eta_3^{\mp})).   
\end{equation}

When $\eta_3$ is unramified, it follows from \cite[p.142]{ABP-G2} that we are in the case of infinitesimal character $t_c$ of \cite[Table~6.1]{Ram}.
The representation $\pi(\eta_3^{\pm 1})$ is a $2$-dimensional square-integrable representation, with non-real central character (see \cite[\S6]{Ram}). The three other irreducible representations are non-tempered. 
It has indexing triple $(t_c,e_{\alpha^\vee}+e_{3\beta^\vee+\alpha^\vee},1)$.  As observed in \cite[p.~25]{Ram}, it follows from \cite[Satz~1']{Stuhler}, that  $e_{\alpha^\vee}+e_{3\beta^\vee+\alpha^\vee}$ is a representative of the subregular unipotent orbit $\rG_2(a_1)$. The representation $J_{\beta}(1,\pi(\delta_3^{\pm 1},\delta_3^{\mp}))$ has dimension $2$ and its indexing triple is $(t_c,0,1)$. The representations $J_\alpha(1/2,\delta(\eta_3^{\pm 1}))$ and $J_\alpha(1/2,\delta(\eta_3^{\mp}))$ are both of  dimension $4$, with indexing triples $(t_c,e_\alpha,1)$ and $(t_c,e_{\alpha+3\beta},1)$. 

\begin{center} 
\begin{table} [ht]
\begin{tabular}{|c|c|c|}
\hline
Indexing triple &unipotent orbit & representation\\
\hline
$(t_c,e_{\alpha^\vee}+e_{3\beta^\vee+\alpha^\vee},1)$& $\rG_2(a_1)$ &$\pi(\eta_3^{\pm})$
 \\
$(t_c, e_{\alpha^\vee}, 1)$&$\rA_1$ &$J_{\alpha}(1/2,\delta(\eta_3^{\pm}))$\\
$(t_c,e_{3\beta^\vee+\alpha^\vee},1)$ & $\rA_1$ & $J_{\alpha}(1/2,\delta(\eta_3^{\mp}))$ \\
$(t_c,0,1)$& 1 &$J_{\beta}(1,\pi(\delta_3^{\pm 1},\delta_3^{\pm}))$\\
\hline
\end{tabular}
\end{table}
\end{center}

In this case, $\mathcal{J}^{\fs}=\rG_2(\C)$ as in \cite[$\mathsection$ 4]{ABP-G2}, and the enhanced $L$-parameters are as in Table~\ref{table:unip-ppal-ser3c-unramified}: 
\begin{center} 
\begin{table} [ht]
\begin{tabular}{|c|c|}
\hline
Representation & enhanced $L$-parameter\cr
\hline\hline
$\pi(\eta_3^{\pm 1})$&$(\varphi_{\sigma,\rG_2(a_1)},1)$ 
\cr
$J_\alpha(1/2,\delta(\eta_3^{\pm 1}))$&$(\varphi'_{\sigma,\rA_1},1)$\cr
$J_\alpha(1/2,\delta(\eta_3^{\mp}))$&$(\varphi''_{\sigma,\rA_1},1)$\cr
$J_{\beta}(1,\pi(\delta_3^{\pm 1},\delta_3^{\mp}))$&$(\varphi_{\sigma,1},1)$\cr
 \hline
\end{tabular}
\vskip0.2cm
\caption{\label{table:unip-ppal-ser3c-unramified}: LLC for the irreducible components of $I(\nu_F^{\pm 1}\eta_3\otimes\eta_3)$ for $\eta_3$ cubic unramified.$\qquad$}
\end{table}
\end{center}
\end{numberedparagraph}

\begin{numberedparagraph}
When $\eta_3$ is ramified, the group $\mcJ^\fs$ is isomorphic to $\SL_3(\C)$ with simple roots $\alpha^\vee$ and $3\beta^\vee+2\alpha^\vee$. We compute the $L$-parameters by combining the case $t_a$ in \cite[Table 4.1]{Ram} with Table~\ref{tab:table-unipotent-Gsigma=SL3}. 
There are four indexing triples as in the following table.
\begin{center} 
\begin{table} [ht]
\begin{tabular}{|c|c|c|}
\hline
Indexing triple &unipotent orbit & representation\\
\hline
$(t_a,0,1)$& 1 &$J_{\beta}(1,\pi(\delta_3^{\pm 1},\delta_3^{\mp}))$\\
$(t_a, e_{\alpha^\vee}, 1)$& $\rA_1$ &$J_\alpha(1/2,\delta(\eta_3^{\pm 1}))$\\
$(t_a,e_{3\beta^\vee+2\alpha^\vee},1)$ &$\rA_1$ &  $J_\alpha(1/2,\delta(\eta_3^{\mp}))$\\
$(t_a,e_{\alpha^\vee}+e_{3\beta^\vee+2\alpha^\vee},1)$& $\rG_2(a_1)$ &$\pi(\eta_3^{\pm 1})$\\
\hline
\end{tabular}
\end{table}
\end{center}
By \cite[Table 4.2]{Ram}, $e_{\alpha^\vee}+e_{3\beta^\vee+2\alpha^\vee}$ corresponds to the regular nilpotent orbit in $\SL_3(\C)$, which gets sent to the subregular nilpotent orbit in $\rG_2(\C)$ by Table~\ref{tab:table-unipotent-Gsigma=SL3}. By \cite[Table 4.2]{Ram} again, $e_{3\beta^\vee+2\alpha^\vee}$ corresponds to the subregular nilpotent orbit in $\SL_3(\C)$, which gets sent to the subminimal nilpotent orbit in $\rG_2$ by Table~\ref{tab:table-unipotent-Gsigma=SL3} again. Here $u'$ and $u''$ are unipotent such that their conjugacy class $[u']=[u'']=\rA_1$ in $\rG_2$, but the restriction $\varphi|_{W_F}$ depends on $'u$ and $u''$ itself, rather than their conjugacy class; therefore we are able to arrive at different $L$-parameters $\varphi'_{\sigma,\rA_1}$ and $\varphi''_{\sigma,\rA_1}$ for the two irreducible constituents $J_{\alpha}(1/2,\delta(\eta_3^{\pm 1}))$ and $J_{\alpha}(1/2,\delta(\eta_3^{\mp }))$.

\begin{center} 
\begin{table} [ht]
\begin{tabular}{|c|c|}
\hline
Representation
& enhanced $L$-parameter
\cr
\hline\hline
$\pi(\eta_3^{\pm 1})$&$(\varphi_{\sigma,\rG_2(a_1)},1)$ 
\cr
\cline{1-2}
$J_\alpha(1/2,\delta(\eta_3^{\pm 1}))$&$(\varphi'_{\sigma,\rA_1},1)$ 
\cr
\cline{1-2}
$J_\alpha(1/2,\delta(\eta_3^{\mp}))$&$(\varphi''_{\sigma,\rA_1},1)$ 
\cr
\cline{1-2}
$J_{\beta}(1,\pi(\delta_3^{\pm 1},\delta_3^{\mp}))$&$(\varphi_{\sigma,1},1)$ 
\cr
 \hline
\end{tabular}
\vskip0.2cm
\caption{\label{table:unip-ppal-series3c-ramified}
LLC for the irreducible components of $I(\nu_F^{\pm 1}\eta_3\otimes\eta_3)$ for $\eta_3$ cubic ramified.
}
\end{table}
\end{center}

We now compute the formal degree of $\pi(\eta_3)$: by \cite[Theorem~10.7]{Roche-principal-series}, up to normalization factors of volumes, we have
\begin{equation}
    \fdeg(\pi(\eta_3))=d(\St_{\PGL_3}^{\mathcal{H}}),
\end{equation}
where $\St_{\PGL_3}$ is the Steinberg representation of $\PGL_3(F)$. 
By~\cite[Theorem~4.11]{Opdam-Selecta},
\begin{equation}    
d(\St_{\PGL_3}^{\mathcal{H}})= 3^{-1}[3]_q^{-1}=3^{-1}\cdot \frac{q^{1/2}-q^{-1/2}}{q^{3/2}-q^{-3/2}}
    =\frac{q}{3(q^2+q+1)}.
\end{equation}
Thus we have 
\begin{equation}\label{fdeg-pi-eta3}
    \fdeg(\pi(\eta_3))=\frac{q}{3(q^2+q+1)},
\end{equation}
which agrees with the formal degree for the singular supercuspidal in \eqref{eqn:fdegx1}.

\end{numberedparagraph}

\subsubsection{{\bf Case (5) within Case $\fs=[T,\xi\otimes\xi]_G$}} \label{subsubsec:case 5-1} \ 

In case (5), $\xi_1=\nu_F^{s_1}\xi=\nu_F^{-2s_2}\xi^{-2}\nu_F^{\pm 1}=\xi_2^{-2}\nu_F^{\pm 1}$. 
Thus  
\begin{equation}
    I(\xi_1\otimes\xi_2)=I(\nu_F^{\pm 1}\xi_2^{-2}\otimes\xi_2)\overset{ababa}{=}I(\nu_F^{\pm 1}\xi_2\otimes\xi_2).
\end{equation}
Thus Lemma~\ref{lem:main-induced} applies to this case, and we have 
\begin{equation}
    I(\xi_1\otimes\xi_2)=I(\nu_F^{\pm 1}\xi_2\otimes\xi_2)=I_{\alpha}(s,\delta(\nu_F^{\pm 1/2-s}\xi_2))+I_{\alpha}(s,\nu_F^{\pm 1/2-s}\xi_2\circ\det)
\end{equation}
Recall from \eqref{eqn:omega} that $\omega=\xi\nu_F^{-t}=\xi\nu_F^{s_2-s\pm 1/2}=\nu_F^{\pm 1/2-s}\xi_2$ is unitary. 
Therefore \cite[Theorem 3.1 (i)]{Muic-G2} applies to this case, and $I_{\alpha}(s,\delta(\nu_F^{\pm 1/2-s}\xi_2))$ and $I_{\alpha}(s,\nu_F^{\pm 1/2-s}\xi_2\circ\det)$ are irreducible unless 
\[
\text{$s=\pm 1/2$, $\nu_F^{\pm 1/2-s}\xi_2=\xi_2$ is quadratic,}\quad \text{$s=\pm 3/2$, $\nu_F^{\pm 1/2-s}\xi_2=1$}\quad \text{or $s=\pm 1/2$, $\nu_F^{\pm 1/2-s}\xi_2$ is cubic.}
\]
When they are irreducible, we compute the $L$-parameters for irreducible constituents and record them in Table~\ref{table:xi-tensor-xi-shorttable-ppal-series-case5} (when $\xi_2$ is unramified) and Table~\ref{table:xi-tensor-xi-shorttable-ppal-series-case5-ramified} (when $\xi_2$ is ramified).

\smallskip
\hrule
\smallskip

$\bullet$ When $\xi_2$ is unramified, we are in case $t_g$ of \cite[Table 6.1]{Ram} (see also \cite[$\mathsection$4]{ABP-G2}) and $\mathcal{J}^{\fs}=G_{2,\C}$, and we compute the $L$-parameters in Table~\ref{table:xi-tensor-xi-shorttable-ppal-series-case5}. Note that since $\delta(\nu_F^{\pm 1/2-s}\xi_2)$ is discrete series, the induced representation $I_{\alpha}(s,\delta(\nu_F^{\pm 1/2-s}\xi_2))$ is tempered, which thus corresponds to the $e_{\alpha_1}$ unipotent class, i.e. the minimal unipotent class $\rA_1$ in $\rG_2$. 
\begin{center}
\begin{table} [ht]
\begin{tabular}{|c|c|c|c|}
\hline
Representation &indexing triple & enhanced $L$-parameter
\cr
\hline\hline
$I_{\alpha}(s,\delta(\nu_F^{\pm 1/2-s}\xi_2))$&$(t_g,e_{\alpha_1},1)$ &$(\varphi_{\sigma,A_1},1)$ 
\cr
\cline{1-3}
 $I_{\alpha}(s,\nu_F^{\pm 1/2-s}\xi_2\circ\det)$&$(t_g,0,1)$ &$(\varphi_{\sigma,1},1)$ 
 \cr
  \hline
\end{tabular}
\vskip0.2cm
\caption{\label{table:xi-tensor-xi-shorttable-ppal-series-case5}: LLC for irred. constituents of case (5) within $I(\xi\otimes\xi)$ where $\xi$ is unramified.}
\end{table}
\end{center}
$\bullet$ When $\xi_2$ is ramified, 
we compute the $L$-parameters in different cases. 
\begin{enumerate}
    \item[(i)] When $\xi$ is neither quadratic nor cubic, we are in case (3) of Proposition~\ref{principal-series-classification-s-WGs}, where $W_G^{\fs}=\Z/2\Z=\{1,s_{\alpha}\}$ and $\mathcal{J}^{\mathfrak{s}}=\GL_2(\C)$. 
We are in case $t_a$ of \cite[Table 2.1]{Ram}.
Again since $\delta(\nu_F^{\pm 1/2-s}\xi_2)$ is discrete series, the induced representation $I_{\alpha}(s,\delta(\nu_F^{\pm 1/2-s}\xi_2))$ is tempered, which thus corresponds to the $e_{\alpha_1}$ unipotent class, which corresponds to $\rA_1$ in $\rG_2$ by Table~\ref{tab:table-unipotent-Gsigma=GL2}. Thus we obtain Table~\ref{table:xi-tensor-xi-shorttable-ppal-series-case5-ramified}.
    \begin{center} 
\begin{table} [ht]
\begin{tabular}{|c|c|c|c|}
\hline
Representation 
&indexing triple& enhanced $L$-parameter
 \cr
\hline\hline
$I_{\alpha}(s,\delta(\nu_F^{\pm 1/2-s}\xi_2))$&$(t_a,e_{\alpha_1},1)$ &$(\varphi_{\sigma,A_1},1)$ \cr
 $I_{\alpha}(s,\nu_F^{\pm 1/2-s}\xi_2\circ\det)$&$(t_a,0,1)$ &$(\varphi_{\sigma,1},1)$  \cr
  \hline
\end{tabular}
\vskip0.2cm
\caption{\label{table:xi-tensor-xi-shorttable-ppal-series-case5-ramified}: LLC for irred. constituents of $I(\xi_1\otimes\xi_2)$ attached to Case (5) of $\mathfrak{s}=[T,\xi\otimes\xi]_G$ where $\xi$ is ramified.}
\end{table}
\end{center}
\item[(ii)] When $\xi$ is quadratic, we are in case (4) of Proposition~\ref{principal-series-classification-s-WGs}, where $W_G^{\fs}=\Z/2\Z\times\Z/2\Z$ and $\mathcal{J}^{\mathfrak{s}}=\SO_4(\C)$. However, this case cannot happen here (although it'll occur elsewhere in our other subsections), because in this case $\xi_2$ is quadratic and $I_{\alpha}(s,\delta(\nu_F^{\pm 1/2-s}\xi_2))$ and $I_{\alpha}(s,\nu_F^{\pm 1/2-s}\xi_2\circ\det)$ are reducible, contradicting our assumption on irreducibility in this current subsection. (In fact, we indeed expect to have a length $4$ table rather than length $2$.)
\item[(iii)] When $\xi$ is cubic, we are in case (5) of Proposition~\ref{principal-series-classification-s-WGs}, where $W_G^{\fs}=S_3$ and $\mathcal{J}^{\mathfrak{s}}=\SL_{3,\C}$. Likewise, this case cannot happen here (although it'll occur elsewhere in our other subsections), because in this case $\xi_2$ is cubic and $I_{\alpha}(s,\delta(\nu_F^{\pm 1/2-s}\xi_2))$ and $I_{\alpha}(s,\nu_F^{\pm 1/2-s}\xi_2\circ\det)$ are reducible, contradicting our assumption on irreducibility in this current subsection. (In fact, we indeed expect to have a length $4$ table rather than length $2$.)
\end{enumerate}
\smallskip

\hrule

\smallskip
When they are reducible, we have in $R(G)$:

(a) Case ($s=\pm 1/2$, $\nu_F^{\pm 1/2-s}\xi_2=\eta_2$ quadratic):
\[I(\xi_1\otimes\xi_2)=I(\nu_F^{s\pm 1/2}\eta_2\otimes\nu_F^{s\mp 1/2}\eta_2)=I(\nu_F^{\pm 1}\nu_F^{s\mp 1/2}\eta_2\otimes \nu_F^{s\mp 1/2}\eta_2)
,\]
and thus Lemma~\ref{lem:main-induced} applies to this case, and we have
\begin{equation}
    I(\xi_1\otimes\xi_2)=I_{\alpha}(s, \delta(\eta_2))+I_{\alpha}(s,\eta_2\circ\det).
\end{equation}
We now plug in $s=\pm 1/2$ and obtain
\[I(\xi_1\otimes\xi_2)=I(\nu_F^{\pm 1}\eta_2\otimes\eta_2)
=
I_\alpha(\pm 1/2,\delta(\eta_2))+I_\alpha(\pm 1/2,\eta_2\circ\det).\]
Note that by \cite[Lemma 5.4(iii)]{BDK}, we have
\begin{equation}
I_{\alpha}(-1/2,\delta(\eta_2))\simeq I_{\alpha}(1/2,\delta(\eta_2))\quad\text{and}\quad
I_{\alpha}(-1/2,\eta_2\circ\det)\simeq I_{\alpha}(1/2,\eta_2\circ\det).
\end{equation}
Therefore in the above we treat $\pm 1/2$ cases together. 
Recall from \eqref{eqn:omega}, we have that $\eta_2=\omega$ is unitary, thus we can apply \cite[Proposition 4.1 (ii)]{Muic-G2} and obtain
\begin{equation}
I(\xi_1\otimes\xi_2)=\pi(\eta_2)+J_\beta(1,\pi(1,\eta_2))+J_\beta(1/2,\delta(\eta_2))+J_{\alpha}(1/2,\delta(\eta_2)).
\end{equation}
The LLC for these irreducible constituents was already computed in Table~\ref{table:unip-ppal-series-case-3aquad} (when $\eta_2$ is unramified) and Table~\ref{table:unip-ppal-series-case-3aquad-ramified} (when $\eta_2$ is ramified).

(b) Case ($s=\pm 3/2$, $\nu_F^{s_2-s\pm 1/2}\xi=1$): we have $\xi_2=\nu_F^{s_2}\xi=\nu_F^{s\mp 1/2}$, and $\xi_1=\nu_F^{\pm}\xi_2=\nu_F^{\pm 1/2+s}$.\footnote{Note that here $\xi_1$ should be thought of as the ``new'' $\xi_1$ after applying the Weyl group actions.} Thus 
\[I(\xi_1\otimes \xi_2)=I(\nu_F^{\pm 1/2+s}\otimes \nu_F^{s\mp 1/2})=I(\nu_F^{\pm 1}\nu_F^{s\mp 1/2}\otimes\nu_F^{s\mp 1/2}),\]
and thus Lemma~\ref{lem:main-induced} applies to this case, and we have 
\begin{equation}
    I(\xi_1\otimes \xi_2)=I_{\alpha}(s,\delta(\nu_F^{\pm 1/2-s}\xi_2))+I_{\alpha}(s,\nu_F^{\pm 1/2-s}\xi_2\circ\det).
\end{equation}
We now plug in $s=\pm 3/2$ and the explicit values for $\xi_2$, and obtain
\[I(\xi_1\otimes\xi_2)=I(\nu_F^{\pm 2}\otimes\nu_F^{\pm 1})
=I_{\alpha}(\pm 3/2,\delta(1))+I_{\alpha}(\pm 3/2,1_{\GL_2}).\]
By \cite[Lemma 5.4(iii)]{BDK}, we have
\begin{equation}
I_{\alpha}(-3/2,\delta(1))\simeq I_{\alpha}(3/2,\delta(1))\quad\text{and}\quad
I_{\alpha}(-3/2,1_{\GL_2})\simeq I_{\alpha}(3/2,1_{\GL_2}).
\end{equation}
Therefore it suffices to treat the $\pm 3/2$ cases together. When $s=3/2$, by \cite[Proposition~4.4]{Muic-G2}, we have the following decomposition into irreducible constituents:
\begin{equation}
I_{\alpha}(3/2,\delta(1))=\St_{\rG_2}+J_\alpha(3/2,\delta(1))\quad\text{and}\quad
I_{\alpha}(3/2,1_{\GL_2})=1_{\rG_2}+J_\beta(5/2,\delta(1)).
\end{equation}
We thus obtain
\begin{equation}
I(\xi_1\otimes\xi_2)=I(\nu_F^{\pm 2}\otimes\nu_F^{\pm 1})=\St_{\rG_2}+J_\alpha(3/2,\delta(1))+1_{\rG_2}+J_\beta(5/2,\delta(1)).
\end{equation}
The LLC for these irreducible constituents was already computed in Table~\ref{table:unip-ppal-series-case-3bbis}.

(c) Case ($s=\pm 1/2$, $\nu_F^{s_2-s\pm 1/2}\xi=\eta_3$ cubic): 
\[
    I(\xi_1\otimes\xi_2)=I(\nu_F^{s\pm 1/2}\eta_3\otimes\nu_F^{s\mp 1/2}\eta_3)=I(\nu_F^{\pm 1}\nu_F^{s\mp 1/2}\eta_3\otimes\nu_F^{s\mp 1/2}\eta_3),
\]
and thus Lemma~\ref{lem:main-induced} applies to this case, and we have the decomposition into irreducible constituents:
\begin{equation}
    I(\xi_1\otimes\xi_2)=I_{\alpha}(s,\delta(\eta_3))+I_{\alpha}(s,\eta_3\circ\det).
\end{equation}
We now plug in $s=\pm 1/2$ and obtain
\[I(\xi_1\otimes\xi_2)=I(\nu_F^{\pm 1}\eta_3\otimes\eta_3)=I_{\alpha}(\pm 1/2,\delta(\eta_3))+I_{\alpha}(\pm 1/2,\eta_3\circ\det).\]
By \cite[Lemma~5.5(iii)]{BDK}, we have
\[I_{\alpha}(-1/2,\delta(\eta_3))=I_{\alpha}( 1/2,\delta(\eta_3^{-1}))\quad\text{and}\quad I_{\alpha}(- 1/2,\eta_3\circ\det)=I_{\alpha}(1/2,\eta_3^{-1}\circ\det).\]
By \cite[Proposition~4.2(ii)]{Muic-G2}, we have
\[I_{\alpha}(1/2,\delta(\eta_3^{\pm 1}))=\pi(\eta_3^{\pm 1})+J_\alpha(1/2,\delta(\eta_3^{\pm 1}))\quad\text{and}\]
\[
I_{\alpha}(1/2,\eta_3^{\pm 1}\circ\det)=J_\beta(1,\pi(\delta_3^{\pm 1},\delta_3^{\mp 1})) + J_\alpha(1/2,\delta(\eta_3^{\mp 1})).\]
In summary, we obtain the following decomposition into irreducible constituents:
\begin{equation}
I(\xi_1\otimes\xi_2)= I(\nu_F^{\pm 1}\eta_3\otimes\eta_3)=\pi(\eta_3^{\pm 1})+J_\alpha(1/2,\delta(\eta_3^{\pm 1}))+J_\beta(1,\pi(\delta_3^{\pm 1},\delta_3^{\mp})) + J_\alpha(1/2,\delta(\eta_3^{\mp})).   
\end{equation}
The LLC for these irreducible representations was already computed in Tables~\ref{table:unip-ppal-ser3c-unramified} and \ref{table:unip-ppal-series3c-ramified}.
\subsection{Case \texorpdfstring{$\fs=[T,\xi\otimes 1]_G$ with $\xi$ ramified}{$\fs=[T,\xi\otimes 1]_G$ with $\xi$ ramified}} \label{principal-series-xi-tensor-1-ramified-nonquad} 
Consider the induced representation $I(\nu_F^{s_1}\xi\otimes\nu_F^{s_2})$. 

--If $I(\nu_F^{s_1}\xi\otimes\nu_F^{s_2})$ is unitary, by \cite[Theorem $\rG_2$]{Keys}, it is reducible if and only if $\nu_F^{s_2}$ and $\nu_F^{s_1}\xi$ are distinct quadratic and unitary. 
When reducible, 
it is of length 2, and we obtain Table~\ref{table:unitary-length-two} (notations as in \cite{ABP-G2}). 
    \begin{center}
\begin{table} [ht]
\begin{tabular}{|c|c|c|}
\hline
Representation 
& Kazhdan-Lusztig triple
\cr
\hline\hline
$\pi(\nu_F^{s_1}\xi\otimes\nu_F^{s_2})$& 
$([s_i,s_i],[1,1],1)$ 
\cr
\hline
 $\pi'(\nu_F^{s_1}\xi\otimes\nu_F^{s_2})$& 
 $([s_i,s_i],[1,1],\sign)$
 \cr
  \hline
\end{tabular}
\vskip0.2cm
\caption{\label{table:unitary-length-two}: LLC for the irreducible constituents of unitary $I(\nu_F^{s_1}\xi\otimes\nu_F^{s_2})$ attached to $\mathfrak{s}=[T,\xi\otimes 1]_G$ with $\xi$ ramified.} 
\end{table}
\end{center}

--If $I(\nu_F^{s_1}\xi\otimes\nu_F^{s_2})$ is non-unitary, by \cite[Proposition 3.1]{Muic-G2}, $I(\nu_F^{s_1}\xi\otimes\nu_F^{s_2})$ is irreducible unless we are in one of the following two cases:
\begin{enumerate}
    \item $\nu_F^{s_2}=\nu_F^{\pm 1}$, i.e. $s_2=\pm 1$ and $s_1$ is arbitrary
    \item $\nu_F^{2s_1+s_2}\xi^2=\nu_F^{\pm 1}$, i.e. $\xi$ is ramified quadratic and $2s_1+s_2=\pm 1$. 
\end{enumerate}

\begin{remark}
The case $\mathfrak{s}=[T,\xi\otimes\xi']$ for $\xi\neq\xi'$ both ramified either reduces to the case for $[T,\xi\otimes 1]$ or the case for $[T,\xi\otimes\xi]$, so there is no need to discuss this case separately. 
\end{remark}

\subsubsection{{\bf Case (1) within Case $\fs=[T,\xi\otimes 1]_G$ with $\xi$ ramified}} \label{subsubsec:case 1-2} \ 

In this case, 
\[I(\xi_1\otimes\xi_2)=I(\nu_F^{s_1}\xi\otimes\nu_F^{\pm 1})=I(\xi_1\otimes\nu_F^{\pm}).\]
\begin{enumerate}
    \item When $\xi_1$ is quadratic, this is similar to the cases covered in Tables \ref{table:unip-ppal-series-case-3aquad} and \ref{table:unip-ppal-series-case-3aquad-ramified}. 
    \item When $\xi_1$ is non-quadratic (already assumed to be ramified), this is similar to the cases covered in Table \ref{table:nu-tensor-xi-shorttable-ppal-series-case2-cubic-ramified} and Table  \ref{table:nu-tensor-xi-shorttable-ppal-series-case2-ramified-neither}.
\end{enumerate}

\subsubsection{{\bf Case (2) within Case $\fs=[T,\xi\otimes 1]_G$ with $\xi$ ramified quadratic}} \label{subsubsec:case 2-2}\ 
When $\xi$ is ramified quadratic, we are in case (4) of Proposition~\ref{principal-series-classification-s-WGs} via the action of $W$, where $W_G^{\fs}=\Z/2\Z\times\Z/2\Z$ and 
\begin{equation} \label{eqn:Js SO4 length 2 unitary}
\mathcal{J}^{\mathfrak{s}}=\SO_4(\C)\simeq\SL_2^{\mathrm {lr}}(\C)\times\SL_2^{\mathrm {sr}}(\C)/\{\pm 1\}. 
\end{equation}

\hrule

In this case, let $\chi:=\nu_F^{s_1}\xi$ where $\xi$ is ramified quadratic, and we have
\begin{equation}
I(\xi_1\otimes\xi_2)=I(\nu_F^{s_1}\xi\otimes\nu_F^{\pm 1-2s_1})=I(\chi\otimes\chi^{-2}\nu_F^{\pm 1})\overset{babab}{=}I(\chi^2\nu_F^{\mp 1}\otimes\chi^{-1})\overset{b}{=}I(\chi\nu_F^{\mp 1}\otimes\chi).
\end{equation}
Thus Lemma~\ref{lem:main-induced} applies to this case, and we have 
\begin{equation}
    I(\xi_1\otimes\xi_2)=I(\nu_F^{s_1\mp 1}\xi\otimes\nu_F^{s_1}\xi)=I_{\alpha}(s,\delta(\nu_F^{\mp 1/2-s}\chi))+I_{\alpha}(s,\nu_F^{\mp 1/2-s}\chi\circ\det).
\end{equation}
Recall from \eqref{eqn:omega} that $\omega=\xi\nu_F^{-t}=\xi\nu_F^{s_2-s\pm 1/2}$ is unitary, thus
\begin{equation}\label{omega-inverse-case2-xi-tensor-one}
\nu_F^{\mp 1/2-s}\chi=\nu_F^{-s+s_2\pm 1/2}\xi=\omega
\end{equation}
is unitary. Therefore \cite[Theorem 3.1]{Muic-G2} applies to this case and $I_{\alpha}(s,\delta(\nu_F^{\mp 1/2-s}\chi))$ and $I_{\alpha}(s,\nu_F^{\mp 1/2-s}\chi\circ\det)$ are irreducible unless 
\[
\text{$s=\mp 1/2$, $\nu_F^{\mp 1/2-s}\chi=\chi$ is quadratic,}\quad \text{$s=\mp 3/2$, $\nu_F^{\mp 1/2-s}\chi=1$}\quad \text{or $s=\mp 1/2$, $\nu_F^{\mp 1/2-s}\chi$ is cubic.}
\]
Clearly the second and third cases cannot happen (since $\xi$ is ramified quadratic), thus we are left with only the first possibility, i.e. $I_{\alpha}(s,\delta(\nu_F^{\mp 1/2-s}\chi))$ and $I_{\alpha}(s,\nu_F^{\mp 1/2-s}\chi\circ\det)$ are irreducible unless \[\text{$s=\mp 1/2$, $\omega=\nu_F^{\mp 1/2-s}\chi=\chi$ is quadratic}.\]
Now, since both $\chi=\nu_F^{s_1}\xi$ and $\xi$ are quadratic, we have $1=\nu_F^{2s_1}$, and thus 
$s_1=0$. Thus we have $\xi=\chi=\omega$ is unitary and quadratic. 

\begin{enumerate}
    \item When they are irreducible i.e. when $s\neq \pm 1/2$, we compute the $L$-parameters for irreducible constituents and record them in Table~\ref{table:xi-tensor-one-shorttable-ppal-series-case2}. 
By Lemma~\ref{lem:generic}~(a), since the representation $\delta(\xi)$ of $\GL_2(F)$ is generic, $I_{\alpha}(s,\nu_F^{\mp 1/2-s}\delta(\xi))=\ii_{P_\alpha}^{G}(\delta(\xi))$ is also generic. By Lemma~\ref{lem:generic}~(b), since the representation $\xi\circ\det$ of $\GL_2(F)$ is not generic, $I_{\alpha}(s,\nu_F^{\mp 1/2-s}\xi\circ\det))=\ii_{P_\alpha}^{G}(\xi\circ\det)$ is also non-generic.

\begin{center}
\begin{table} [ht]
\begin{tabular}{|c|c|c|}
\hline
Representation 
& enhanced $L$-parameter
\cr
\hline\hline
$\ii_{M_\alpha}^{G}(\delta(\xi)))$& $(\varphi_{\sigma,\rA_1},1)$ 
\cr
\hline
 $\ii_{M_\alpha}^{G}(\xi\circ\det)$& $(\varphi_{\sigma,1},1)$ 
 \cr
  \hline
\end{tabular}
\vskip0.2cm
\caption{\label{table:xi-tensor-one-shorttable-ppal-series-case2}: LLC for the irreducible constituents of $I(\xi\nu_F^{\mp 1}\otimes\xi)$ attached to Case (2) of $\mathfrak{s}=[T,\xi\otimes 1]_G$ with $\xi$ ramified quadratic.} 
\end{table}
\end{center}
\smallskip

\hrule

\smallskip

\item When they are reducible 
i.e. $s=\mp 1/2$ and $\omega=\nu_F^{\mp 1/2-s}\chi=\chi$ is quadratic, we have in $R(G)$:

(a) Case ($s=\pm 1/2$, $\nu_F^{\mp 1/2-s}\chi=\eta_2$ quadratic)\footnote{Note that although we have used the same notation $\eta_2$ as in the previous sections, this $\eta_2$ is not necessarily the same character as the $\eta_2$ from the previous sections; here we are simply abusing notations.}:
\[I(\xi_1\otimes\xi_2)=I(\nu_F^{s\pm 1/2}\eta_2\otimes\nu_F^{s\mp 1/2}\eta_2)=I(\nu_F^{\pm 1}\nu_F^{s\mp 1/2}\eta_2\otimes \nu_F^{s\mp 1/2}\eta_2)
,\]
and thus Lemma~\ref{lem:main-induced} applies to this case, and we have
\begin{equation}
    I(\xi_1\otimes\xi_2)=I_{\alpha}(s, \delta(\eta_2))+I_{\alpha}(s,\eta_2\circ\det).
\end{equation}
We now plug in $s=\pm 1/2$ and obtain
\[I(\xi_1\otimes\xi_2)=I(\nu_F^{\pm 1}\eta_2\otimes\eta_2)
=
I_\alpha(\pm 1/2,\delta(\eta_2))+I_\alpha(\pm 1/2,\eta_2\circ\det).\]
Note that by \cite[Lemma 5.4(iii)]{BDK}, we have
\begin{equation}
I_{\alpha}(-1/2,\delta(\eta_2))\simeq I_{\alpha}(1/2,\delta(\eta_2))\quad\text{and}\quad
I_{\alpha}(-1/2,\eta_2\circ\det)\simeq I_{\alpha}(1/2,\eta_2\circ\det).
\end{equation}
Therefore in the above we treat $\pm 1/2$ cases together. 
As in \eqref{omega-inverse-case2-xi-tensor-one}, $\eta_2=\omega$ is unitary, by \cite[Proposition 4.1 (ii)]{Muic-G2} we have
\begin{equation}
I(\xi_1\otimes\xi_2)=\pi(\eta_2)+J_\beta(1,\pi(1,\eta_2))+J_\beta(1/2,\delta(\eta_2))+J_{\alpha}(1/2,\delta(\eta_2)).
\end{equation}
The computation of LLC for these irreducible constituents was already done in Table~\ref{table:unip-ppal-series-case-3aquad-ramified}.

(b) Case ($s=\pm 3/2$, $\nu_F^{\mp 1/2-s}\chi=1$): 
recall that $\chi=\nu_F^{s_1}\xi$. Plugging it in we get
\begin{equation}\label{Case2-xi-tensor-one-b}
\nu_F^{\mp 1/2-s}\nu_F^{s_1}\xi=\nu_F^{s_1\pm 1}\xi=1.
\end{equation}
Since $\xi$ is ramified quadratic in this case, the above equation \eqref{Case2-xi-tensor-one-b} cannot happen.

(c) Case ($s=\pm 1/2$, $\nu_F^{\mp 1/2-s}\chi=\eta_3$ cubic): This case also cannot happen since $\xi$ is ramified quadratic.
\end{enumerate}

\section{Main results}
\subsection{List of properties of the LLC} \label{subsec:expected properties}
In this subsection, we state several properties that are expected to be satisfied by the local Langlands correspondence. 
Recall from Definition~\ref{eqn:L-packet} that the $L$-packet of irreducible representations of $G$ attached to the $L$-parameter $\varphi$ is denoted by $\Pi_\varphi(G)$.

\begin{property} \label{property:L-packets} {\rm \cite[\S10.3]{Borel-Corvallis}}
Let $\varphi$ be an $L$-parameter for $G$.
\begin{enumerate}
\item $\varphi$ is bounded if and only if one element (equivalently any element) of $\Pi_\varphi(G)$ is tempered;
\item $\varphi$ is discrete  if and only if one element (equivalently any element) of $\Pi_\varphi(G)$ is square-integrable modulo center;
\item $\varphi$ is supercuspidal if and only if all the elements of $\Pi_\varphi(G)$ are supercuspidal.
\end{enumerate}
\end{property}

\begin{property} \label{property:size of L-packets} {\rm (\cite[\S2]{Arthur-Note}, and \cite[Conjecture~B]{Kaletha-LLC})} 
The elements of $\Pi_\varphi(G)$ are in bijection with $\Irr(S_\varphi)$. 
\end{property}

\begin{property} \label{property:Langlands quotient} {\rm \cite[\S7.2]{Silberger-Zink}} 
Let $(P,\pi,\nu)$ be a standard triple for $G$. We have 
\begin{equation} \label{eqn:PhiJ}
\varphi_{J(P,\pi,\nu)}=\iota_{L^\vee}\circ\varphi_{\pi\otimes\chi_\nu},
\end{equation}
where is $J(P,\pi,\nu)$ is the Langlands quotient defined in \eqref{eqn:Langlands quotient} and $\iota_{L^\vee}\colon L^\vee\hookrightarrow G^\vee$ is the canonical embedding.
\end{property}

In general, the bijection mentioned in Property~\ref{property:size of L-packets} depends on the choice of a Whittaker datum ${\mathfrak w}:=(U,\chi)$--up to conjugation by $G$--(see for instance the $\SL_2(F)$-example in \cite[pp.484-485]{Gross-Reeder}). We will denote this bijection as 
\begin{equation} \label{eqn:Whittaker}
 \iota_{\mathfrak w}\colon \Pi_\varphi(G)\to \Irr(S_\varphi).   
\end{equation}
An $L$-packet is called \textit{${\mathfrak w}$-generic} if it contains an element which is ${\mathfrak w}$-generic in the sense of Definition~\ref{defn:generic}.
\begin{property}\label{property:generic tempered L-packets}{\rm \cite[Conjecture~9.4]{ShahidiAnnalsAproofof}}
If $\varphi$ is bounded, then the $L$-packet $\Pi_\varphi(G)$ is ${\mathfrak w}$-generic for some Whittaker datum $\mathfrak w$. Moreover, the conjectural bijection $\iota_{\mathfrak w}\colon \Pi_\varphi(G)\to \Irr(S_\varphi)$ maps the $\mathfrak w$-generic representation to the trivial representation of $S_\varphi$.
\end{property}

\begin{lem}\label{sc-packet-generic-lemma}
Supercuspidal $L$-packets satisfy Property~\ref{property:generic tempered L-packets}. 
\end{lem}
\begin{proof}
The depth-zero case follows from \cite[Lemma 4.2.1]{Kaletha-nonsingular}. For the reader's convenience, we recall the argument here. 
By \cite[Lemma H.1]{Kaletha-nonsingular}, $\mathfrak{w}$ determines, uniquely up to $G$-conjugacy, an absolutely special vertex $x\in\mathcal{B}(G,F)$ such that $\chi$ has depth-zero at $x$ for $(U,\chi)=\mathfrak{w}$. By $\mathfrak{w}$-genericity, the supercuspidal $\pi\in\Pi_{\varphi}$ is induced from an irreducible representation of $G_x$ containing a $\chi_x$-generic representation of $G_{x,0}$. By \cite[Lemma 3.4.12]{Kal-reg}, up to $G$-conjugacy, there exists exactly one admissible embedding $j\colon\bS\to\bG$ such that the vertex $x$ corresponds to the maximal torus $j(\bS)\subset \bG$. A $\mathfrak{w}$-generic member of the $L$-packet $\Pi_{\varphi}$ necessarily arises from the non-singular Deligne-Lusztig packet $[\pi_{(j(\bS),\theta_j)}]$ (as defined in \cite[p.35]{Kaletha-nonsingular}). By \cite[Proposition~3.10]{Digne-Lehrer-Michel}, there exists a unique $\chi_x$-generic irreducible component of the Deligne-Lusztig character $\kappa_{(j(\bS),\theta_j^{\circ})}$ (as defined in \cite[Definition 2.6.8 and p.34]{Kaletha-nonsingular}), and hence a unique irreducible representation of $G_x$ containing it after taking the restriction. Its compact induction to $G(F)$ gives the unique $\mathfrak{w}$-generic element of the Deligne-Lusztig packet $[\pi_{(j(S),\theta_j)}]$.   
As remarked in \textit{loc.cit.}, the positive-depth case can be deduced from the depth-zero case via the local character expansions of \cite{Spice-LCE}. 
\end{proof}

We now state a compatibility property of the LLC with supercuspidal supports. 
First we recall the following conjecture from \cite[Conjecture~7.18]{Vogan-LLC-1993}, or equivalently \cite[Conjecture~5.2.2]{Haines-Bst}.
\begin{conj}\label{infinitesimal-param-conj}
Let $P$ be a parabolic subgroup of $G$ with Levi subgroup $L$ and $\sigma$ a smooth irreducible supercuspidal representation of $L$. For any irreducible constituent $\pi$ of $\ii_{P}^G\sigma$, the infinitesimal $L$-parameters $\lambda_{\iota_{L^\vee}\circ\varphi_\sigma}$ and $\lambda_{\varphi_\pi}$ are $G^{\vee}$-conjugate.
\end{conj}
The following Property~\ref{property:AMS-conjecture-7.8} generalizes Conjecture~\ref{infinitesimal-param-conj} (see Remark~\ref{remark:infinitesimal parameter}).
Let $\cL(G)$ be a set of representatives for the conjugacy classes of Levi subgroups of $G$. 
By \cite[Proposition~3.1]{ABPS-CM}, for any $L\in\cL(G)$ there is a canonical isomorphism between $W_G(L)$ and $W_{G^\vee}(L^\vee)$.
\begin{property} \label{property:AMS-conjecture-7.8} {\rm \cite[Conjecture~7.8]{AMS18}}
The following diagram is commutative
\begin{equation} \label{eqn:conjecture-7.8}
\begin{tikzcd}
\Irr(G)\arrow[]{d}[swap]{\Sc}\arrow[]{r}{\LLC}[swap]{1\text{-}1} &\Phi_\enh(G)\arrow[]{d}{\Sc}\\
\bigsqcup_{L\in\cL(G)}\Irr_{\superc}(L)/W_G(L)\arrow[]{r}{\LLC}[swap]{1\text{-}1} & \bigsqcup_{L\in\cL(G)}\Phi_{\enh,\cusp}(L)/W_G(L).
\end{tikzcd}.
\end{equation}
\end{property}
\begin{remark} \label{remark:conj-cuspidality}
{\rm When $L=G$, the diagram~\eqref{eqn:conjecture-7.8} collapses to the bottom horizontal line, and Property~\ref{property:AMS-conjecture-7.8} states that the (isomorphism classes of) irreducible supercuspidal representions of $G$ correspond under the local Langlands correspondence to the ($G^\vee$-conjugacy classes of) cuspidal enhanced $L$-parameters (see also \cite[Conjecture~6.10]{AMS18} and \cite[Conjecture~5.2]{Aubert-Pune})}:
\begin{equation}
    \LLC\colon\Irr_\superc(G)\overset{1-1}{\longrightarrow}\Phi_{\enh,\cusp}(G).
\end{equation}
\end{remark}

Note that Property~\ref{property:AMS-conjecture-7.8} is known to hold for unipotent representations by \cite[Theorem~2]{FOS}, for all representations of general linear groups and split classical $p$-adic groups by \cite{Moussaoui-Bernstein-center}, and all representations of special linear groups by \cite{AMS18}.

\begin{conj} \label{conj:matching} {\rm \cite[Conjecture~2]{AMS18}} For any $\fs=[L,\sigma]_G\in\fB(G)$, the LLC for $L$ given by $\sigma\mapsto(\varphi_\sigma,\rho_\sigma)$ induces a bijection
\begin{equation} \label{eqn:matching}
    \Irr^{\fs}(G)\xrightarrow{1--1} \Phi_\enh^{\fs^{\vee}}(G),
\end{equation}
where $\fs^\vee=[L^\vee,(\varphi_\sigma,\rho_\sigma)]_{G^\vee}$.
\end{conj}
Conjecture~\ref{conj:matching} is proved  for split classical groups  
\cite[\S5.3]{Moussaoui-Bernstein-center}, for $\GL_n(F)$ and $\SL_n(F)$  \cite[Theorems 5.3 and 5.6]{ABPS-SL}, for principal series representations of split groups \cite[\S16]{ABPS-LMS}. 
For the group $\rG_2$, a bijection between $\Irr^{\fs}(G)$ and $\Phi_\enh^{\fs^{\vee}}(G)$ has been constructed in \cite[Theorem 3.1.19]{aubert-xu-Hecke-algebra}. 

Recall from \cite{aubert-xu-Hecke-algebra}, we have an isomorphism 
\begin{equation}\label{Bernstein-block-isom}
    \mathrm{Irr}^{\fs}(G)\xrightarrow{1--1}\Phi_\enh^{\fs^{\vee}}(G)
\end{equation}
for each Bernstein series $\mathrm{Irr}^{\fs}(G)$ of \textit{intermediate series}. On the other hand, the bijection \eqref{Bernstein-block-isom} holds for \textit{principal series} blocks thanks to \cite{Roche-principal-series,Reeder-isogeny,ABPS-KTheory,AMS18}. 

Moreover, we verify the following two properties for our LLC. 
\begin{property}\label{property: Shahidi-fdeg}\cite{ShahidiAnnalsAproofof} The quantity $\frac{\fdeg(\pi)}{\dim(\rho)}$ is constant in an $L$-packet. 
\end{property}
By Harish-Chandra (see \cite[Proposition III.4.1]{Wal}), any tempered non-discrete series irreducible representation $\pi$ is a subrepresentation of $\ii_P^G(\delta)$, where $P$ is a parabolic subgroup of $G$ with Levi factor $L$ and $\delta$ is a discrete series representation of $L$, and the $G$-conjugacy class of pair $(L,\delta)$ is uniquely determined. 

The following is a standard expected property of LLC. 
\begin{property}\label{property:Levi-inclusion}
The $L$-parameter of $\pi$ is the composition of the $L$-parameter of $\delta$ with the natural inclusion $L^{\vee}\to G^{\vee}$ and $\rho_\pi=\rho_\delta$.
\end{property}

\begin{property}[DeBacker, Kaletha]\label{property:atomic-stability}
Let $\varphi$ be a discrete $L$-parameter. There exists a non-zero $\C$-linear combination  
\begin{equation}\label{stable-distribution-Lpacket}
S\Theta_{\varphi}:=\sum\limits_{\pi\in\Pi_{\varphi}}\dim(\rho_\pi)\Theta_{\pi},\quad\text{for }z_{\pi}\in\C,
\end{equation}
which is stable. In fact, one can take $z_\pi=\dim(\rho_\pi)$ where $\rho_\pi$ is the enhancement of the $L$-parameter. Moreover, no proper subset of $\Pi_{\varphi}$ has this property. 
\end{property}

\subsection{Main result}
Construction of the Local Langlands Correspondence
\begin{equation} \label{eqn:LLC}
\begin{split}
    \LLC\colon\mathrm{Irr}(G)&\xrightarrow{1\text{-}1}\Phi_\enh(G)\\
    \pi &\mapsto (\varphi_{\pi},\rho_{\pi}).
\end{split}
\end{equation}
Recall from \eqref{eqn:Bernstein decomposition} and \eqref{eqn:decPhi_e} that we have 
\[\Irr^{\mathfrak{s}}(G)=\bigsqcup\limits_{\fs\in \mathcal{B}(G)}\Irr^{\fs}(G)\;\;\text{and}\;\;
\Phi_\enh(G)=\bigsqcup\limits_{\mathfrak{s}^{\vee}\in \mathcal{B}^{\vee}(G)}\Phi_\enh^{\fs^{\vee}}(G).\]
When $\pi\in\Irr(G)$ is not supercuspidal, we have $\fs=[L,\sigma]_G$ where $L$ is a proper Levi subgroup of $G$. Hence, $L$ is isomorphic to either $F^\times\times F^\times$ or $\GL_2(F)$. Let $\varphi_\sigma\colon W'_F\to L^\vee$ be the $L$-parameter attached to $\sigma$ by the Local Langlands Correspondence for $L$ (see \cite{Bushnell-Henniart-GL2}). The $L$-packet $\Pi_{\varphi_\sigma}(L)$ is always a singleton (in particular, the enhancement $\rho_\sigma$ is trivial). The $L^\vee$-conjugacy class of $\varphi_\sigma$ is uniquely determined by $\sigma$, and we have $\varphi_{(\chi\circ\det) \otimes\sigma}=\varphi_\sigma\otimes \varphi_\chi$, i.e.~\cite[Property~3.12(1)]{aubert-xu-Hecke-algebra} holds. This allows us to define 
\begin{equation}
\fs^\vee:=[L^\vee,(\varphi_\sigma,1)]_{G^\vee}.
\end{equation}
Let $\pi\mapsto (\varphi_\pi,\rho_\pi$) be the bijection 
\begin{equation}
\Irr^{\fs}(G)\xrightarrow{1--1}\Phi_\enh^{\fs^{\vee}}(G),
\end{equation}
established in \cite[Main Theorem]{aubert-xu-Hecke-algebra} (for intermediate series) and in \cite{ABPS-KTheory} (for principal series). We have given explicit Kazhdan-Lusztig triples and $L$-packets in $\mathsection$\ref{sec:ppal-series} and $\mathsection$\ref{sec:interm_series} (see also Tables \ref{tableprin} and \ref{tableint}). 

We consider now the case where $\pi$ is supercuspidal. Hence we have $\fs=[G,\pi]_G$ for $\pi$ an irreducible supercuspidal representation of $G$. 

\begin{enumerate}
\item[(a)]
When $\pi$ is non-singular supercuspidal, we define $(\varphi_\pi,\rho_\pi)$
to be the enhanced $L$-parameter constructed in \cite{Kaletha-nonsingular}. In particular, when $\pi$ is regular (in the sense of Definition~\ref{defn:Kaletha regular}),  the enhanced $L$-parameter $(\varphi_\pi,\rho_\pi)$ coincides with the one constructed in \cite{Kal-reg}. 

When $F$ has characteristic zero and $p\ge 7(e(F|\Qp)+2)$, where $e(F|\Qp)$ is the ramification index of $F/\Qp$ (see the proof of \cite[Proposition 4.3.2]{Fintzen-Kaletha-Spice}), the $G^\vee$-conjugacy class of $(\varphi_\pi,\rho_\pi)$ is uniquely determined by $\pi$ \cite[Theorem~4.4.4]{Fintzen-Kaletha-Spice}, since we have $S_{\varphi_\pi}^+=S_{\varphi_\pi}$ as $\rG_2$ is adjoint.

\item[(b)]
When $\pi$ is a unipotent supercuspidal representation of $G$, we define $(\varphi_\pi,\rho_\pi)$ to be the enhanced $L$-parameter constructed in \cite{Lu-padicI} and \cite[\S~5.6]{Morris-ENS}. 

\item[(c)]\label{summary-section-d0-mixed-packets}
Let $\pi$ be a non-unipotent depth-zero \textit{singular} supercuspidal representation of $G$. As recalled in~\eqref{eqn:iGxG}, we have
$\pi=\ii_{\bbG_{x}}^G\tau$, where $x$ is a vertex of the Bruhat-Tits building of $G$ and $\tau\in\cE(\bbG_x,s)$ with $s\ne 1$. We have three cases:

$\bullet$ $x=x_1$: From~\S\ref{d0s}\eqref{Gx=G2}(c), when $q\equiv -1 (3)$, we have one rational Lusztig series $\cE(\bbG_{x_1},s_1)$ giving a singular depth-zero non-unipotent supercuspidal representation $\pi_1$, which by Property \ref{property:AMS-conjecture-7.8} lives in the same $L$-packet as an intermediate series representation $\pi(\sigma)$ given precisely in Table~\ref{num: M long 0}, as explained in \ref{d0p}(\ref{2b}). The $L$-packet in this case is given by
\begin{equation}
    \Pi_{\varphi}(G)=\{\pi_1,\pi(\sigma)\}.
\end{equation}
Property~\ref{property: Shahidi-fdeg} is verified in this case by comparing \eqref{fdeg-singular-nonunipotent-sc} and the quantities in Table~\ref{num: M long 0}, as explained in Paragraph~\ref{long-root-intermediate-series-mixed-packet}. 

$\bullet$ $x=x_2$: From~\S\ref{d0s}\eqref{Gx=SL3}, when $q\equiv 1 (3)$, we have two possible rational Lusztig series $\cE(\bbG_{x_2},s_2[\zeta_3])$ and $\cE(\bbG_{x_2},s_2[\zeta_3^2])$ for the primitive third root of unity $\zeta_3$, where  $s_2[\zeta]=\mathrm{diag}\left(1,\zeta,\zeta^2\right)\mod \F_q^{\times}$ for $\zeta\in\{\zeta_3,\zeta_3^2\}$. Each Lusztig series $\cE(\bbG_{x_2},s[\zeta])$ contains three cuspidal representations: $\tau_2^1[\zeta]$, $\tau_2^2[\zeta]$, $\tau_2^3[\zeta]$. Let $\zeta\in\{\zeta_3,\zeta_3^2\}$. Let $i\in\{1,2,3\}$, and let $\pi_2^i(\zeta)$ denote the representation $\ii_{\bbG_{x_2}}^G(\tau_2^i[\zeta])$. 

There are six (depth-zero) ramified cubic characters $\eta_3^{(1)\pm}$, $\eta_3^{(2)\pm}$ $\eta_3^{(3)\pm}$ of $F^\times$, corresponding to three ramified cubic extensions $E_3^1$, $E_3^2$, $E_3^3$ over $F$. For each $i\in\{1,2,3\}$, let $\sigma^i_3:=\nu_F^{\pm 1}\eta_3^i\otimes\eta_3^i$ (a character of $T\simeq F^\times\times F^\times$), and let $\varphi(\eta_3^i)\colon W_F\times\SL_2(\C)\to G^\vee$ be the $L$-parameter for $G$ defined by 
\begin{equation} \label{eqn:Lpar_2_ii}
    \varphi(\eta_3^i):=\varphi_{\sigma_3^i,u},
\end{equation} 
using the formula \eqref{defn:phisu}, where the $G^\vee$-conjugacy class of $u$ is $\rG_2(a_1)$ (the subregular unipotent class in $\rG_2(\C)$). The restriction of $\varphi(\eta_3^i)$ to $W_F$ factors through $\Gal(E_3^i/F)$.
We define 
\begin{equation}
\varphi_{\pi_2^i(\zeta_3)}=\varphi_{\pi_2^i(\zeta_3^2)}:=\varphi(\eta_3^i).
\end{equation}
We have $\cG_{\varphi}\simeq\SL_3(\C)$, the unipotent element $u$ is regular in $\cG_{\varphi}$ (see Table~\ref{tab:table-unipotent-Gsigma=SL3}), and $S_\varphi\simeq\mu_3$. From Table~\ref{table:unip-ppal-series3c-ramified}, we have $\varphi(\eta_3^i)=\varphi_{\pi(\eta_3^i)}$, where $\pi(\eta_3^i)$ is a discrete series representation in $\Irr^\fs(G)$ for $\fs=[T,\sigma^i_3]_G$.
    
Thus we obtain three $L$-packets of size $3$, for each $i=1,2,3$,
\begin{equation}
    \Pi_{\varphi_2^i}(G):=\{\pi_2^i(\zeta_3),\pi_2^i(\zeta_3^2),\pi(\eta_3^i)\}.
\end{equation}
Each $L$-packet $\Pi_{\varphi(\eta_3^i)}(G)$ contains the two depth-zero \textit{singular} supercuspidal representations $\pi_2^i(\zeta)$ for $\zeta\in\{\zeta_3,\zeta_3^2\}$, and a depth-zero discrete series $\pi(\eta_3^i)$ in the principal series such that $\mcJ^\fs=\SL_3(\C)$. Here $\pi(\eta_3^i)$ is given precisely in Table~\ref{table:unip-ppal-series3c-ramified} with $\eta_3^i$ ramified of depth-zero. Property~\ref{property: Shahidi-fdeg} is verified in this case by comparing \eqref{fdeg-pi-eta3} and \eqref{eqn:fdegx1}. 

$\bullet$ $x=x_3$: From~\S\ref{d0s}\eqref{Gx=SO4}, we have one rational Lusztig series $\cE(\bbG_{x_3},s_3[\zeta_2])$ for the primitive square root of unity $\zeta_2$, where  $s_3[\zeta_2]=\mathrm{diag}\left(1,\zeta_2,\zeta_2,1\right)\mod \F_q^{\times}$. The Lusztig series $\cE(\bbG_{x_3},s_3[\zeta_2])$ contains two cuspidal representations: $\tau_3^1[\zeta_2]$ and $\tau_3^2[\zeta_2]$. Let $i\in\{1,2\}$, and let $\pi_3^i(\zeta_2)$ denote the representation $\ii_{\bbG_{x_3}}^G(\tau_3^i[\zeta_2])$. 

There are two (depth-zero) ramified quadratic characters $\eta_2^1$, $\eta_2^2$ of $F^\times$, corresponding to two ramified cubic extensions $E^1_2$, $E^2_2$ over $F$. For each $i\in\{1,2\}$, let $\sigma^i_2:=\nu_F\otimes \eta_2^i$ (a character of the torus $T\simeq F^\times\times F^\times$), and let $\varphi(\eta_2^i)\colon W_F\times\SL_2(\C)\to G^\vee$ be the $L$-parameter for $G$ defined by 
\begin{equation} \label{eqn:Lpar_2_i}
    \varphi(\eta_2^i):=\varphi_{\sigma_2^i,u},
\end{equation} 
using the formula \eqref{defn:phisu}, where the $G^\vee$-conjugacy class of $u$ is $\rG_2(a_1)$ (the subregular unipotent class in $\rG_2(\C)$). The restriction of $\varphi(\eta_2^i)$ to $W_F$ factors through $\Gal(E_2^i/F)$.
We define 
\begin{equation}
\varphi_{\pi_2^i(\zeta_2)}:=\varphi(\eta_2^i).
\end{equation}
We have $\cG_{\varphi}\simeq\SO_4(\C)$, the unipotent element $u$ is regular in $\cG_{\varphi}$ (see Table~\ref{table:SO4}), and $S_\varphi\simeq\mu_2$. From Table~\ref{table:unip-ppal-series-case-3aquad-ramified}, we have $\varphi(\eta_2^i)=\varphi_{\pi(\eta_2^i)}$, where $\pi(\eta_2^i)$ is a discrete series representation in $\Irr^\fs(G)$ for $\fs=[T,\sigma^i_2]_G$.

Thus we obtain two $L$-packets of size $2$, for each $i=1,2$,
\begin{equation}
    \Pi_{\varphi(\eta_2^i)}(G):=\{\pi_3^i(\zeta_2),\pi(\eta_2^i)\}.
\end{equation}
Each $L$-packet $\Pi_{\varphi(\eta_2^i)}(G)$ contains one depth-zero \textit{singular} supercuspidal representations $\pi_3^i(\zeta_2)$ and a depth-zero discrete series $\pi(\eta_2^i)$ in the principal series such that $\mcJ^\fs=\SO_4(\C)$. Here $\pi(\eta_2^i)$ is given precisely in Table~\ref{table:unip-ppal-series-case-3aquad-ramified} with $\eta_2^{i}$ ramified of depth-zero. Property~\ref{property: Shahidi-fdeg} is verified in this case by comparing \eqref{fdeg-pi-eta2} and \eqref{eqn:fdegx2}.

\item[(d)] Let $\pi$ be a positive-depth singular supercuspidal representation of $G$. As in \S\ref{d+s}(3b), such a singular supercuspidal representation necessarily arises from an $L$-parameter of $\SL_2$-type as in Definition~\ref{defn:SL2-type}, and lives in a mixed $L$-packet together with another non-supercuspidal representation described as
follows. 
In this case, $\rZ_{G^{\vee}}(\varphi)=\rZ_{\SL_2(\C)}(\varphi(W_F))=\mu_2$ by Lemma~\ref{solv}. Hence the $L$-packet $\Pi_\varphi(G)$ consists of a non-supercuspidal representation and a supercuspidal representation by Lemma~\ref{cusp}. Let $P$ be a parabolic subgroup of $G$ with Levi subgroup isomorphic to $\GL_2(F)$. Here the $\GL_2(F)$ corresponds to the short (resp.~long) root if $\varphi$ is of short (resp.~long) $\SL_2$-type. By Property~\ref{property:AMS-conjecture-7.8}, the non-supercuspidal member has the following supercuspidal support: $\varphi|_{W_F}\colon W_F\rightarrow \SL_2(\C)$, which gives us a supercuspidal $\tau$ of $\PGL_2(F)$. We consider $\tau_s:=\tau\otimes(\nu_F\circ\det)^s$ as a representation of $\GL_2(F)$. Then the non-supercuspidal of $G$ in this $L$-packet is the generic constituent $\pi(\tau_{\pm \frac{1}{2}})$ in $\ii_{P}^{G}\tau_{\pm\frac{1}{2}}$ (both $s=\pm\frac{1}{2}$ work), as detailed in Tables~\ref{tab:LLC} and \ref{num: M long 0}. Our $L$-packets are given by
\begin{equation}
    \Pi_{\varphi}(G)=\{\pi^{\mathrm{sing}},\pi(\tau_{\pm\frac{1}{2}})\}
\end{equation}
\end{enumerate}
Let $G$ be the group of $F$-rational points of the exceptional group $\rG_2$. We suppose that the residual characteristic of $F$ is different from $2$ and $3$. 
\begin{theorem}\label{main-thm} 
The local Langlands correspondence defined in \eqref{eqn:LLC} satisfies Properties~{\rm \ref{property:L-packets},
\ref{property:size of L-packets}, \ref{property:Langlands quotient},  \ref{property:AMS-conjecture-7.8}, \ref{property:generic tempered L-packets}} and {\rm \ref{property:Levi-inclusion}}, and satisfies Property~{\rm\ref{property: Shahidi-fdeg}} for depth-zero $L$-packets\footnote{we certainly expect this property to hold for positive-depth $L$-packets as well.}.

Moreover, Properties~{\rm \ref{property:L-packets},
\ref{property:size of L-packets}, \ref{property:Langlands quotient},  \ref{property:AMS-conjecture-7.8}},  \ref{property:Levi-inclusion} and \ref{property:atomic-stability} uniquely determine the bijection~\eqref{eqn:LLC}. 
\end{theorem}
\begin{proof}
By Property~\ref{property:Langlands quotient}, the $L$-parameter $\varphi_\pi$ of each irreducible non-tempered representation $\pi$ of $G$ is uniquely determined by \eqref{eqn:PhiJ}. Since the $L$-packets of the representations of the proper Levi subgroups of $G$ are all singletons, the $L$-packet $\Pi_{\varphi_\pi}(G)$ is a singleton (see also Theorem~\ref{constituents-part-one}). Hence, by Property~\ref{property:size of L-packets}, we have $\rho_\pi=1$. Thus the map 
\eqref{eqn:LLC} is uniquely characterized for non-tempered representations. 
This finishes the case of non-discrete series tempered representations. 

Property~\ref{property:generic tempered L-packets} holds for supercuspidal $L$-packets by Lemma~\ref{sc-packet-generic-lemma}. For the mixed $L$-packets, this can be seen directly from paragraph \eqref{summary-section-d0-mixed-packets} and the tables \textit{loc.cit.}, where we specify which member in a given $L$-packet is generic. 
By \cite{G2-stability}, the $L$-packets are uniquely pinned down by stability. 
Since we have already treated the discrete series in \eqref{summary-section-d0-mixed-packets}, we are done. 
\end{proof}

\subsection{Summary of $L$-packets}
\subsubsection{Unipotent \texorpdfstring{$L$}{L}-packets} \label{subsec:udL-packets}
For unipotent irreducible representations, a Local Langlands Correspondence was constructed by Lusztig in \cite{Lu-padicI} for any simple $p$-adic reductive group $\bG$ of adjoint type. In the special case where $G=\rG_2(F)$, when the representation is also supercuspidal, it follows from \cite{Morris-ENS}.
Moreover, a set of enhanced $L$-parameters 
\begin{equation}
\{(\varphi[\zeta],\rho[\zeta])\,:\, \zeta\in \{1,-1,\zeta_3,\zeta_3^2\}\}
\end{equation}
 was described by Morris in \cite[\S~5.6]{Morris-ENS}. Writing $s'[\zeta]:=s_{\varphi[\zeta]}$ and $u[\zeta]:=u_{\varphi[\zeta]}$, we have 
\begin{enumerate}
\item[(1)]\label{unipotent-case1} $\rZ_{G^\vee}(s'[1])\simeq \rG_2(\C)$, $u[1]\in \rG_2(a_1)$, $S_{\varphi[1]}\simeq S_3$, and $\rho[1]$ is the sign representation of $S_3$ (i.e., it sends every element to its sign: for example, $(12)\mapsto -1$ and $(132)\mapsto 1$);
\item[(2)]\label{unipotent-case2} $\rZ_{G^\vee}(s'[-1])\simeq \SO_4(\C)$, $u[-1]\in \rG_2(a_1)$, $S_{\varphi[-1]}\simeq \ZZ/2\ZZ$, and $\rho[-1]$ is the non-trivial representation of $\ZZ/2\ZZ$;
\item[(3)]\label{unipotent-case3} $\rZ_{G^\vee}(s'[\zeta])\simeq \SL_3(\C)$, $u[\zeta]\in \rG_2(a_1)$, $S_{\varphi[\zeta]}\simeq \ZZ/3\ZZ$, and $\rho[\zeta]$ is a non-trivial representation of $\ZZ/3\ZZ$ for $\zeta\in\{\zeta_3,\zeta_3^2\}$.
\end{enumerate}
As remarked in \cite[\S5.6]{Morris-ENS}, the unipotent class is always the subregular ones $\rG_2(a_1)$. However, since $s'[1]\neq s'[-1]$,  $s'[1]\ne s'[\zeta_3]$, and $s'[-1]\ne s'[\zeta_3]$, cases (1), (2) and (3) 
give different $L$-packets. We describe unipotent packets in these sections:
\begin{itemize}
    \item $\mathsection$\ref{d0s}(\ref{Gx=G2})(\ref{d0s-Gx=G2=unipotent}),
    \item $\mathsection$\ref{d0p}(\ref{d0-unipotent}).
\end{itemize} 
The unipotent discrete series of $\rG_2(F)$ belong to the following $L$-packets (see also \cite[p.~482]{Reeder-Iwahori-spherical-discrete-series} and \cite[Table~2.7.1]{CFZ}):
\begin{enumerate}
\item[(1)] $\Pi_{\varphi_1}(G)=\{\St_{\rG_2}\}$,
\item[(2)] $\Pi_{\varphi_2}(G)=\{\pi(1)',\pi(1),\pi[1]\}$,
\item[(3)] $\Pi_{\varphi_3}(G)=\{\pi(\eta_2),\pi[-1]\}$, with $\eta_2$ unramified,
\item[(4)] $\Pi_{\varphi_4}(G)=\{\pi(\eta_3),\pi[\zeta_3],\pi[\zeta_3^2]\}$, with $\eta_3$ unramified,
\end{enumerate}
where $\St_{\rG_2}$ is the Steinberg representation of $G$, and the representations $\St_{\rG_2}$, $\pi(1)'$, $\pi_(1)$, $\pi(\eta_2)$ and $\pi(\eta_3)$ are in the principal series of $G$ (these representations are Iwahori-spherical, i.e.~they have non-zero Iwahori invariant vectors).~Their enhanced $L$-parameters are also computed in \cite{Reeder-Iwahori-spherical-discrete-series}, using the Kazhdan-Lusztig parametrization established in \cite{Kazhdan-Lusztig}. 
Note that $\{\mathrm{St}_{G_2}\}$ from Table~\ref{table:unip-ppal-series-case-3bbis} is a singleton $L$-packet, and coincides precisely with the $L$-packet labeled as $\widetilde{\mathcal{E}}_2(\tau_1)$ in \cite[p.480]{Reeder-Iwahori-spherical-discrete-series}. 
The enhanced $L$-parameters for the representations $\St_{\rG_2}$, $\pi(1)'$, $\pi_(1)$, $\pi(\eta_2)$ and $\pi(\eta_3)$ belong to the same series $\Phi_\enh^{\fs^\vee}(G)$, where 
$\fs^\vee=[T^\vee,(\varphi_0,1)]_{G^\vee}$. Here $\varphi_0\colon W_F/I_F\to T^\vee$, and $1$ is the trivial representation of $S_{\varphi_0}=\{1\}$. 

\subsubsection{Explicit Supercuspidal $L$-packets}\
\begin{enumerate}
    \item \textit{Depth-zero supercuspidal $L$-packets}
\begin{itemize}
    \item \ref{d0p}(\ref{regular-sc-depthzero}): regular supercuspidal $L$-packets of sizes 4,3,1 
    \item \ref{d0s}(\ref{Gx=G2})(\ref{regular-d0s}): regular supercuspidal $L$-packets
    \item \ref{d0s}(\ref{Gx=SL3})(\ref{Gx=SL3-regular}): regular supercuspidal $L$-packets
    \item \ref{d0s}(\ref{Gx=SO4})(\ref{Gx=SO4-regular}): regular supercuspidal $L$-packets
\end{itemize}

\item \textit{Positive-depth supercuspidal $L$-packets}
\begin{itemize}
    \item \ref{d+p}(\ref{posdep-GL1}),  \ref{d+p}(\ref{posdep-2b}), \ref{d+p}(\ref{posdep-2c}): regular supercuspidal $L$-packets of size $1,2,3,4$.
    \item \ref{d+s}(\ref{posdep-most-regular}): regular supercuspidal $L$-packets
    \item \ref{d+s}(\ref{posdep-PGL2-reg}): regular supercuspidal $L$-packets
    \item \ref{d+p}(\ref{ns}) $\longleftrightarrow$ \ref{d+s}(\ref{posdep-PGL2-nonsing}): non-singular (non-regular) positive-depth supercuspidal $L$-packet of size 4. 
    \item \ref{d+s}(\ref{inner-form-PGL2-reg}): regular supercuspidal $L$-packets
\end{itemize}
\end{enumerate}

\subsubsection{Non-unipotent non-supercuspidal depth-zero packets}
We only list the non-singleton packets here. For a complete description of the singleton packets, see $\mathsection$\ref{sec:Galois-G2}. 
\begin{itemize}
\item \ref{d0s}(\ref{u3}) $\longleftrightarrow$ \ref{d0p}(\ref{2b}): In this case, we obtain one $L$-packet, consisting of one singular supercuspidal representation coming from reductive quotient $\bbG_{x_0}\simeq G_2(\Fq)$ and $\Cent_{\bbbG_{x_0}^{\vee}}(s)\simeq \mathrm{SU}_3(\Fq)$, and an intermediate series representation $\pi(\sigma)$ whose cuspidal support lives in $\mathrm{GL}_2^{\mathrm{l.r.}}$ and corresponds to a simple module of a Hecke algebra $\mathcal{H}^{\mathfrak{s}}(G)$ with unequal parameters $\{q^3,q\}$. In this case the formal degree for the $L$-packet should be $\frac{q^2}{q^3+1}$. This case only occurs when $q\equiv -1\mod 3$. 
\item \ref{d0s}(\ref{chals}) $\longleftrightarrow$ \ref{d0p}(\ref{chal}): In this case, we obtain three $L$-packets, each $L$-packet consisting of two singular supercuspidal representations coming from reductive quotient $\bbG_{x_1}\simeq \SL_3(\Fq)$ and $\rZ_{\bbbG_{x_1}^{\vee}}(s)=\bbbT^{\vee}\rtimes\mu_3$, and a generic principal series representation $\pi(\eta_3)$ from Table \ref{table:unip-ppal-series3c-ramified}, where $\eta_3$ is a ramified cubic character. The unipotent class attached to these three $L$-packets is the subregular unipotent class $\rG_2(a_1)$, which comes from the regular unipotent of $\SL_3(\C)$. 
\item \ref{d0s}(\ref{toys}) $\longleftrightarrow$ \ref{d0p}(\ref{toy}): In this case, we obtain two $L$-packets, each $L$-packet consisting of one singular supercuspidal representation coming from reductive quotient $\bbG_{x_2}\simeq \SO_4(\Fq)$ and $\Cent_{\bbbG_{x_2}^{\vee}}(s)\simeq \rS(\rO_2\times \rO_2)(\Fq)$ (here we take the non-split form of $\rO_2$), and a generic principal series representation $\pi(\eta_2)$ from Table \ref{table:unip-ppal-series-case-3aquad-ramified}, where $\eta_2$ is a ramified quadratic character.  
\end{itemize}

\subsubsection{Non-unipotent non-supercuspidal (i.e.~singular) positive-depth packets}
\begin{itemize}
    \item \ref{d+p}(\ref{posdep-2a}) $\longleftrightarrow$  \ref{d+s}(\ref{inner-form-PGL2-singular}): size two $L$-packet as described in Paragraph \ref{paragraph-after-SL2-type}, mixing an intermediate series representation with a singular supercuspidal. 
\end{itemize}

\appendix

\bibliographystyle{amsalpha}
\bibliography{bibfile}

\providecommand{\bysame}{\leavevmode\hbox to3em{\hrulefill}\thinspace}
\providecommand{\MR}{\relax\ifhmode\unskip\space\fi MR }
\providecommand{\MRhref}[2]{%
  \href{http://www.ams.org/mathscinet-getitem?mr=#1}{#2}
}
\providecommand{\href}[2]{#2}
\begin{thebibliography}{ABPS17b}

\bibitem[ABP11]{ABP-G2}
Anne-Marie Aubert, Paul Baum, and Roger Plymen, \emph{Geometric structure in
  the principal series of the {$p$}-adic group {${\rm G}_2$}}, Represent.
  Theory \textbf{15} (2011), 126--169. \MR{2772586}

\bibitem[ABPS16a]{ABPS-KTheory}
Anne-Marie Aubert, Paul Baum, Roger Plymen, and Maarten Solleveld,
  \emph{Geometric structure for the principal series of a split reductive
  {$p$}-adic group with connected centre}, J. Noncommut. Geom. \textbf{10}
  (2016), no.~2, 663--680. \MR{3519048}

\bibitem[ABPS16b]{ABPS-SL}
\bysame, \emph{The local {L}anglands correspondence for inner forms of
  {$\SL_n$}}, Res. Math. Sci. \textbf{3} (2016), 2--34. \MR{3579297}

\bibitem[ABPS17a]{ABPS-CM}
\bysame, \emph{Conjectures about {$p$}-adic groups and their noncommutative
  geometry}, Around {L}anglands correspondences, Contemp. Math., vol. 691,
  Amer. Math. Soc., Providence, RI, 2017, pp.~15--51. \MR{3666049}

\bibitem[ABPS17b]{ABPS-LMS}
\bysame, \emph{The principal series of {$p$}-adic groups with disconnected
  center}, Proc. Lond. Math. Soc. (3) \textbf{114} (2017), no.~5, 798--854.
  \MR{3653247}

\bibitem[AMS18]{AMS18}
Anne-Marie Aubert, Ahmed Moussaoui, and Maarten Solleveld,
  \emph{Generalizations of the {S}pringer correspondence and cuspidal
  {L}anglands parameters}, Manuscripta Math. \textbf{157} (2018), no.~1-2,
  121--192. \MR{3845761}

\bibitem[AMS21]{AMS2}
\bysame, \emph{Affine {H}ecke algebras for {L}anglands parameters},
  arXiv:1701.03593 (2021), 70 pp.

\bibitem[Art06]{Arthur-Note}
James Arthur, \emph{A note on {$L$}-packets}, Pure Appl. Math. Q. \textbf{2}
  (2006), no.~1, Special Issue: In honor of John H. Coates. Part 1, 199--217.
  \MR{2217572}

\bibitem[Art13]{Arthur-classification}
\bysame, \emph{The endoscopic classification of representations}, American
  Mathematical Society Colloquium Publications, vol.~61, American Mathematical
  Society, Providence, RI, 2013, Orthogonal and symplectic groups. \MR{3135650}

\bibitem[Aub19]{Aubert-Pune}
Anne-Marie Aubert, \emph{Local {L}anglands and {S}pringer correspondences},
  Representations of reductive $p$-adic groups, Progr. Math.,
  Birkh\"auser/Springer, Singapore, 2019, pp.~1--37. \MR{3889758}

\bibitem[AX22]{aubert-xu-Hecke-algebra}
Anne-Marie Aubert and Yujie Xu, \emph{Hecke algebras for $p$-adic reductive
  groups and local langlands correspondence for {B}ernstein blocks},
  arXiv:2202.01305 (2022), 39pp.

\bibitem[BBD82]{BBD}
Alexander Be\u{\i}linson, Joseph Bernstein, and Pierre Deligne, \emph{Faisceaux
  pervers}, Analysis and topology on singular spaces, {I} ({L}uminy, 1981),
  Ast\'{e}risque, vol. 100, Soc. Math. France, Paris, 1982, pp.~5--171.
  \MR{751966}

\bibitem[BDK86]{BDK}
Joseph Bernstein, Pierre Deligne, and David Kazhdan, \emph{Trace
  {P}aley-{W}iener theorem for reductive {$p$}-adic groups}, J. Analyse Math.
  \textbf{47} (1986), 180--192. \MR{874050}

\bibitem[Ber84]{Bernstein-centre}
Joseph Bernstein, \emph{Le ``centre'' de {B}ernstein}, Representations of
  reductive groups over a local field, Travaux en Cours, Hermann, Paris, 1984,
  Edited by P. Deligne, pp.~1--32. \MR{771671}

\bibitem[BH06a]{Bushnell-Henniart}
Colin~J. Bushnell and Guy Henniart, \emph{The local {L}anglands conjecture for
  {$\rm GL(2)$}}, Grundlehren der mathematischen Wissenschaften [Fundamental
  Principles of Mathematical Sciences], vol. 335, Springer-Verlag, Berlin,
  2006. \MR{2234120}

\bibitem[BH06b]{Bushnell-Henniart-GL2}
\bysame, \emph{The local {L}anglands conjecture for {$\rm GL(2)$}}, Grundlehren
  der mathematischen Wissenschaften [Fundamental Principles of Mathematical
  Sciences], vol. 335, Springer-Verlag, Berlin, 2006. \MR{2234120}

\bibitem[BL94]{Bernstein-Lunts}
Joseph Bernstein and Valery Lunts, \emph{Equivariant sheaves and functors},
  Lecture Notes in Mathematics, vol. 1578, Springer-Verlag, Berlin, 1994.
  \MR{1299527}

\bibitem[BLR90]{Bosch-Lutkebohmer-Raynaud}
Siegfried Bosch, Werner L\"utkebohmer, and Michel Raynaud, \emph{N\'eron
  model}, Ergebnisse der Mathematik und ihrer Grenzgebiete, Springer-Verlag,
  Berlin, 1990. \MR{1045822}

\bibitem[BM97]{Barbasch-Moy-LCE}
Dan Barbasch and Allen Moy, \emph{Local character expansions}, Ann. Sci.
  \'{E}cole Norm. Sup. (4) \textbf{30} (1997), no.~5, 553--567. \MR{1474804}

\bibitem[Bon05]{Bonnafe}
C\'{e}dric Bonnaf\'{e}, \emph{Quasi-isolated elements in reductive groups},
  Comm. Algebra \textbf{33} (2005), no.~7, 2315--2337. \MR{2153225}

\bibitem[Bor79]{Borel-Corvallis}
Armand Borel, \emph{Automorphic {$L$}-functions}, Automorphic forms,
  representations and {$L$}-functions ({P}roc. {S}ympos. {P}ure {M}ath.,
  {O}regon {S}tate {U}niv., {C}orvallis, {O}re., 1977), {P}art 2, Proc. Sympos.
  Pure Math., XXXIII, Amer. Math. Soc., Providence, R.I., 1979, pp.~27--61.
  \MR{546608}

\bibitem[Bou68]{Bou}
Nicolas Bourbaki, \emph{\'{E}l\'{e}ments de math\'{e}matique. {F}asc. {XXXIV}.
  {G}roupes et alg\`ebres de {L}ie. {C}hapitre {IV}: {G}roupes de {C}oxeter et
  syst\`emes de {T}its. {C}hapitre {V}: {G}roupes engendr\'{e}s par des
  r\'{e}flexions. {C}hapitre {VI}: syst\`emes de racines}, Hermann, Paris,
  1968. \MR{0240238}

\bibitem[BT72]{Bruhat-Tits-I}
{Fran\c cois} Bruhat and Jacques Tits, \emph{Groupes r\'eductifs sur un corps
  local, {I}: {D}onn\'ees radicielles valu\'ees}, Publ. Math. I.H.E.S.
  \textbf{41} (1972), 1--251.

\bibitem[BT84]{Bruhat-Tits-II}
\bysame, \emph{Groupes r\'eductifs sur un corps local, {II}: {D}onn\'ees
  radicielles valu\'ees}, Publ. Math. I.H.E.S. \textbf{60} (1984), 197--376.
  \MR{86c:20042}

\bibitem[BZ77]{Bernstein-Zelevinsky}
Joseph Bernstein and Andrei Zelevinsky, \emph{Induced representations of
  reductive {$p$}-adic groups. {I}}, Ann. Sci. \'{E}cole Norm. Sup. (4)
  \textbf{10} (1977), no.~4, 441--472. \MR{579172}

\bibitem[Car93]{Carter-book}
Roger~W. Carter, \emph{Finite groups of {L}ie type}, Wiley Classics Library,
  John Wiley \& Sons, Ltd., Chichester, 1993, Conjugacy classes and complex
  characters, Reprint of the 1985 original, A Wiley-Interscience Publication.
  \MR{1266626}

\bibitem[CFZ21]{CFZ}
Clifton Cunningham, Andrew Fiori, and Qing Zhang, \emph{Toward the endoscopic
  classification of unipotent representations of $p$-adic {$G_2$}},
  arXiv:2101.0457 (2021), 53pp.

\bibitem[CKK12]{Ciubotaru-Kato-Kato}
Dan Ciubotaru, Midori Kato, and Syu Kato, \emph{On characters and formal
  degrees of discrete series of affine {H}ecke algebras of classical types},
  Invent. Math. \textbf{187} (2012), no.~3, 589--635. \MR{2891878}

\bibitem[CM84]{Collingwood-McGovern}
David~H. Collingwood and William~M. McGovern, \emph{Nilpotent {O}rbits {I}n
  {S}emisimple {L}ie {A}lgebra: {A}n {I}ntroduction}, Nilpotent {O}rbits {I}n
  {S}emisimple {L}ie {A}lgebra: {A}n {I}ntroduction, Travaux en Cours, Hermann,
  Paris, 1984, Edited by P. Deligne, pp.~1--32. \MR{771671}

\bibitem[CR74]{Chang-Ree}
Bomshik Chang and Rimhak Ree, \emph{The characters of {$\rG_2(q)$}}, 1974.

\bibitem[DeB06]{DeBacker-parameterizing-conjugacy-classes}
Stephen DeBacker, \emph{Parameterizing conjugacy classes of maximal unramified
  tori via {B}ruhat-{T}its theory}, Michigan Math. J. \textbf{54} (2006),
  no.~1, 157--178. \MR{2214792}

\bibitem[DL76]{Deligne-Lusztig}
Pierre Deligne and George Lusztig, \emph{Representations of reductive groups
  over finite fields}, Ann. of Math. \textbf{103} (1976), no.~1, 103--161.

\bibitem[DLM92]{Digne-Lehrer-Michel}
F.~Digne, G.~I. Lehrer, and J.~Michel, \emph{The characters of the group of
  rational points of a reductive group with nonconnected centre}, J. Reine
  Angew. Math. \textbf{425} (1992), 155--192. \MR{1151318}

\bibitem[DM91]{Digne-Michel}
Fran\c{c}ois Digne and Jean Michel, \emph{Representations of finite groups of
  {L}ie type}, London Mathematical Society Student Texts, vol.~21, Cambridge
  University Press, Cambridge, 1991. \MR{1118841}

\bibitem[DR09]{DeBacker-Reeder}
Stephen DeBacker and Mark Reeder, \emph{Depth-zero supercuspidal {$L$}-packets
  and their stability}, Ann. of Math. (2) \textbf{169} (2009), no.~3, 795--901.
  \MR{2480618}

\bibitem[DS00]{Debacker-Sally-germs}
Stephen DeBacker and Paul~J. Sally, Jr., \emph{Germs, characters, and the
  {F}ourier transforms of nilpotent orbits}, The mathematical legacy of
  {H}arish-{C}handra ({B}altimore, {MD}, 1998), Proc. Sympos. Pure Math.,
  vol.~68, Amer. Math. Soc., Providence, RI, 2000, pp.~191--221. \MR{1767897}

\bibitem[Fin21]{Fintzen}
Jessica Fintzen, \emph{Types for tame {$p$}-adic groups}, Ann. of Math. (2)
  \textbf{193} (2021), no.~1, 303--346. \MR{4199732}

\bibitem[FKS19]{Fintzen-Kaletha-Spice}
Jessica Fintzen, Tasho Kaletha, and Loren Spice, \emph{A twisted {Y}u
  construction, {H}arish-{C}handra characters, and endoscopy}, Duke Math. J.,
  to appear (arXiv:1912.03286) (2019), 45 pp.

\bibitem[FOS20]{FOS}
Yongqi Feng, Eric Opdam, and Maarten Solleveld, \emph{Supercuspidal unipotent
  representations: {L}-packets and formal degrees}, J. \'{E}c. polytech. Math.
  \textbf{7} (2020), 1133--1193. \MR{4167790}

\bibitem[FOS21]{FOS-Jussieu}
\bysame, \emph{On formal degrees of unipotent representations}, Journal of the
  Institute of Mathematics of Jussieu (2021), 1–53.

\bibitem[FS21]{Fargues-Scholze}
Laurent Fargues and Peter Scholze, \emph{Geometrization of the local
  {L}anglands correspondence}, arXiv preprint arXiv:2102.13459 (2021), 350.

\bibitem[GL17]{Genestier-Lafforgue}
Alain Genestier and Vincent Lafforgue, \emph{Chtoucas restreints pour les
  groupes r\'eductifs et param\'etrisation de {L}anglands locale},
  arXiv:1709.00978 (2017), 42.

\bibitem[GM83]{Goresky-MacPherson}
Mark Goresky and Robert MacPherson, \emph{Intersection homology. {II}}, Invent.
  Math. \textbf{72} (1983), no.~1, 77--129. \MR{696691}

\bibitem[GR10]{Gross-Reeder}
Benedict Gross and Mark Reeder, \emph{Arithmetic invariants of discrete
  langlands parameters}, Duke Math. J. \textbf{154} (2010), no.~3, 431--508.
  \MR{2730575}

\bibitem[Hai14]{Haines-Bst}
Thomas~J. Haines, \emph{The stable {B}ernstein center and test functions for
  {S}himura varieties}, Automorphic forms and {G}alois representations. {V}ol.
  2, London Math. Soc. Lecture Note Ser., vol. 415, Cambridge Univ. Press,
  Cambridge, 2014, pp.~118--186. \MR{3444233}

\bibitem[HC99]{harish-chandra-local-character}
Harish-Chandra, \emph{Admissible invariant distributions on reductive
  {$p$}-adic groups}, University Lecture Series, vol.~16, American Mathematical
  Society, Providence, RI, 1999, With a preface and notes by Stephen DeBacker
  and Paul J. Sally, Jr. \MR{1702257}

\bibitem[Hei11]{Heiermann-intertwining-operators-Hecke-algebras}
Volker Heiermann, \emph{Op\'{e}rateurs d'entrelacement et alg\`ebres de {H}ecke
  avec param\`etres d'un groupe r\'{e}ductif {$p$}-adique: le cas des groupes
  classiques}, Selecta Math. (N.S.) \textbf{17} (2011), no.~3, 713--756.
  \MR{2827179}

\bibitem[Hen00]{Henniart-LLC-GLn}
Guy Henniart, \emph{Une preuve simple des conjectures de {L}anglands pour
  {$\GL(n)$} sur un corps {$p$}-adique}, Invent. Math. \textbf{139} (2000),
  no.~2, 439--455. \MR{1738446}

\bibitem[HII08]{Hiraga-Ichino-Ikeda}
Kaoru Hiraga, Atsushi Ichino, and Tamotsu Ikeda, \emph{Formal degrees and
  adjoint {$\gamma$}-factors}, J. Amer. Math. Soc. \textbf{21} (2008), no.~1,
  283--304. \MR{2350057}

\bibitem[HKT19]{Harris-Khare-Thorne}
Michael Harris, Chandrashekhar~B. Khare, and Jack~A. Thorne, \emph{A local
  {L}anglands parameterization for generic supercuspidal representations of
  $p$-adic {$\rG_2$}}, arXiv:1909.05933 (2019), 32pp.

\bibitem[HS12]{Hiraga-Saito}
Kaoru Hiraga and Hiroshi Saito, \emph{On {$L$}-packets for inner forms of
  {$\SL_n$}}, Mem. Amer. Math. Soc. \textbf{215} (2012), no.~1013, vi+97.
  \MR{2918491}

\bibitem[HT01]{Harris-Taylor}
Michael Harris and Richard Taylor, \emph{The geometry and cohomology of some
  simple {S}himura varieties}, Annals of Mathematics Studies, vol. 151,
  Princeton University Press, Princeton, NJ, 2001, With an appendix by Vladimir
  G. Berkovich. \MR{1876802}

\bibitem[Kal16]{Kaletha-LLC}
Tasho Kaletha, \emph{The local {L}anglands conjectures for non-quasi-split
  groups}, Families of automorphic forms and the trace formula, Simons Symp.,
  Springer, [Cham], 2016, pp.~217--257. \MR{3675168}

\bibitem[Kal19]{Kal-reg}
\bysame, \emph{Regular supercuspidal representations}, J. Amer. Math. Soc.
  \textbf{32} (2019), no.~4, 1071--1170. \MR{4013740}

\bibitem[Kal21]{Kaletha-nonsingular}
\bysame, \emph{Supercuspidal {$L$}-packets}, arXiv:1912.03274 (2021), 83.

\bibitem[Key82]{Keys}
Charles~David Keys, \emph{On the decomposition of reducible principal series
  representations of {$p$}-adic {C}hevalley groups}, Pacific J. Math.
  \textbf{101} (1982), no.~2, 351--388. \MR{675406}

\bibitem[Kim07]{Kim-exhaustion}
Ju-Lee Kim, \emph{Supercuspidal representations: an exhaustion theorem}, J.
  Amer. Math. Soc. \textbf{20} (2007), no.~2, 273--320. \MR{2276772}

\bibitem[KL87]{Kazhdan-Lusztig}
David Kazhdan and George Lusztig, \emph{Proof of the {D}eligne-{L}anglands
  conjecture for {H}ecke algebras}, Invent. math. \textbf{87} (1987), 153--215.
  \MR{0862716}

\bibitem[Kon03]{Konno}
Takuya Konno, \emph{A note on the {L}anglands classification and irreducibility
  of induced representations of {$p$}-adic groups}, Kyushu J. Math. \textbf{57}
  (2003), no.~2, 383--409. \MR{2050093}

\bibitem[Kos59]{Kostant}
Bertram Kostant, \emph{The principal three-dimensional subgroup and the {B}etti
  numbers of a complex simple {L}ie group}, Amer. J. Math. \textbf{81} (1959),
  973--1032. \MR{0114875}

\bibitem[KS23]{Kret-Shin}
Arno Kret and Sug~Woo Shin, \emph{Galois representations for general symplectic
  groups}, J. Eur. Math. Soc. (JEMS) \textbf{25} (2023), no.~1, 75--152.
  \MR{4556781}

\bibitem[KY17]{Kim-Yu-types}
Ju-Lee Kim and Jiu-Kang Yu, \emph{Construction of tame types}, Representation
  theory, number theory, and invariant theory, Progr. Math., vol. 323,
  Birkh\"{a}user/Springer, Cham, 2017, pp.~337--357. \MR{3753917}

\bibitem[Lus78]{Lusztig-Madison}
George Lusztig, \emph{Representations of finite {C}hevalley groups}, CBMS
  Regional Conference Series in Mathematics, American Mathematical Society,
  Providence, R.I., 1978, Expository lectures from the {CBMS R}egional
  Conference held at Madison, Wis., August 8--12, 1977. \MR{0518617}

\bibitem[Lus84a]{Lusztig-characters-Princeton-book}
\bysame, \emph{Characters of reductive groups over a finite field}, Annals of
  Mathematics Studies, vol. 107, Princeton University Press, Princeton, NJ,
  1984. \MR{742472}

\bibitem[Lus84b]{Lusztig-IC}
\bysame, \emph{Intersection cohomology complexes on a reductive group}, Invent.
  Math. \textbf{75} (1984), no.~2, 205--272. \MR{732546}

\bibitem[Lus88]{Lusztig-disconnected-center}
\bysame, \emph{On the representations of reductive groups with disconnected
  center}, Ast\'erisque (1988), 157--166.

\bibitem[Lus95]{Lu-padicI}
\bysame, \emph{Classification of unipotent representations of simple {$p$}-adic
  groups}, Internat. Math. Res. Notices (1995), no.~11, 517--589. \MR{1369407}

\bibitem[Moe11]{Moeglin}
Colette Moeglin, \emph{Multiplicit\'{e} $1$ dans les paquets d'{A}rthur aux
  places {$p$}-adiques}, On certain {$L$}-functions, Clay Math. Proc., vol.~13,
  Amer. Math. Soc., Providence, RI, 2011, pp.~333--374. \MR{2767522}

\bibitem[Mor96]{Morris-ENS}
Lawrence Morris, \emph{Tamely ramified supercuspidal representations}, Ann.
  Sci. \'Ecole Norm. Sup. \textbf{29} (1996), no.~5, 639--667. \MR{1399618}

\bibitem[Mor99]{Morris-Level-zero-types}
\bysame, \emph{Level zero {$G$}-types}, Compositio Math. \textbf{118} (1999),
  no.~3, 135--157.

\bibitem[Mou17]{Moussaoui-Bernstein-center}
Ahmed Moussaoui, \emph{Centre de {B}ernstein dual pour les groupes classiques},
  Represent. Theory \textbf{21} (2017), 172--246. \MR{3694312}

\bibitem[MP96a]{Moy-PrasadII}
Allen Moy and Gopal Prasad, \emph{Jacquet functors and unrefined minimal
  {$K$}-types}, Comment. Math. Helvetici \textbf{71} (1996), no.~3, 98--121.
  \MR{1371680}

\bibitem[MP96b]{Moy-Prasad-types}
\bysame, \emph{Jacquet functors and unrefined minimal {$K$}-types}, Comment.
  Math. Helv. \textbf{71} (1996), no.~1, 98--121. \MR{1371680}

\bibitem[Mui97]{Muic-G2}
Goran Mui\'{c}, \emph{The unitary dual of {$p$}-adic {$G_2$}}, Duke Math. J.
  \textbf{90} (1997), no.~3, 465--493. \MR{1480543}

\bibitem[Oha21]{Ohara-fdegr}
Kazuma Ohara, \emph{On the formal degree conjecture for non-singular
  supercuspidal representations}, 2021.

\bibitem[Opd16]{Opdam-Selecta}
Eric Opdam, \emph{Spectral transfer morphisms for unipotent affine {H}ecke
  algebras}, Selecta Math. (N.S.) \textbf{22} (2016), no.~4, 2143--2207.
  \MR{3573955}

\bibitem[OT21]{Oi-Tokimoto}
Masao Oi and Kazuki Tokimoto, \emph{Local {L}anglands correspondence for
  regular supercuspidal representations of {$\GL(n)$}}, Int. Math. Res. Not.
  IMRN (2021), no.~3, 2007--2073. \MR{4206603}

\bibitem[Ram03]{Ram}
Arun Ram, \emph{Representations of rank two affine {H}ecke algebras}, Advances
  in algebra and geometry ({H}yderabad, 2001), Hindustan Book Agency, New
  Delhi, 2003, pp.~57--91. \MR{1986143}

\bibitem[Ree94]{Reeder-Iwahori-spherical-discrete-series}
Mark Reeder, \emph{On the {I}wahori-spherical discrete series for {$p$}-adic
  {C}hevalley groups; formal degrees and {$L$}-packets}, Ann. Sci. \'{E}cole
  Norm. Sup. (4) \textbf{27} (1994), no.~4, 463--491. \MR{1290396}

\bibitem[Ree97]{Reeder-HA}
\bysame, \emph{{H}ecke algebras and harmonic analysis on $p$-adic groups},
  American Journal of Mathematics \textbf{119} (1997), 225--249.

\bibitem[Ree02]{Reeder-isogeny}
\bysame, \emph{Isogenies of {H}ecke algebras and a {L}anglands correspondence
  for ramified principal series representations}, Represent. Theory \textbf{6}
  (2002), 101--126. \MR{1915088}

\bibitem[Ree10]{Reeder-torsion}
\bysame, \emph{Torsion automorphisms of simple {L}ie algebras}, Enseign. Math.
  (2) \textbf{56} (2010), no.~1-2, 3--47. \MR{2674853}

\bibitem[Roc98]{Roche-principal-series}
Alan Roche, \emph{Types and {H}ecke algebras for principal series
  representations of split reductive {$p$}-adic groups}, Ann. Sci. \'{E}cole
  Norm. Sup. (4) \textbf{31} (1998), no.~3, 361--413. \MR{1621409}

\bibitem[Rod75]{Rodier-Whittaker-models}
Fran\c{c}ois Rodier, \emph{Mod\`eles de {W}hittaker des repr\'esentations
  admisssibles des groupes r\'eductifs $p$-adiques quasi-d\'eploy\'es}, 1975.

\bibitem[Rod81]{Rodier-principal-series}
\bysame, \emph{D\'{e}composition de la s\'{e}rie principale des groupes
  r\'{e}ductifs {$p$}-adiques}, Noncommutative harmonic analysis and {L}ie
  groups ({M}arseille, 1980), Lecture Notes in Math., vol. 880, Springer,
  Berlin-New York, 1981, pp.~408--424. \MR{644842}

\bibitem[Sch13]{Scholze-LLC}
Peter Scholze, \emph{The local {L}anglands correspondence for {$\GL_n$} over
  $p$-adic fields}, Invent. Math. \textbf{192} (2013), no.~3, 663--715.
  \MR{3049932}

\bibitem[Sch21]{Schwein}
David Schwein, \emph{Formal degree of regular supercuspidals}, 2021.

\bibitem[Sha89]{ShahidiThird}
Freydoon Shahidi, \emph{Third symmetric power {$L$}-functions for {${\rm
  GL}(2)$}}, Compositio Math. \textbf{70} (1989), no.~3, 245--273. \MR{1002045}

\bibitem[Sha90]{ShahidiAnnalsAproofof}
\bysame, \emph{A proof of {L}anglands' conjecture on {P}lancherel measures;
  complementary series for {$p$}-adic groups}, Ann. of Math. (2) \textbf{132}
  (1990), no.~2, 273--330. \MR{1070599}

\bibitem[Sha91]{ShahidiLanglandsconjecture}
\bysame, \emph{Langlands' conjecture on {P}lancherel measures for {$p$}-adic
  groups}, Harmonic analysis on reductive groups ({B}runswick, {ME}, 1989),
  Progr. Math., vol. 101, Birkh\"{a}user Boston, Boston, MA, 1991,
  pp.~277--295. \MR{1168488}

\bibitem[Sol18]{Solleveld-LLC-unipotent}
Maarten Solleveld, \emph{A local {L}anglands correspondence for unipotent
  representations}, 2018.

\bibitem[Sol20]{Solleveld-Lparameters}
\bysame, \emph{Langlands parameters, functoriality and {H}ecke algebras},
  Pacific J. Math. \textbf{304} (2020), no.~1, 209--302. \MR{4053201}

\bibitem[Sol22]{Solleveld-endomorphism-algebra}
\bysame, \emph{Endomorphism algebras and {H}ecke algebras for reductive
  $p$-adic groups}, Journal of Algebra (2022), no.~606, 371--470.

\bibitem[Spi18]{Spice-LCE}
Loren Spice, \emph{Explicit asymptotic expansions for tame supercuspidal
  characters}, Compos. Math. \textbf{154} (2018), no.~11, 2305--2378.
  \MR{3867302}

\bibitem[Ste68]{Steinberg-Memoir}
Robert Steinberg, \emph{Endomorphisms of linear algebraic groups}, Memoirs of
  the American Mathematical Society, No. 80, American Mathematical Society,
  Providence, R.I., 1968. \MR{0230728}

\bibitem[Stu71]{Stuhler}
Ulrich Stuhler, \emph{Unipotente und nilpotente {K}lassen in einfachen
  {G}ruppen und {L}iealgebren vom {T}yp {$G_{2}$}}, Indag. Math. \textbf{33}
  (1971), 365--378. \MR{0302723}

\bibitem[SX23a]{G2-stability}
Kenta Suzuki and Yujie Xu, \emph{The explicit {L}ocal {L}anglands
  correspondence for {$G_2$} {II}: character formulas and stability}, 2023.

\bibitem[SX23b]{LLC-GSp4}
\bysame, \emph{The explicit {L}ocal {L}anglands correspondence for {$\GSp_4$},
  {$\Sp_4$} and stability (with an application to modularity lifting)}, 2023.

\bibitem[SZ18]{Silberger-Zink}
Allan~J. Silberger and Ernst-Wilhelm Zink, \emph{Langlands classification for
  {$L$}-parameters}, J. Algebra \textbf{511} (2018), 299--357. \MR{3834776}

\bibitem[Tat79]{Tate-Corvallis}
John Tate, \emph{Number theoretic background}, Automorphic forms,
  representations and {$L$}-functions ({P}roc. {S}ympos. {P}ure {M}ath.,
  {O}regon {S}tate {U}niv., {C}orvallis, {O}re., 1977), {P}art 2, Proc. Sympos.
  Pure Math., XXXIII, Amer. Math. Soc., Providence, R.I., 1979, pp.~3--26.
  \MR{546607}

\bibitem[Vog93]{Vogan-LLC-1993}
David~A. Vogan, Jr., \emph{The local {L}anglands conjecture}, Representation
  theory of groups and algebras, Contemp. Math., vol. 145, Amer. Math. Soc.,
  Providence, RI, 1993, pp.~305--379. \MR{1216197}

\bibitem[Wal03]{Wal}
Jean-Loup Waldspurger, \emph{La formule de {P}lancherel pour les groupes
  {$p$}-adiques (d'apr\`es {H}arish-{C}handra)}, J. Inst. Math. Jussieu
  \textbf{2} (2003), no.~2, 235--333. \MR{1989693}

\bibitem[Yu01]{Yu}
Jiu-Kang Yu, \emph{Construction of tame supercuspidal representations}, J.
  Amer. Math. Soc. (2001), no.~3, 579--622. \MR{1824988}

\bibitem[Zel80]{Zelevinsky}
Andrei Zelevinsky, \emph{Induced representations of reductive {$p$}-adic
  groups. {II}. {O}n irreducible representations of {${\rm GL}(n)$}}, Ann. Sci.
  \'{E}cole Norm. Sup. (4) \textbf{13} (1980), no.~2, 165--210. \MR{584084}

\bibitem[Zhu20]{Zhu-LLC}
Xinwen Zhu, \emph{Coherent sheaves on the stack of langlands parameters}, 2020.

\end{thebibliography}

\end{document}